\documentclass{amsart}
\usepackage{amsmath}
\usepackage{graphicx}
\usepackage{amsfonts}
\usepackage{amssymb}
\providecommand{\U}[1]{\protect\rule{.1in}{.1in}}
\providecommand{\U}[1]{\protect\rule{.1in}{.1in}}
\makeatletter
\providecommand{\bigsqcap}{\mathop{\mathpalette\@updown\bigsqcup}}
\newcommand*{\@updown}
[2]{\rotatebox[origin=c]{180}{$\m@th#1#2$}}
\makeatother
\newtheorem{theorem}{Theorem}
\theoremstyle{plain}

\newtheorem{algorithm}{Algorithm}

\newtheorem{case}{Case}

\newtheorem{conjecture}{Conjecture}

\newtheorem{corollary}{Corollary}

\newtheorem{definition}{Definition}
\newtheorem{example}{Example}

\newtheorem{lemma}{Lemma}

\newtheorem{proposition}{Proposition}
\newtheorem{remark}{Remark}

\numberwithin{equation}{section}
\begin{document}		
	\title{Framed Matrices and  $A_{\infty}$-Bialgebras}
	\author{Samson Saneblidze$^{1}$}
	\address{ A. Razmadze Mathematical Institute, I.Javakhishvili Tbilisi State University 2, Merab Aleksidze II Lane\\
		0193 Tbilisi, Georgia}
	\email{sane@rmi.ge}
	\author{Ronald Umble $^{2}$}
	\address{Department of Mathematics\\
		Millersville University of Pennsylvania\\
		Millersville, PA. 17551}
	\email{ron.umble@millersville.edu}
	\thanks{$^{1} $ The research reported in this publication was made possible in part
		by grant SRNSF/217614}
	\thanks{$^{2}$ This research funded in part by multiple Millersville University faculty research grants.}
	\date{}
\begin{abstract}
	We complete the construction of the biassociahedra $KK$, construct the free matrad
	$\mathcal{H}_{\infty}$, realize $\mathcal{H}_{\infty}$ as the cellular chains of $KK,$ and define an $A_{\infty}$-bialgebra as an
	algebra over $\mathcal{H}_{\infty}.$ We construct the bimultiplihedra $JJ,$ construct the
	relative free matrad $r\mathcal{H}_{\infty}$ as a $\mathcal{H}_{\infty}$-bimodule,
	realize $r\mathcal{H}_{\infty}$ as the cellular chains of $JJ$, and define a
	morphism of $A_{\infty}$-bialgebras as a bimodule over $\mathcal{H}_{\infty}$.
	We prove that the homology of every $A_{\infty }$-bialgebra over a commutative ring with unity admits an induced $A_{\infty}
	$-bialgebra structure. We extend the Bott-Samelson isomorphism to an
	isomorphism of $A_{\infty}$-bialgebras and determine the $A_{\infty}
	$-bialgebra structure of $H_{\ast}\left(  \Omega\Sigma X;\mathbb{Q}\right)  $.
	For each $n\geq2$, we construct a space $X_{n}$ and identify an induced
	nontrivial $A_{\infty}$-bialgebra operation $\omega_{2}^{n}: H^{\ast}\left(
	\Omega X_{n};\mathbb{Z}_{2}\right)  ^{\otimes2}\rightarrow H^{\ast}\left(
	\Omega X_{n};\mathbb{Z}_{2}\right)  ^{\otimes n}$.
\end{abstract}
\maketitle
\vspace{-.2in}
\tableofcontents
\vspace{-.3in}

\section{Introduction}

Let $m,n\in \mathbb{N}, mn\ge 2$. The biassociahedron $KK_{n,m}$ is a contractible $(m+n-3)$-dimensional
polytope with a single top dimensional cell $e^{m+n-1}$; in particular, $KK_{1,n}\cong KK_{n,1}$ is Stasheff's associahedron $K_{n}$ \cite{Stasheff}. Our construction of $KK_{n,m}$ in \cite{SU4}, which is valid for all $m$ when $n\leq3$ and for all $n$ when $m\leq3$, extends M. Markl's construction for $m+n\leq 6$ in \cite{Markl2} and \cite{Markl3}.

In this paper we introduce and apply the theory of framed matrices to construct $KK_{m,n}$ for all $m$ and $n$. A \emph{framed matrix} is an equivalence class of paths of \emph{generalized bipartition matrices} whose entries involve \emph{augmented bipartitions}, which are pairs of partitions  $(A_1|\cdots|A_r,B_1|\cdots|B_r)$ of finite sets of positive integers in which $A_i$ and $B_j$ are possibly null (cf. Sections \ref{BM} and \ref{FM}).

Let $\Theta=\{\theta_m^n:\theta_1^1=\mathbf{1}\}_{m,n\in \mathbb{N}}$ be a bigraded set with at most one element of bidegree $(m,n)$ and let $F^{pre}(\Theta)$ denote the \emph{free prematrad generated by} $\Theta$ introduced in \cite{SU4}.  Then $F^{pre}(\Theta)$ is the $\mathcal{A}_{\infty}$-operad when $\Theta =\{\theta_m^1\}_{m\ge 1}$ or $\Theta =\{\theta_1^n\}_{n\ge 1}$, and  $F^{{pre}}\left(  \Theta\right)  /\sim,$ where $A\sim B$ if ${bideg}\left(  A\right)  ={bideg}\left(  B\right)$, is the \emph{bialgebra prematrad} $\mathcal{H}^{{pre}}=\langle c_{n,m}\rangle_{m,n\in \mathbb{N}}$ when  $\Theta=\{\mathbf{1},\theta_2^1,\theta_1^2\}$ (see Example \ref{bialg}). 	
The \emph{free matrad} $\mathcal{H}_\infty$ is a proper submodule of $F^{pre}(\Theta)$ when $\Theta$ contains
exactly one element of each bidegree.

Consider the canonical prematrad projection
$\rho^{{pre}}:F^{pre}(\Theta)\rightarrow{\mathcal{H}}^{pre}$ under
which	\[
\rho^{pre}(\theta_{m}^{n})=\left\{
\begin{array}
	[c]{ll}
	c_{n,m}, & m+n\leq3\\
	0, & \text{otherwise.}
\end{array}
\right.
\]
Suppose we wish to define a differential $\partial^{pre}$ such that
$\rho^{pre}:\left(F^{{pre}}(\Theta),\partial^{pre}\right)\rightarrow\left(
{\mathcal{H}}^{pre},0\right)  $ is a free resolution in the category of
prematrads. Then $H_{\ast}\left(  F^{pre}(\Theta),\partial^{pre}\right)$ $\approx\mathcal{H}^{pre}$ implies
\[
\begin{array}
	[c]{l}%
	\partial^{pre}(\theta_{2}^{1})=\partial^{pre}(\theta_{1}
	^{2})=0\smallskip\\
	\partial^{pre}(\theta_{3}^{1})=\gamma(\theta_{2}^{1};\mathbf{1}
	,\theta_{2}^{1})-\gamma(\theta_{2}^{1};\theta_{2}^{1},\mathbf{1})\smallskip\\
	\partial^{pre}(\theta_{2}^{2})=\gamma(\theta_{1}^{2};\theta_{2}
	^{1})-\gamma(\theta_{2}^{1}\theta_{2}^{1};\theta_{1}^{2}\theta_{1}
	^{2})\smallskip\\
	\partial^{pre}(\theta_{1}^{3})=\gamma(\mathbf{1},\theta_{1}^{2}
	;\theta_{1}^{2})-\gamma(\theta_{1}^{2},\mathbf{1};\theta_{1}^{2}),
\end{array}
\]
where $\gamma(\theta_{2}^{1};\mathbf{1}
,\theta_{2}^{1}):=\theta_{2}^{1}(\mathbf{1}\otimes
\theta_{2}^{1})$ and $\gamma(\theta_{2}^{1}\theta_{2}^{1};\theta_{1}^{2}
\theta_{1}^{2}):=\left(  \theta_{2}^{1}\otimes\theta_{2}^{1}\right)  \left(
2,3\right)  \left(  \theta_{1}^{2}\otimes\theta_{1}^{2}\right)  $. Although it
is difficult to extend $\partial^{pre}$ in such a way that acyclicity is
easy to verify, there is a canonical differential
$\partial$ on $\mathcal{H}_{\infty}$ such that the canonical projection
$\varrho:\left(  \mathcal{H}_{\infty},\partial\right)  \rightarrow\left(
{\mathcal{H}}^{pre},0\right)  $ is a free resolution in the category of
prematrads.

The precise definition of $\mathcal{H}_{\infty}$ requires the
notion of \emph{coherent} framed matrices defined in terms of the S-U diagonal $\Delta_P$ on the permutahedra $P=\{P_n\}_{n\geq 1}$ \cite{SU2}, and once defined, $\mathcal{H}
_{\infty}$ is the \emph{minimal} submodule of $F^{pre}(\Theta).$ But unlike the resolution differential defined in \cite{SU4}, which fails to preserve the underlying coherence in bidegree $(4,4)$  (cf. Example \ref{33obstruction}), the more delicate resolution differential defined here preserves the underlying coherence in every bidegree.

Let $\mathfrak{n}=\left\{  1,2,\ldots,n\right\}$ and let $P_{n}$ denote the $\left(  n-1\right)  $-dimensional
permutahedron. Guided by the combinatorial join of
permutahedra $P_{m}\ast_{c}P_{n}=P_{m+n}$ defined in \cite{SU4},
we introduce
the \emph {balanced framed join} $\mathfrak{m}
\circledast_{pp}\mathfrak{n}$ and the \emph{reduced balanced
	framed join} $\mathfrak{m}\circledast_{kk}\mathfrak{n}=\mathfrak{m}\circledast_{pp}\mathfrak{n}/\sim$.
The bipermutahedron $PP_{n,m}$ is the geometric realization of
$\mathfrak{m}\circledast_{pp}\mathfrak{n}$ and the biassociahedron
$KK_{n+1,m+1}=PP_{n,m}/\sim$ is the geometric realization of $\mathfrak{m}\circledast_{kk}\mathfrak{n}$, where $\sim$ extends the
standard relation $K_{n+1}=P_{n}/\sim$.
Thus $KK_{m,n}=K_{m}\ast_{kk}K_{n}$ is the reduced balanced framed join of the associahedra $K_m$ and $K_n$  for all $m$ and $n.$

Let $KK=\{KK_{n,m}\}_{mn\ge 2}$. We realize $\mathcal{H}_{\infty}$ as the
cellular chains $C_*(KK)$ and define an $A_{\infty}$-\emph{bialgebra} as an
algebra over $\mathcal{H}_{\infty}.$ 	
We define a \emph{relative prematrad}, construct the
\emph{relative free matrad} $r\mathcal{H}_{\infty}$ as a $\mathcal{H}_{\infty}$-bimodule,
realize $r\mathcal{H}_{\infty}$ as the cellular chains of
bimultiplihedra $JJ=\{JJ_{n,m}\}_{mn\geq 1},$ and define a
\emph{morphism $G:A\Rightarrow B$ of $A_{\infty}$-bialgebras} as a bimodule
over $\mathcal{H}_{\infty}$.

Let $R$ be a commutative ring with unity. 	An $A_{\infty}$-bialgebra is a graded $R$-module $H$ together with a 	family of multilinear operations $\{\omega_{m}^{n}\in
Hom_{m+n-3}(H^{\otimes m},$ $H^{\otimes n})\}_{m,n\geq1}$ and a
chain map $\alpha:C_{\ast}\left(  KK\right)  \rightarrow \mathcal{E}nd_{TH}$ such that
$\alpha\left(  e^{m+n-3}\right)  =\omega_{m}^{n}.$ 	
Given an $A_{\infty}$-bialgebra $B$ whose homology $H_{\ast}(B)$ is a free
$R$-module, and a homology isomorphism $g:H_{\ast}\left(  B\right)
\rightarrow B$, we prove that the $A_{\infty}$-bialgebra structure on $B$
pulls back along $g$ to an $A_{\infty}$-bialgebra structure on $H_{\ast
}\left(  B\right)  $, and any two such structures so obtained are isomorphic.
This is a special case of our main result:\smallskip

\noindent\textbf{Theorem} \ref{AAHopf}. \textit{Let }$A$\textit{\ be a free
}$R$\textit{-module, let }$B$\textit{\ an }$A_{\infty}$\textit{-bialgebra, and
	let }$g:A\rightarrow B$\textit{\ be a homology isomorphism. Then}\smallskip

\noindent(i) (Existence) $g$\textit{\ induces an }$A_{\infty}$%
\textit{-bialgebra structure }$\omega_{A}$\textit{\ on }$A$\textit{\ and
	extends to a map\ }	$G:A\Rightarrow B$ \textit{of}

$A_{\infty}$-\textit{bialgebras and}\smallskip

\noindent(ii) (Uniqueness) $\left(  \omega_{A},G\right)  $\textit{\ is unique
	up to isomorphism.}\smallskip

Our proof of Theorem \ref{AAHopf} follows from a new transfer algorithm based on the interpretation of $JJ_{n,m}$ as a subdivision of $KK_{n,m}\times I$. In the special case of $A_{\infty}$-(co)algebras, in which $J_{n}$ is interpreted as a subdivision of $K_{n}\times I$, our approach differs significantly from the standard perturbation method for transferring $A_{\infty}$-structures. Indeed, our definition of a morphism between two $A_{\infty}$-(bi)structures in terms $C_*(JJ_{n,m})$ allows us to extend a given morphism $A\to B$ of DG modules to a morphism $A\Rightarrow B$ of $A_{\infty}$-(bi)structures inductively. For some remarks on the history of
perturbation theory see \cite{Huebschmann} and \cite{Hueb-Kade}.

We conclude with several applications. Given a
topological space $X$ and a field $F,$ the bialgebra structure of simplicial
singular chains $S_{\ast}\left(  \Omega X;F\right)  $ of Moore base pointed
loops induces an $A_{\infty}$-bialgebra structure on $H_{\ast}\left(  \Omega
X;F\right)  $ whose $A_{\infty}$-algebra substructure was observed by Kadeishvili
\cite{Kadeishvili1} and whose $A_{\infty}$-coalgebra substructure was observed by Gugenheim \cite{Gugenheim}. Furthermore, the $A_{\infty}$-coalgebra structure on
${H}_{\ast}(X;F)$ extends to an $A_{\infty}$-bialgebra structure on the tensor
algebra $T^{a}\tilde{H}_{\ast}(X;F)$, which is trivial if and only if the
$A_{\infty}$-coalgebra structure on ${H}_{\ast}(X;F)$ is trivial, and the
Bott-Samelson isomorphism $t_{\ast}:T^{a}\tilde{H}_{\ast}(X;F)\overset{\approx
}{\rightarrow}H_{\ast}(\Omega\Sigma X;F)$ extends to an isomorphism of
$A_{\infty}$-bialgebras (Theorem \ref{B-S}). The $A_{\infty}%
$-bialgebra structure of $H_{\ast}(\Omega\Sigma X;\mathbb{Q})$ provides the
first nontrivial rational homology invariant for $\Omega\Sigma X$ (Corollary
\ref{invariant}). Finally, for each $n\geq2,$ we construct a space $X_{n}$ and
identify a nontrivial $A_{\infty}$-bialgebra operation $\omega_{2}^{n}%
:H^{\ast}(\Omega X_{n};\mathbb{Z}_{2})^{\otimes2}\rightarrow H^{\ast}(\Omega
X_{n};\mathbb{Z}_{2})^{\otimes n}$ defined in terms of the action of the
Steenrod algebra $\mathcal{A}_{2}$ on $H^{\ast}(X_{n};\mathbb{Z}_{2}).$

\section{Combinatorial and Topological Tools}

\label{topological}We begin with a review of the combinatorial and topological
tools we need for our constructions.

\subsection{Partitions and Permutahedra}

By an \emph{ordered set} we mean the empty set or a finite strictly increasing
set of positive integers. We denote the special ordered sets $\mathfrak{0}:=\varnothing$ and $\mathfrak{n}:=\left\{
1,2,\ldots,n\right\}  ,$ $n\geq1,$ and define $\min\varnothing=\max
\varnothing:=0.$

Let $\mathbf{a}$ be an ordered set and let
$\#\mathbf{a}$  denote its cardinality. When $A_{1}|\cdots|A_{n}$ is an ordered partition of $A$ in
which $A_{i}=\varnothing,$ we write $A_{1}|\cdots|A_{i-1}|0|\cdots|A_{n}$. The
set of \emph{(ordered) partitions of }$\mathbf{a}$ is defined and denoted by
\[
P\left(  \mathbf{a}\right)  :=\left\{
\begin{array}
	[c]{cl}%
	\left\{  0\right\}  , & \mathbf{a}=\varnothing\\
	\left\{  \text{standard ordered partitions of }\mathbf{a}\right\}  , &
	\text{otherwise},
\end{array}
\right.
\]
and decomposes as the disjoint union of subsets  $P_{n}\left(  \mathbf{a}\right):=\left\{  A_{1}|\cdots|A_{n}\in P\left(  \mathbf{a}\right)  \right\}$ of \emph{length }$n$ \emph{partitions of} $\mathbf{a}$, i.e.,
$P\left(  \mathbf{a}\right) =\bigcup_{1\leq n\leq \#\mathbf{a}}P_{n}\left(  \mathbf{a}\right) $.

Let $\alpha\in P_{n}\left(  \mathbf{a}\right)$ and let $\left\vert \alpha\right\vert $ denote its \emph{dimension}.  When $\mathbf{a}=\varnothing,$ define
$\left\vert \mathbf{a}\right\vert :=0.$ When $\mathbf{a}\neq\varnothing,$ define
$\left\vert \alpha\right\vert :=\#\mathbf{a}-n$ via the correspondence
$P_{n}(\mathbf{a})\leftrightarrow\{$codimension $n-1$ cells of
$P_{\#\mathbf{a}}\},$ where $P_{\#\mathbf{a}}$ denotes the permutahedron whose
cells (or faces) are indexed by the partitions in $P\left(\mathbf{a}%
\right) $. Thus, when $\mathbf{a}\neq\varnothing$ we identify
$P\left(  \mathbf{a}\right)  $ with the permutahedron $P_{\#\mathbf{a}}$.

The set of \emph{augmented (ordered) partitions of }$\mathbf{a}$ is defined
and denoted by
\[
P^{\prime}\left(  \mathbf{a}\right)  :=\left\{
\begin{array}
	[c]{cl}%
	\{\underset{k}{\underbrace{0|\cdots|0}}:k\in\mathbb{N}\}, & \mathbf{a}
	=\varnothing\\
	\{A_{1}|\cdots|A_{k}:k\in\mathbb{N}\text{, where }A_{i_{1}}|\cdots|A_{i_{r}
	}\in P(\mathbf{a}) & \\
	\text{for some }i_{1}<\cdots<i_{r}\}, & \mathbf{a}\neq\varnothing\text{.}%
\end{array}
\right.
\]
Let $\alpha=A_{1}|\cdots|A_{n}\in P^{\prime}\left(  \mathbf{a}\right)  ;$ its
\emph{length }$l\left(  \alpha\right)  :=n.$ Then $P_{n}^{\prime}%
(\mathbf{a}):=\{\alpha\in P^{\prime}\left(  \mathbf{a}\right)  :l\left(
\alpha\right)  =n\}$ is the subset of \emph{length\ }$n$\emph{\ augmented
	partitions} \emph{in} $P^{\prime}\left(  \mathbf{a}\right)  .$ Note that
$P_{1}^{\prime}\left(  \mathbf{a}\right)  =P_{1}\left(  \mathbf{a}\right)
=\left\{  \mathbf{a}\right\}  .$

Let
$\alpha:=A_{1}|\cdots|A_{n+1}\in P_{n+1}^{\prime}(\mathbf{a}).$ Given a subset $M^k\subseteq A_k,$  the
\emph{(partitioning) action of }$M^{k}$ \emph{on }$\alpha$ is the partition
\begin{equation} \label{partitioning-action}
	\partial_{M^{k}}\alpha:=A_{1}|\cdots|A_{k-1}|M^{k}|A_{k}\smallsetminus
	M^{k}|A_{k+1}|\cdots|A_{n+1}
\end{equation}
(cf. \cite{SU2}).
The set $M^{k}$ is \emph{extreme }if $M^{k}=\varnothing$ or
$M^{k}=A_{k}.$ Note that when $M^{k}\neq\varnothing,$ the index $k$ can be omitted.

Given an ordered subset $\lambda\subseteq\mathfrak{n,}$ let $k=\#\lambda.$
Define $\lambda^{0}:=0$ and $\lambda^{k+1}:=n+1;$ when $k>0,$ write
$\lambda=\left\{  \lambda^{1}<\cdots<\lambda^{k}\right\}  .$ The $\lambda$-\emph{projection}
\begin{equation}\label{mu}
	\mu_{\lambda}:P_{n+1}^{\prime
	}(\mathbf{a})\rightarrow P_{k+1}^{\prime}(\mathbf{a})
\end{equation}
is defined by
$\mu_{\lambda}\left(  \alpha\right)  :=\bar{A}_{1}|\cdots|\bar{A}_{k+1},$
where $\bar{A}_{i}=A_{\lambda^{i-1}+1}\cup\cdots\cup A_{\lambda^{i}}.$ In the
extreme cases $\lambda=\varnothing$ and $\lambda=\mathfrak{n}$ we have
$\mu_{\varnothing}\left(  \alpha\right)  =\mathbf{a}$ and $\mu_{\mathfrak{n}%
}\left(  \alpha\right)  =\alpha.$ When $\lambda_{i}=\mathfrak{n}%
\smallsetminus\left\{  i\right\}  ,$ $1\leq i\leq n,$ we write
\begin{equation}
	\alpha\left[  i\right]  :=\mu_{\lambda_{i}}\left(  \alpha\right)
	=A_{1}|\cdots|A_{i}\cup A_{i+1}|\cdots|A_{n+1}. \label{special-lambda}%
\end{equation}
Then $\partial_{A_{k}}\alpha\left[  k\right]  =\alpha$ for all $k.$
The \emph{canonical projection} \[\pi:P^{\prime}(\mathbf{a})\rightarrow
P(\mathbf{a})\]
is given by discarding empty blocks when $\mathbf{a}\neq\varnothing$ and by $\pi(0|\cdots|0):=\varnothing$ otherwise.

Define the \emph{dimension} and \emph{vacuosity }of $\alpha$ by%
\begin{equation}
	\left\vert \alpha\right\vert :=\left\vert \pi\left(  \alpha\right)
	\right\vert \text{ \ and \ }v\left(  \alpha\right)  :=l\left(  \alpha\right)
	-l\left(  \pi\left(  \alpha\right)  \right)  , \label{dim-vac}%
\end{equation}
respectively; then $v\left(  \alpha\right)  $ counts the number of empty
blocks in $\alpha.$

Let $\mathbf{a}$ and $\mathbf{b}$ be ordered sets. When $\mathbf{a}%
\subseteq\mathbf{b},$ the precise way in which $\mathbf{a}$ embeds in
$\mathbf{b}$ is encoded by a special element of $P^{\prime}\left(
\mathbf{a}\right)  $, called the \emph{embedding partition of }$\mathbf{a}%
$\emph{\ in} $\mathbf{b}$. When $\mathbf{a}\neq\varnothing,$ this particular
partition is constructed block-by-block as the elements of $\mathbf{a}$ are
analyzed sequentially.
\begin{definition}\label{embed}
	Let $\mathbf{a}\subseteq\mathbf{b}.$ The\textbf{\ embedding partition
		of}\emph{\ }$\mathbf{a}$ \textbf{in} $\mathbf{b,}$ denoted by $EP_{\mathbf{b}
	}\mathbf{a,}$ is the following element of $P^{\prime}\left(  \mathbf{a}
	\right)  $:\smallskip
	
	\noindent If $\mathbf{a}=\varnothing,$ define $EP_{\mathbf{b}}\mathbf{a:}
	=0|\cdots|0\in P_{\#\mathbf{b}+1}^{\prime}\left(  \mathbf{a}\right)
	.$\smallskip
	
	\noindent If $\mathbf{a}\neq \varnothing,$ write $\mathbf{a}=\left\{  a_{1}<\cdots<a_{m}\right\} $ and define $a_{0}:=0$ and $a_{m+1}:=\infty$. \smallskip
	
	\noindent For $i=0,1,\ldots,m:$	
	define $j_{i}:=\#\{b\in\mathbf{b}:a_{i}	<b<a_{i+1}\}.$\smallskip
	
	\noindent Then $EP_{\mathbf{b}}\mathbf{a:}=\mathbf{b}_{1}|\cdots|\mathbf{b}_{q},$ where
	
	\begin{enumerate}

		\item $\mathbf{b}_{k}=\varnothing$ for all $k\leq j_{0}$ and $k>q-j_{m},$
		
		\item $a_{1}\in\mathbf{b}_{j_{0}+1},$ and
		
		\item for $i=1,2,\ldots,m-1:$ if $a_{i}\in\mathbf{b}_{k},$
		
		\noindent then $a_{i+1}\in\mathbf{b}_{k+j_{i}}$
		and $\mathbf{b}_{k+1}=\cdots=\mathbf{b}_{k+j_{i}-1}=\varnothing$ when
		$j_{i}>1.$
	\end{enumerate}
\end{definition}

\noindent Heuristically speaking, the empty blocks at the extremes are place
holders for the extreme elements of $\mathbf{b}\smallsetminus\mathbf{a}.$ The
elements of a non-empty block $\mathbf{b}_{k}$ are consecutive in $\mathbf{b}%
$, and $\mathbf{b}_{k}$ is maximal in the sense that $\max\mathbf{b}_{k}$ and
$\min\mathbf{b}_{k+1}$ are \emph{not} consecutive in $\mathbf{b}$ whenever
$k<q$ and $\mathbf{b}_{k+1}\neq\varnothing.$ For example, $EP_{\mathbf{b}%
}\mathbf{b}=\mathbf{b}$ and $EP_{\mathfrak{9}}\left\{  2,3,7,9\right\}
=0|23|0|0|7|9.$ Note that
\begin{equation}
	\#\left(  \mathbf{b}\smallsetminus\mathbf{a}\right)  =\sum\limits_{i=0}
	^{m}j_{i}=q-1. \label{relative-complement}%
\end{equation}
The algorithmic formulation in Definition \ref{embed} is due to D. Freeman and the second author in \cite{Freeman}.

\subsection{Partitions and Planar Leveled Trees}

Denote the sets of down-rooted and up-rooted Planar Leveled Trees (PLTs,
aka \textit{planar rooted trees with levels}) with $n$ leaves by $\vee(n)
$ and $\wedge(n)$, respectively, and consider a PLT $T\in\vee(n)$ with $k$
levels. Set $T=T_{1}$ and let $\left( \curlyvee^{n_{11}},\ldots
,\curlyvee^{n_{1s_{1}}}\right) $ be the sequence of top level corollas in $%
T_{1}.$ Then $\mathbf{n}_{1}=(n_{11},\ldots,n_{1s_{1}})$ is the \emph{first
	leaf sequence of} $T.$ Form the subtree $T_{2}$ by removing the top level
of $T_{1}$ and let $\left(
\curlyvee^{n_{21}},\ldots,\curlyvee^{n_{2s_{2}}}\right) $ be the sequence of
top level corollas in $T_{2}.$ Then $\mathbf{n}_{2}=\left(
n_{21},\ldots,n_{2s_{2}}\right) $ is the \emph{second leaf sequence of} $%
T.$ Continue in this manner until the process terminates and the $k^{th}$
\emph{leaf sequence} $\mathbf{n}_{k}=\left( n_{k1}\right) $ is obtained ($%
s_{k}=1$). Then $\left( \mathbf{n}_{1},\ldots,\mathbf{n}_{k}\right) $ is the
\emph{leaf decomposition} \emph{of }$T.$

To recover the down-rooted tree $T$ from its leaf decomposition $\left(
\mathbf{n}_{1},\ldots ,\mathbf{n}_{k}\right) ,$ set $T^{\mathbf{n}%
	_{k}}=\curlyvee ^{n_{k1}}$ and construct the subtree $T^{\mathbf{n}_{k-1},%
	\mathbf{n}_{k}}$ by attaching corolla $\curlyvee ^{n_{k-1,j}}$ to leaf $j$
of $T^{\left( \mathbf{n}_{k}\right) }$. Construct the subtree $T^{\mathbf{n}%
	_{k-2},\mathbf{n}_{k-1},\mathbf{n}_{k}}$ by attaching corolla $\curlyvee
^{n_{k-2,j}}$ to leaf $j$ of $T^{\mathbf{n}_{k-1},\mathbf{n}_{k}}.$ Continue
in this manner until the process terminates and $T=T^{\mathbf{n}_{1},\ldots ,%
	\mathbf{n}_{k}}$ is obtained. Dually, a tree $T\in \wedge \left( n\right) $
has the leaf decomposition $\left( \mathbf{n}_{k},\ldots ,\mathbf{n}%
_{1}\right) $ and $T=T_{\mathbf{n}_{k},\ldots ,\mathbf{n}_{1}},$ where $%
\mathbf{n}_{1}$ is the bottom level leaf sequence of $T$.

Let $\mathbf{a}$ be a non-empty ordered set. Given $B_{1}|\cdots |B_{k}\in P\left(
\mathbf{a}\right) ,$ identify $B_{1}|\cdots |B_{k}$ with the down-rooted PLT $T^{%
	\mathbf{n}_{1},\ldots ,\mathbf{n}_{k}}$, where $\left( \mathbf{n}_{1},\ldots
,\mathbf{n}_{k}\right) $ is constructed as follows: Set $A_{1}=\mathbf{a}$ and let $%
\beta _{11}|\cdots |\beta _{1s_{1}}:=EP_{A_{1}}B_{1};$ then $\mathbf{n}%
_{1}=(\#\beta _{11}+1,\ldots ,\#\beta _{1s_{1}}+1).$ Set $%
A_{2}=A_{1}\smallsetminus B_{1}$ and consider $B_{2}|\cdots |B_{k}\in
P\left( A_{2}\right) .$ Let $\beta _{21}|\cdots |\beta
_{2s_{2}}:=EP_{A_{2}}B_{2};$ then $\mathbf{n}_{2}=(\#\beta _{21}+1,\ldots ,$
$\#\beta _{2s_{2}}+1).$ Continue in this manner until the process terminates
and $\mathbf{n}_{k}$ is obtained. Denote the componentwise correspondence
between partitions and leaf sequences established above by $\epsilon (B_{i})=%
\mathbf{n}_{i}.$ Then the map
\begin{equation}
	\overset{\vee }{\epsilon }:P\left( \mathbf{a}\right) \rightarrow \vee \left(
	\#\mathbf{a}+1\right)   \label{down}
\end{equation}%
given by $\overset{\vee }{\epsilon }\left( B_{1}|\cdots |B_{k}\right)
=T^{\epsilon (B_{1}),\ldots ,\epsilon (B_{k})}$ is a bijection, and $\overset{\vee}{\epsilon}\hspace*{0.01in}^{-1}$ is given by the
following standard construction when $\mathbf{a}=\mathfrak{n}:$ Given $T\in\vee\left(
n+1\right)  $, number the leaves of $T$ from left-to-right and assign the
label $\ell$ to the vertex of $T$ at which the branch containing leaf $\ell$
meets the branch containing leaf $\ell+1$.
Let $B_{i}$ denote the set of vertex labels in level $i;$ then
$\overset{\vee}{\epsilon}\left(  B_{1}|\cdots|B_{k}\right)  =T.$ For example,
$\overset{\vee}{\epsilon}(5|13|24)=T^{(11112),(221),(3)}\in\vee\left(
6\right)  $ is pictured in Figure 1.

\vspace{0.4in}

\unitlength 1.5mm 
\linethickness{0.8pt}
\ifx\plotpoint\undefined\newsavebox{\plotpoint}\fi
\begin{picture}(77.572,16.187)(0,0)
	\put(37.0,18.443){\makebox(0,0)[cc]{$1\ \ \ \ 2\ \ \ \ 3\ \ \ 4\ \ \ 5\ \ \ 6$}}
	\put(43.8,12.8){\makebox(0,0)[cc]{$5$}}
	\put(31.4,11.0){\makebox(0,0)[cc]{$1$}}
	\put(35.2,11.0){\makebox(0,0)[cc]{$3$}}
	\put(36.2,7.0){\makebox(0,0)[cc]{$2\ \ 4$}}
	\multiput(28.019,15.977)(.070074384,-.067379215){117}{\line(1,0){.070074384}}
	\multiput(36.217,8.094)(.078612859,.067130082){119}{\line(1,0){.078612859}}
	\put(36.3,16.0){\line(0,-1){15.361}}
	\multiput(36.217,11.3)(.078612859,.067130082){43}{\line(1,0){.078612859}}
	\put(32.5,16.0){\line(0,-1){4.361}}
	\put(39.5,16.0){\line(0,-1){1.9}}
	\put(42.5,16.0){\line(0,-1){2.6}}
\end{picture}

\begin{center}
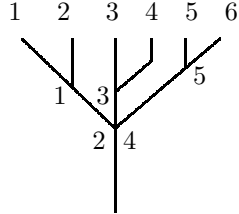
	Figure 1. The down-rooted PLT corresponding to $5|13|24$.
\end{center}
\bigskip

Dually, let $\tau :P\left( \mathbf{a}\right) \rightarrow P\left( \mathbf{a}\right) $ denote
the \emph{reversing map}
\begin{equation}
	\tau \left( B_{1}|\cdots |B_{k}\right) =B_{k}|\cdots |B_{1},  \label{tau}
\end{equation}%
let $\sigma $ be a reflection of the plane in some horizontal axis, and
define
\begin{equation}
	\overset{\wedge }{\epsilon }:=\sigma \overset{\vee }{\epsilon }\tau :P\left(
	\mathbf{a}\right) \rightarrow \wedge \left( \#\mathbf{a}+1\right) .  \label{up}
\end{equation}
Let $\left( \mathbf{n}_{1},\ldots ,\mathbf{n}_{k}\right) $ and $\left(
\mathbf{n}_{k}^{\prime },\ldots ,\mathbf{n}_{1}^{\prime }\right) $ be the
leaf decompositions of $\overset{\vee }{\epsilon }
\left( B_{1}|\cdots |B_{k}\right) $ and $\overset{\wedge }{\epsilon }\left(
B_{1}|\cdots |B_{k}\right)$, respectively. Then the (typically non-isomorphic)
PLTs $T^{\mathbf{n}_{1},\ldots ,\mathbf{n}_{k}}$ and $T_{\mathbf{n}%
	_{k}^{\prime },\ldots ,\mathbf{n}_{1}^{\prime }}\ $index the same cell of $%
P_{n}\ $(cf. \cite{Loday}, \cite{SU2}).

\subsection{The Combinatorial Join of Permutahedra}

In this subsection we present a slight reformulation of the combinatorial join constructed in \cite{SU4}.

The $n$ \emph{right-shift }of an ordered set $\mathbf{a}$
is the ordered set
\[
\mathbf{a}+n:=\left\{
\begin{array}
	[c]{cl}%
	\varnothing, & \mathbf{a}=\varnothing\\
	\left\{  a+n:a\in\mathbf{a}\right\}  , & \mathbf{a}\neq\varnothing.
\end{array}
\right.
\]
The \emph{partitioned union} of $A_{1}|\cdots|A_{n}\in P_{n}^{\prime}\left(
\mathbf{a}\right)  $ \emph{with} $B_{1}|\cdots|B_{n}\in P_{n}^{\prime}\left(
\mathbf{b}\right)  $ is the partition $A_{1}|\cdots|A_{n}\Cup B_{1}%
|\cdots|B_{n}:=$%
\begin{equation}
	\left\{
	\begin{array}
		[c]{cl}%
		A_{1}\cup\left(  B_{1}+\max\mathbf{a}\right)  |\cdots|A_{n}\cup\left(
		B_{n}+\max\mathbf{a}\right)  , & \min\mathbf{b}\leq\max\mathbf{a}\\
		A_{1}\cup B_{1}|\cdots|A_{n}\cup B_{n}, & \text{otherwise.}%
	\end{array}
	\right.  \label{partition-union}%
\end{equation}
When $n=1$, formula (\ref{partition-union}) defines the \emph{ordered set
	union} \emph{of }$\mathbf{a}$\emph{\ with} $\mathbf{b}.$ Note that ordered set
union is associative but non-commutative in general; however, $\mathfrak{m}%
\Cup\mathfrak{n}=\mathfrak{n}\Cup\mathfrak{m}.$

Define $P_{n}^{\prime}(\mathbf{a})\Cup P_{n}^{\prime}(\mathbf{b}):=\left\{
\alpha\Cup\beta:\left(  \alpha,\beta\right)  \in P_{n}^{\prime}(\mathbf{a}%
)\times P_{n}^{\prime}(\mathbf{b})\right\}  ;$ then clearly, $P_{n}^{\prime
}(\mathbf{a})\Cup P_{n}^{\prime}(\mathbf{b})\subseteq P_{n}^{\prime
}(\mathbf{a}\Cup\mathbf{b}).$ Conversely, given $C_{1}|\cdots|C_{n}\in
P_{n}^{\prime}\left(  \mathbf{a}\Cup\mathbf{b}\right)  ,$ for each
$i\in \mathfrak{n}$ let
\[
A_{i}=C_{i}\cap\mathbf{a}\text{ and }B_{i}=\left\{
\begin{array}
	[c]{cc}%
	\left(  C_{i}\smallsetminus\mathbf{a}\right)  -\max\mathbf{a}, & \text{if
	}\min\mathbf{b}\leq\max\mathbf{a}\\
	C_{i}\smallsetminus\mathbf{a}, & \text{otherwise.}%
\end{array}
\right.
\]
Then $C_{1}|\cdots|C_{n}=A_{1}|\cdots|A_{n}\Cup B_{1}|\cdots|B_{n}\in
P_{n}^{\prime}\left(  \mathbf{a}\right)  \Cup P_{n}^{\prime}\left(
\mathbf{b}\right)  \ $so that $P_{n}^{\prime}(\mathbf{a})\Cup P_{n}^{\prime
}(\mathbf{b})=P_{n}^{\prime}(\mathbf{a}\Cup\mathbf{b}).$
\begin{definition}\label{comb-join}
	Let $\mathbf{a}$ and $\mathbf{b}$ be ordered sets. The \textbf{combinatorial
		join of} $P_{m}(\mathbf{a})$ \textbf{with} $P_{n}(\mathbf{b})$ is the set
	\[
	P_{m}(\mathbf{a})\ast_{c}P_{n}(\mathbf{b}):=\left\{  \alpha\Cup\beta\in
	P_{r}(\mathbf{a}\Cup\mathbf{b}):\left(  \alpha,\beta\right)  \in P_{r}
	^{\prime}(\mathbf{a})\times P_{r}^{\prime}(\mathbf{b}), 1\leq r\leq
	m+n\right\}.
	\]
	Define
	\[
	P(\mathbf{a})\ast_{c}P(\mathbf{b}):=\bigcup_{(1,1)\leq (m,n)\leq (\#\mathbf{a},\#\mathbf{b})}P_{m}(\mathbf{a})\ast_{c}P_{n}(\mathbf{b}).
	\]
\end{definition}

\noindent Since $P_{r}(\mathbf{a}\Cup\mathbf{b})\subset P_{r}^{\prime
}(\mathbf{a}\Cup\mathbf{b})=P_{r}^{\prime}(\mathbf{a})\Cup P_{r}^{\prime
}(\mathbf{b}),$ an element $c\in P(\mathbf{a})\ast_{c}P(\mathbf{b})$
decomposes uniquely as $c=\alpha\Cup\beta$ for some $\left(  \alpha
,\beta\right)  \in P_{r}^{\prime}(\mathbf{a})\times P_{r}^{\prime}%
(\mathbf{b}).$ Hence
\begin{equation}
	P(\mathbf{a})\ast_{c}P(\mathbf{b})=P\left(  \mathbf{a}\Cup\mathbf{b}\right)  .
	\label{join}%
\end{equation}
Thus in view of (\ref{join}), the \emph{combinatorial join of permutahedra} $P_{m}\ast_{c}%
P_{n}=P_{m+n}.$

Let $c\in P\left(  \mathbf{a}\Cup\mathbf{b}\right)  $ and let $e_{c}$ be the
cell of $P_{m+n}$ indexed by $c.$ Then $c$ decomposes uniquely as
$c=\alpha\Cup\beta$ for some $\left(  \alpha,\beta\right)  \in P_{r}^{\prime
}(\mathbf{a})\times P_{r}^{\prime}(\mathbf{b}),$ and the projections
$\pi\left(  \alpha\right)  \in P\left(  \mathbf{a}\right)  $ and $\pi\left(
\beta\right)  \in P\left(  \mathbf{b}\right)  $ index cells $e_{\alpha}\subset
P_{m}$ and $e_{\beta}\subset P_{n}.$ Thus, there is a decomposition map
\[
W:\left\{  \text{cells of }P_{m+n}\right\}  \rightarrow\left\{  \text{cells of
}P_{m}\right\}  \times\left\{  \text{cells of }P_{n}\right\}
\]
given by $W\left(  e_{c}\right)  =e_{\alpha}\times e_{\beta}.$ Furthermore, if
$e\times e^{\prime}\subset P_{m}\times P_{n}$ is indexed by $A_{1}%
|\cdots|A_{s}\times B_{1}|\cdots|B_{t}\in P\left(  \mathbf{a}\right)  \times
P\left(  \mathbf{b}\right)  ,$ set
\[
\alpha\times\beta=A_{1}|\cdots|A_{s}|0|\cdots|0\ \times\ 0|\cdots|0|B_{1}%
|\cdots|B_{t}\in P_{s+t}^{\prime}\left(  \mathbf{a}\right)  \times
P_{s+t}^{\prime}\left(  \mathbf{b}\right)  ;
\]
then $W\left(  e_{\alpha\Cup\beta}\right)  =e_{\alpha}\times e_{\beta}=e\times
e^{\prime}$ and $W$ is a surjection. For example, when $m=n=1,$ $W\left(
e_{12}\right)  =e_{1}\times e_{1},$ $W\left(  e_{1|2}\right)  =e_{1|0}\times
e_{0|1},$ and $W\left(  e_{2|1}\right)  =e_{0|1}\times e_{1|0}.$ Of course,
$W$ is not an injection since $e_{1}\times e_{1}=e_{1|0}\times e_{0|1}%
=e_{0|1}\times e_{1|0}.$ Note that $\left\vert e_{\alpha\Cup\beta}\right\vert
\geq\left\vert e_{\alpha}\right\vert +\left\vert e_{\beta}\right\vert ,$ where
equality holds when $\alpha\Cup\beta$ is a shuffle permutation of $\left(
\mathfrak{m};\mathfrak{n}+m\right)  .$

\begin{example}
	\label{comb-join-ex}Let us compute the combinatorial join $P_{1}\ast_{c}
	P_{2}=P_{3}:$\smallskip
	\[
	\begin{tabular}
		[c]{|c|c|c|c|}\hline
		$r$ & $\alpha\in P_{r}^{\prime}\left(  \mathfrak{1}\right)  $ & $\beta\in
		P_{r}^{\prime}\left(  \mathfrak{2}\right)  $ & $\alpha\Cup\beta\in
		P_{r}(\mathfrak{1}\Cup\mathfrak{2})$\\\hline\hline
		$1$ & $1$ & $12$ & $1\Cup12=123$\\\hline
		$2$ & $%
		\begin{array}
			[c]{c}%
			1|0\\
			0|1
		\end{array}
		$ & $%
		\begin{array}
			[c]{r}%
			1|2\\
			2|1\\
			0|12\\
			12|0
		\end{array}
		$ & $%
		\begin{array}
			[c]{r}%
			1|0\Cup1|2=12|3\\
			1|0\Cup2|1=13|2\\
			1|0\Cup0|12=1|23\\
			0|1\Cup1|2=2|13\\
			0|1\Cup2|1=3|12\\
			0|1\Cup12|0=23|1
		\end{array}
		$\\\hline
		$3$ & $%
		\begin{array}
			[c]{c}%
			1|0|0\\
			0|1|0\\
			0|0|1
		\end{array}
		$ & $%
		\begin{array}
			[c]{r}%
			0|1|2\\
			0|2|1\\
			1|0|2\\
			2|0|1\\
			1|2|0\\
			2|1|0\\
			12|0|0\\
			0|12|0\\
			0|0|12
		\end{array}
		$ & $%
		\begin{array}
			[c]{c}%
			1|0|0\Cup0|1|2=1|2|3\\
			1|0|0\Cup0|2|1=1|3|2\\
			0|1|0\Cup1|0|2=2|1|3\\
			0|1|0\Cup2|0|1=3|1|2\\
			0|0|1\Cup1|2|0=2|3|1\\
			0|0|1\Cup2|1|0=3|2|1
		\end{array}
		$\\\hline
	\end{tabular}
	\]
\end{example}

\begin{remark}
	M. Markl's description of the permutahedron $P_{m+n}$ in \cite{Markl3} is a
	translation of the description of the combinatorial join decomposition
	$P_{m+n}=P_{m}\ast_{c}P_{n}$ defined in terms of partitions as above into a description defined in
	terms of PLTs.
\end{remark}

\subsection{Diagonals on Permutahedra and Associahedra}

Let $X$ be an $n$-dimen- sional polytope that admits a cellular projection
$p:P_{n+1}\rightarrow X$ and a realization as a subdivision of the $n$-cube
$I^{n}$. In this subsection we review the S-U diagonal $\Delta_P$ (see \cite{SU2})
and discuss the diagonal $\Delta_X$ induced by $p$. The diagonal
$\Delta_{K_{n+2}}$ is obtained by setting $X=K_{n+2}$.

The permutahedron $P_{n}$ can be realized as a subdivision of $I^{n-1}$ in the following way:
Identify the faces of $P_{n}$ with the partitions of $\mathfrak{n}$; then $P_{1}$ is identified
with the partition $1.$ If $P_{n-1}$ has been constructed and $a=A_{1}|\cdots|A_{p}$ is a face
of $P_{n-1}$, define $a_{0}:=0$, $a_{j}:=\#\left(A_{p-j+1}\cup\cdots\cup A_{p}\right)$ for
$0<j<p$, $a_{p}:=\infty$, and $\frac{1}{2^{\infty}}:=0$. Define
$I\left(a\right):=I_{1}\cup I_{2}\cup\cdots\cup I_{p},$ where $I_{j}:=[1-{2^{-a_{j-1}} },1-{2^{-a_{j}}}];$
then $P_{n}=\bigcup\nolimits_{a\in P_{n-1}}a\times I\left(  a\right)  ,$ where faces of
$a\times I(a)$ are identified with partitions of $\mathfrak{n}$ as follows (see Figures 1 and 2):
\[
\begin{tabular}
	[c]{c|rr}%
	$_{\mathstrut}^{\mathstrut}$\textbf{Face of }$a\times I\left(  a\right)  $ &
	\textbf{Partition of }$\mathfrak{n}$\ \ \ \ \ \ \ \ \ \ \  & \\\hline
	$_{\mathstrut}^{\mathstrut}a\times0$ & \multicolumn{1}{|c}{$A_{1}|\cdots
		|A_{p}|n$} & \multicolumn{1}{l}{}\\
	$_{\mathstrut}^{\mathstrut}a\times(I_{j}\cap I_{j+1})$ &
	\multicolumn{1}{|c}{$A_{1}|\cdots|A_{p-j}|n|A_{p-j+1}|\cdots|A_{p},$} &
	\multicolumn{1}{l}{$1\leq j\leq p-1$}\\
	$_{\mathstrut}^{\mathstrut}a\times1$ & \multicolumn{1}{|c}{$n|A_{1}%
		|\cdots|A_{p},$} & \multicolumn{1}{l}{}\\
	$_{\mathstrut}^{\mathstrut}a\times I_{j}$ & \multicolumn{1}{|c}{$A_{1}%
		|\cdots|A_{p-j+1}\cup n|\cdots|A_{p},$} & \multicolumn{1}{l}{$1\leq j\leq p.$}%
\end{tabular}
\]

\noindent A vertex common to $P_{n}$ and $I^{n-1}$ is a \emph{cubical vertex}.
Thus $a$ is a cubical\ vertex of $P_{n}$ if and only if $a|n$ and $n|a$ are
cubical vertices of $P_{n+1}.$

Let $e$ be a cell of $P_n$ and denote the set of vertices of $e$ by $\mathcal{V}_{e}$.
The subset $\mathcal{V}_{e}\subseteq S_{n}$ determines the components of
$\Delta_{P}(e)$ in the following way: Let $\sigma=x_{1}|\cdots|x_{n}%
\in\mathcal{V}_{e}.$ Reading $\sigma$ from left-to-right and from
right-to-left, construct the partitions $\overleftarrow{\sigma}_{1}%
|\cdots|\overleftarrow{\sigma}_{p}$ and $\overrightarrow{\sigma}_{q}%
|\cdots|\overrightarrow{\sigma}_{1}$ of maximal decreasing subsets and form
the \emph{Strong Complementary Pair }(SCP)
\[
a_{\sigma}\times b_{\sigma}:\overleftarrow{\sigma}_{1}|\cdots
|\overleftarrow{\sigma}_{p}\times\overrightarrow{\sigma}_{q}|\cdots
|\overrightarrow{\sigma}_{1}\in P(\mathfrak{n})\times P(\mathfrak{n}).
\]
Then $\sigma=\max a_{\sigma}=\min b_{\sigma},$ $\min\overleftarrow{\sigma}
_{j}<\max\overleftarrow{\sigma}_{j+1}$ for all $j<p$, and $\min
\overrightarrow{\sigma}_{i}<\max\overrightarrow{\sigma}_{i+1}$ for all $i<q.$
Let $\left\{  c_{ij}\right\}  :=\overrightarrow{\sigma}_{i}\cap
\overleftarrow{\sigma}_{j};$ then $a_{\sigma}\times b_{\sigma}$ is represented
by the $q\times p$ \emph{step matrix} $C_{\sigma}=\left(  c_{ij}\right)  $
\emph{over }$\mathfrak{n\cup}\left\{  0\right\}  ,$ whose positive entries in
each row and column are contiguous and increasing.\bigskip

\begin{center}
	\setlength{\unitlength}{0.0004in}\begin{picture}
		(2975,2685)(3126,-2038) \thicklines \put(3601,239){\line(
			1,0){1800}} \put(5401,239){\line( 0,-1){1800}}
		\put(5401,-1561){\line(-1, 0){1800}} \put(3601,-1561){\line(
			0,1){1800}} \put(3601,239){\makebox(0,0){$\bullet$}}
		\put(3601,-661){\makebox(0,0){$\bullet$}}
		\put(3601,-1561){\makebox(0,0){$\bullet$}}
		\put(5401,239){\makebox(0,0){$\bullet$}}
		\put(5401,-661){\makebox(0,0){$\bullet$}}
		\put(5401,-1561){\makebox(0,0){$\bullet$}}
		\put(4500,-680){\makebox(0,0){$123$}}
		\put(2980,-1861){\makebox(0,0){$1|2|3$}}
		\put(2980,-699){\makebox(0,0){$1|3|2$}}
		\put(2980,464){\makebox(0,0){$3|1|2$}}
		\put(6000,-1861){\makebox(0,0){$2|1|3$}}
		\put(6000,-699){\makebox(0,0){$2|3|1$}}
		\put(6000,464){\makebox(0,0){$3|2|1$}}
		\put(3040,-1260){\makebox(0,0){$1|23$}}
		\put(4550,530){\makebox(0,0){$3|12$}}
		\put(3040,-111){\makebox(0,0){$13|2$}}
		\put(5960,-111){\makebox(0,0){$23|1$}}
		\put(5960,-1260){\makebox(0,0){$2|13$}}
		\put(4550,-1890){\makebox(0,0){$12|3$}}
	\end{picture}\vspace{0.1in}
	
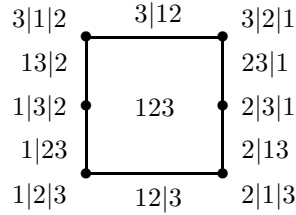
	Figure 1: $P_{3}$ as a subdivision of $P_{2}\times I$.\vspace{0.2in}
	
	\setlength{\unitlength}{0.007in} \begin{picture}
		(500,-500) \thicklines
		\put(0,-120){\line( 0,-1){120}} \put(120,0){\line( 0,-1){360}}
		\put(180,-120){\line( 0,-1){120}} \put(240,0){\line( 0,-1){360}}
		\put(360,-120){\line( 0,-1){120}} \put(420,-120){\line(
			0,-1){120}} \put(480,-120){\line( 0,-1){120}}
		\put(120,0){\line( 1,0){120}} \put(0,-120){\line( 1,0){480}} \put
		(120,-150){\line( 1,0){60}} \put(180,-180){\line( 1,0){60}} \put
		(240,-150){\line( 1,0){120}} \put(420,-150){\line( 1,0){60}} \put
		(0,-180){\line( 1,0){120}} \put(360,-180){\line( 1,0){60}}
		\put(0,-240){\line( 1,0){480}} \put(120,-360){\line( 1,0){120}}
		\put(120,0){\makebox(0,0){$\bullet$}}
		\put(180,0){\makebox(0,0){$\bullet$}}
		\put(240,0){\makebox(0,0){$\bullet$}}
		\put(0,-120){\makebox(0,0){$\bullet$}}
		\put(120,-120){\makebox(0,0){$\bullet$}} \put(180,-120){\makebox
			(0,0){$\bullet$}} \put(240,-120){\makebox(0,0){$\bullet$}} \put
		(360,-120){\makebox(0,0){$\bullet$}}
		\put(420,-120){\makebox(0,0){$\bullet$}}
		\put(480,-120){\makebox(0,0){$\bullet$}} \put(120,-150){\makebox
			(0,0){$\bullet$}} \put(180,-150){\makebox(0,0){$\bullet$}} \put
		(240,-150){\makebox(0,0){$\bullet$}}
		\put(360,-150){\makebox(0,0){$\bullet$}}
		\put(0,-180){\makebox(0,0){$\bullet$}}
		\put(182.25,-180){\makebox(0,0){$\bullet$ }}
		\put(240,-180){\makebox(0,0){$\bullet$}} \put(420,-180){\makebox
			(0,0){$\bullet$}} \put(480,-180){\makebox(0,0){$\bullet$}} \put
		(0,-150){\makebox(0,0){$\bullet$}}
		\put(120,-240){\makebox(0,0){$\bullet$}}
		\put(360,-240){\makebox(0,0){$\bullet$}} \put(420,-240){\makebox
			(0,0){$\bullet$}} \put(480,-150){\makebox(0,0){$\bullet$}} \put
		(0,-240){\makebox(0,0){$\bullet$}}
		\put(120,-180){\makebox(0,0){$\bullet$}}
		\put(180,-240){\makebox(0,0){$\bullet$}} \put(240,-240){\makebox
			(0,0){$\bullet$}} \put(360,-180){\makebox(0,0){$\bullet$}} \put
		(420,-150){\makebox(0,0){$\bullet$}}
		\put(480,-240){\makebox(0,0){$\bullet$}}
		\put(120,-360){\makebox(0,0){$\bullet$}} \put(180,-360){\makebox
			(0,0){$\bullet$}} \put(240,-360){\makebox(0,0){$\bullet$}}
		\put(55,-152){\makebox(0,0){$124|3$}}
		\put(55,-210){\makebox(0,0){$12|34$}}
		\put(150,-134){\makebox(0,0){$24|13$}}
		\put(150,-195){\makebox(0,0){$2|134$}}
		\put(177,-55){\makebox(0,0){$4|123$}}
		\put(177,-295){\makebox(0,0){$123|4$}}
		\put(210,-152){\makebox(0,0){$234|1$}}
		\put(210,-210){\makebox(0,0){$23|14$}}
		\put(300,-134){\makebox(0,0){$34|12$}}
		\put(300,-195){\makebox(0,0){$3|124$}}
		\put(390,-152){\makebox(0,0){$134|2$}}
		\put(390,-200){\makebox(0,0){$13|24$}}
		\put(450,-134){\makebox(0,0){$14|23$}}
		\put(450,-194){\makebox(0,0){$1|234$}}
	\end{picture}\vspace*{2.6in}
	
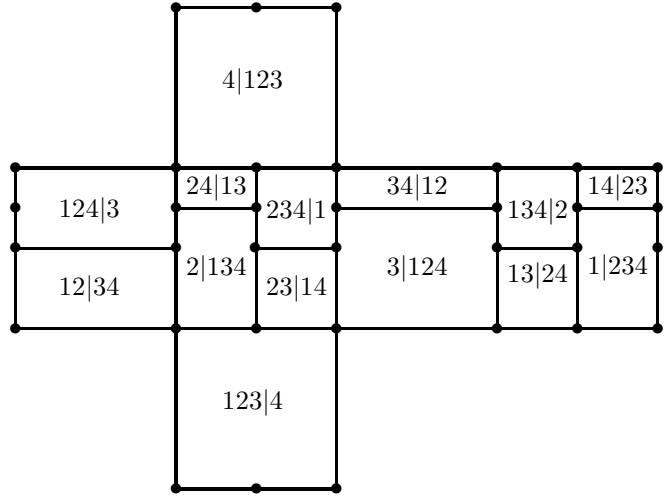
	Figure 2: The facets of $P_{4}$ as a subdivision of $I^{3}$.\bigskip
\end{center}

\begin{example}
	\label{scp}Consider the vertex $\sigma=2|1|3|5|4$ of $P_{5};$ then
	$\overleftarrow{\sigma}_{1}|\overleftarrow{\sigma}_{2}|\overleftarrow{\sigma
	}_{3}=21|3|54$ and $\overrightarrow{\sigma}_{3}|\overrightarrow{\sigma}%
	_{2}|\overrightarrow{\sigma}_{1}=2|135|4$ so that $a_{\sigma}\times b_{\sigma
	}=21|3|54\times2|135|4\ $and
	\[
	C_{\sigma}=\left(
	\begin{tabular}
		[c]{lll}%
		$0$ & $0$ & $4$\\
		$1$ & $3$ & $5$\\
		$2$ & $0$ & $0$%
	\end{tabular}
	\ \right)  \ .
	\]
	
\end{example}

Let $a=A_{1}|\cdots|A_{p}\in P(\mathfrak{n})$. For $1\leq j<p,$ choose a subset
$M_{j}\subseteq(A_{j}\smallsetminus\{\min A_{j}\})$ such that $\min M_{j}>\max
A_{j+1}$ when $M_{j}\neq\varnothing.$ Define the \emph{right-shift }$M_{j}$
\emph{action}
\[
R_{M_{j}}(a):=\left\{
\begin{array}
	[c]{cl}%
	A_{1}|\cdots|A_{j}\smallsetminus M_{j}|A_{j+1}\cup M_{j}|\cdots|A_{k}, &
	M_{j}\neq\varnothing\\
	a, & M_{j}=\varnothing.
\end{array}
\right.
\]
Choose $M_{1}\subset A_{1}$ so that $R_{M_{1}}(a)$ is defined, choose
$M_{2}\subset A_{2}\cup M_{1}$ so that $R_{M_{2}}R_{M_{1}}(a)$ is defined,
choose $M_{3}\subset A_{3}\cup M_{2}$ so that $R_{M_{3}}R_{M_{2}}R_{M_{1}}(a)$
is defined, and so on. Having chosen \textbf{$M$}$:=(M_{1},M_{2}%
,\ldots,M_{p-1})$, denote the composition ${R}_{M_{p-1}}\cdots R_{M_{2}%
}R_{M_{1}}(a)$ by $R_{\mathbf{M}}\left(  a\right)  $.

Dually, let $b=B_{q}|\cdots|B_{1}\in P(\mathfrak{n}).$ For $1\leq i<q,$ choose a subset
$N_{i}\subseteq(B_{i}\smallsetminus\linebreak\{\min B_{i}\})$ such that $\min N_{i}>\max
B_{i+1}$ when $N_{i}\neq\varnothing.$ Define the \emph{left-shift }$N_{i}%
$\emph{ action}
\[
L_{N_{i}}(b):=\left\{
\begin{array}
	[c]{cl}%
	B_{q}|\cdots|B_{i+1}\cup N_{i}|B_{i}\smallsetminus N_{i}|\cdots|B_{1}, &
	N_{i}\neq\varnothing\\
	b, & N_{i}=\varnothing.
\end{array}
\right.
\]
Choose $N_{1}\subset B_{1}$ so that $L_{N_{1}}(b)$ is defined, choose
$N_{2}\subset B_{2}\cup N_{1}$ so that $L_{N_{2}}L_{N_{1}}(b)$ is defined,
choose $N_{3}\subset B_{3}\cup N_{2}$ so that $L_{N_{3}}L_{N_{2}}L_{N_{1}}(b)$
is defined, and so on. Having chosen \textbf{$N$}$:=(N_{1},N_{2}%
,\ldots,N_{q-1})$, denote the composition $L_{N_{q-1}}\cdots L_{N_{2}}%
L_{N_{1}}(b)$ by $L_{\mathbf{N}}\left(  b\right)  $.

Now choose $\mathbf{M}$ and $\mathbf{N}$ so that ${R}_{\mathbf{M}}(a_{\sigma
})$ and $L_{\mathbf{N}}\left(  b_{\sigma}\right)  $ are defined, and form the
\emph{Complementary Pair} (CP) ${R}_{\mathbf{M}}(a_{\sigma})\times
L_{\mathbf{N}}\left(  b_{\sigma}\right)  $. Define
\[
A_{\sigma}\times B_{\sigma}:=\bigcup\limits_{\mathbf{M,N}}\left\{
{R}_{\mathbf{M}}(a_{\sigma})\times L_{\mathbf{N}}\left(  b_{\sigma}\right)
\right\}  ;
\]
then
\begin{equation}
	\Delta_{P}(e)=\bigcup_{\sigma\in\mathcal{V}_{e}}A_{\sigma}\times B_{\sigma}.
	\label{Delta_P}%
\end{equation}
Note that $R_{M_{j}}(a_{\sigma})$ and $L_{N_{i}}\left(  b_{\sigma}\right)  $
are realized by a right-shift $M_{j}$ action and a down-shift $N_{i}$ action
on $C_{\sigma}.$ Thus \emph{the components of }$\Delta_{P}(e^{n-1})$\emph{ are
	generated by all possible right-shift actions together with all possible
	down-shift actions on all possible step matrices over} $\mathfrak{n}%
\cup\left\{  0\right\}  $.

\begin{example}
	On the top dimensional cell $e^{2}$ of $P_{3}$, $\Delta_{P}(e^{2})$ is the
	union of
	\[
	\hspace{-0.1in}%
	\begin{array}
		[c]{ll}%
		A_{1|2|3}\times B_{1|2|3}=\left\{  1|2|3\times123\right\}  , & A_{1|3|2}\times
		B_{1|3|2}=\left\{  1|32\times13|2\right\}  ,\\
		A_{2|1|3}\times B_{2|1|3}=\left\{  21|3\times2|13,\text{ }21|3\times
		23|1\right\}  , & A_{2|3|1}\times B_{2|3|1}=\{2|31\times23|1\},\\
		A_{3|1|2}\times B_{3|1|2}=\{31|2\times3|12,\text{ }1|32\times3|12\}, &
		A_{3|2|1}\times B_{3|2|1}=\{321\times3|2|1\}.
	\end{array}
	\]
	
\end{example}

Note that $R_{M_{j}}(a_{\sigma})$ and $L_{N_{i}}\left(  b_{\sigma}\right)  $
are realized by a right-shift $M_{j}$ action and a down-shift $N_{i}$ action
on $C_{\sigma}.$ Thus \emph{the components of }$\Delta_{P}(e^{n-1})$\emph{ are
	generated by all possible right-shift actions together with all possible
	down-shift actions on all possible step matrices over} $\mathfrak{n}%
\cup\left\{  0\right\}  $.

\begin{example}
	Continuing Example \ref{scp}, the left-shift action $L_{\left\{  5\right\}
	}\left(  2|135|4\right)  =25|13|4$ on the right-hand factor of the SCP
	$a_{\sigma}\times b_{\sigma}=12|3|45\times2|135|4$ produces the CP
	$12|3|45\times25|13|4$ and induces the corresponding down-shift action on the
	step matrix
	\[
	C_{\sigma}=\left(
	\begin{tabular}
		[c]{lll}%
		$0$ & $0$ & $4$\\
		$1$ & $3$ & $5$\\
		$2$ & $0$ & $0$%
	\end{tabular}
	\ \right)  \overset{L_{\left\{  5\right\}  }}{\longrightarrow}\left(
	\begin{tabular}
		[c]{lll}%
		$0$ & $0$ & $4$\\
		$1$ & $3$ & $0$\\
		$2$ & $0$ & $5$%
	\end{tabular}
	\ \right)  \text{\ }.
	\]
	
\end{example}

Define $\Delta_P$ on cellular chains $C_{\ast}(P_{n})$ by
\begin{equation}
	\Delta_{P}(e^{n-1}):=\sum_{\substack{{e}_{1}\times e_{2}\in A_{\sigma}\times
			B_{\sigma}\\\sigma\in S_{n}}} sgn(e_{1},e_{2})\,\,e_{1}\otimes e_{2} ,
	\label{DeltaP}%
\end{equation}
where $sgn(e_{1}, e_{2})$ denotes the sign specified in \cite{SU2}, and on proper
cells by extending comultiplicatively. Define the boundary operator $\partial$
on $C_{\ast}(P_{n})$ by
\begin{equation}
	\partial e^{n-1}:=\sum_{A|B\sqsubset P_n}(-1)^{\#A}sgn(A|B)\,\,A|B \label{Psign}%
\end{equation}
and on proper cells by extending as a derivation, i.e.,
\[
\partial(A_{1}|\cdots|A_{k})=\sum(-1)^{\#(A_{1}\cup\cdots\cup A_{r-1}
	)+r+1}A_{1}|\cdots|A_{r-1}|\partial A_{r}|A_{r+1}|\cdots|A_{k}.
\]
Then $(C_{\ast}(P_{n});\partial,\Delta_{P})$ is a (non-coassociative) DG coalgebra.

Now identify the vertices of $P_{n+1}$ with the permutations in $S_{n+1}$ and extend the
\emph{weak order} on $S_{n+1}$ given by $\cdots|x_{i}|x_{i+1}|\cdots
<\cdots|x_{i+1}|x_{i}|\cdots$ if $x_{i}<x_{i+1}$ to a partial order (p-o).
Then the associated Hasse diagram orients the $1$-skeleton of $P_{n+1}$
\cite{CP}. Denote the minimal and maximal vertices of a face $e$ of $P_{n+1}$
by $\min e$ and $\max e$, respectively, and define $e\leq e^{\prime}$ if there
exists an oriented edge-path in $P_{n+1}$ from $\max e$ to $\min e^{\prime}$.

When $X$ is an $n$-dimensional polytope as above, the projection
$p:P_{n+1}\to X$ induces a p-o on the cells of $X$. For example, when the faces of
$P_{n+1}$ are indexed by planar leveled trees (PLTs) with $n+2$ leaves and the
faces of $K_{n+2}$ are indexed by planar rooted trees (PRTs) with $n+2$ leaves
(without levels), Tonks' projection $p=\theta$ given by forgetting levels
\cite{Tonks} induces the \emph{Tamari order} on the faces $\{\theta(T_{i})\}$
of $K_{n+2}$ given by $\theta(T_{i})\leq\theta(T_{j})$ if $T_{i}\leq T_{j}$.

Let $e$ be a cell of $X$ and let $\left\vert e\right\vert $ denote its
dimension$.$ A $k$\emph{-subdivision cube of }$e$ is a set of faces of $e$
whose union is a $k$-subcube of $I^{n}$ for some $k\leq n.$ For example, when
$e$ is the top dimensional cell of $P_{4},$ the facets in $\left\{
2|134,24|13\right\}  $ and $\{2|134,24|13,23|14,234|1\}$ form $2$-subdivision
cubes of $e$, but any three in the latter do not (see Figure 2). Denote the
set of vertices of $e$ by $\mathcal{V}_{e}$ (when $e=X$ we suppress the
subscript $e$). Given a vertex $v\in\mathcal{V}_{e},$ let $I_{v,1}^{k_{1}}$
and $I_{v,2}^{k_{2}}$ be $k_{i}$-subdivision cubes of $e$ such that $\max
I_{v,1}^{k_{1}}=\min I_{v,2}^{k_{2}}=v$ and $k_{1}+k_{2}=|e|;$ then $\left(
I_{v,1}^{k_{1}},I_{v,2}^{k_{2}}\right)  $ is a \emph{pair of} $\left(
k_{1},k_{2}\right)  $-\emph{subdivision cubes of }$e.$ Denote the set of all
such pairs by $e_{v}$ and let $(\mathbf{I}_{v,1}^{k_{1}},\mathbf{I}%
_{v,2}^{k_{2}})_{e}$ denote its unique maximal element; then $\left(
I_{v,3}^{k_{3}},I_{v,4}^{k_{4}}\right)  \subseteq(\mathbf{I}_{v,1}^{k_{1}%
},\mathbf{I}_{v,2}^{k_{2}})_{e}$ for all $(I_{v,3}^{k_{3}},I_{v,4}^{k_{4}})\in
e_{v}.$ For example, when $e$ is the top dimensional cell of $P_{4}$ and
$v=4|3|2|1,$ we have $(\mathbf{I}_{v,1}^{2},\mathbf{I}_{v,2}^{1})_{e}=\left(
\left\{  2|134,24|13\right\}  ,\left\{  4|23|1\right\}  \right)  $. For an
explicit description of $\left(  \mathbf{I}_{v,1}^{k_{1}},\mathbf{I}%
_{v,2}^{k_{2}}\right)  _{\!e}$ when $e\subseteq P_{n}$ see (\ref{maxpair1}) below.

If in addition, the cellular projection $p:P_{n+1}\rightarrow X$ preserves
maximal pairs of $\left(  k_{1},k_{2}\right)  $-subdivision cubes, i.e., for
every cell $e$ of $P_{n+1}$ we have
\[
p\left(  \mathbf{I}_{v,1}^{k_{1}},\mathbf{I}_{v,2}^{k_{2}}\right)
_{e}=\left(  \mathbf{I}_{p(v),1}^{k_{1}},\mathbf{I}_{p(v),2}^{k_{2}}\right)
_{p(e)},
\]
the components of the induced diagonal $\Delta_{X}$ on a cell $f$ of $X$ form
the set of product cells
\begin{equation}
	\Delta_{X}(f):=\bigcup_{\substack{\left(  e^{k_{1}},\,e^{k_{2}}\right)
			\in\left(  \mathbf{I}_{v,1}^{k_{1}},\mathbf{I}_{v,2}^{k_{2}}\right)
			_{\!f}\\v\in\mathcal{V}_{f}}}\{e^{k_{1}}\times e^{k_{2}}\}. \label{thepair}%
\end{equation}
In particular, $p=\theta$ preserves maximal pairs of $(k_{1},k_{2}%
)$-subdivision cubes and $\Delta_{K}(e)$ is given by setting $X=K_{n+2}$
(see (\ref{S-U}) below). Note that $\left(  e^{k_{1}},e^{k_{2}%
}\right)  \in\left(  \mathbf{I}_{v,1}^{k_{1}},\mathbf{I}_{v,2}^{k_{2}}\right)
_{\!X}$ implies $e^{k_{1}}\leq e^{k_{2}}.$ Thus $e^{k_{1}}\times e^{k_{2}}$ is
a \textquotedblleft matching pair\textquotedblright\ in the sense of Masuda,
Thomas, Tonks, and Vallette in \cite{MTTV} (see Definition \ref{mp}). Furthermore,
since $f=p\left(  e\right)  $ for some
$e=P_{n_{1}}\times\cdots\times P_{n_{s}}$ and $p\left(  e\right)  =p(P_{n_{1}%
})\times\cdots\times p(P_{n_{s}}),$ the diagonal $\Delta_{X}(f)$ is
automatically the comultiplicative extension of its values on the factors of
$f$, i.e., $\Delta_{X}(f)=\Delta_{X}(p(P_{n_{1}}))\times\cdots\times\Delta
_{X}(p(P_{n_{s}}))$.

When $X=P_{n+1}$, Formulas (\ref{Delta_P}) and (\ref{thepair}) are equivalent.
The maximal $\left(  k_{1},k_{2}\right)  $-subdivision pair with respect to a
vertex $\sigma$ of $P_{n+1}$ is\
\begin{equation}
	\left(  \mathbf{I}_{\sigma,1}^{k_{1}},\mathbf{I}_{\sigma,2}^{k_{2}}\right)
	=\left(  \bigcup_{e_{1}\in A_{\sigma}}e_{1},\bigcup_{e_{2}\in B_{\sigma}}%
	e_{2}\right)  . \label{maxpair1}%
\end{equation}

\begin{definition}
	A positive dimensional face $e$ of $P_{n}$ is \textbf{non-degenerate }if
	$|\theta(e)|=|e|$. A positive dimensional partition $a=A_{1}|\cdots|A_{p}\in
	P(\mathfrak{n})$ is \textbf{degenerate }if for some $j$ and some $k>0$, there exist
	$x,z\in A_{j}$ and $y\in A_{j+k}$ such that $x<y<z;$ otherwise $a$ is
	\textbf{non-degenerate}. A CP $\alpha\times\beta$ is \textbf{non-degenerate}
	if $\alpha$ and $\beta$ are non-degenerate.
\end{definition}

\noindent Define $\Delta_{K}(K_{n+1})=\Delta_{K}(\theta(P_{n})):=(\theta
\times\theta)\Delta_{P}(P_{n});$ then%
\begin{equation}
	\Delta_{K}(e^{n-1})=\bigcup_{\substack{\text{non-degenerate CPs }%
			\\\alpha\times\beta\in A_{\sigma}\times B_{\sigma}\\\sigma\in S_{n}}%
	}\{\theta\left(  \alpha\right)  \times\theta\left(  \beta\right)  \}.
	\label{S-U}%
\end{equation}

\begin{definition}
	\label{mp} A pair of faces $a\times b\subseteq K_{n+1}\times K_{n+1}$ is a
	\textbf{Matching Pair }(MP) if $a\leq b$ and $\left\vert a\right\vert
	+\left\vert b\right\vert =n-1.$
\end{definition}

\noindent J.-L. Loday proposed the following \textquotedblleft magical formula\textquotedblright\ for a diagonal $\Delta_K'$ on associahedra:
\begin{equation}
	\Delta_{K}^{\prime}\left(  e^{n-1}\right)  =\bigcup_{\substack{\text{MPs of
				faces}\\a\times b\subseteq K_{n+1}\times K_{n+1}}}\{a\times b\}.
	\label{magical}
\end{equation}
Formula (\ref{magical}), derived by Markl and Shnider in \cite{MS} and more recently by Masuda, Thomas, Tonks, and Vallette in \cite{MTTV}, agrees with Formula (\ref{S-U}) (see \cite{SU5}).

\subsection{The Subdivision Complex of a Diagonal Approximation}

Recall that every map of $CW$-complexes is homotopic to a cellular map and a
cellular map of $CW$-complexes induces a chain map of cellular chains. Let $X$
be a $CW$ complex, let $\Delta:X\rightarrow X\times X$ be the geometric
diagonal, and let $C_{\ast}\left(  X\right)  $ denote the cellular chains of
$X$. A \emph{diagonal approximation of }$\Delta$ is a cellular map $\Delta
_{X}$ homotopic to $\Delta,$ and the induced chain map $\Delta_{X}:C_{\ast
}\left(  X\right)  \rightarrow C_{\ast}\left(  X\times X\right)  \approx
C_{\ast}\left(  X\right)  \otimes C_{\ast}\left(  X\right)  $ is a
\emph{diagonal on }$C_{\ast}\left(  X\right)  $.

Let $X$ be a polytope. For each $k\geq1,$ there is a diagonal approximation
$\Delta_{X}$ whose \emph{(left)} $k$\emph{-fold iterate}
\[
\Delta_{X}^{\left(  k\right)  }:=\left(  \Delta_{X}\times{\mathbf{Id}}%
^{\times\left(  k-1\right)  }\right)  \cdots\left(  \Delta_{X}\times
{\mathbf{Id}}\right)  \Delta_{X}%
\]
is an embedding $X\hookrightarrow X^{\times\left(  k+1\right)  }$.
Given a diagonal approximation $\Delta_{X},$ define $\Delta_{X}^{\left(0\right)}
:=${$\mathbf{Id}$}. Then for each $k\geq0,$ there exists a unique cellular
complex $X^{\left(  k\right)  },$ called the $k^{th}$ \emph{subdivision}
\emph{complex} \emph{of} $X$ (\emph{with respect to} $\Delta_{X}$), with the
following property: Given a cell $e\sqsubseteq X$ and a subdivision cell
$e^{\prime}\sqsubseteq e$ in $X^{\left(  k\right)  },$ there exist unique
cells $u_{1},\ldots,u_{k+1}\sqsubseteq e$ such that $\Delta_{X}^{\left(
	k\right)  }\left(  e^{\prime}\right)  =u_{1}\times\cdots\times u_{k+1}$. Thus
$\Delta_{X}^{\left(  k\right)  }$ extends to a cellular inclusion $\Delta
_{X}^{(k)}:X^{(k)}\hookrightarrow X^{\times\left(  k+1\right)  },$ and
$X^{(k)}$ is the geometric representation of $\Delta_{X}^{(k)}\left(
X\right)  $ obtained by gluing the cells of $\Delta_{X}^{(k)}(X)$ together
along their common boundaries in the only possible way.

\begin{example}
	Let $I=\left[  0,1\right]  $ and consider the Serre diagonal $\Delta_{I}.$
	For  each $k\geq0$ we have
	\[
	\Delta_{I}^{\left(  k\right)  }\left(  I\right)  =\bigcup\limits_{i=0}
	^{k}0^{\times i}\times I\times1^{\times\left(  k-i\right)  },
	\]
	and the $k^{th}$ subdivision complex of $I$ is\
	\[
	I^{\left(  k\right)  }=\left[  1-2^{-k},1\right]  \cup\left[  1-2^{1-k}
	,1-2^{-k}\right]  \cup\cdots\cup\left[  0,1-2^{-1}\right]  .
	\]
	A subdivision $1$-cell
	\[
	e_{i}=\left\{
	\begin{array}
		[c]{ll}%
		\left[  1-2^{-k},1\right]  , & i=0\\
		\left[  1-2^{i-k},1-2^{i-k-1}\right]  , & 0<i\leq k
	\end{array}
	\right.
	\]
	is uniquely determined by the product cell $0^{\times i}\times I\times
	1^{\times k-i}\ $via $\Delta_{I}^{\left(  k\right)  }\left(  e_{i}\right)
	=0^{\times i}\times I\times1^{\times\left(  k-i\right)  }$ and a subdivision
	$0$-cell $1-2^{i-k},$ $0\leq i<k,$ is uniquely determined by the product cell
	$0^{\times i}\times1^{\times\left(  k-i+1\right)  }$ via $\Delta_{I}^{\left(
		k\right)  }\left(  1-2^{i-k}\right)  =0^{\times i}\times1^{\times\left(
		k-i+1\right)  }.$ Thus $\Delta_{I}^{\left(  k\right)  }\left(  I^{\left(
		k\right)  }\right)  =\bigcup\limits_{i=0}^{k}0^{\times i}\times I\times
	1^{\times\left(  k-i\right)  }$ and $\Delta_{I}^{\left(  k\right)  }$ extends
	to a cellular inclusion $I^{\left(  k\right)  }\hookrightarrow I^{\times
		\left(  k+1\right)  }.$ When $k=1$ we have
	\[
	I^{\left(  1\right)  }=\left[  \tfrac{1}{2},1\right]  \cup\left[  0,\tfrac
	{1}{2}\right]  \overset{\Delta_{I}^{\left(  1\right)  }}{\mapsto}I\times
	1\cup0\times I\hookrightarrow I^{2}.
	\]
	
\end{example}

Consider the (left) $k$-fold iterated diagonal ${\Delta}_{P}^{(k)}$ on the
permutahedron $P_{n}$. If $e$ is a cell of $P_{n}$ and $X\sqsubseteq\Delta
_{P}^{\left(  k\right)  }(e),$ then $\left\vert X\right\vert \leq\left\vert
e\right\vert .$ A \emph{diagonal component} \emph{of} ${\Delta}_{P}^{(k)}(e)$
is product cell $X\sqsubseteq\Delta_{P}^{\left(  k\right)  }(e)$ such that
$\left\vert X\right\vert =\left\vert e\right\vert ,$ in which case we write
$X\sqsubseteq_{diag}{\Delta}_{P}^{(k)}(e)$. For example, $1|3|2|4\times
13|24\sqsubseteq1|23|4\times13|24\sqsubseteq_{diag}\Delta_{P}\left(
P_{4}\right)  $ and $\left\vert 1|3|2|4\times13|24_{\mathstrut}^{\mathstrut
}\right\vert <\left\vert 1|23|4\times13|24_{\mathstrut}^{\mathstrut
}\right\vert =\left\vert P_{4}\right\vert .$

Let $\mathbf{x}=\alpha_{1}\times\cdots
\times\alpha_{q}$ be a product of partitions and write $\alpha_{i}%
:=A_{i1}|\cdots|A_{ir_{i}}.$ Given $i\in \mathfrak{q}$,
$k,\ell<r_{i},$ and $M^{k}\subseteq A_{ik}$, define
\[
\alpha_{j}^{\prime}:=A_{j1}^{\prime}|\cdots|A_{jr_{j}^{\prime}}^{\prime
}=\left\{
\begin{array}
	[c]{cc}%
	\partial_{M^{k}}\alpha_{i}, & j=i\\
	\alpha_{j}, & j\neq i,
\end{array}
\right.  \text{ where }r_{j}^{\prime}=\left\{
\begin{array}
	[c]{cc}%
	{r_{i}+1}, & j=i\\
	r_{j}, & j\neq i,
\end{array}
\right.
\]
\[
\alpha_{j}^{\prime\prime}:=A_{j1}^{\prime\prime}|\cdots|A_{jr_{j}
	^{\prime\prime}}^{\prime\prime}=\left\{
\begin{array}
	[c]{cc}%
	{\alpha_{i}[\ell]}, & j=i\\
	\alpha_{j}, & j\neq i,
\end{array}
\right.  \text{ where }r_{j}^{\prime\prime}=\left\{
\begin{array}
	[c]{cc}%
	{r_{i}-1}, & j=i\\
	r_{j}, & j\neq i,
\end{array}
\right.
\]
\[
\partial_{M^{k}}^{i}\mathbf{x}:=\alpha_{1}^{\prime}\times\cdots\times
\alpha_{q}^{\prime}\text{\ \ and\ \ }\mathbf{x}_{i\ell}:=\alpha_{1}^{\prime\prime
}\times\cdots\times\alpha_{q}^{\prime\prime}\text{.}%
\]
Then $\left(  \partial_{M^{k}}^{j}\mathbf{x}\right)  _{jk}=\partial_{A_{jk}%
}^{j}\mathbf{x}_{jk}=\mathbf{x}$ for all $j,$ $k,$ and $M^{k}.$ Note that if
$e$ is a cell of $P_{n},$ $n\geq2,$ and $\mathbf{x}\sqsubseteq_{diag}{\Delta
}_{P}^{(q)}(\partial e)\mathbf{,}$ these equations (with non-extreme $M^{k}$)
determine the cell structure of the boundary subdivision complex $\partial
P_{n}^{\left(  q\right)  }.$

\begin{proposition}
	\label{ab} Given a cell $e\sqsubseteq P_{n},$ $n\geq2,$ and a product
	cell $\mathbf{e:}=e_{1}\times\cdots\times e_{q}\sqsubseteq e^{\times q}$,
	write $e_{i}:=B_{i1}|\cdots|B_{ir_{i}}.$
	
	\begin{enumerate}

		\item Let $\mathbf{e}\sqsubseteq_{diag}{\Delta}_{P}^{(q-1)}(e)$ and
		$\mathbf{x}:=\partial_{M^{k}}^{i}\mathbf{e}\sqsubseteq\hspace*{-0.15in}
		\diagup_{diag}{\ \Delta}_{P}^{(q-1)}(\partial e)$ for some $i$ and some
		non-extreme $M^{k}\subset B_{ik}$; write $\mathbf{x}=e_{1}^{\prime}
		\times\cdots\times e_{q}^{\prime}$ and $e_{j}^{\prime}:=B_{j1}^{\prime}
		|\cdots|B_{jr_{j}^{\prime}}^{\prime}.$ There exist unique positive integers
		$j\leq q$ and $\ell<r_{j}^{\prime}$ such that $\mathbf{x}_{j\ell}
		\neq\mathbf{e}$ and $\mathbf{x}_{j\ell}\sqsubseteq_{diag}{\Delta}_{P}
		^{(q-1)}(e).$ \medskip
		
		\noindent$
		\begin{array}
			[c]{rcccl}%
			{\exists!\text{ }e_{j}} & \overset{proj\smallskip}{\longleftarrow} &
			\mathbf{e}=e_{1}\times\cdots\times e_{q}\medskip & \overset{\partial_{M^{k}
				}^{i}}{\longrightarrow} & \mathbf{x}=\partial_{B_{j\ell}^{\prime}}
			^{j}\mathbf{x}_{j\ell}\text{ }.\\
			& \multicolumn{1}{l}{} & \left\vert \bigsqcap\right.  _{diag}\medskip &
			\multicolumn{1}{l}{} & \nearrow_{\partial_{B_{j\ell}^{\prime}}^{j}}\\
			\mathbf{x}\sqsubseteq\hspace*{-0.13in}\diagup_{diag\text{ }}{\Delta}
			_{P}^{(q-1)}(\partial e) & \multicolumn{1}{l}{\overset{\partial
				}{\longleftarrow}} & {\Delta}_{P}^{(q-1)}(e)\sqsupseteq_{diag} &
			\multicolumn{1}{l}{\mathbf{x}_{j\ell}} &
		\end{array}
		\bigskip$
		
		\item If $\mathbf{x:}=\mathbf{e}\sqsubseteq_{diag}{\Delta}_{P}^{(q-1)}
		(\partial e)$, there exist unique positive integers $j\leq q$ and $\ell
		<r_{j}$  such that $\mathbf{x}_{j\ell}\sqsubseteq_{diag}\Delta_{P}%
		^{(q-1)}(e).$
		
		\noindent$
		\begin{array}
			[c]{cccc}
			& \mathbf{x}=e_{1}\times\cdots\times e_{q}\medskip & \overset{proj_{\smallskip
			}}{\longrightarrow} & e_{j}=\partial_{_{B_{j\ell}}}e_{j}[\ell]\\
			& \left\vert \bigsqcap\right.  _{diag}\medskip &  & \uparrow_{\partial
				_{B_{j\ell}}}\\
			\mathbf{x}_{j\ell}\sqsubseteq_{diag}\text{ }\Delta_{P}^{(q-1)}(e)\hspace
			{0.1in}\overset{\partial}{\longrightarrow} & {\Delta}_{P}^{(q-1)}(\partial
			e)\text{ } & \multicolumn{1}{l}{} & \exists!\text{ }e_{j}[\ell].
		\end{array}
		$\medskip
	\end{enumerate}
\end{proposition}

\begin{proof}
	The proof follows immediately from the proof of Theorem 1 in \cite{SU2}.
\end{proof}

\begin{example}
	Consider the top dimensional cell $e=123\sqsubseteq P_{3}$ and its diagonal
	component $\mathbf{e}=e_{1}\times e_{2}=13|2\times3|12.$ Set $\mathbf{x}
	=\partial_{\left\{  1\right\}  }^{1}\mathbf{e}=1|3|2\times3|12\sqsubseteq
	\hspace*{-0.13in}\diagup_{diag}$ $\Delta_{P}(\partial e)$ in Proposition
	\ref{ab}, part (1). By uniqueness, $\mathbf{x}_{12}=1|23\times3|12$ is the
	unique diagonal component of $e$ distinct from $\mathbf{e}$ that can be
	obtained from $\mathbf{x}$ by a single factor replacement. On the other hand,
	set $\mathbf{x}=3|1|2\times3|12\sqsubseteq_{diag}\Delta_{P}(\partial e)$ in
	Proposition \ref{ab}, part (2). By uniqueness, $\mathbf{x}_{11}=13|2\times
	3|12$ is the unique diagonal component of $e$ that can be obtained from
	$\mathbf{x}$ by a single factor replacement.
\end{example}

\section{Bipartition Matrices}\label{BM}

\subsection{Bipartition Matrices Defined}

Let $\mathbf{a}$ and $\mathbf{b}$ be ordered sets. A \emph{bipartition on
}$\left(  \mathbf{a,b}\right)  $ \emph{of length }$r$ is a pair $\left(
\alpha,\beta\right)  \in P_{r}^{\prime}(\mathbf{a})\times P_{r}^{\prime
}(\mathbf{b})$; it is \emph{elementary} when $r=1$. We will often denote a bipartition $\left(  \alpha,\beta\right)  $ as the
fraction $\beta/\alpha$ and its length $r$ by $l(\beta/\alpha).$ The $k^{th}$
\emph{biblock} of a bipartition $\frac{B_{1}|\cdots|B_{r}}{A_{1}|\cdots|A_{r}%
}$ is the pair $\left(  A_{k},B_{k}\right)  $; we refer to the components $A_{k}$ and
$B_{k}$ as the $k^{th}$ \emph{input} \emph{block} and\emph{\ the }$k^{th}$
\emph{output block}, respectively. Denote the $i^{th}$ row and
$j^{th}$ column of a matrix $C$ by $C_{i\ast}$ and $C_{\ast j}$.

\begin{definition}
	Given ordered sets $\mathbf{a}$ and $\mathbf{b,}$ let $\left(  \mathbf{a}%
	_{\ast},\mathbf{b}_{\ast}\right)  :=\left(  \mathbf{a}_{1}|\cdots
	|\mathbf{a}_{p},\mathbf{b}_{1}|\cdots|\mathbf{b}_{q}\right)  $ $\in P^{\prime
	}\left(  \mathbf{a}\right)  \times P^{\prime}\left(  \mathbf{b}\right)  ,$ let
	$\left(  r_{ij}\right)  $ be a $q\times p$ matrix of positive integers, and
	choose a bipartition $c_{ij}\in P_{r_{ij}}^{\prime}\left(  \mathbf{a}%
	_{j}\right)  \times P_{r_{ij}}^{\prime}\left(  \mathbf{b}_{i}\right)  $ for
	each $\left(  i,j\right)  ;$ then $C=\left(  c_{ij}\right)  $ is a $q\times p$
	\textbf{bipartition matrix over} $\left(  \mathbf{a}_{\ast},\mathbf{b}_{\ast
	}\right)  $. The sets $\mathbf{is}\left(  C\right)  :=\mathbf{a}$ and
	$\mathbf{os}\left(  C\right)  :=\mathbf{b}$ are the \textbf{input and output
		sets of }$C,$ respectively. A bipartition matrix $C=\left(  c_{ij}\right)  $
	is \textbf{elementary }if $c_{ij}=\frac{\mathbf{b}_{i}}{\mathbf{a}_{j}}$ for
	all $\left(  i,j\right)  ,$ \textbf{null }if $\mathbf{is}\left(  C\right)
	=\mathbf{os}\left(  C\right)  =\varnothing$, and \textbf{semi-null} if either
	
	\begin{enumerate}
		\item $\mathbf{is}\left(  C\right)  =\varnothing,$ $\mathbf{os}\left(
		C\right)  \neq\varnothing,$ and $C_{\ast1}=\cdots=C_{\ast p}$ or
		
		\item $\mathbf{is}\left(  C\right)  \neq\varnothing,$ $\mathbf{os}\left(
		C\right)  =\varnothing,$ and $C_{1\ast}=\cdots=C_{q\ast}$\textbf{. }
	\end{enumerate}
\end{definition}

\noindent Note that if $C$ is a $q\times p$ bipartition matrix over $\left(
\mathbf{a}_{\ast},\mathbf{b}_{\ast}\right)  $, all numerators in $C_{i\ast}$
lie in $P^{\prime}\left(  \mathbf{b}_{i}\right)  $ and all denominators in
$C_{\ast j}$ lie in $P^{\prime}\left(  \mathbf{a}_{j}\right)  $.
Thus, an elementary matrix has constant denominators in each column and
constant numerators in each row.

\subsection{Products of Bipartition Matrices}

Before we can impose a binary product on the set of bipartition matrices we
need some definitions that apply the $\lambda$-projection map $\mu_\lambda$ defined in (\ref{mu}).

\begin{definition}
	\label{equalizers}Let $C=\left(  \frac{\beta_{ij}}{\alpha_{ij}}\right)  $ be
	a  $q\times p$ bipartition matrix and let $r_{ij}=l\left(  \frac{\beta_{ij}
	}{\alpha_{ij}}\right)  $. A \textbf{row equalizer of }$C$ is a $q\times p$
	matrix of ordered sets $\lambda^{row}\left(  C\right)  =\left(  \lambda
	_{ij}^{row}\right)  ,$ $\lambda_{ij}^{row}\subseteq\mathfrak{r_{ij}
	}\smallsetminus\left\{  r_{ij}\right\}  ,$ such that for each $i$,
	$\#\lambda_{ij}^{row}=s_{i}$ is constant for all $j,$ and
	\begin{equation}
		\mu_{\lambda_{i1}^{row}}\left(  \beta_{i1}\right)  =\cdots=\mu_{\lambda
			_{ip}^{row}}\left(  \beta_{ip}\right)  \ \ \text{for each}\ \ i=1,2,\ldots,q;
		\label{roweq}
	\end{equation}
	it is \textbf{maximal }if $s_{i}$ is maximal for all $i$. A \textbf{column
		equalizer of }$C$ is a $q\times p$ matrix of ordered sets $\lambda
	^{col}\left(  C\right)  =\left(  \lambda_{ij}^{col}\right)  ,\,\lambda
	_{ij}^{col}\subseteq\mathfrak{r_{ij}}\smallsetminus\left\{  r_{ij}\right\}
	,$  such that for each $j$, $\#\lambda_{ij}^{col}=t_{j}$ is constant for all
	$i,$  and
	\begin{equation}
		\mu_{\lambda_{1j}^{col}}\left(  \alpha_{1j}\right)  =\cdots=\mu_{\lambda
			_{qj}^{col}}\left(  \alpha_{qj}\right)  \ \ \text{for each}\ \ j=1,2,\ldots,p;
		\label{coleq}%
	\end{equation}
	it is \textbf{maximal }if $t_{j}$ is maximal for all $j$. An \textbf{equalizer
		of }$C$ is a $q\times p$ matrix of ordered sets $\lambda\left(  C\right)
	=\left(  \lambda_{ij}\right)  ,\,\lambda_{ij}\subseteq\mathfrak{r_{ij}
	}\smallsetminus\left\{  r_{ij}\right\}  ,$ such that $\#\lambda_{ij}=r$ is
	constant for all $(i,j),$ and
	\begin{equation}
		\mu_{\lambda_{i1}}\left(  \beta_{i1}\right)  =\cdots=\mu_{\lambda_{ip}}\left(
		\beta_{ip}\right)  \text{ and }\mu_{\lambda_{1j}}\left(  \alpha_{1j}\right)
		=\cdots=\mu_{\lambda_{qj}}\left(  \alpha_{qj}\right)  \label{equa1}
	\end{equation}
	for each $\left(  i,j\right)  ;$ it is \textbf{maximal }if $r$ is maximal.
	
	The \textbf{row and column equalizations of}\emph{\ }$C$\textbf{\ with respect
		to} $\lambda^{row}\left(  C\right)  $ \textbf{and} $\lambda^{col}\left(
	C\right)  $ are the bipartition matrices
	\[
	\left(  \frac{\mu_{\lambda_{ij}^{row}}(\beta_{ij})}{\mu_{\lambda_{ij}^{row}
		}(\alpha_{ij})}\right)  \text{ and }\ \left(  \frac{\mu_{\lambda_{ij}^{col}
		}(\beta_{ij})}{\mu_{\lambda_{ij}^{col}}(\alpha_{ij})}\right)  .
	\]
	The \textbf{equalization of }$C$ \textbf{with respect to }$\lambda\left(
	C\right)  $ is the bipartition matrix
	\[
	\left(  \frac{\mu_{\lambda_{ij}}(\beta_{ij})}{\mu_{\lambda_{ij}}(\alpha_{ij}
		)}\right)  .
	\]
	
\end{definition}

Maximal equalizers are particularly important. However, when the entries of a
bipartition matrix contain null bipartition blocks, multiple maximal
equalizers can exist and produce different equalizations. For example, the
bipartition matrix%
\[
C=\left(
\begin{array}
	[c]{cc}%
	\frac{0|0|1}{1|2|0} & \frac{0|1}{0|3}%
\end{array}
\right)
\]
has two maximal row equalizers
\[
\lambda_{1}^{row}\left(  C\right)  =\left(
\begin{array}
	[c]{cc}%
	\left\{  1\right\}  & \left\{  1\right\}
\end{array}
\right)  \text{ and }\lambda_{2}^{row}\left(  C\right)  =\left(
\begin{array}
	[c]{cc}%
	\left\{  2\right\}  & \left\{  1\right\}
\end{array}
\right)
\]
with respective equalizations
\[
\left(
\begin{array}
	[c]{cc}%
	\frac{0|1}{1|2} & \frac{0|1}{0|3}%
\end{array}
\right)  \text{ and }\left(
\begin{array}
	[c]{cc}%
	\frac{0|1}{12|0} & \frac{0|1}{0|3}%
\end{array}
\right)  .
\]
Thus we distinguish one particular maximal (row/column) equalizer in the
following way: Given a $q\times p$ bipartition matrix $C$, denote its class of
maximal equalizers by $\mathcal{E}\left(  C\right)  $ and note that
$\lambda_{1}\neq\lambda_{2}\in\mathcal{E}\left(  C\right)  \ $implies $r>0.$
So assume $\lambda_{1}\neq\lambda_{2},$ list the entries of $\lambda_{i}$ in
row order, and write $\lambda_{1}=\left(  \lambda_{1}^{1},\ldots,\lambda
_{pq}^{1}\right)  $ and $\lambda_{2}=\left(  \lambda_{1}^{2},\ldots
,\lambda_{pq}^{2}\right)  .$ Define $\lambda_{1}<\lambda_{2}$ if and only if
for some $i$ and all $j<i,$ $\lambda_{j}^{1}=\lambda_{j}^{2}$ and $\lambda
_{i}^{1}<\lambda_{i}^{2}$ with respect to the lexicographic ordering. Order
the classes $\mathcal{E}^{row}\left(  C\right)  $ and $\mathcal{E}%
^{col}\left(  C\right)  $ of maximal row and maximal column equalizers in like
manner and denote the minimal elements of $\mathcal{E}\left(  C\right)  ,$
$\mathcal{E}^{row}\left(  C\right)  ,$ and $\mathcal{E}^{col}\left(  C\right)
,$ by $\Lambda\left(  C\right)  =\left(  \Lambda_{ij}\right)  ,$
$\Lambda^{row}\left(  C\right)  =\left(  \Lambda_{ij}^{row}\right)  ,$ and
$\Lambda^{col}\left(  C\right)  =\left(  \Lambda_{ij}^{col}\right)  ,$ respectively.

\begin{definition}
	Let $C=\left(  \frac{\beta_{ij}}{\alpha_{ij}}\right)  $ be a bipartition
	matrix. The \textbf{(canonical) maximal row and column equalizers} of $C$ are
	$\Lambda^{row}\left(  C\right)  $ and $\Lambda^{col}\left(  C\right)  ;$ the
	\textbf{(canonical) maximal row and column equalizations} of $C$ are
	\[
	{C}^{req}:=\left(  \frac{\mu_{\mathbf{\Lambda}_{ij}^{row}}(\beta_{ij})}
	{\mu_{\mathbf{\Lambda}_{ij}^{row}}(\alpha_{ij})}\right)  \text{ \ and }
	\ {C}^{ceq}:=\left(  \frac{\mu_{\mathbf{\Lambda}_{ij}^{col}}(\beta_{ij})}
	{\mu_{\mathbf{\Lambda}_{ij}^{col}}(\alpha_{ij})}\right)  .
	\]
	The \textbf{(canonical) maximal equalizer of }$C$ is $\Lambda\left(  C\right)
	;$ the \textbf{(canonical) maximal equalization of }$C$ is
	\[
	{C}^{eq}:=\left(  \frac{\mu_{\mathbf{\Lambda}_{ij}}(\beta_{ij})}
	{\mu_{\mathbf{\Lambda}_{ij}}(\alpha_{ij})}\right)  .
	\]
	
\end{definition}

\begin{example}
	\label{wedge-c}For
	\[
	C=\left(
	\begin{array}
		[c]{cc}%
		\frac{1|2|3}{1|2|3}\smallskip & \frac{1|2|3}{4|5|6}\\
		\frac{4|5|6}{1|2|3} & \frac{45|6}{45|6}%
	\end{array}
	\right)
	\]
	we have
	\[
	\Lambda^{row}\left(  C\right)  =\left(
	\begin{array}
		[c]{cc}%
		\left\{  1,2\right\}  \smallskip & \left\{  1,2\right\} \\
		\left\{  2\right\}  & \left\{  1\right\}
	\end{array}
	\right)  \text{ and }\,\, \Lambda^{col}\left(  C\right)  =\left(
	\begin{array}
		[c]{cc}%
		\left\{  1,2\right\}  \smallskip & \left\{  2\right\} \\
		\left\{  1,2\right\}  & \left\{  1\right\}
	\end{array}
	\right)
	\]
	so that
	\[
	{C}^{req}=\left(
	\begin{array}
		[c]{cc}%
		\frac{1|2|3}{1|2|3}\smallskip & \frac{1|2|3}{4|5|6}\\
		\frac{45|6}{12|3} & \frac{45|6}{45|6}%
	\end{array}
	\right)  \text{ and }\,\,{C}^{ceq}=\left(
	\begin{array}
		[c]{cc}%
		\frac{1|2|3}{1|2|3}\smallskip & \frac{12|3}{45|6}\\
		\frac{4|5|6}{1|2|3} & \frac{45|6}{45|6}%
	\end{array}
	\right)  .
	\]
	Furthermore,
	\[
	\Lambda\left(  C\right)  =\left(
	\begin{array}
		[c]{cc}%
		\left\{  2\right\}  & \left\{  2\right\} \\
		\left\{  2\right\}  & \left\{  1\right\}
	\end{array}
	\right)  \text{ and }C^{eq}=\left(
	\begin{array}
		[c]{cc}%
		\frac{12|3}{12|3}\smallskip & \frac{12|3}{45|6}\\
		\frac{45|6}{12|3} & \frac{45|6}{45|6}%
	\end{array}
	\right)  .
	\]
	
\end{example}

\begin{remark}
	\label{non-null}Since a maximal equalizer $\Lambda\left(  C\right)  $ is
	simultaneously a row and column equalizer, $\Lambda_{ij}\subseteq\Lambda
	_{ij}^{row}\cap\Lambda_{ij}^{col}$ for all $\left(  i,j\right)  .$
	Furthermore, when an equalizer $\lambda\left(  C\right)  =\left(  \lambda
	_{ij}\right)  $ is non-null, all entries in the corresponding equalization
	have constant length greater than $1,$ and consequently, $\lambda_{ij}
	^{row}\cap\lambda_{ij}^{col}\neq\varnothing$ for all $(i,j).$ However, as our
	next example demonstrates, $\lambda_{ij}^{row}\cap\lambda_{ij}^{col}
	\neq\varnothing$ for all $(i,j)$ does not imply the existence of a non-null equalizer.
\end{remark}

\begin{example}
	Although the maximal equalizer $\Lambda\left(  C\right)  $ of the bipartition
	matrix
	\[
	C=\left(
	\begin{array}
		[c]{cc}%
		\frac{1|2|3}{1|2|3}\smallskip & \frac{1|3|2}{4|5|6}\\
		\frac{4|5|6}{1|2|3} & \frac{5|4|6}{4|5|6}%
	\end{array}
	\right)
	\]
	is null, we have
	\[
	\Lambda^{row}\left(  C\right)  =\left(
	\begin{array}
		[c]{cc}%
		\left\{  1\right\}  & \left\{  1\right\} \\
		\left\{  2\right\}  & \left\{  2\right\}
	\end{array}
	\right)  \text{ and }\Lambda^{col}\left(  C\right)  =\left(
	\begin{array}
		[c]{cc}%
		\left\{  1,2\right\}  & \left\{  1,2\right\} \\
		\left\{  1,2\right\}  & \left\{  1,2\right\}
	\end{array}
	\right)
	\]
	so that $\Lambda_{ij}^{row}\cap\Lambda_{ij}^{col}\neq\varnothing$ for all
	$(i,j).$
\end{example}

Let $C$ be a $q\times p$ bipartition matrix $C$ with a non-null maximal
equalizer $\Lambda\left(  C\right)  =\left(  \Lambda_{ij}\right)  .$ Write
$\Lambda_{ij}=\{\Lambda_{ij}^{1}<\cdots<\Lambda_{ij}^{r}\},$ $\Lambda
_{ij}^{row}=\{\left(  \Lambda_{ij}^{row}\right)  ^{1}<\cdots<\left(
\Lambda_{ij}^{row}\right)  ^{s_{i}}\},$ and $\Lambda_{ij}^{col}=\{\left(
\Lambda_{ij}^{col}\right)  ^{1}<\cdots<\left(  \Lambda_{ij}^{col}\right)
^{t_{j}}\}.$ By Remark \ref{non-null}, there is a subsequence $\{v_{i}%
^{1},\ldots,v_{i}^{k_{i}}\}$ $\subseteq\left\{  1,2,\ldots,s_{i}\right\}  $
such that $\Lambda_{i1}=\{\left(  \Lambda_{i1}^{row}\right)  ^{v_{i}^{1}%
}<\cdots<\left(  \Lambda_{i1}^{row}\right)  ^{v_{i}^{k_{i}}}\}$ for each
$i\leq q,$ and a subsequence $\{w_{j}^{1},\ldots,w_{j}^{l_{j}}\}\subseteq
\left\{  1,2,\ldots,l_{j}\right\}  $ such that $\Lambda_{1j}=\{\left(
\Lambda_{1j}^{col}\right)  ^{w_{j}^{1}}<\cdots<\left(  \Lambda_{1j}%
^{col}\right)  ^{w_{j}^{l_{j}}}\}$ for each $j\leq p.$ The following
proposition gives a method for constructing the maximal equalizer from the
maximal row and maximal column equalizers:

\begin{proposition}
	\label{uniqueness}Given a $q\times p$ bipartition matrix $C$ with a non-null
	maximal equalizer $\Lambda\left(  C\right)  =\left(  \Lambda_{ij}\right)  ,$
	write $\Lambda_{ij}=\{\Lambda_{ij}^{1}<\cdots<\Lambda_{ij}^{r}\},$
	$\Lambda_{i1}=\{\left(  \Lambda_{i1}^{row}\right)  ^{v_{i}^{1}}<\cdots<\left(
	\Lambda_{i1}^{row}\right)  ^{v_{i}^{k_{i}}}\}$ for each $i\leq q,$ and
	$\Lambda_{1j}=\{\left(  \Lambda_{1j}^{col}\right)  ^{w_{j}^{1}}<\cdots<\left(
	\Lambda_{1j}^{col}\right)  ^{w_{j}^{l_{j}}}\}$ for each $j\leq p.$ Then for
	all $\left(  i,j\right)  ,$
	
	\begin{enumerate}

		\item $r:=k_{i}=l_{j}$ is constant and
		
		\item $\Lambda_{ij}=\{\left(  \Lambda_{ij}^{row}\right)  ^{v_{i}^{1}}
		<\cdots<\left(  \Lambda_{ij}^{row}\right)  ^{v_{i}^{r}}\}=\{\left(
		\Lambda_{ij}^{col}\right)  ^{w_{j}^{1}}<\cdots<\left(  \Lambda_{ij}
		^{col}\right)  ^{w_{j}^{r}}\}.$
	\end{enumerate}
\end{proposition}

\begin{proof}
	First, $\Lambda_{ij}^{row}\cap\Lambda_{ij}^{col}\neq\varnothing$ for all
	$(i,j)$ by Remark \ref{non-null}. Then $\Lambda_{11}^{row}\cap\Lambda
	_{11}^{col}=\{\left(  \Lambda_{11}^{row}\right)  ^{m_{1}}<\cdots<\left(
	\Lambda_{11}^{row}\right)  ^{m_{s}}\}=\{\left(  \Lambda_{11}^{col}\right)
	^{n_{1}}<\cdots<\left(  \Lambda_{11}^{col}\right)  ^{n_{s}}\}.$ Consider the
	corresponding subsets $\{\left(  \Lambda_{1j}^{row}\right)  ^{m_{1}}
	<\cdots<\left(  \Lambda_{1j}^{row}\right)  ^{m_{s}}\}\subseteq\Lambda
	_{1j}^{row}$ for each $j$ and $\{\left(  \Lambda_{i1}^{col}\right)  ^{n_{1}%
	}<$  $\left.  \cdots<\left(  \Lambda_{i1}^{col}\right)  ^{n_{s}}\right\}
	\subseteq\Lambda_{i1}^{col}$ for each $i,$ and let $r_{1}$ be the smallest
	positive integer such that
	
	\begin{itemize}

		\item $(\Lambda_{1j}^{row})^{m_{r_{1}}}\in\Lambda_{1j}^{col}$ for all $j,$
		
		\item $(\Lambda_{i1}^{col})^{n_{r_{1}}}\in\Lambda_{i1}^{row}$ for all $i,$ and
		
		\item $(\Lambda_{ij}^{col})^{m_{r_{1}}}=(\Lambda_{ij}^{row})^{n_{r_{1}}}$ for
		all $\left(  i,j\right)  .$
	\end{itemize}
	
	\noindent Then $n_{r_{1}}=v_{i}^{1}$ for all $i,$ $m_{r_{1}}=w_{j}^{1}$ for
	all $j$, and $\Lambda_{ij}^{1}=\left(  \Lambda_{ij}^{row}\right)  ^{v_{i}^{1}
	}=\left(  \Lambda_{ij}^{col}\right)  ^{w_{j}^{1}}$ for all $\left(
	i,j\right)  .$
	
	Inductively, assume that for some $k>1,$ the set $\{\Lambda_{ij}^{1}
	<\cdots<\Lambda_{ij}^{k-1}\}$ has been constructed for each $\left(
	i,j\right)  .$ Let $r_{k}>r_{k-1}$ be the smallest integer such that
	
	\begin{itemize}

		\item $(\Lambda_{1j}^{row})^{m_{r_{k-1}}}<(\Lambda_{1j}^{row})^{m_{r_{k}}}
		\in\Lambda_{1j}^{col}$ for all $j,$
		
		\item $(\Lambda_{i1}^{col})^{n_{r_{k-1}}}<(\Lambda_{i1}^{col})^{n_{r_{k}}}
		\in\Lambda_{i1}^{row}$ for all $i,$ and
		
		\item $(\Lambda_{ij}^{col})^{m_{r_{k}}}=(\Lambda_{ij}^{row})^{n_{r_{k}}}$ for
		all $\left(  i,j\right)  .$
	\end{itemize}
	
	\noindent Then $n_{r_{k}}=v_{i}^{k}$ for all $i,$ $m_{r_{k}}=w_{j}^{k}$ for
	all $j$, and $\Lambda_{ij}^{k}=(\Lambda_{ij}^{row})^{v_{i}^{k}}=(\Lambda
	_{ij}^{col})^{w_{j}^{k}}$ for all $\left(  i,j\right)  .$ The induction
	terminates after $r$ steps and produces (2).
\end{proof}

\begin{definition}
	\label{transverse-decomp}Given a bipartition $c=\frac{B_{1}|\cdots|B_{r}}{A_{1}|\cdots|A_{r}
	}$  over $\left(  \mathbf{a,b}\right)  $ with $r>1,$ let
	$\lambda\in\left\{  1,2,\ldots,r-1\right\}  ,$ let
	$\mathbf{a}_{1}|\cdots|\mathbf{a}_{p}:=EP_{\mathbf{a}}\left(  A_{\lambda+1}
	\cup\cdots\cup A_{r}\right)$, and let $\mathbf{b}_{1}|\cdots|\mathbf{b}
	_{q}:= EP_{\mathbf{b}}\left(  B_{1}\cup\cdots\cup B_{\lambda}\right).	$
	For each $(i,j)\in \mathfrak{q}\times \mathfrak{p}$, let $	B_{1}^{j}|\cdots|B_{\lambda}^{j}:=(B_{1}\cap\mathbf{b}_{j})|\cdots|(B_{\lambda
	}\cap\mathbf{b}_{j})$ and let  $A_{\lambda+1}^{i}|\cdots|A_{r}^{i}
	:=(A_{\lambda+1}\cap\mathbf{a}_{i})|\cdots|(A_{r}\cap\mathbf{a}_{i}).$
	The \textbf{transverse decomposition of }$c$ \textbf{with respect to}
	$\lambda$ is the formal product
	\begin{equation}
		A\cdot B=\left(
		\begin{tabular}
			[c]{c}%
			$\frac{B_{1}^{1}|\cdots|B_{\lambda}^{1}}{A_{1}|\cdots|A_{\lambda}}$\\
			$\vdots$\\
			$\frac{B_{1}^{q}|\cdots|B_{\lambda}^{q}}{A_{1}|\cdots|A_{\lambda}}$%
		\end{tabular}
		\right)  \left(
		\begin{array}
			[c]{c}%
			\frac{B_{\lambda+1}|\cdots|B_{r}}{A_{\lambda+1}^{1}|\cdots|A_{r}^{1}}\text{
			}\cdots\text{ }\frac{B_{\lambda+1}|\cdots|B_{r}}{A_{\lambda+1}^{p}
				|\cdots|A_{r}^{p}}%
		\end{array}
		\right)  . \label{deco1}%
	\end{equation}
	
\end{definition}

\noindent It follows immediately that
\begin{equation}%
	\begin{array}
		[c]{rll}%
		B_{1}|\cdots|B_{\lambda} & = & B_{1}^{1}|\cdots|B_{\lambda}^{1}\Cup\cdots\Cup
		B_{1}^{q}|\cdots|B_{\lambda}^{q}\text{ and}\vspace{1mm}\\
		A_{\lambda+1}|\cdots|A_{r} & = & A_{\lambda+1}^{1}|\cdots|A_{r}^{1}\Cup
		\cdots\Cup A_{\lambda+1}^{p}|\cdots|A_{r}^{p}.
	\end{array}
	\label{deco2}%
\end{equation}

\begin{definition}
	\label{TP}A pair of bipartition matrices $\left(  A^{q\times1},B^{1\times
		p}\right)  $ is a \textbf{Transverse Pair}\emph{\ }(TP) if $A\cdot B$ is the
	transverse decomposition of some bipartition. A pair of bipartition matrices
	$(A^{q\times s},B^{t\times p})$ is a\textbf{\ Block Transverse Pair} (BTP) if
	there exist $t\times s$ block decompositions $A=\left(  A_{ij}\right)  $ and
	$B=\left(  B_{ij}\right)  $ such that $\left(  A_{ij},B_{ij}\right)  $ is a
	TP  for all $\left(  i,j\right)  $. When $\left(  A,B\right)  $ is a BTP, the
	value of the \textbf{formal product }$A\cdot B$ is the \textbf{formal matrix}
	$AB:=\left(  A_{ij}B_{ij}\right)  ,$ which is a bipartition matrix if
	$A_{ij}B_{ij}$ is a bipartition for all $\left(  i,j\right)  $. When
	$C_{1}\cdots C_{r}$ is a formal product of bipartition matrices,
	$(C_{k},C_{k+1})$ is a BTP for each $k.$ A bipartition matrix $C$ is
	\textbf{indecomposable} if $\Lambda\left(  C\right)  $ is null; otherwise $C$
	is \textbf{decomposable. }A \textbf{factorization}\emph{\ }of $C$ is a formal
	product $C_{1}\cdots C_{r}$ such that $C=C_{1}\cdots C_{r}$ and $C_{k}$ is a
	bipartition matrix for all $k;$ it is \textbf{indecomposable} if\textbf{
	}$C_{k}$ is indecomposable for all $k.$
\end{definition}

\noindent If $\left(  A^{q\times1},B^{1\times p}\right)  $ is a TP, formula
(\ref{relative-complement}) implies $\left(  \mathbf{\#is}(A)\,,\#\mathbf{os}%
(B)\right)  =(p-1,q-1)$. Furthermore, if $(A^{q\times s},B^{t\times p})$ is a
BTP with $A=\left(  A_{ij}^{q_{ij}\times1}\right)  $ and $B=\left(
B_{ij}^{1\times p_{ij}}\right)  $, then $(\#\mathbf{is}(A_{ij}%
)\,,\#\mathbf{os}(B_{ij}))=(p_{ij}-1,q_{ij}-1)$ so that $\#\mathbf{is}\left(
A\right)  =\sum\nolimits_{j}\#\mathbf{is}\left(  A_{ij}\right)  =\sum
\nolimits_{j}p_{ij}-s=p-s$ and $\#\mathbf{os}\left(  B\right)  =\sum
\nolimits_{i}\#\mathbf{os}\left(  B_{ij}\right)  =\sum\nolimits_{i}%
q_{ij}-t=q-t.$ Thus%
\begin{equation}
	(\#\mathbf{is}(A^{q\times s})\,,\#\mathbf{os}(B^{t\times p}))=(p-s,q-t).
	\label{input-output}%
\end{equation}
When $C=C_{1}\cdots C_{r}$ define
\[
\mathbf{is}(C):=\bigcup_{k\in\mathfrak{r}}\mathbf{is}(C_{k})\ \ \text{and}%
\ \ \mathbf{os}(C):=\bigcup_{k\in\mathfrak{r}}\mathbf{os}(C_{k}).
\]
Our next proposition is immediate.

\begin{proposition}
	\label{decomposition-formula}If a bipartition matrix $C$ has a non-null
	equalizer $\lambda\left(  C\right)  =\left(  \lambda_{ij}\right)  $, there is
	a factorization $C=C_{1}\cdots C_{r+1}$, where $r=\#\lambda_{ij}.$ If
	$\lambda\left(  C\right)  =\Lambda\left(  C\right)  $, the factorization
	$C=C_{1}\cdots C_{r+1}$ is unique and indecomposable.
\end{proposition}

\begin{example}
	\label{decomposable-ex1}The bipartition matrix
	\[
	C=\left(  c_{ij}\right)  =\left(
	\begin{array}
		[c]{cc}%
		\frac{2|3|0|4}{1|3|2|0} & \frac{23|0|4}{7|5|6}\vspace{1mm}\\
		\frac{7|56|0}{13|0|2} & \frac{7|0|56}{5|7|6}%
	\end{array}
	\right)
	\]
	over $(\mathbf{a}_{1}=\{1,2,3\},$ $\mathbf{a}_{2}=\{5,6,7\},$ $\mathbf{b}
	_{1}=\{2,3,4\},$ $\mathbf{b}_{2}=\{5,6,7\})$ has maximal equalizer
	\[
	\left(  \Lambda_{ij}\right)  =\left(
	\begin{array}
		[c]{cc}%
		\left\{  2\right\}  & \left\{  2\right\} \\
		\left\{  1\right\}  & \left\{  2\right\}
	\end{array}
	\right)  .
	\]
	To compute the unique indecomposable factorization, apply the formulas in
	Definition \ref{transverse-decomp} with $c=c_{ij}$ and $\lambda =\Lambda
	_{ij}$ for each $\left( i,j\right) $:%
	\[
	\begin{array}{lll}
		c_{11}:\left\{
		\begin{array}{l}
			\mathbf{b}_{1}^{1}|\mathbf{b}_{2}^{1}=23|0 \\
			\mathbf{a}_{1}^{1}|\mathbf{a}_{2}^{1}|\mathbf{a}_{3}^{1}=0|2|0%
		\end{array}%
		\right.  & \Rightarrow  &
		\begin{array}{l}
			B_{1}^{1}|B_{2}^{1}=2|3,\text{ }B_{1}^{2}|B_{2}^{2}=0|0 \\
			A_{3}^{1}|A_{4}^{1}=0|0,\text{ }A_{3}^{2}|A_{4}^{2}=2|0,\text{ }%
			A_{3}^{3}|A_{4}^{2}=0|0,%
		\end{array}
		\\
		&  &  \\
		c_{12}:\left\{
		\begin{array}{l}
			\mathbf{b}_{1}^{1}|\mathbf{b}_{2}^{1}=23|0 \\
			\mathbf{a}_{1}^{2}|\mathbf{a}_{2}^{2}|\mathbf{a}_{3}^{2}=0|6|0%
		\end{array}%
		\right.  & \Rightarrow  &
		\begin{array}{l}
			B_{1}^{1}|B_{2}^{1}=23|0,\text{ }B_{1}^{2}|B_{2}^{2}=0|0 \\
			A_{3}^{1}=0,\text{ }A_{3}^{2}=6,\text{ }A_{3}^{3}=0,%
		\end{array}
		\\
		&  &  \\
		c_{21}:\left\{
		\begin{array}{l}
			\mathbf{b}_{1}^{2}|\mathbf{b}_{2}^{2}|\mathbf{b}_{3}^{2}=0|07 \\
			\mathbf{a}_{1}^{1}|\mathbf{a}_{2}^{1}|\mathbf{a}_{3}^{1}=0|2|0%
		\end{array}%
		\right.  & \Rightarrow  &
		\begin{array}{l}
			B_{1}^{1}=0,\text{ }B_{1}^{2}=0,\text{ }B_{1}^{3}=7 \\
			A_{2}^{1}|A_{3}^{1}=0|0,\text{ }A_{2}^{2}|A_{3}^{2}=0|2,\text{ }%
			A_{2}^{3}|A_{3}^{3}=0|0,%
		\end{array}
		\\
		&  &  \\
		c_{22}:\left\{
		\begin{array}{l}
			\mathbf{b}_{1}^{2}|\mathbf{b}_{2}^{2}|\mathbf{b}_{3}^{2}=0|0|7 \\
			\mathbf{a}_{1}^{2}|\mathbf{a}_{2}^{2}|\mathbf{a}_{3}^{2}=0|6|0%
		\end{array}%
		\right.  & \Rightarrow  &
		\begin{array}{l}
			B_{1}^{1}|B_{2}^{1}=0|0,\text{ }B_{1}^{2}|B_{2}^{2}=0|0,\text{ }%
			B_{1}^{3}|B_{2}^{3}=7|0 \\
			A_{3}^{1}=0,\text{ }A_{3}^{2}=6,\text{ }A_{3}^{3}=0.%
		\end{array}%
	\end{array}%
	\]
	Then
	\[
	\left(  A_{ij}B_{ij}\right)  =\left(
	\begin{array}
		[c]{lll}%
		\left(
		\begin{array}
			[c]{c}%
			\frac{2|3}{1|3}\vspace{1mm}\\
			\frac{0|0}{1|3}%
		\end{array}
		\right)  \left(
		\begin{array}
			[c]{ccc}%
			\frac{0|4}{0|0} & \frac{0|4}{2|0} & \frac{0|4}{0|0}%
		\end{array}
		\right)  &  & \left(
		\begin{array}
			[c]{c}%
			\frac{23|0}{7|5}\vspace{1mm}\\
			\frac{0|0}{7|5}%
		\end{array}
		\right)  \left(
		\begin{array}
			[c]{ccc}%
			\frac{4}{0} & \frac{4}{6} & \frac{4}{0}%
		\end{array}
		\right) \\
		&  & \\
		\left(
		\begin{array}
			[c]{c}%
			\frac{0}{13}\smallskip\\
			\frac{0}{13}\smallskip\\
			\frac{7}{13}%
		\end{array}
		\right)  \left(
		\begin{array}
			[c]{ccc}%
			\frac{56|0}{0|0} & \frac{56|0}{0|2} & \frac{56|0}{0|0}%
		\end{array}
		\right)  &  & \left(
		\begin{array}
			[c]{c}%
			\frac{0|0}{5|7}\vspace{1mm}\\
			\frac{0|0}{5|7}\vspace{1mm}\\
			\frac{7|0}{5|7}%
		\end{array}
		\right)  \left(
		\begin{array}
			[c]{ccc}%
			\frac{56}{0} & \frac{56}{6} & \frac{56}{0}%
		\end{array}
		\right)
	\end{array}
	\right)
	\]
	and
	\[
	C=AB=\left(
	\begin{array}
		[c]{cc}%
		\frac{2|3}{1|3} & \frac{23|0}{7|5}\vspace{1mm}\\
		\frac{0|0}{1|3} & \frac{0|0}{7|5}\vspace{1mm}\\
		\frac{0}{13} & \frac{0|0}{5|7}\vspace{1mm}\\
		\frac{0}{13} & \frac{0|0}{5|7}\vspace{1mm}\\
		\frac{7}{13} & \frac{7|0}{5|7}%
	\end{array}
	\right)  \left(
	\begin{array}
		[c]{cccccc}%
		\frac{0|4}{0|0} & \frac{0|4}{2|0} & \frac{0|4}{0|0} & \frac{4}{0} & \frac
		{4}{6} & \frac{4}{0}\vspace{1mm}\\
		\frac{56|0}{0|0} & \frac{56|0}{0|2} & \frac{56|0}{0|0} & \frac{56}{0} &
		\frac{56}{6} & \frac{56}{0}%
	\end{array}
	\right)  .
	\]

	\noindent The matrix dimensions $\left(  q,s\right)  =\left(  5,2\right)  $
	and $\left(  t,p\right)  =\left(  2,6\right)  $ together with $\left(
	\#\mathbf{is}\left(  A\right)  \right.  ,$ $\left.  \#\mathbf{os}\left(
	B\right)  \right)  =\left(  4,3\right)  $ verify Formula (\ref{input-output}%
	).  To recover the matrix $C,\ $apply the formulas in (\ref{deco2}) and
	obtain  $c_{ij}=A_{ij}B_{ij}.$ Then $\mathbf{is}\left(  C\right)  =\left\{
	1,2,3\right\}  \cup\left\{  5,6,7\right\}  =\left\{  1,3,5,7\right\}
	\cup\left\{  2,6\right\}  =\mathbf{is}\left(  A\right)  \cup\mathbf{is}\left(
	B\right)  ,$ and $\mathbf{os}\left(  C\right)  =\left\{  2,3,4\right\}
	\cup\left\{  5,6,7\right\}  =\left\{  2,3,7\right\}  \cup\left\{
	4,5,6\right\}  =\mathbf{os}\left(  A\right)  \cup\mathbf{os}\left(  B\right)
	$ as required. Note that the second row of $A$ factors as
	\[
	\left(
	\begin{array}
		[c]{cc}%
		\frac{0|0}{1|3} & \frac{0|0}{7|5}%
	\end{array}
	\right)  =\left(
	\begin{array}
		[c]{cc}%
		\frac{0}{1} & \frac{0}{7}%
	\end{array}
	\right)  \left(
	\begin{array}
		[c]{cccc}%
		\frac{0}{0} & \frac{0}{3} & \frac{0}{5} & \frac{0}{0}%
	\end{array}
	\right)  .
	\]
	Thus the rows (and columns) of an indecomposable matrix may be decomposable.
\end{example}

\noindent The indecomposable factorization of a bipartition is given by

\begin{algorithm}
	Let $\frac{B_{1}|\cdots|B_{r}}{A_{1}|\cdots|A_{r}}$ be a bipartition with
	$r>1$.
	
	For $k=1$ to $r:$
	
	\qquad Let $\mathbf{a}_{k1}|\cdots|\mathbf{a}_{kp_{k}}\,:=EP_{A_{1}\cup
		\cdots\cup A_{k}}A_{k}.$
	
	\qquad Let $\mathbf{b}_{k1}|\cdots|\mathbf{b}_{kq_{k}}:=EP_{B_{k}\cup
		\cdots\cup B_{r}}B_{k}.$
	
	\qquad Form the $q_{k}\times p_{k}$ elementary matrix $C_{k}=\left(
	\frac{\mathbf{b}_{ki}}{\mathbf{a}_{kj}}\right)  .$
	
	Obtain the unique factorization
	\begin{equation}
		\frac{B_{1}|\cdots|B_{r}}{A_{1}|\cdots|A_{r}}=C_{1}\cdots C_{r}.
		\label{bipart-factor}%
	\end{equation}
	
\end{algorithm}

\begin{example}
	\label{decomposable-ex2}To compute the indecomposable factorization of the
	bipartition $\frac{56|7|8}{1|23|4},$ set
	\[
	\begin{array}
		[c]{cc}%
		\mathbf{a}_{11}=1 & \mathbf{b}_{11}|\mathbf{b}_{12}|\mathbf{b}_{13}=56|0|0\\
		\mathbf{a}_{21}|\mathbf{a}_{22}=0|23 & \mathbf{b}_{21}|\mathbf{b}_{22}=7|0\\
		\mathbf{a}_{31}|\mathbf{a}_{32}|\mathbf{a}_{33}|\mathbf{a}_{34}=0|0|0|4 &
		\mathbf{b}_{31}=8.
	\end{array}
	\]
	Then
	\[
	\frac{56|7|8}{1|23|4}=\left(
	\begin{array}
		[c]{c}%
		\frac{56}{1}\smallskip\\
		\frac{0}{1}\smallskip\\
		\frac{0}{1}%
	\end{array}
	\right)  \left(
	\begin{array}
		[c]{cc}%
		\frac{7}{0} & \frac{7}{23}\smallskip\\
		\frac{0}{0} & \frac{0}{23}%
	\end{array}
	\right)  \left(
	\begin{array}
		[c]{cccc}%
		\frac{8}{0} & \frac{8}{0} & \frac{8}{0} & \frac{8}{4}%
	\end{array}
	\right)  =\mathbf{C}_{1}\mathbf{C}_{2}\mathbf{C}_{3}.
	\]
	
\end{example}

\noindent The fact that $\mathbf{C}_{1},$ $\mathbf{C}_{2},$ and $\mathbf{C}_{3}$  are elementary matrices illustrates the fact that the indecomposable factorization of a
bipartition is an elementary product.

\subsection{Partitioning the Entries of a Bipartition Matrix}

\label{partitioning}
An essential action on a bipartition matrix $C$ is to partition its entries.
Such a partitioning is called a \emph{partitioning action on} $C$. Let
$\frac{\beta}{\alpha}=\frac{B_{1}|\cdots|B_{r}}{A_{1}|\cdots|A_{r}}$ be a
bipartition and let $\left(  M^{k},N^{k}\right)  \subseteq\left(  A_{k}%
,B_{k}\right)  $ for some $k.$ The \emph{(partitioning) action of }$\left(
M^{k},N^{k}\right)  $ \emph{on }$\frac{\beta}{\alpha}$ is the bipartition
\begin{equation}
	\partial_{M^{k},N^{k}}\left(  \frac{\beta}{\alpha}\right)  :=\frac
	{\partial_{N^{k}}\beta}{\partial_{M^{k}}\alpha} \label{biface1}%
\end{equation}
(see \ref{partitioning-action}). The pair $\left(  M^{k},N^{k}\right)  $ is
\emph{extreme}\textbf{\ }if $M^{k}$ and $N^{k}$ are extreme, and is
\emph{strongly extreme} when $\left(  M^{k},N^{k}\right)  =\left(
\varnothing,\varnothing\right)  $ or $\left(  M^{k},N^{k}\right)  =\left(
A_{k},B_{k}\right)  .$ When $\left(  M^{1},N^{1}\right)  =\left(
\varnothing,\varnothing\right)  $ or $\left(  M^{r},N^{r}\right)  =\left(
A_{r},B_{r}\right)  ,$ we denote the special strongly extreme cases
$\partial_{M^{1},N^{1}}$ and $\partial_{M^{r},N^{r}}$ by $\eta_{1}$ and
$\eta_{2}$ respectively; thus
\[
\eta_{1}\left(  \frac{\beta}{\alpha}\right)  =\frac{0|B_{1}|\cdots|B_{r}%
}{0|A_{1}|\cdots|A_{r}}\ \ \text{and}\ \ \eta_{2}\left(  \frac{\beta}{\alpha
}\right)  =\frac{B_{1}|\cdots|B_{r}|0}{A_{1}|\cdots|A_{r}|0}.
\]
Given a positive integer $k<r,$ let
\[
\frac{\beta}{\alpha}[k]:=\frac{\beta\lbrack k]}{\alpha\lbrack k]}=\frac
{B_{1}|\cdots|B_{k}\cup B_{k+1}|\cdots|B_{r}}{A_{1}|\cdots|A_{k}\cup
	A_{k+1}|\cdots|A_{r}};
\]
then $\partial_{A_{k},B_{k}}\left(  \frac{\beta}{\alpha}[k]\right)
=\frac{\beta}{\alpha}.$

\begin{definition}Let $C=\left(  c_{ij}\right)  =\left(  \frac{B_{ij}^{1}|\cdots|B_{ij}^{r_{ij}%
	}}{A_{ij}^{1}|\cdots|A_{ij}^{r_{ij}}}\right)  $ be a $q\times p$ bipartition
	matrix, let $\mathcal{U}\subseteq\mathfrak{q}\times\mathfrak{p,\,}$choose a
	family of pairs
	\[
	\left(  \mathbf{M,N}\right)  =\left\{  \left(  M_{ij}^{k_{ij}},N_{ij}^{k_{ij}%
	}\right)  \subseteq\left(  A_{ij}^{k_{ij}},B_{ij}^{k_{ij}}\right)  \right\}
	_{\left(  i,j\right)  \in\mathcal{U}},
	\]
	and define
	\begin{equation}
		\partial_{ij}\left(  c_{ij}\right)  :=\left\{
		\begin{array}
			[c]{ll}%
			\partial_{M_{ij}^{k_{ij}},N_{ij}^{k_{ij}}}\left(  c_{ij}\right)  , & \text{if
			}\left(  i,j\right)  \in\mathcal{U},\\
			\mathbf{Id}, & \text{otherwise.}%
		\end{array}
		\right.  \label{bipartitioning}%
	\end{equation}
	The \textbf{(partitioning) action of }$\left(  \mathbf{M,N}\right)  $ \textbf{on
	}$C$ is the matrix $\partial_{\mathbf{M,N}}\left(  C\right)  =\left(
	\partial_{ij}\left(  c_{ij}\right)  \right)  $. When $\mathcal{U}=\varnothing$
	there is the \textbf{trivial partitioning action} $\mathbf{Id}\left(  C\right)
	=C.$
	\medskip
	\noindent Given $\left(  i,j\right)  \in\mathfrak{q}\times\mathfrak{p,}$ let
	$\mathcal{U}_{i\ast}:=\left\{  i\right\}  \times\mathfrak{p},$ $\mathcal{U}%
	_{\ast j}:=\mathfrak{q}\times\left\{  j\right\}$, and $\mathcal{V}_{ij}\subseteq \mathfrak{q\times p}\setminus(\mathcal{U}_{i\ast}
	\cup\mathcal{U}_{\ast j})$. Then $\partial
	_{\mathbf{M,N}}\left(  C\right)  $ is a/an
	
	\begin{itemize}	
		\item $\left(  i,j\right)  $\textbf{ entry action} if $\mathcal{U}=\left\{(i,j)\right\}  ,$ $M_{ij}^{k}$ is not extreme when $q>1,$ and
		$N_{ij}^{k}$ is not extreme when $p>1$.
		
		\item \textbf{left (respt. right) row} $i$\textbf{\ action} if
		$\mathcal{U}=\mathcal{U}_{i\ast}$, $p\geq2$, $N_{ij}^{k_{ij}}=\varnothing$ for each $j$
		(respt. $N_{ij}^{k_{ij}}=B_{ij}^{k_{ij}}$ for each $j$), and $M_{ij}^{k}$ is
		not extreme for each $j$ when $q>1$.
		
		\item \textbf{row} $i$\textbf{\ action }if $\partial_{\mathbf{M,N}}\left(
		C\right)  $ is a left or right row $i$ action.
		
		\item \textbf{left (respt. right) column} $j$\textbf{\ action} if
		$\mathcal{U}=\mathcal{U}_{\ast j}$, $q\geq2$, $M_{ij}^{k_{ij}}=\varnothing$ for each $i$
		(respt. $M_{ij}^{k_{ij}}=A_{ij}^{k_{ij}}$ for each $i$), and $N_{ij}^{k}$ is
		not extreme for each $i$ when $p>1$.
		
		\item \textbf{column} $j$\textbf{\ action }if\textbf{\ }$\partial_{\mathbf{M,N}
		}\left(  C\right)  $ is a left or right column $j$ action.
		
		\item \textbf{row} $i/$\textbf{column} $j$\textbf{\ action} if $\partial_{\mathbf{M},\mathbf{N}}\left(  C\right)  $ is a row $i$ action for $\mathcal{U}=\mathcal{U}_{i\ast}$,  $\partial_{\mathbf{M},\mathbf{N}}\left(  C\right)  $ is a column $j$ action for $\mathcal{U}=\mathcal{U}_{\ast j}$, the \textbf{pivoting pair} $(M_{ij}^{k_{ij}},N_{ij}^{k_{ij}})$ common to both actions is not strongly extreme, and for some $\mathcal{V}_{ij}$ there exist strongly extreme pairs $\{(M_{st}^{k_{st}},N_{st}^{k_{st}})\}_{(s,t)\in\mathcal{V}_{ij}}$ such that $\partial_{\mathbf{M,N}}\left(C\right)$ is decomposable for $\mathcal{U}=\mathcal{U}_{i\ast} \cup \mathcal{U}_{\ast j}\cup \mathcal{V}_{ij}$.
	\end{itemize}
\end{definition}

\noindent Hence, a left (respt. right) row $i$ action inserts a null partition
block to the left (respt. right) of $B_{ij}^{k_{ij}}$ for each $j$, and dually
for column $j$ actions. Thus left (and right) row or column actions are always decomposable and the pivoting pair in a row $i$/column $j$ action is extreme.

\begin{example}
	\label{bi-22}Let
	\[
	C=\left(
	\begin{array}
		[c]{cc}%
		\frac{0|1}{0|1}\smallskip & \frac{0|1}{3|0}\\
		\frac{3}{1} & \frac{3|0}{3|0}%
	\end{array}
	\right)  ,
	\]
	$\left(  \mathbf{M}_{1}\mathbf{,N}_{1}\right)  =\left\{  (\left\{
	1\right\}  _{11},\varnothing_{11}),(\left\{  1\right\}  _{21},\varnothing
	_{21})\right\}\text{, and} \left(  \mathbf{M}_{2}\mathbf{,N}_{2}\right)
	=\{  (\left\{  1\right\}  _{21},
	\varnothing_{21}),(\left\{  3\right\}_{22},\varnothing_{22})\}  .$
	Then $\partial_{\mathbf{M}_{1}
		,\mathbf{N}_{1}}\left(  C\right)  $ is the right column action
	\[
	\left(
	\begin{array}
		[c]{c}
		\frac{0|0|1}{0|1|0}\smallskip\\
		\frac{0|3}{1|0}
	\end{array}
	\right)
	=\left(
	\begin{array}
		[c]{c}
		\frac{0|0}{0|1}\smallskip\\
		\frac{0|0}{0|1}\smallskip\\
		\frac{0}{1}\smallskip\\
		\frac{0}{1}
	\end{array}
	\right)  \left(
	\begin{array}
		[c]{cc}%
		\frac{1}{0} & \frac{1}{0} \smallskip\\
		\frac{3}{0} & \frac{3}{0}
	\end{array}
	\right) ,
	\]
	$\partial
	_{\mathbf{M}_{2},\mathbf{N}_{2}}\left(  C\right)  $ is the left row action
	\[
	\left(
	\begin{array}
		[c]{cc}%
		\frac{0|3}{1|0} & \frac{0|3|0}{3|0|0}%
	\end{array}
	\right)  =\left(
	\begin{array}
		[c]{cc}%
		\frac{0}{1} & \frac{0}{3}\smallskip\\
		\frac{0}{1} & \frac{0}{3}
	\end{array}
	\right)  \left(
	\begin{array}
		[c]{cccc}%
		\frac{3}{0} & \frac{3}{0} & \frac{3|0}{0|0} & \frac{3|0}{0|0}%
	\end{array}
	\right) ,
	\]
	and $\partial_{\left(  \mathbf{M}_{1},\mathbf{N}_{1}\right)  \cup\left(
		\mathbf{M}_{2}\mathbf{,N}_{2}\right)  }\left(  C\right)  $  with $\mathcal{V}_{21}=\varnothing$ is the row
	$2$/column  $1$ action
	\[
	\left(
	\begin{array}
		[c]{cc}%
		\frac{0|0|1}{0|1|0}\smallskip & \frac{0|1}{3|0}\\
		\frac{0|3}{1|0} & \frac{0|3|0}{3|0|0}%
	\end{array}
	\right)
	=\left(
	\begin{array}
		[c]{cc}%
		\frac{0|0}{0|1}& \frac{0}{3}\smallskip\\
		\frac{0|0}{0|1}& \frac{0}{3}\smallskip\\
		\frac{0}{1}& \frac{0}{3}\smallskip\\
		\frac{0}{1}& \frac{0}{3}
	\end{array}
	\right)  \left(
	\begin{array}
		[c]{cccc}%
		\frac{1}{0} & \frac{1}{0} & \frac{1}{0} & \frac{1}{0}\smallskip\\
		\frac{3}{0} & \frac{3}{0} & \frac{3|0}{0|0} & \frac{3|0}{0|0}
	\end{array}
	\right) .
	\]
\end{example}

\subsection{Coherent Bipartition Matrices}
In this subsection we define the \emph{coherence} of a bipartition matrix in terms of $\Delta_{P}^{\left(  k\right)  }$. Coherence is fundamentally important in our development
because the \emph{dimension} of a coherent bipartition matrix is under control.

Let $C=(\beta_{ij}/\alpha_{ij})$ be a $q\times p$ bipartition matrix. Recall
that a row equalization of $C$ has equal numerators in each row (\ref{roweq})
and a column equalization of $C$ has equal denominators in each column
(\ref{coleq}). Consider the maximal row and column equalizers $\Lambda
^{row}\left(  C\right)  =\left(  \Lambda_{ij}^{row}\right)  $ and
$\Lambda^{col}\left(  C\right)  =\left(  \Lambda_{ij}^{col}\right)  .$ The
$i^{th}$ \emph{input and }$j^{th}$ \emph{output partitions} \emph{of }$C$ are%
\begin{align*}
	\overset{\wedge}{\alpha}_{i}\left(  C\right)   &  :=\mu_{\Lambda_{i1}^{row}
	}(\alpha_{i1})\Cup\cdots\Cup\mu_{\Lambda_{ip}^{row}}(\alpha_{ip})\in
	P^{\prime}(\mathbf{is}(C))\\
	\overset{\vee}{\beta}_{j}\left(  C\right)   &  :=\mu_{\Lambda_{1j}^{col}
	}(\beta_{1j})\Cup\cdots\Cup\mu_{\Lambda_{qj}^{col}}(\beta_{qj})\in P^{\prime
	}(\mathbf{os}(C)),
\end{align*}
which may vary in length with $i$ and $j.$ Consider the maximal equalization
$C^{eq}$ with respect to the maximal equalizer $\Lambda\left(  C\right)
=\left(  \Lambda_{ij}\right)  $.\ The \emph{input and output partitions of
}$C^{eq}$ are
\begin{align*}
	\overset{\wedge}{eq}\left(  C\right)   &  :=\mu_{\Lambda_{11}}(\alpha
	_{11})\Cup\cdots\Cup\mu_{\Lambda_{1p}}(\alpha_{1p})\in P^{\prime}
	(\mathbf{is}(C))\\
	\overset{\vee}{eq}\left(  C\right)   &  :=\mu_{\Lambda_{11}}(\beta_{11}
	)\Cup\cdots\Cup\mu_{\Lambda_{q1}}(\beta_{q1})\in P^{\prime}(\mathbf{os}(C)).
\end{align*}
Then $\overset{\wedge}{eq}\left(  C\right)  $ is obtained by merging the
denominators in the first (or any) row of $C^{eq}$ and $\overset{\wedge
}{\alpha}_{i}\left(  C\right)  $ is obtained by appropriately partitioning
$\overset{\wedge}{eq}\left(  C\right)  $ (and dually for $\overset{\vee
}{eq}\left(  C\right)  $). When $\mathbf{is}(C)$ and $\mathbf{os}(C)$ are non-empty, $\pi\overset{\wedge
}{eq}\left(  C\right)  $ is identified with a cell of $P_{\#\mathbf{is}(C)}$.
In particular, when $C$ is indecomposable (in which case $\Lambda(C)$ is null), $\pi\overset{\wedge}{eq}\left(
C\right)  $ is identified with the top dimensional cell of $P_{\#\mathbf{is}(C)}$.

The \emph{input and output products of }$C$ are
\begin{align*}
	\overset{\wedge}{\mathbf{e}}(C)  &  :=\pi\overset{\wedge}{\alpha}_{q}\left(
	C\right)  \times\cdots\times\pi\overset{\wedge}{\alpha}_{1}\left(  C\right)
	\sqsubseteq P(\mathbf{is}(C))^{\times q}\\
	\overset{\vee}{\mathbf{e}}(C)  &  :=\pi\overset{\vee}{\beta_{1}}\left(
	C\right)  \times\cdots\times\pi\overset{\vee}{\beta_{p}}\left(  C\right)
	\sqsubseteq P(\mathbf{os}(C))^{\times p}.
\end{align*}

\begin{example}
	\label{twomatrices}Let $C=\left(
	\begin{array}
		[c]{ccc}%
		\frac{1|2|3|4|5}{0|0|0|1|2} & \frac{1|2|34|5}{0|0|3|4} & \frac{12|4|3|5}
		{0|0|5|6}%
	\end{array}
	\right)  .$ Then $\Lambda\left(  C\right)  =(\left\{  2,4\right\}  $ $\left\{
	2,3\right\}  $ $\left\{  1,3\right\}  )$ and $C^{eq}=\left(
	\begin{array}
		[c]{ccc}%
		\frac{12|34|5}{0|1|2} & \frac{12|34|5}{0|3|4} & \frac{12|34|5}{0|5|6}%
	\end{array}
	\right)  \ $so that $\overset{\wedge}{eq}\left(  C\right)  =\overset{\wedge
	}{\alpha}_{1}(C)=0|135|246,$ $\overset{\wedge}{\mathbf{e}}(C)=\pi
	\overset{\wedge}{\alpha}_{1}(C)=135|246,$ $\overset{\vee}{eq}(C)=12|34|5,$
	and  $\overset{\vee}{\mathbf{e}}(C)=1|2|3|4|5\times1|2|34|5\times
	12|4|3|5\sqsubseteq_{diag}\Delta_{P}^{\left(  2\right)  }(\overset{\vee
	}{eq}\left(  C\right)  ).$
	
	Consider the matrix $D=\left(
	\begin{array}
		[c]{ccc}%
		\frac{1|2|3|4|5}{0|0|0|1|2} & \frac{1|2|34|5}{0|0|3|4} & \frac{12|4|35}
		{0|0|56}%
	\end{array}
	\right)  $ obtained from $C$ via the replacement $\frac{12|4|3|5}
	{0|0|5|6}\leftarrow\frac{12|4|35}{0|0|56};$ then $C=\partial_{\left\{
		5\right\}  _{13},\left\{  3\right\}  _{13}}\left(  D\right)  $. Furthermore,
	$\Lambda\left(  D\right)  =(\left\{  2\right\}  $ $\left\{  2\right\}  $
	$\left\{  1\right\}  )$ and $D^{eq}=\left(
	\begin{array}
		[c]{ccc}%
		\frac{12|345}{0|12} & \frac{12|345}{0|34} & \frac{12|345}{0|56}%
	\end{array}
	\right)  \ $so that $\overset{\wedge}{eq}\left(  D\right)  =0|123456,$
	$\overset{\wedge}{\mathbf{e}}(D)=123456,$ $\overset{\vee}{eq}(D)=12|345,$ and
	$\overset{\vee}{\mathbf{e}}(D)=1|2|3|4|5\times1|2|34|5\times12|4|35\sqsubseteq
	_{diag}\Delta_{P}^{\left(  2\right)  }(\overset{\vee}{eq}\left(  D\right)
	).$  Let $e=\overset{\vee}{eq}\left(  D\right)  ;$ then $\overset{\vee
	}{eq}(C)\sqsubseteq\partial e.$ Set $\mathbf{x=}\overset{\vee}{\mathbf{e}
	}(C);$ by uniqueness in Proposition \ref{ab}, $\mathbf{x}_{33}=\overset{\vee
	}{\mathbf{e}}(D)$ is the unique diagonal component of $\overset{\vee
	}{eq}\left(  D\right)  $ that can be obtained from $\overset{\vee}{\mathbf{e}
	}(C)$ by a single factor replacement.
\end{example}

\begin{definition}
	\label{defn-coherence}A $q\times p$ bipartition matrix $C$ is
	
	\begin{itemize}

		\item \textbf{row precoherent }if $\mathbf{is}(C)\neq\varnothing$ and
		\begin{equation}
			\overset{\wedge}{\mathbf{e}}(C)\sqsubseteq\Delta_{P}^{(q-1)}(\pi
			\overset{\wedge}{eq}\left(  C\right)  ); \label{deco-iterwedge}%
		\end{equation}

		\item \textbf{maximally row precoherent }if $\mathbf{is}(C)\neq\varnothing$
		and
		\begin{equation}
			\overset{\wedge}{\mathbf{e}}(C)\sqsubseteq_{diag}\Delta_{P}^{(q-1)}
			(\pi\overset{\wedge}{eq}\left(  C\right)  ); \label{mdeco-iterwedge}%
		\end{equation}

		\item \textbf{column precoherent }if $\mathbf{os}(C)\neq\varnothing$ and
		\begin{equation}
			\overset{\vee}{\mathbf{e}}(C)\sqsubseteq\Delta_{P}^{(p-1)}(\pi\overset{\vee
			}{eq}\left(  C\right)  ); \label{deco-itervee}%
		\end{equation}

		\item \textbf{maximally column precoherent }if $\mathbf{os}(C)\neq
		\varnothing$  and
		\begin{equation}
			\overset{\vee}{\mathbf{e}}(C)\sqsubseteq_{diag}\Delta_{P}^{(p-1)}
			(\pi\overset{\vee}{eq}\left(  C\right)  ); \label{mdeco-itervee}%
		\end{equation}

		\item (\textbf{maximally) precoherent }if $C$ is (maximally) row and
		(maximally) column precoherent;
		
		\item \textbf{(maximally) row coherent} if $\mathbf{is}(C)=\varnothing$ or
		its  indecomposable factors are (maximally) row precoherent;
		
		\item \textbf{(maximally) column coherent} if $\mathbf{os}(C)=\varnothing$ or
		its indecomposable factors are (maximally) column precoherent;
		
		\item \textbf{(maximally) coherent }if $C$ is (maximally) row and (maximally)
		column coherent;
		
		\item \textbf{(maximally) totally coherent} if $C$ is (maximally) coherent
		and  its indecomposable factors have (maximally) coherent rows, columns, and
		entries.
	\end{itemize}
\end{definition}

Null bipartition matrices are maximally coherent. An indecomposable
bipartition matrix is coherent if and only if it is precoherent, but a
precoherent decomposable bipartition matrix is not necessarily coherent.

\begin{example}
	\label{not-totally-coherent}The maximally precoherent matrix
	\[
	\left(
	\begin{array}
		[c]{cc}%
		\frac{0|1}{12|0} & \frac{0|1}{45|0}%
	\end{array}
	\right)  =\left(
	\begin{array}
		[c]{cc}%
		\frac{0}{12}\smallskip & \frac{0}{45}\\
		\frac{0}{12} & \frac{0}{45}%
	\end{array}
	\right)  \left(
	\begin{array}
		[c]{cccccc}%
		\frac{1}{0} & \frac{1}{0} & \frac{1}{0} & \frac{1}{0} & \frac{1}{0} & \frac
		{1}{0}%
	\end{array}
	\right)  =C_{1}C_{2}
	\]
	is incoherent because its factor $C_{1}$ is incoherent. Since $C_{1}$ has
	coherent entries but fails to be coherent, coherence is a global property of
	bipartition matrices.
\end{example}

\begin{example}
	\label{D-coherent}The matrices
	\[
	C=\left(
	\begin{array}
		[c]{ccc}%
		\frac{1|2|3|4|5}{0|0|0|1|2} & \frac{1|2|34|5}{0|0|3|4} & \frac{12|4|3|5}
		{0|0|5|6}%
	\end{array}
	\right)  \text{ and }D=\left(
	\begin{array}
		[c]{ccc}%
		\frac{1|2|3|4|5}{0|0|0|1|2} & \frac{1|2|34|5}{0|0|3|4} & \frac{12|4|35}
		{0|0|56}%
	\end{array}
	\right)
	\]
	discussed in Example \ref{twomatrices} are maximally column precoherent since
	$\overset{\vee}{\mathbf{e}}(C)\sqsubseteq_{diag}$\linebreak $\Delta_{P_{5}
	}^{\left(  2\right)  }(\overset{\vee}{eq}(C))$ and $\overset{\vee}{\mathbf{e}
	}(D)\sqsubseteq_{diag}\Delta_{P}^{(2)}(\overset{\vee}{eq}\left(  D\right)
	).$  Furthermore, by Proposition \ref{ab}, part (2), $D$ is the unique
	maximally  column precoherent matrix such that $\partial_{\mathbf{M,N}}\left(
	D\right)   =C$ for some $\left(  \mathbf{M,N}\right)  $.
\end{example}

\begin{proposition}
	A coherent bipartition matrix  is precoherent.
\end{proposition}

\begin{proof}
	When i/o sets are empty, the proof is trivial.
	Let $C$ be a $q\times p$ coherent bipartition matrix with indecomposable
	factorization $C=C_{1}\cdots C_{r}$ and let $\mathbf{is}(C)\neq \varnothing.$ We claim that  $C$ is row
	precoherent.
	If $r=1$ or $q=1$ there is nothing to prove. So
	assume $r,q>1.$ By hypothesis $C_{k}$ is coherent for each $k.$  The input
	product $\overset{\wedge}{\mathbf{e}}(C_{k}),$ which has the form
	\[
	\overset{\wedge}{\mathbf{e}}(C_{k})=(e_{k_{q}}\times\cdots\times e_{k_{q-1}
		+1})\times\cdots\times(e_{k_{1}}\times\cdots\times e_{1}),
	\]
	is a subcomplex of $\Delta_{P}^{\left(  k_{q}-1\right)  }\left(
	P_{\#\mathbf{is}(C_{k})}\right)  $ by the row coherence of $C_{k}.$ On the
	other hand, $\overset{\wedge}{\mathbf{e}}(C)$ has the form
	\[
	\overset{\wedge}{\mathbf{e}}(C)=E_{q}\times\cdots\times E_{1},
	\]
	where $E_{i}=\bar{e}_{i,1}\times\cdots\times\bar{e}_{i,r}$ and $\bar{e}
	_{i,k}\sqsubseteq P_{\#\mathbf{is}(C_{k})}.$ Now for all $i$ and $k,$ each
	factor $e_{\nu}$ of $e_{k_{i}}\times\cdots\times e_{k_{i-1}+1}$ is a cell of
	$\bar{e}_{i,k},$ and an analysis of $\Delta_P$ reveals that
	$e_{k_{i}}\times\cdots\times e_{k_{i-1}+1}\sqsubseteq\Delta_{P}^{\left(
		k_{i}-k_{i-1}-1\right)  }(\bar{e}_{i,k})$ and
	\begin{equation}
		\bar{e}_{q,k}\times\cdots\times\bar{e}_{1,k}\sqsubseteq\Delta_{P}^{\left(
			q-1\right)  }(P_{\#\mathbf{is}(C_{k})}). \label{bars}%
	\end{equation}
	Since $\Delta_{p}$ acts multiplicatively on product cells, the inclusion in
	(\ref{bars}) yields
	\begin{align*}
		\overset{\wedge}{\mathbf{e}}(C)  &  =E_{q}\times\cdots\times E_{1}=\left(
		\bar{e}_{q,1}\times\cdots\times\bar{e}_{q,r}\right)  \times\cdots\times\left(
		\bar{e}_{1,1}\times\cdots\times\bar{e}_{1,r}\right) \\
		&  \sqsubseteq\Delta_{P}^{\left(  q-1\right)  }\left(  P_{\#\mathbf{is}\left(
			C_{1}\right)  }\times\cdots\times P_{\#\mathbf{is}\left(  C_{r}\right)
		}\right)  =\Delta_{P}^{\left(  q-1\right)  }\left(  \pi\overset{\wedge
		}{eq}\left(  C\right)  \right)  ,
	\end{align*}
	and it follows that $C$ is row precoherent. A dual argument for column
	precoherence completes the proof.
\end{proof}

The class of coherent elementary matrices is highly constrained.
\pagebreak

\begin{proposition}
	\label{el-coherence}Let $\mathbf{C}$ be a $q\times p$ elementary matrix with
	$pq>1$.
	
	\begin{enumerate}

		\item If $\mathbf{is}\left(  \mathbf{C}\right)  \neq\varnothing$ and
		$q\geq2,$  then $\mathbf{C}$ is row coherent if and only if $\mathbf{C}$ is
		maximally row  coherent if and only if $\#\mathbf{is}\left(  \mathbf{C}%
		\right)  =1;$
		
		\item If $\mathbf{os}\left(  \mathbf{C}\right)  \neq\varnothing$ and
		$p\geq2,$  then $\mathbf{C}$ is column coherent if and only if $\mathbf{C}$ is
		maximally  column coherent if and only if $\#\mathbf{os}\left(  \mathbf{C}%
		\right)  =1;$
		
		\item If $\mathbf{is}\left(  \mathbf{C}\right)  ,\mathbf{os}\left(
		\mathbf{C}\right)  \neq\varnothing,$ and $p,q\geq2,$ then $\mathbf{C}$ is
		coherent if and only if $\mathbf{C}$ is maximally coherent if and only if
		$\#\mathbf{is}\left(  \mathbf{C}\right)  =\#\mathbf{os}\left(  \mathbf{C}
		\right)  =1$.
	\end{enumerate}
\end{proposition}

\begin{proof}
	(1) Since $\mathbf{C}$ is elementary, $\pi\overset{\wedge}{eq}\left(
	\mathbf{C}\right)  =\pi\overset{\wedge}{\alpha}_{i}\left(  \mathbf{C}\right)
	=\mathbf{is}(\mathbf{C})$ for all $i.$ But $q\geq2$ implies $\#\mathbf{is}
	\left(  \mathbf{C}\right)  =1$ if and only if
	\[
	\overset{\wedge}{\mathbf{e}}\left(  \mathbf{C}\right)  =\underset{q\text{
			factors}}{\underbrace{\mathbf{is}\left(  \mathbf{C}\right)  \times\cdots
			\times\mathbf{is}\left(  \mathbf{C}\right)  }}=\Delta_{P}^{\left(  q-1\right)
	}(\mathbf{is}\left(  \mathbf{C}\right)  )
	\]
	if and only if $\mathbf{C}$ is maximally row coherent. The proof of (2) is
	completely dual and the proof of (3) follows immediately from (1) and (2).
\end{proof}

\noindent In fact, every coherent elementary matrix is maximally totally coherent.

Now if $\mathbf{C}$ is a $q\times p$ (maximally) coherent elementary matrix
such that $q\geq2$ and $\mathbf{is}\left(  \mathbf{C}\right)  \neq
\varnothing,$ then $\#\mathbf{is}\left(  \mathbf{C}\right)  =1$ by Proposition
\ref{el-coherence}. Since $\mathbf{C}$ has constant denominators in each
column, row coherence and $\#\mathbf{is}\left(  \mathbf{C}\right)  =1$ imply
that denominators in exactly one column of $\mathbf{C}$ are non-empty constant
singletons and all other denominators are null. Dually, if $p\geq2$ and
$\mathbf{os}\left(  \mathbf{C}\right)  \neq\varnothing,$ numerators in exactly
one row of $\mathbf{C}$ are non-empty constant singletons and all other
numerators are null. Thus if $p,q\geq2$ and $\mathbf{is}\left(  \mathbf{C}%
\right)  =\mathbf{os}\left(  \mathbf{C}\right)  =\left\{  1\right\}  ,$
exactly one row of $\mathbf{C}$ has the form $\left(  \frac{1}{0}\cdots
\frac{1}{0}\text{ }\frac{1}{1}\text{ }\frac{1}{0}\cdots\frac{1}{0}\right)  $,
exactly one column of $\mathbf{C}$ has the form $\left(  \frac{0}{1}%
\cdots\frac{0}{1}\text{ }\frac{1}{1}\text{ }\frac{0}{1}\cdots\frac{0}%
{1}\right)  ^{T}$, and all other entries are null. For example, when $\left(  p,q\right)
=\left(  3,4\right)  $ we have
\[
\mathbf{C}=\left(
\begin{array}
	[c]{cccc}%
	\frac{0}{0}\smallskip & \frac{0}{0} & \frac{0}{1} & \frac{0}{0}\vspace{1mm}\\
	\frac{1}{0}\smallskip & \frac{1}{0} & \frac{1}{1} & \frac{1}{0}\vspace{1mm}\\
	\frac{0}{0}\smallskip & \frac{0}{0} & \frac{0}{1} & \frac{0}{0}
\end{array}
\right)  .
\]
By similar calculations, if $p,q\geq2,$ $\mathbf{is}\left(  \mathbf{C}\right)
=\varnothing$ and $\mathbf{os}\left(  \mathbf{C}\right)  =\left\{  1\right\}
,$ exactly one row of $\mathbf{C}$ has the form $\left(  \frac{1}{0}%
\cdots\frac{1}{0}\right)  $ and all other entries are null.

Bipartitions viewed as $1\times 1 $ matrices are (trivially) precoherent but not necessarily coherent. Indeed,
the set of coherent bipartitions is highly constrained.

\begin{proposition}
	\label{coh-bipartition}A bipartition $c=\frac{B_{1}|\cdots|B_{r}}{A_{1}
		|\cdots|A_{r}}$ is coherent whenever one of the following conditions is satisfied:
	
	\begin{enumerate}

		\item $r=1.$
		
		\item $A_{1}=\cdots=A_{i-1}=B_{i+1}=\cdots=B_{r}=\varnothing$ for $r\geq2$
		and  some $i\leq r.$
		
		\item $\#A_{1},\ldots,\#A_{r-1},\#B_{2},\ldots,\#B_{r}\in\{0,1\}$ for
		$r\geq2.$
	\end{enumerate}
\end{proposition}

\begin{proof}
	When $r=1$ there is nothing to prove. When $r=2,$ consider the indecomposable
	factorization $c=C_{1}C_{2}$. Since $C_{1}$ and $C_2$ are elementary column and row matrices, respectively, their coherence follows easily from (2) and (3) independently. Hence $c$ is 	coherent.  Conversely, if $\#A_{1}\geq2$ and $\#B_{2}\geq1$, the
	partition 	$EP_{B_{1}\cup B_{2}}B_{1}$ has $k\geq2$ blocks so that $C_{1}$ has $k$ rows.
	Since the  $k$-fold product cell $\overset{\wedge}{\mathbf{e}}\left(
	C_{1}\right)   =A_{1}\times\cdots\times A_{1}\sqsubseteq\hspace*{-0.15in}%
	\diagup\Delta_{P}^{\left(  k-1\right)  }(\pi\overset{\wedge}{eq}\left(
	C_{1}\right)  )$,  it follows that $C_{1}$ is incoherent and so is $c.$ On the
	other hand,  $\#A_{1}\geq1$ and $\#B_{2}\geq2$ imply that $C_{2}$ is incoherent by
	a symmetric  argument. Cases with $r>2$ follow by similar arguments.
\end{proof}

\noindent Note that Condition (2) reduces to $A_{1}=\cdots=A_{r-1}%
=\varnothing$ when $i=r,$ and in particular, to $A_{k}=\varnothing$ for all
$k$ when $\mathbf{is}\left(  c\right)  =\varnothing;$ dually, Condition (2)
reduces to $B_{2}=\cdots=B_{r}=\varnothing$ when $i=1,$ and in particular, to
$B_{k}=\varnothing$ for all $k$ when $\mathbf{os}\left(  c\right)
=\varnothing.$ Thus every bipartition with a null input set or a null output
set is coherent. A coherent bipartition satisfying Condition (2) has the form%
\[
\frac{\text{\ }B_{1}|\cdots|B_{i}|0|\cdots|0}{0|\cdots|0|A_{i}|\cdots|A_{r}}.
\]
An example of a coherent bipartition satisfying Condition (3) is%
\[
\frac{123|0|0|4|5\text{ \ \ }}{\text{ \ \ }1|0|2|0|345}.
\]

\subsection{Coheretization}

\label{coheretizing} In this subsection we discuss conditions under which a sequence of
partitioning actions transforms a bipartition matrix into a precoherent matrix.

\begin{definition}
	\label{coheretizable}A bipartition matrix $C$ is (\textbf{pre)coheretizable}
	if there exists a sequence of partitioning operators $\partial_{\mathbf{M_{1}
			,N_{1}}},\dots,\partial_{\mathbf{M}_{k}\mathbf{,N}_{k}}$ such that
	\[
	C^{\prime}=\partial_{\mathbf{M}_{k}\mathbf{,N}_{k}}\cdots\partial
	_{\mathbf{M_{1},N_{1} }}\left(  C\right)
	\]
	is (pre)coherent; the matrix $C^{\prime}$ is a  (\textbf{pre)coheretization of
	}$C$.
\end{definition}

\noindent Coherent matrices $C$ are trivially coheretizable via the identity
partitioning action $\mathbf{Id}\left(  C\right)  $. An indecomposable
precoheretization is a coheretization.

Consider a $q\times p$ semi-null elementary matrix $C$ with $\mathbf{is}%
\left(  C\right)  \neq\varnothing,$ an augmented partition $\alpha\in
P^{\prime}\left(  \mathbf{\#is}\left(  C\right)  \right)  ,$ and a diagonal
component $X_{q}:=x_{q}\times\cdots\times x_{1}\sqsubseteq_{diag}\Delta
_{P}^{\left(  q-1\right)  }\left(  P_{\#\mathbf{is}(C)}\right)  .$ Since the
rows of $C$ are equal and elementary, there is a row $i$ partitioning action
$\partial_{\mathbf{M}_{i}\mathbf{,N}_{i}}\left(  C\right)  $ for each $i$ such
that $\pi\overset{\wedge}{\alpha}_{i}(\partial_{\mathbf{M}_{i}\mathbf{,N}_{i}%
}\left(  C\right)  )=x_{i}.$ Hence $C^{\prime}:=\partial_{\mathbf{M}
	_{q}\mathbf{,N}_{q}}\cdots\partial_{\mathbf{M_{1},N_{1}}}\left(  C\right)  $
is a precoheretization such that $\overset{\wedge}{\mathbf{e}}(C^{\prime
})=X_{q}.$
In fact, $C^{\prime}$ is a coheretization: When $q=2,$ the partition factors
of $X_{2}=x_{2}\times x_{1}$ are distinct and $\Lambda\left(  C^{\prime
}\right)  $ is null. When $q=3,$ consider $X_{3}=x_{3}\times x_{2}\times
x_{1}\sqsubseteq_{diag}\left(  \Delta_{P}\times\mathbf{1}\right)  \Delta
_{P}\left(  P_{\#\mathbf{is}(C)}\right)  .$ Let $D$ be the matrix obtained
from $C$ by deleting the first row, and let $D^{\prime}=\partial
_{\mathbf{M}_{3}\mathbf{,N}_{3}}\partial_{\mathbf{M_{2},N_{2}}}\left(
D\right)  .$ Then $\overset{\wedge}{\mathbf{e}}(D^{\prime})=x_{3}\times x_{2}$
and $\Lambda\left(  D^{\prime}\right)  $ is null. Hence $\Lambda\left(
C^{\prime}\right)  $ is null, and so on inductively. It follows that
$C^{\prime}$ is indecomposable. The dual statement with $\mathbf{os}%
(C)\neq\varnothing$ follows by a similar argument.

\begin{example}
	Consider the semi-null incoherent matrix
	\[
	C=\left(
	\begin{array}
		[c]{cc}%
		\frac{0}{12} & \frac{0}{45}\vspace{1mm}\\
		\frac{0}{12} & \frac{0}{45}%
	\end{array}
	\right)  .
	\]
	To construct the coheretization $C^{\prime}$ such that $\overset{\wedge
	}{\mathbf{e}}\left(  C^{\prime}\right)  =1|24|5\times14|25\sqsubseteq
	_{diag}\Delta_{P}\left(  P_{4}\right)  ,$ let $(\mathbf{M_{1},N_{1}})=\left\{
	(\{1\}_{i1},\varnothing_{i1}),(\{4\}_{i2},\varnothing_{i2})\right\}
	_{i=1,2}$  and $(\mathbf{M_{2},N_{2}})=\{(\{2\}_{21},\varnothing_{21}^{2}),$
	$(\varnothing_{22}^{1},\varnothing_{22}^{1})\}.$ Then
	\[
	C^{\prime}=\partial_{\mathbf{M_{2},N_{2}}}\partial_{\mathbf{M_{1},N_{1}}
	}\left(  C\right)  =\left(
	\begin{array}
		[c]{cc}%
		\frac{0|0}{1|2} & \frac{0|0}{4|5}\vspace{1mm}\\
		\frac{0|0|0}{1|2|0} & \frac{0|0|0}{0|4|5}%
	\end{array}
	\right)  .
	\]
	
\end{example}
\noindent When $p\geq2,$ a coherent elementary row matrix has the form $C=\left(
\frac{\mathbf{b}}{\mathbf{a}_{1}}\cdots\frac{\mathbf{b}}{\mathbf{a}_{p}%
}\right)  ,$ where $\#\mathbf{b\in}\left\{  0,1\right\}  .$ When
$\#\mathbf{b}=1,$ a family of pairs $\left(  \mathbf{M,N}\right)  =\left\{
\left(  M^{j},\varnothing\right)  \subseteq\left(  \mathbf{a}_{j}%
,\mathbf{b}\right)  \right\}  _{j\in\mathfrak{p}}$ defines a row action
\begin{align}
	\partial_{\mathbf{M,N}}\left(  C\right)   &  =\left(  \frac{0|\mathbf{b}
	}{M^{1}|\mathbf{a}_{1}\smallsetminus M^{1}}\text{ }\cdots\text{ }
	\frac{0|\mathbf{b}}{M^{p}|\mathbf{a}_{p}\smallsetminus M^{p}}\right)
	\bigskip\nonumber\\
	&  =\left(
	\begin{array}
		[c]{ccc}%
		\frac{0}{M^{1}} & \cdots & \frac{0}{M^{p}}\vspace{1mm}\\
		\frac{0}{M^{1}} & \cdots & \frac{0}{M^{p}}%
	\end{array}
	\right)  \left(  \frac{\mathbf{b}}{A_{11}}\frac{\mathbf{b}}{A_{12}}\text{
	}\cdots\text{ }\frac{\mathbf{b}}{A_{p1}}\frac{\mathbf{b}}{A_{p2}}\right)
	=C_{1}C_{2}, \label{rowdecomp}%
\end{align}
where $A_{j1}|A_{j2}:=EP_{\mathbf{a}_{j}}\left(  \mathbf{a}_{j}\smallsetminus
M^{j}\right)  .$ For example, when $\#M^{j}=1$ for all $j,$ the factor $C_{1}$
is incoherent whenever $\#\mathbf{is}\left(  C_{1}\right)  \geq2.$ However,
given a \emph{primitive }component $X\sqsubseteq_{diag}\Delta_{P}\left(
P_{\cup M^{j}}\right)  $, a row action $C_{1}^{\prime}=\partial_{\mathbf{M,N}%
}\left(  C_{1}\right)  $ such that $\overset{\wedge}{\mathbf{e}}\left(
C_{1}^{\prime}\right)  =X$ can be expressed as compositions of $\eta_{\ast}$
actions on the entries of $C_{1}.$

\begin{example}
	\label{eta-action}Let $C=\left(  \tfrac{1}{12}\text{ }\tfrac{1}{45}\text{
	}\tfrac{1}{78}\right)  $ and $\left(  \mathbf{M,N}\right)  =\left\{  \left(
	\left\{  j\right\}  ,\varnothing\right)  :j=1,4,7\right\}  ;$ then
	\[
	\partial_{\mathbf{M,N}}\left(  C\right)  =\left(  \tfrac{0|1}{1|2}\text{
	}\tfrac{0|1}{4|5}\text{ }\tfrac{0|1}{7|8}\right)  =\left(
	\begin{array}
		[c]{ccc}%
		\frac{0}{1} & \frac{0}{4} & \frac{0}{7}\vspace{1mm}\\
		\frac{0}{1} & \frac{0}{4} & \frac{0}{7}%
	\end{array}
	\right)  \left(  \tfrac{1}{0}\text{ }\tfrac{1}{2}\text{ }\tfrac{1}{0}\text{
	}\tfrac{1}{5}\text{ }\tfrac{1}{0}\text{ }\tfrac{1}{8}\right)  =C_{1}C_{2}.
	\]
	The coheretization $C_{1}^{\prime}$ of $C_{1}$ such that $\overset{\wedge
	}{\mathbf{e}}\left(  C_{1}^{\prime}\right)  =1|4|7\times147\sqsubseteq
	_{diag}\Delta_{P}^{\left(  1\right)  }\left(  P_{3}\right)  $ is given by
	\[
	C_{1}^{\prime}=\left(
	\begin{array}
		[c]{ccc}%
		\frac{0}{1} & \frac{0}{4} & \frac{0}{7}\vspace{1mm}\\
		\eta_{2}\eta_{2}\left(  \frac{0}{1}\right)  & \eta_{1}\eta_{2}\left(  \frac
		{0}{4}\right)  & \eta_{1}\eta_{1}\left(  \frac{0}{7}\right)
	\end{array}
	\right)  =\left(
	\begin{array}
		[c]{ccc}%
		\frac{0}{1} & \frac{0}{4} & \frac{0}{7}\vspace{1mm}\\
		\frac{0|0|0}{1|0|0} & \frac{0|0|0}{0|4|0} & \frac{0|0|0}{0|0|7}%
	\end{array}
	\right)  .
	\]
	
\end{example}

More generally, let $\mathbf{C}=\left(  \mathbf{b}_{i}/\mathbf{a}_{j}\right)  $ be a
$q\times p$ elementary matrix over $(\mathbf{a}_{\ast},\mathbf{b}_{\ast}).$ If
$\mathbf{a}_{j}\neq\varnothing,$ write $\mathbf{a}_{j}=\{a_{j_{1}}^{j}%
<\cdots<a_{j_{s}}^{j}\};$ if $\mathbf{b}_{i}\neq\varnothing,$ write
$\mathbf{b}_{i}=\{b_{i_{1}}^{i}<\cdots<b_{i_{t}}^{i}\}.$ There exists a unique
$q\times p$ totally coherent bipartition matrix $prim(\mathbf{C}),$ called the
\emph{primitive coheretization of} $\mathbf{C},$ with the following properties:

\begin{itemize}
	\item $\overset{\wedge}{\alpha}_{1}\left(  prim(\mathbf{C})\right)  =\mathbf{a}_{1}%
	\cup\cdots\cup\mathbf{a}_{p}$ and $\overset{\wedge}{\alpha}_{i}\left(
	prim(\mathbf{C})\right)  =a_{1}^{1}|\cdots|a_{u_{i}}^{1}\Cup\cdots\Cup a_{1}^{p}%
	|\cdots|a_{u_{i}}^{p}$ for $2\leq i\leq q,$ where $\pi(a_{1}^{j}%
	|\cdots|a_{u_{i}}^{j})=a_{j_{1}}^{j}|\cdots|a_{j_{s}}^{j}$ when $\mathbf{a}%
	_{j}\neq\varnothing,$
	
	\item $\overset{\vee}{\beta}_{1}\left(  prim(\mathbf{C})\right)  =\mathbf{b}_{1}\cup
	\cdots\cup\mathbf{b}_{q}$ and $\overset{\vee}{\beta}_{j}\left(  prim(\mathbf{C})\right)
	=b_{v_{j}}^{1}|\cdots|b_{1}^{1}\Cup\cdots\Cup b_{v_{j}}^{q}|\cdots|b_{1}^{q}$
	for $2\leq j\leq p,$ where $\pi(b_{v_{j}}^{i}|\cdots|b_{1}^{i})=b_{i_{t}}%
	^{i}|\cdots|b_{i_{1}}^{i}$ when $\mathbf{b}_{i}\neq\varnothing,$
	
	\item non-empty partition blocks in the entries of $prim\left(  \mathbf{C}\right)  $ are
	in left-most positions,
	
	\item and $|prim(\mathbf{C})|=\#\mathbf{is}(\mathbf{C})+\#\mathbf{os}(\mathbf{C})-1.$
\end{itemize}
\begin{example}
	The primitive coheretization of the elementary matrix
	\[
	\mathbf{C}=\left(
	\begin{array}
		[c]{cccc}
		\frac{1}{1}\smallskip & \frac{1}{2} & \frac{1}{0} & \frac{1}{3}\vspace{1mm}\\
		\frac{23}{1} & \frac{23}{2} & \frac{23}{0} & \frac{23}{3}%
	\end{array}
	\right)  \ \text{is}\ \
	prim\left(  \mathbf{C}\right)  =
	\left(
	\begin{array}
		[c]{cccc}
		\frac{1}{1}\smallskip & \frac{0|0|1}{0|2|0} & \frac{0|0|1}{0|0|0} &
		\frac{0|0|1}{0|0|3}\\
		\frac{3|2|0}{1|0|0} & \frac{3|2|0}{0|2|0} & \frac{3|2|0}{0|0|0} &
		\frac{3|2|0}{0|0|3}
	\end{array}
	\right)
	\]
	so that
	$
	\overset{\wedge}{\mathbf{e}}\left(  prim\left(  \mathbf{C}\right)  \right)  =1|2|3\times
	123\text{ \ and \ }\overset{\vee}{\mathbf{e}}\left( prim\left(  \mathbf{C}\right)
	\right)  =123\times3|2|1\times3|2|1\times3|2|1.
	$
\end{example}

\subsection{Generalized Bipartition Matrices}

Given an elementary bipartition $\mathbf{c}=$ $\mathbf{b/a},$ choose a bipartition
$c^{1}\in P_{r}^{\prime}\left(  \mathbf{a}\right)  \times P_{r}^{\prime
}\left(  \mathbf{b}\right)  $ for some $r\geq1$ (the \emph{trivial choice} is
$c^{1}=\mathbf{c}$) and let $c^{1}=\mathbf{C}_{1}^{1}\cdots \mathbf{C}_{r}^{1}$ be the
elementary factorization. The $1$\emph{-formal product} $T^{1}\left(
\mathbf{c}\right)  :=\mathbf{C}_{1}^{1}\cdots \mathbf{C}_{r}^{1}$ is a $1$\emph{-formal
	bipartition on} $\mathbf{c}$.

For each $k,$ write $\mathbf{C}_{k}^{1}=(\mathbf{c}_{k}^{ij}),$ choose a $1$-formal
bipartition $T^{1}(\mathbf{c}_{k}^{ij})$ for each $\left(  i,j\right)  ,$ and
form the $1$-\emph{formal (bipartition) matrix} $C_{k}^{2}=T^{1}\left(
\mathbf{C}_{k}^{1}\right)  :=(T^{1}(\mathbf{c}_{k}^{ij}))$. The $2$\emph{-formal
	product} $c^{2}=T^{2}\left(  \mathbf{c}\right)  :=C_{1}^{2}\cdots C_{r}^{2}$
of $1$-formal matrices is a $2$\emph{-formal bipartition on }$\mathbf{c}$.

For each $\left(  i,j,k\right)  ,$ choose a $2$-formal bipartition
$T^{2}(\mathbf{c}_{k}^{ij})$ and form the $2$\emph{-formal matrix} $C_{k}%
^{3}=T^{2}\left(  \mathbf{C}_{k}^{1}\right)  :=(T^{2}(\mathbf{c}_{k}^{ij}))$. Given an
elementary matrix $\mathbf{C}$ and a $2$-formal matrix $C=T^{2}\left(
\mathbf{C}\right)  ,$ define $\mathbf{is}\left(  C\right)  :=\mathbf{is}%
\left(  \mathbf{C}\right)  $ and $\mathbf{os}\left(  C\right)  :=\mathbf{os}%
\left(  \mathbf{C}\right)  .$ Recall that if $\left(  A^{q\times1},B^{1\times
	p}\right)  $ is a TP of bipartition matrices, then $\left(  \mathbf{\#is}%
(A),\#\mathbf{os}(B)\right)  =(p-1,q-1)$. A pair of $2$-formal matrices
$\left(  A^{q\times1},B^{1\times p}\right)  $ is a \emph{(formal) Transverse
	Pair (TP)} if $\left(  \mathbf{\#is}(A),\#\mathbf{os}(B)\right)  =(p-1,q-1);$
then the notion of a formal product of bipartition matrices extends to
$2$-formal matrices in the obvious way. The $3$-\emph{formal product}
$c^{3}=T^{3}\left(  \mathbf{c}\right)  :=C_{1}^{3}\cdots C_{r}^{3}\ $of
$2$-formal matrices is a $3$\emph{-formal} \emph{bipartition on }$\mathbf{c}$.

Continue inductively. After $h$ steps we obtain a sequence $\left(
c^{1},\ldots,c^{h}\right)  $ of $i$-formal bipartitions $c^{i}=T^{i}\left(
\mathbf{c}\right)  =C_{1}^{i}\cdots C_{r}^{i}\mathbf{,}$ where $C_{k}%
^{i}=T^{i-1}\left(  \mathbf{C}_{k}^{1}\right)  $ is an $\left(  i-1\right)  $-formal
matrix for each $k.$

\begin{definition}
	\label{GBPM}Let $\mathbf{a}$ and $\mathbf{b}$ be ordered sets. A
	\textbf{generalized bipartition}\emph{ }\textbf{on}\emph{ }$\mathbf{b/a}$ is
	an $i$-formal bipartition on\emph{ }$\mathbf{b/a}$ for some $i\geq1.$ Let
	$\mathbf{C}=\left(  \mathbf{b}_{s}\mathbf{/a}_{t}\right)  $ be a $q\times p$
	elementary matrix over $\left(  \mathbf{a}_{\ast},\mathbf{b}_{\ast}\right)  $
	with $pq\geq1,$ let $c_{st}$ be a generalized bipartition on $\mathbf{b}%
	_{s}\mathbf{/a}_{t}$ for each $\left(s,t\right),$ and let $C=(c_{st})$.
	If $B=\left(
	b_{ij}\right)  $ is a $v\times u$ matrix within $c_{st}$, and $b_{ij}%
	=B_{1}\cdots B_{k-1}\cdot\mathbf{0}\cdot B_{k+1}\cdots B_{r}$ for some $k,$
	identify $b_{ij}$ with $B_{1}\cdots B_{k-1}B_{k+1}\cdots B_{r}$ if either
	$uv=1$ and $B=C$ or $uv>1$ and $B_{i\ast}$ and $B_{\ast j}$ are indecomposable. Then
	$C$ with this identification is a \textbf{Generalized BiPartition Matrix }(GBPM)
	\textbf{over} $\left(  \mathbf{a}_{\ast},\mathbf{b}_{\ast}\right)  .$
	
	Given a bipartition matrix $C^{1}=T^1(\mathbf{C})$ and its
	indecomposable factorization $C^{1}=C_{1}^{1}\cdots$ $ C_{r}^{1}$, let $C=\left(
	C^{1},\ldots,C^{h}\right)  $ be a sequence of GBPMs over $\left(
	\mathbf{a}_{\ast},\mathbf{b}_{\ast}\right)  $ such that $C^{i}=(c_{st}%
	^{i}=T_{st}^{i-1}(C_{1}^{1})\cdots$\linebreak $T_{st}^{i-1}(C_{r}^{1}))$ for $2\leq i\leq
	h;$ then $C$ is an $h$-\textbf{level path (of GBPMs) from}\emph{ }$C^{1}$
	\textbf{to} $C^{h}$ with \textbf{initial bipartition matrix }$C^{1}$ and
	$i^{th}$ \textbf{level} $C^{i}$. The sequence $\tilde{C}=\left(
	\mathbf{C},C^{1},\ldots,C^{h}\right)  $ is an \textbf{augmented }$h$-level
	path from $C^{1}$ to $C^{h}$ with \textbf{augmentation} $\mathbf{C}$\textbf{.}
\end{definition}

\begin{example}
	\label{framed-ex}Let $\mathbf{C=}\left(
	\begin{array}
		[c]{c}%
		\frac{123}{123}\medskip\\
		\frac{456}{123}%
	\end{array}
	\right)  $ and define
	\begin{align*}
		C^{1}  &
		=T^{1}\left(  \mathbf{C}\right)  :=\left(
		\begin{array}
			[c]{c}%
			\frac{\!\!\!123|0}{\,\,\,\,\,\,1|23}=\left(  \frac{123}{1}\right)\!  \left(  \frac{0}{0}\text{
			}\frac{0}{23}\right)  \smallskip\\
			\frac{45|6}{12|3}=\left(
			\begin{array}
				[c]{c}%
				\frac{45}{12}\smallskip\\
				\frac{0}{12}%
			\end{array}
			\right)\!  \left(  \frac{6}{0}\text{ }\frac{6}{0}\text{ }\frac{6}{3}\right)
		\end{array}
		\right)  ,\\
		&
	\end{align*}
	\vspace{-.3in}
	\begin{align*}
		C^{2}  &  =T^{2}\left(  \mathbf{C}\right)  :=\left(
		\begin{array}
			[c]{c}%
			\left(  \frac{12|3}{\,\,0|1}\right)  \left(  \frac{0}{0}\text{ }\frac{0}
			{23}\right)  =\left(
			\begin{array}
				[c]{c}%
				\frac{12}{0}\smallskip\\
				\frac{0}{0}%
			\end{array}
			\right)  \left(  \frac{3}{1}\right)  \left(  \frac{0}{0}\text{ }\frac{0}
			{23}\right)  \smallskip\\
			\left(
			\begin{array}
				[c]{c}%
				\frac{4|5}{1|2}=\left(
				\begin{array}
					[c]{c}%
					\frac{4}{1}\smallskip\\
					\frac{0}{1}%
				\end{array}
				\right)  \left(  \frac{5}{0}\text{ }\frac{5}{2}\right)  \medskip\\
				\frac{0}{12}%
			\end{array}
			\right)  \left(  \frac{6}{0}\text{ }\frac{6}{0}\text{ }\frac{6}{3}\right)
		\end{array}
		\right)  ,\\
		&
	\end{align*}
	\vspace{-.3in}
	\begin{align*}
		\text{and\ \  }C^{3}  &  =T^{3}\left(  \mathbf{C}\right)  :=\left(
		\begin{array}
			[c]{c}%
			\left(
			\begin{array}
				[c]{c}%
				\frac{12}{0}\smallskip\\
				\frac{0}{0}%
			\end{array}
			\right)  \left(  \frac{3}{1}\right)  \left(  \frac{0}{0}\text{ }\frac{0}
			{23}\right)  \smallskip\\
			\left(
			\begin{array}
				[c]{c}%
				\left(
				\begin{array}
					[c]{c}%
					\frac{4}{1}\smallskip\\
					\frac{0}{1}%
				\end{array}
				\right)  \left(  \frac{5}{0}\text{ }\frac{5|0}{0|2}\right)  \smallskip\\
				\frac{0}{12}%
			\end{array}
			\right)  \left(  \frac{6}{0}\text{ }\frac{6}{0}\text{ }\frac{6}{3}\right)
		\end{array}
		\right)  .\\
		&
	\end{align*}
	Then $\left(  \mathbf{C,}\text{ }C^{1},C^{2},C^{3}\right)  $ is an augmented
	$3$-level path of GBPMs from $C^{1}$ to $C^{3}$.
\end{example}

\subsection{The Dimension of a Generalized Bipartition Matrix}

Given a bipartition matrix $C$, let $\pi\left(  C\right)$ denote the matrix obtained from $C$ by discarding empty biblocks $0/0$ in non-null entries and by reducing null entries to $0/0$.

\begin{definition}
	\label{dim-GBPM}Let $C$ be a $q\times p$ GBPM with $pq\geq1$. Denote the
	dimension of $C$ by $\left\vert C\right\vert .$ If $C$ is null, define
	$\left\vert C\right\vert :=0.$ Otherwise, define
	\begin{equation}
		\left\vert C\right\vert :=\left\vert C\right\vert ^{row}+\left\vert
		C\right\vert ^{col}+|C|^{ent}, \label{dim-framed}%
	\end{equation}
	where the \textbf{row dimension }$\left\vert C\right\vert ^{row}$\textbf{,
		column dimension }$\left\vert C\right\vert ^{col}$, and \textbf{entry
		dimension} $\left\vert C\right\vert ^{ent}$ are independent and given by the
	following recursive algorithms\textbf{:\medskip}
	
	\noindent\textbf{Row Dimension Algorithm. }Let $C$ be a $q\times p$ GBPM over
	$\left(  \mathbf{a}_{\ast},\mathbf{b}_{\ast}\right)  $ with $pq\geq1$ and
	indecomposable factorization $C=C_{1}\cdots C_{r}.$
	\smallskip
	
	\noindent\smallskip If $r>1,$ define
	$\left\vert C\right\vert ^{row}:=\sum_{k\in\mathfrak{r}}\left\vert
	C_{k}\right\vert ^{row},$
	where $\left\vert C_{k}\right\vert ^{row}$ is given by setting $C=C_{k}$\linebreak 
	\hspace*{.2in} and continuing recursively.\smallskip
	
	\noindent If $q>1$, define $\left\vert C\right\vert ^{row}:=\sum_{i\in\mathfrak{q}}\left\vert C_{i\ast
	}\right\vert ^{row},$
	where $\left\vert C_{i\ast}\right\vert ^{row}$ is given by setting
	$C=C_{i\ast}$\linebreak 
	\hspace*{.2in} and continuing recursively.\smallskip
	
	\noindent Otherwise, write $C=\left(  c_{1}\cdots c_{p}\right)  .$\smallskip
	
	If $C$ is a bipartition matrix, set $C=\pi(C)$.\smallskip
	
	If $C=\left(  \frac{\mathbf{b}}{\mathbf{a}}\right)  $ with $\mathbf{b}
	\neq\varnothing,$ define $\left\vert C\right\vert ^{row}:=0.$\smallskip
	
	If $C=\left(  \frac{0}{\mathbf{a}_{1}}\cdots\frac{0}{\mathbf{a}_{p}}\right)
	,$ define $
	\left\vert C\right\vert ^{row}:=\left\{
	\begin{array}
		[c]{cc}%
		0, & C\text{ is null}\\
		\#\mathbf{is}(C)-1, & \text{otherwise.}%
	\end{array}
	\right.$\smallskip
	
	Otherwise, define $	\left\vert C\right\vert ^{row}:=\sum_{j\in\mathfrak{p}}\left\vert
	c_{j}\right\vert ^{row},$ where $\left\vert c_{j}\right\vert ^{row}$ is given by setting
	\linebreak 
	\hspace*{.4in}$C=c_{j}$ and continuing recursively.\medskip
	
	\noindent\textbf{Column Dimension Algorithm. }Let $C$ be a $q\times p$ GBPM
	over $\left(  \mathbf{a}_{\ast},\mathbf{b}_{\ast}\right)  $ with $pq\geq1$
	and  indecomposable factorization $C=C_{1}\cdots C_{r}.$\smallskip
	
	\noindent If $r>1,$ define $\left\vert C\right\vert ^{col}:=\sum_{k\in\mathfrak{r}}\left\vert
	C_{k}\right\vert ^{col},$
	where $\left\vert C_{k}\right\vert ^{col}$ is given by setting $C=C_{k}$\linebreak 
	\hspace*{.2in} and continuing recursively.\smallskip
	
	\noindent If $p>1$, define $\left\vert C\right\vert ^{col}:=\sum_{j\in\mathfrak{p}}\left\vert C_{\ast
		j}\right\vert ^{col},$ where $\left\vert C_{\ast j}\right\vert ^{col}$ is given by setting $C=C_{\ast j}$\linebreak 
	\hspace*{.2in} and continuing recursively.\smallskip
	
	\noindent Otherwise, write $C=\left(  c_{1}\cdots c_{q}\right)  ^{T}.$\smallskip
	
	If $C$ is a bipartition matrix, set $C=\pi(C)$.\smallskip
	
	If $C=\left(  \frac{\mathbf{b}}{\mathbf{a}}\right)  $ with $\mathbf{a}
	\neq\varnothing,$ define $\left\vert C\right\vert ^{col}:=0.$\smallskip
	
	If $C=\left(  \frac{\mathbf{b}_{1}}{0}\cdots\frac{\mathbf{b}_{q}}{0}\right)
	^{T},$ define $
	\left\vert C\right\vert ^{col}:=\left\{
	\begin{array}
		[c]{cc}%
		0, & C\text{ is null}\\
		\#\mathbf{os}(C)-1, & \text{otherwise.}%
	\end{array}
	\right.$\smallskip	
	
	Otherwise, define $|C|^{col}:=\sum_{i\in\mathfrak{q}}\left\vert c_{i}\right\vert ^{col},$ where $\left\vert c_{i}\right\vert ^{col}$ is given by setting
	$C=c_{i}$\linebreak 
	\hspace*{.4in}and continuing recursively.\medskip
	
	\noindent\textbf{Entry Dimension Algorithm. }Let $C=\left(  c_{ij}\right)  $
	be a $q\times p$ GBPM over $\left(  \mathbf{a}_{\ast},\mathbf{b}_{\ast
	}\right)   $ with $pq\geq1$ and indecomposable factorization $C=C_{1}\cdots
	C_{r}.$\smallskip
	
	\noindent If $r>1,$ define $\left\vert C\right\vert ^{ent}:=\sum_{k\in\mathfrak{r}}\left\vert
	C_{k}\right\vert ^{ent},$ where $\left\vert C_{k}\right\vert ^{ent}$ is given by setting $C=C_{k}$ and\linebreak 
	\hspace*{.2in}continuing recursively.\smallskip
	
	\noindent Otherwise, define
	$|C|^{ent}:=\sum_{\left(  i,j\right)  \in\mathfrak{q}\times\mathfrak{p}
	}\left\vert c_{ij}\right\vert ^{ent},$ where $\left\vert c_{ij}\right\vert ^{ent}$ is given by setting\linebreak 
	\hspace*{.2in}$C=c_{ij}$ and continuing recursively unless $c_{ij}=\frac{\mathbf{b}_{i}}{\mathbf{a}_{j}},$ in which case define
	\[
	\left\vert c_{ij}\right\vert ^{ent}:=\left\{
	\begin{array}
		[c]{cl}%
		\#\mathbf{a}_{j}+\#\mathbf{b}_{i}-1, & \mathbf{a}_{j},\mathbf{b}_{i}
		\neq\varnothing\\
		0, & \text{otherwise.}%
	\end{array}
	\right.
	\]
\end{definition}
\smallskip

\noindent We highlight several important facts:\ \medskip

\begin{enumerate}

	\item If $C=C_{1}\cdots C_{r},$ then
	\begin{equation}
		\left\vert C\right\vert =\sum_{k\in\mathfrak{r}}\left(  \left\vert
		C_{k}\right\vert ^{row}+\left\vert C_{k}\right\vert ^{col}+\left\vert
		C_{k}\right\vert ^{ent}\right)  =\sum_{k\in\mathfrak{r}}\left\vert
		C_{k}\right\vert . \label{Cdecomp}%
	\end{equation}

	\item $\left\vert C\right\vert $ \emph{is the sum of the dimensions of all
		terminating matrices in the recursion}, i.e., elementary bipartitions
	$\frac{\mathbf{b}}{\mathbf{a}}$ with $\mathbf{a},\mathbf{b}\neq\varnothing$,
	elementary row matrices $\left(  \frac{0}{\mathbf{a}_{1}}\cdots\frac
	{0}{\mathbf{a}_{p}}\right)  ,$ and elementary column matrices $\left(
	\frac{\mathbf{b}_{1}}{0}\cdots\frac{\mathbf{b}_{q}}{0}\right)  ^{T}$. Indeed, if $C$ is an elementary matrix,
	$\left\vert C_{i\ast}\right\vert ^{row}=0$ when the denominators in $C_{i\ast
	}$  are empty or the (equal) numerators in $C_{i\ast}$ are non-empty; dually,
	$\left\vert C_{\ast j}\right\vert ^{col}=0$ when the numerators in $C_{\ast
		j}$ are empty or the (equal) denominators in $C_{\ast j}$ are non-empty. Thus
	if $C$ is an elementary matrix over $\left(  \mathbf{a}_{\ast},\mathbf{b}
	_{\ast}\right)  $ with $\mathbf{a}_{j},\mathbf{b}_{i}\neq\varnothing$ for all
	$\left(  i,j\right)  ,$ then $\left\vert C\right\vert ^{row}=\left\vert
	C\right\vert ^{col}=0$ and $\left\vert C\right\vert =\left\vert C\right\vert
	^{ent}.$
	
	\item $\left\vert C\right\vert $\emph{\ is not the sum of the dimensions of
		its entries}. For an underlying geometric example, consider a facet
	$A|B\sqsubset P_{n}$. Let $\mathbf{a}_{1}
	|\cdots|\mathbf{a}_{p}:=EP_{\mathfrak{n}}B;$ then $A|B$ corresponds to the
	monomial $C_{1}C_{2}=\left(  \frac{0}{A}\right)  \left(  \frac{0}
	{\mathbf{a}_{1}}\cdots\frac{0}{\mathbf{a}_{p}}\right)  \in\mathfrak{n}
	\circledast \mathfrak{0}.$ By the Row Dimension Algorithm, $|C_{1}C_{2}|=|C_{1}
	|+|C_{2}|=\#\mathbf{is}(C_{1})-1+\#\mathbf{is}(C_{2})-1=n-2$ as required, but if $\mathbf{a}_{i},\mathbf{a}_{j}\neq\varnothing$ for some $i\neq
	j,$ then $\left\vert \frac{0}{\mathbf{a}_{1}}\right\vert +\cdots+\left\vert
	\frac{0}{\mathbf{a}_{p}}\right\vert <|C_{2}|$ .
	
	\item When $C$ is an indecomposable row or column bipartition matrix,
	the row and column dimension algorithms reduce the
	calculation to that of $\left\vert \pi\left(  C\right)  \right\vert $. This
	reduction is critical. For example, if $C=\left(  \frac{0}{1}\text{ \ }
	\frac{0|0}{0|2}\right)  $, then $\left\vert C\right\vert =\left\vert
	\pi\left(  C\right)  \right\vert =\left\vert \left(  \frac{0}{1}\text{
		\ }\frac{0}{2}\right)  \right\vert =1$, but without reduction to
	$|\pi\left(  C\right) | $, the computed dimension of $C$ would be $0$.
\end{enumerate}

\begin{example}
	Let us compute the dimension of the indecomposable bipartition matrix
	\[
	C=\left(
	\begin{array}
		[c]{cc}%
		\frac{0}{12} & \frac{0}{0}\smallskip\\
		\frac{3|4}{1|2} & \frac{34}{0}%
	\end{array}
	\right)  .
	\]
	\bigskip
	
	\noindent Then $\left\vert C\right\vert =|C|^{row}+|C|^{col}+\left\vert
	C\right\vert ^{ent},$ where
	
	\noindent$|C|^{row}=\left\vert \left(
	\begin{array}
		[c]{cc}%
		\frac{0}{12} & \frac{0}{0}%
	\end{array}
	\right)  \right\vert ^{row}+\left\vert \left(
	\begin{array}
		[c]{cc}%
		\left(  \frac{34}{12},\left(
		\begin{array}
			[c]{c}%
			\frac{3}{1}\smallskip\\
			\frac{0}{1}%
		\end{array}
		\right)  \left(
		\begin{array}
			[c]{cc}%
			\frac{4}{0} & \frac{4}{2}%
		\end{array}
		\right)  \right)  & \left(  \frac{34}{0},\frac{34}{0}\right)
	\end{array}
	\right)  \right\vert ^{row}$
	\[
	=1+\left\vert \left(
	\begin{array}
		[c]{c}%
		\frac{3}{1}\smallskip\\
		\frac{0}{1}%
	\end{array}
	\right)  \right\vert ^{row}+\left\vert \left(
	\begin{array}
		[c]{cc}%
		\frac{4}{0} & \frac{4}{2}%
	\end{array}
	\right)  \right\vert ^{row}+0=1;
	\]

	\noindent$|C|^{col}=\left\vert \left(
	\begin{array}
		[c]{c}%
		\frac{0}{12}\smallskip\\
		\frac{3|4}{1|2}%
	\end{array}
	\right)  \right\vert ^{col}+\left\vert \left(
	\begin{array}
		[c]{c}%
		\frac{0}{0}\smallskip\\
		\frac{34}{0}%
	\end{array}
	\right)  \right\vert ^{col}=0+\left\vert \left(
	\begin{array}
		[c]{c}%
		\frac{3}{1}\smallskip\\
		\frac{0}{1}%
	\end{array}
	\right)  \right\vert ^{col}+\left\vert \left(
	\begin{array}
		[c]{cc}%
		\frac{4}{0} & \frac{4}{2}%
	\end{array}
	\right)  \right\vert ^{col}+1=1;$\bigskip
	
	\noindent$\left\vert C\right\vert ^{ent}=\left\vert \tfrac{0}{12}\right\vert
	^{ent}+\left\vert \frac{0}{0}\right\vert ^{ent}+\left\vert \left(
	\begin{array}
		[c]{c}%
		\frac{3}{1}\smallskip\\
		\frac{0}{1}%
	\end{array}
	\right)  \right\vert ^{ent}+\left\vert \left(
	\begin{array}
		[c]{cc}%
		\frac{4}{0} & \frac{4}{2}%
	\end{array}
	\right)  \right\vert ^{ent}+\left\vert \tfrac{34}{0}\right\vert ^{ent}
	=0+0+1+1+0=2.$\bigskip
	
	\noindent Therefore $\left\vert C\right\vert =4.$
\end{example}

\begin{example}
	\label{22framed}  Let us compute the dimension of the indecomposable
	bipartition  matrix
	\[
	C=\left(\!\!
	\begin{array}
		[c]{cc}%
		\frac{0|1}{1|0}\smallskip & \frac{0|1}{2|0}\\
		\frac{0|2}{1|0} & \frac{0|2|0}{0|2|0}
	\end{array}\!\!
	\right)  =
	\left(\!\!
	\begin{array}
		[c]{ccc}
		\left(
		\begin{array}
			[c]{c}%
			\frac{0}{1}\smallskip\\
			{\frac{0}{1}}%
		\end{array}
		\right)  \left(  \frac{1}{0}\text{ \ }\frac{1}{0}\right)  &  & \left(
		\begin{array}
			[c]{c}%
			\frac{0}{2}\smallskip\\
			{\frac{0}{2}}%
		\end{array}
		\right)  \left(  \frac{1}{0}\text{ \ }\frac{1}{0}\right) \\
		&  & \\
		\left(
		\begin{array}
			[c]{c}%
			\frac{0}{1}\smallskip\\
			{\frac{0}{1}}%
		\end{array}
		\right)  \left(  \frac{2}{0}\text{ \ }\frac{2}{0}\right)  &  & \frac
		{0|2|0}{0|2|0}=\left(
		\begin{array}
			[c]{c}%
			\frac{0}{0}\smallskip\\
			{\frac{0}{0}}%
		\end{array}
		\right)  \left(  \frac{2|0}{0|0}\right)  =\left(  \tfrac{0|2}{0|2}\right)
		\left(  \frac{0}{0}\text{ \ }\frac{0}{0}\right)
	\end{array}
	\!\!
	\right)\!  .
	\]
	\noindent Then
	\[
	\begin{array}
		[c]{ccc}%
		C_{1\ast}=\left(
		\begin{array}
			[c]{cc}%
			\frac{0}{1} & \frac{0}{2}\smallskip\\
			\frac{0}{1} & \frac{0}{2}%
		\end{array}
		\right)  \left(
		\begin{array}
			[c]{cccc}%
			\frac{1}{0} & \frac{1}{0} & \tfrac{1}{0} & \tfrac{1}{0}%
		\end{array}
		\right)  , &  & C_{2\ast}=\left(
		\begin{array}
			[c]{cc}%
			\frac{0}{1} & \frac{0}{0}\smallskip\\
			\frac{0}{1} & \frac{0}{0}%
		\end{array}
		\right)  \left(
		\begin{array}
			[c]{ccc}%
			\frac{2}{0} & \frac{2}{0} & \tfrac{2|0}{2|0}%
		\end{array}
		\right)  ,\\
		&  & \\
		C_{\ast1}=\left(
		\begin{array}
			[c]{c}%
			\frac{0}{1}\smallskip\\
			\frac{0}{1}\smallskip\\
			\frac{0}{1}\smallskip\\
			\frac{0}{1}%
		\end{array}
		\right)  \left(
		\begin{array}
			[c]{cc}%
			\frac{1}{0} & \frac{1}{0}\smallskip\\
			\frac{2}{0} & \frac{2}{0}%
		\end{array}
		\right)  , & \text{and} & C_{\ast2}=\left(
		\begin{array}
			[c]{c}%
			\frac{0}{2}\smallskip\\
			\frac{0}{2}\smallskip\\
			\tfrac{0|2}{0|2}%
		\end{array}
		\right)  \left(
		\begin{array}
			[c]{cc}%
			\frac{1}{0} & \frac{1}{0}\smallskip\\
			\frac{0}{0} & \frac{0}{0}%
		\end{array}
		\right)
	\end{array}
	\]
	\bigskip
	
	\noindent so that $\left\vert C\right\vert ^{row}=\left\vert \left(
	\begin{array}
		[c]{cc}%
		\frac{0}{1} & \frac{0}{2}\smallskip\\
		\frac{0}{1} & \frac{0}{2}%
	\end{array}
	\right)  \right\vert ^{row}=1+1=2,$ $\left\vert C\right\vert ^{col}=\left\vert
	\left(
	\begin{array}
		[c]{cc}%
		\frac{1}{0} & \frac{1}{0}\smallskip\\
		\frac{2}{0} & \frac{2}{0}%
	\end{array}
	\right)  \right\vert ^{col}=1+1=2,$\bigskip\linebreak and $\left\vert
	C\right\vert ^{ent}=\left\vert \frac{0|2|0}{0|2|0}\right\vert =1.$ Therefore
	$\left\vert C\right\vert =|C|^{row}+|C|^{col}+|C|^{ent}=5.$
\end{example}

\begin{example}
	Consider the indecomposable GBPM $C=C^{3}$ in Example \ref{framed-ex}
	:$\bigskip$
	\[
	C=\left(
	\begin{array}
		[c]{c}%
		\left(
		\begin{array}
			[c]{c}%
			\frac{12}{0}\smallskip\\
			\frac{0}{0}%
		\end{array}
		\right)  \left(  \frac{3}{1}\right)  \left(  \frac{0}{0}\text{ }\frac{0}
		{23}\right)  \smallskip\\
		\left(
		\begin{array}
			[c]{c}%
			\left(
			\begin{array}
				[c]{c}%
				\frac{4}{1}\smallskip\\
				\frac{0}{1}%
			\end{array}
			\right)  \left(  \frac{5}{0}\text{ \ }\left(  \frac{5}{0}\right)  \left(
			\frac{0}{2}\right)  \right)  \smallskip\\
			\frac{0}{12}%
		\end{array}
		\right)  \left(  \frac{6}{0}\text{ }\frac{6}{0}\text{ }\frac{6}{3}\right)
	\end{array}
	\right)  .
	\]
	\smallskip
	
	\noindent Then $|C|^{row}=\left\vert \left(  \frac{0}{0}\text{ }\frac{0}
	{23}\right)  \right\vert ^{row}+\left\vert \left(  \frac{0}{12}\right)
	\right\vert ^{row}=2,$ $\ |C|^{col}=\left\vert \left(
	\begin{array}
		[c]{c}%
		\frac{12}{0}\smallskip\\
		\frac{0}{0}%
	\end{array}
	\right)  \right\vert ^{col}=1,$ and$\bigskip$
	
	\noindent$\left\vert C\right\vert ^{ent}=\left\vert \frac{3}{1}\right\vert
	^{ent}+\left\vert \frac{4}{1}\right\vert ^{ent}+\left\vert \frac{6}
	{3}\right\vert ^{ent}=3$ imply $\left\vert C\right\vert =6.\bigskip$
\end{example}
Our next proposition follows immediately from Proposition \ref{el-coherence}.

\begin{proposition}
	\label{coh-elem}If $\mathbf{C}$ is a $q\times p$ coherent elementary matrix
	with $p,q\geq2,\ $then
	\[
	\left\vert \mathbf{C}\right\vert =\left\{
	\begin{array}
		[c]{cc}%
		1, & \mathbf{is}\left(  \mathbf{C}\right)  ,\mathbf{os}\left(  \mathbf{C}
		\right)  \neq\varnothing,\\
		0, & \text{otherwise.}%
	\end{array}
	\right.
	\]
	
\end{proposition}

In the special case of bipartition matrices, computer implementations of the algorithms in Definition \ref{dim-GBPM} were given by D. Freeman and the second author in \cite{Freeman}.

\section{Framed Matrices}\label{FM}

\subsection{The Framed Join $\mathfrak{m}{\circledast}\mathfrak{n}$}

Given an elementary bipartition $\mathbf{c,}$ choose a $1$-formal bipartition
$c^{1}=T^{1}\left(  \mathbf{c}\right)  $ with elementary factorization
$c^{1}=\mathbf{C}_{1}^{1}\cdots \mathbf{C}_{r}^{1}.$ Then $\left(  c^{1}\right)  $ is a
$1$-level path of generalized bipartitions on $\mathbf{c.}$ A\emph{ structure
	element within }$c^{1}$ is an elementary bipartition entry of some
$\mathbf{C}_{k}^{1}.$ Define $\mathcal{T}_{1}\left(  \mathbf{c}\right)  :=\{1$-level
paths of generalized bipartitions on $\mathbf{c\}}$.

Let $\left(  c^{1},c^{2}=C_{1}^{2}\cdots C_{r}^{2}\right)  $ be a $2$-level
path of generalized bipartitions\ on $\mathbf{c.}$ A \emph{structure element
	within }$c^{2}$ is a bipartition entry of some $C_{k}^{2}.$ Define
$\mathcal{T}_{2}\left(  \mathbf{c}\right)  :=\{2$-level paths of generalized
bipartitions on $\mathbf{c}\}$.

Let $\left(  c^{1},c^{2},c^{3}=C_{1}^{3}\cdots C_{r}^{3}\right)  $ be a
$3$-level path of generalized bipartitions\ on $\mathbf{c.}$ A \emph{structure
	element within }$c^{3}$ is a structure element within an entry of some
$C_{k}^{3}.$ Define $\mathcal{T}_{3}\left(  \mathbf{c}\right)  :=\{3$-level
paths of generalized bipartitions on $\mathbf{c\}}$.

Continuing inductively, let $\left(  c^{1},\ldots,c^{h}=C_{1}^{h}\cdots
C_{r}^{h}\right)  ,$ $h>1,$ be an $h$-level path of generalized
bipartitions\ on $\mathbf{c.}$ Then a structure element within $c^{h}$ is a
structure element within an entry of some $C_{k}^{h}$ and $\mathcal{T}%
_{h}\left(  \mathbf{c}\right)  =\{h$-level paths of generalized bipartitions
on $\mathbf{c\}}$. Note that if $i<j\leq h$ and all choices from level $i$
through level $j-1$ are trivial, then $c^{i}=c^{i+1}=\cdots=c^{j}$.

\begin{definition}
	\label{framed-prejoin}Let $\mathbf{a}$ and $\mathbf{b}$ be ordered sets and
	let $\mathbf{c=b/a.}$ The \textbf{framed join of }$\mathbf{a}$ \textbf{with
	}$\mathbf{b}$ is defined and denoted by
	\begin{equation}
		\mathbf{a}\circledast\mathbf{b}:=\left.  \bigcup_{h\geq1}{\mathcal{T}}
		_{h}\left(  \mathbf{c}\right)  \right/  \sim,\text{ where} \label{hierarchy}%
	\end{equation}

	\begin{enumerate}

		\item $\left(  c^{1},\ldots,c^{h}\right)  \sim\left(  d^{1},\ldots
		,d^{h}\right)  $ if $c^{h}=d^{h},$
		
		\item $\left(  c^{1},\ldots,c^{h}\right)  \sim\left(  c^{1},\ldots,c^{i}
		,c^{i},\ldots,c^{h}\right)  $ for all $i$.
	\end{enumerate}
	
	A class $[c]\in\mathbf{a}\circledast\mathbf{b}$ is a \textbf{framed element on
	}$\mathbf{c.}$ Let $c=\left(  c^{1},\ldots,c^{h}\right)  \in\left[  c\right]
	.$ A \textbf{level }$i$ \textbf{structure element within} $c$ is a structure
	element within $c^{i}.$ The \textbf{canonical path} in\textbf{ }$\left[  c\right]  ,$
	denoted by $\bar{c}=(\bar{c}^{1},\ldots,\bar{c}^{\bar{h}}),$ has a minimum
	number of levels $\bar{h}$ and components $\bar{c}^{i}$ of maximal
	(indecomposable) factorization length. The \textbf{height }$h\left[  c\right]
	:=\bar{h}$\textbf{. }
	
	Let $\mathbf{C}=(\mathbf{b}_{s}/\mathbf{a}_{t})$ be a $q\times p$ elementary
	matrix with $pq\geq1.$ For each $\left(  s,t\right)  ,$ let $\left[
	c_{st}\right]  \in\mathbf{a}_{t}\circledast\mathbf{b}_{s};$ Let $h=max\left\{
	h\left[  c_{st}\right]  \right\}  ,$ choose $(c_{st}^{1},\ldots,c_{st}
	^{h_{st}})\in\left[  c_{st}\right]  $ such that $h_{st}\leq h,$ and define
	\begin{equation}
		c_{st}:=(c_{st}^{1},\ldots,c_{st}^{h_{st}},\underset{h-h_{st}
		}{\underbrace{c_{st}^{h_{st}},\ldots,c_{st}^{h_{st}}}}). \label{adjusted}%
	\end{equation}
	Then $\left[  C\right]  =\left(  \left[  c_{st}\right]  \right)  $ is a
	$q\times p$ \textbf{framed matrix} \textbf{on }$\mathbf{C}.$ Let
	$C^{i}=\left(  c_{st}^{i}\right)  ;$ then $C=\left(  C^{1},\ldots
	,C^{h}\right)  \in\left[  C\right]  $, the \textbf{initial bipartition
		matrix}  of $C$ is $C^{1},$ the $i^{th}$ \textbf{level of }$C$ is $C^{i},$ and
	a  \textbf{structure matrix} \textbf{within} $C$ is an indecomposable factor
	of  either $C^{1}$ or a structure element within some entry of some $C^{i},$
	$i>1.$ Set $c_{st}=\bar{c}_{st}$ in (\ref{adjusted})\textbf{; }the
	\textbf{canonical path }in $\left[  C\right]  $ is $\bar{C}:=(\bar{c}_{st})$
	and the \textbf{height }$h\left[  C\right]  :=\bar{h}.$ The \textbf{input set
	}$\mathbf{is}[C]:=\mathbf{is}(\mathbf{C}),$ the \textbf{output set}
	$\mathbf{os}[C]:=\mathbf{os}(\mathbf{C}),$ and the \textbf{dimension
	}$\left\vert \left[  C\right]  \right\vert :=\left\vert C\right\vert
	:=\left\vert C^{h}\right\vert .$
\end{definition}

In view of Definition \ref{GBPM}, a framed matrix has finitely many
representatives.
A
bipartition  can be thought of as a framed element of height $1,$ and a
bipartition matrix can be thought of as a framed matrix of height $1$. In other words, the set of all bipartitions on $\left(  \mathbf{a,b}\right)$
is a subset of $  \mathbf{a} \circledast\mathbf{b},$ and the inclusion is a bijection when
$\left(  \mathbf{a,b}\right)  =\left(  \mathfrak{n},
\mathfrak{0}\right)  $ (or $\left(  \mathfrak{0},\mathfrak{n}\right)  $). On the other hand, such bipartitions
can be identified with ordinary
partitions in $P\left( \mathfrak{n}\right)  $ by discarding null numerators
(or null denominators). Thus $\mathfrak{0}\circledast\mathfrak{n}
=\mathfrak{n}\circledast\mathfrak{0}=P\left(  \mathfrak{n}\right)  .$

Let $\left[  c\right]  $ be a framed element and let $c=\left(  c^{1}%
,\ldots,c^{h}\right)  \in\left[  c\right]  .$ Definition \ref{framed-prejoin},
part (1), says that all paths in $\left[  c\right]  $ terminate at $c^{h},$
and part (2) says that repeated components
in $c$ can be suppressed. Thus the paths in $\left[  c\right]  $ form a
connected digraph with maximal element $c^{h}$. The canonical path $\bar{c}%
\in\left[  c\right]  $ has minimal height $\bar{h}$ and components of maximal
factorization length. Thus, if $c$ is a non-canonical path of minimal height
$\bar{h},$ some component $c^{i}$ has non-maximal factorization length.

\begin{example}
	\label{canonical}Let $\mathbf{a}=\mathbf{b}=\{1,2,3,4\}.$ The initial
	bipartition of the canonical path
	\[
	\bar{c}=\left(  \frac{1|2|34}{1|2|34},\text{ }\left(
	\begin{array}
		[c]{c}%
		\frac{1}{1}\smallskip\\
		\frac{0}{1}\smallskip\\
		\frac{0}{1}\smallskip\\
		\frac{0}{1}%
	\end{array}
	\right)  \left(
	\begin{array}
		[c]{cc}%
		\frac{2}{0}\smallskip & \frac{2}{2}\smallskip\\
		\frac{0}{0}\smallskip & \frac{0}{2}\smallskip\\
		\frac{0}{0} & \frac{0}{2}%
	\end{array}
	\right)  \left(  \tfrac{4|3}{0|0}\text{ }\tfrac{3|4}{0|0}\text{ }\tfrac
	{3|4}{4|3}\right)  \right)
	\]
	has factorization length $3.$ Since
	\[
	\left(
	\begin{array}
		[c]{c}%
		\frac{1}{1}\smallskip\\
		\frac{0}{1}\smallskip\\
		\frac{0}{1}\smallskip\\
		\frac{0}{1}%
	\end{array}
	\right) \!\! \left(  \tfrac{2|4|3}{0|0|0}\text{ }\tfrac{2|3|4}{2|4|3}\right)
	=\left(
	\begin{array}
		[c]{c}%
		\frac{1}{1}\smallskip\\
		\frac{0}{1}\smallskip\\
		\frac{0}{1}\smallskip\\
		\frac{0}{1}%
	\end{array}
	\right)\!\!  \left(
	\begin{array}
		[c]{cc}%
		\frac{2}{0}\smallskip & \frac{2}{2}\smallskip\\
		\frac{0}{0}\smallskip & \frac{0}{2}\smallskip\\
		\frac{0}{0} & \frac{0}{2}%
	\end{array}
	\right) \!\! \left(  \tfrac{4|3}{0|0}\text{ }\tfrac{3|4}{0|0}\text{ }\tfrac
	{3|4}{4|3}\right)  =\left(
	\begin{array}
		[c]{c}%
		\frac{1|2}{1|2}\smallskip\\
		\frac{0|0}{1|2}\smallskip\\
		\frac{0|0}{1|2}%
	\end{array}
	\right)\!\!  \left(  \tfrac{4|3}{0|0}\text{ }\tfrac{3|4}{0|0}\text{ }\tfrac
	{3|4}{4|3}\right)  ,
	\]
	the framed element $\left[  \bar{c}\right]  $ is also represented by
	\[
	\left(  \frac{1|234}{1|234},\text{ }\left(
	\begin{array}
		[c]{c}%
		\frac{1}{1}\smallskip\\
		\frac{0}{1}\smallskip\\
		\frac{0}{1}\smallskip\\
		\frac{0}{1}%
	\end{array}
	\right)  \left(
	\begin{array}
		[c]{cc}%
		\frac{2|4|3}{0|0|0}\text{ } & \frac{2|3|4}{2|4|3}%
	\end{array}
	\right)  \right)
	\]
	\ and
	\[
	\left(  \frac{12|34}{12|34},\text{ }\left(
	\begin{array}
		[c]{c}%
		\frac{1|2}{1|2}\smallskip\\
		\frac{0|0}{1|2}\smallskip\\
		\frac{0|0}{1|2}%
	\end{array}
	\right)  \left(
	\begin{array}
		[c]{ccc}%
		\frac{4|3}{0|0} & \frac{4|3}{0|0} & \frac{3|4}{4|3}%
	\end{array}
	\right)  \right)
	\]
	whose initial bipartitions have non-maximal factorization length $2.$
\end{example}

\subsection{The Structure Tree of a Framed Element}

Let $\left[  c\right]  $ be a framed element, let $c=\left(  c^{1}%
,\ldots,c^{h}\right)  \in\left[  c\right]  $ and consider the augmented path
$\tilde{c}=\left(  \mathbf{c}_{1}^{0},c^{1},\ldots,c^{h}\right)  .$ Then
$\left(  \mathbf{c}_{1}^{0}\right)  $ is an $m_{0}$-tuple with $m_{0}=1.$
Inductively, if $i\geq0$ and the $m_{i}$-tuple $c^{i}=\left(  \mathbf{c}%
_{1}^{i},\ldots,\mathbf{c}_{m_{i}}^{i}\right)  $ has been constructed, write
$T^{1}(\mathbf{c}_{j}^{i})=\mathbf{C}_{1}^{i}\cdots\mathbf{C}_{r_{ij}}^{i}$.
Identify $\mathbf{C}_{k}^{i}$ with the tuple of its entries listed in row
order, identify $T^{1}(\mathbf{c}_{j}^{i})$ with the concatenation of tuples
$(\mathbf{C}_{1}^{i},\ldots,\mathbf{C}_{r_{ij}}^{i})$, and identify $c^{i+1}$
with the $m_{i+1}$-tuple $\left(  \mathbf{c}_{1}^{i+1},\ldots,\mathbf{c}%
_{m_{i+1}}^{i+1}\right)  :=\left(  T^{1}\left(  \mathbf{c}_{1}^{i}\right)
,\ldots,T^{1}\left(  \mathbf{c}_{m_{i}}^{i}\right)  \right)  $.

The \emph{structure tree} \emph{of} $c,$ denoted by $\mathbb{T}\left(
c\right)  ,$ is the PLT with $h+1$ levels constructed as
follows: If $h=1,$ then%
\[
\mathbb{T}\left(  c\right)  :=%
\begin{array}
	[c]{c}%
	\mathbf{c}_{1}^{1}\text{ }\cdots\text{ }\mathbf{c}_{m_{1}}^{1}\\
	\diagdown|\diagup\\
	\mathbf{c}_{1}^{0}%
\end{array}
\]
is the $2$-level tree with root in level $0$ and vertices $\mathbf{c}_{j}^{1}$
in level $1.$ Inductively, if $1\leq i<h$ and level $i$ has been constructed,
construct level $i+1$ by connecting each level $i$ vertex $\mathbf{c}_{j}^{i}$
to the level $i+1$ vertices in $T^{1}(\mathbf{c}_{j}^{i})$ (see Figure 3).

\[
\begin{tabular}
	[c]{c}
	$
	\begin{array}
		[c]{c}%
		\mathbf{c}_{1}^{h}\cdots\cdots\cdots\mathbf{c}_{m_{h}}^{h}\\
		\vdots\text{ \ }\ \ \ \ \ \text{ \ \ }\vdots\\
		\diagdown|\diagup\cdots\diagdown|\diagup\\
		\text{ }\mathbf{c}_{1}^{1}\text{ }\cdots\text{ }\mathbf{c}_{m_{1}}^{1}\\
		\diagdown|\diagup\\
		\mathbf{c}_{1}^{0}%
	\end{array}
	$
\end{tabular}
\]
\medskip
\begin{center}
	Figure 3. The structure tree $\mathbb{T}\left(  c\right)$.
	\medskip
\end{center}
When the marked units on a measuring stick are levels, the measured height of
$\mathbb{T}\left(  c\right)  $ is $h$. The \emph{structure tree of} $\left[  c\right]  $ is defined and denoted by
$\mathbb{T}\left[  c\right]  :=\mathbb{T}\left(  \bar{c}\right)  .$

Let $1\leq i<h$. Given a bipartition $b$ within $c^{i+1}$, let $b=\mathbf{B}_{1}%
\cdots\mathbf{B}_{r}$ be the elementary factorization and let $b=T^{1}\left(
\mathbf{b}\right);$ note that $\mathbf{b}$ lies
within $c^{i}$ when $b$ is a structure element. Let $\left\{  \mathbf{b}_{1},\ldots,\mathbf{b}%
_{t}\right\}  $ be the entries of $\mathbf{B}_{1},\ldots,\mathbf{B}_{r}$
listed in row order as above. For each $s,$ let $b_{s}=T^{1}\left(  \mathbf{b}%
_{s}\right)  $ and denote the maximal subtree of
$\mathbb{T}\left(  c\right)  $ with root $\mathbf{b}_{s}$ by $\mathbb{T}(c|_{b_{s}})$. The \emph{restriction of }%
$c$\emph{ to }$b$ is the representative $c|_{b}:=\left(  T^{1}%
(\mathbf{b}),\dots,T^{h-i}(\mathbf{b})\right)  $ whose structure tree
$\mathbb{T}\left(  c|_{b}\right)$ is obtained by
attaching the root $\mathbf{b}_{s}$ of each subtree $\mathbb{T}(c|_{b_{s}})$ to $\mathbf{b}$. The class $\left[  c|_{b}\right]  $ is a
\emph{subframed element in }$\left[  c\right]  .$ Let $C=\left(  C^{1}%
,\ldots,C^{h}\right)  $ be a framed matrix representative and consider a
bipartition matrix $B=\left(  b_{uv}\right)  $ within $C^{i+1}.$ The
\emph{restriction of }$C$\emph{ to }$B,$ defined and denoted by $C|_{B}%
:=\left(  C|_{b_{uv}}\right)  ,$ represents the \emph{subframed matrix
}$\left[  C|_{B}\right]  $ \emph{in }$\left[  C\right]  $.

Let $\Gamma_{m}^{n}$ denote the upward directed double corolla with $m$
incoming leaves (inputs) and $n$ outgoing leaves (outputs):
\vspace{0.2in}

\[
\hspace{-.7in}\Gamma_{m}^{n}=
\]

\vspace{-0.55in} \unitlength=1.00mm\linethickness{0.4pt}
\hspace{.2in}
\ifx\plotpoint\newsavebox{\plotpoint}\fi
\begin{picture}(61.69,19.771)(0,0)
	\multiput(49.649,15.757)(.0341096942,-.0336885868){353}{\line(1,0){.0341096942}}
	\multiput(61.541,15.906)(-.0336885868,-.0341096942){353}{\line(0,-1){.0341096942}}
	\multiput(52.43,15.906)(.033642043,-.062589848){190}{\line(0,-1){.062589848}}
	\put(56.785,15.757){\makebox(0,0)[cc]{$\cdots$}}
	\put(54.109,4.162){\makebox(0,0)[cc]{$\cdots$}}
	\put(56.19,19.00){\makebox(0,0)[cc]{$n$}}
	\put(55.893,1.635){\makebox(0,0)[cc]{$m$}}
\end{picture}

\noindent Given ordered sets $\mathbf{a}$ and $\mathbf{b,}$ there is the
projection
\begin{equation}
	\frac{\mathbf{b}}{\mathbf{a}}\rightarrow\Gamma_{\#\mathbf{a}+1}^{\#\mathbf{b}
		+1}. \label{elementary-correspondence}%
\end{equation}
When the context is clear, we suppress the scripts and simply write $\Gamma$.
The \emph{graph structure tree of} $\left[  c\right]  ,$ denoted by
$\Gamma\left[  c\right]  ,$ is the PLT given by
replacing each vertex $\mathbf{c}_{j}^{i}$ in $\mathbb{T}_{h}\left[  c\right]
$ with the corresponding double corolla $\Gamma_{ij}.$

\begin{equation}
	\begin{tabular}
		[c]{c}
		$
		\begin{array}
			[c]{c}
			\Gamma_{h1}\cdots\cdots\cdots\Gamma_{hm_{h}}\\
			\vdots\text{ \ }\ \ \ \ \ \text{ \ \ }\vdots\\
			\diagdown|\diagup\cdots\diagdown|\diagup\\
			\text{ }\Gamma_{11}\text{\thinspace}\cdots\,\Gamma_{1m_{1}}\\
			\diagdown|\diagup\\
			\Gamma_{01}
		\end{array}
		$
	\end{tabular}
	\ \ \ \ \ \label{g-hierarchy}%
\end{equation}
\medskip
\begin{center}
	Figure 4. The graph structure tree $\Gamma\left[  c\right]$.
\end{center}
\medskip

The \emph{graph matrix} of an elementary matrix $\mathbf{C=}\left(
\mathbf{c}_{ij}\right)  $ is the matrix $\mathbb{C=}\left(  \Gamma
_{ij}\right)  $ given by the replacement $\mathbf{c}_{ij}\leftarrow\Gamma
_{ij}$ for all $\left(  i,j\right)  $. The \emph{graph} of a bipartition
$\frac{\beta}{\alpha}=\mathbf{C}_{1}\cdots\mathbf{C}_{r}$ is the evaluated
formal product
\begin{equation}
	\mathbb{G}:=\mathbb{C}_{1}\cdots\mathbb{C}_{r}, \label{graph-matrices}%
\end{equation}
i.e., the corresponding iterated elementary fraction constructed by M. Markl in \cite{Markl2}. The
\emph{graph matrix }of\emph{\ }a bipartition matrix $C=\left(  c_{ij}\right)
$ is the matrix $\mathbb{C}=\left(  \mathbb{G}_{ij}\right)  $ given by the
replacement $c_{ij}\leftarrow\mathbb{G}_{ij}$ for all $\left(  i,j\right)  .$

\subsection{Formal Decomposability}
\begin{definition}
	\label{formal}For $1\leq i\leq r,$ let $C_{1},\ldots,C_{r}$ be framed matrix
	representatives on $\mathbf{C}_{1},\ldots,\mathbf{C}_{r},$ respectively. The
	string $\rho:=C_{1}\cdots C_{r}$ is a \textbf{formal product }if $\rho
	^{1}:=C_{1}^{1}\cdots C_{r}^{1}$ is a formal product of bipartition matrices,
	in which case
	\[
	(\mathbf{is}(\rho),(\mathbf{os}(\rho)):=\left(  \mathbf{is}(\rho
	^{1}),\mathbf{os}(\rho^{1})\right)  =\left(  \mathbf{is}(C_{r}),\mathbf{os}
	(C_{1})\right)  .
	\]
	A framed matrix representative $F$ is \textbf{formally decomposable }if there
	exist framed matrix representatives $F_{1},\ldots,F_{r}$ such that
	$F=F_{1}\cdots F_{r}$ is a formal product with $r>1$; when $r$ is maximal,
	the  \textbf{length }$l(F):=r.$
\end{definition}

For each $r\geq1,$ let $\mathbf{a}\circledast^{r}\mathbf{b:=}\left\{
[c]\in\mathbf{a}{\circledast}\mathbf{b}:l[c]=r\right\}  .$ There is the
canonical decomposition
\[
\mathbf{a}\circledast\mathbf{b}:=\bigcup_{r\geq1}\mathbf{a}\circledast
^{r}\mathbf{b.}%
\]
\begin{example}
	Let $c_{i}^{o}$ and $d_{i}^{o}$ denote representatives of framed elements
	with  $i$ inputs and $o$ outputs. A matrix of the form
	\[
	F=\left(
	\begin{array}
		[c]{cc}%
		\left(
		\begin{array}
			[c]{c}%
			c_{1}^{2}\smallskip\\
			c_{1}^{1}%
		\end{array}
		\right)  \smallskip\left(  d_{1}^{2}\right)  & \left(
		\begin{array}
			[c]{c}%
			c_{2}^{2}\smallskip\\
			c_{2}^{1}%
		\end{array}
		\right)  \left(  d_{1}^{2}\text{ \ }d_{2}^{2}\right) \\
		\left(  c_{1}^{1}\right)  \left(  d_{1}^{1}\right)  & \left(  c_{2}
		^{1}\right)  (d_{1}^{1}\text{ \ }d_{2}^{1})
	\end{array}
	\right)
	\]
	is formally decomposable and factors as
	\[
	F=\left(
	\begin{array}
		[c]{cc}%
		c_{1}^{2}\smallskip & c_{2}^{2}\smallskip\\
		c_{1}^{1}\smallskip & c_{2}^{1}\smallskip\\
		c_{1}^{1} & c_{2}^{1}%
	\end{array}
	\right)  \left(
	\begin{array}
		[c]{ccc}%
		d_{1}^{2}\smallskip & d_{1}^{2} & d_{2}^{2}\\
		d_{1}^{1} & d_{1}^{1} & d_{2}^{1}%
	\end{array}
	\right)  .
	\]
	
\end{example}

When $r>1,$ a path $C=\left(  C^{1},\ldots,C^{h}\right)  =\left(  C_{1}%
^{1}\cdots C_{r}^{1},\ldots,C_{1}^{h}\cdots C_{r}^{h}\right)  $ is formally
decomposable as a matrix of formal products and factors as
\[
C=C_{1}\cdots C_{r}=\left(  C_{1}^{1},\ldots,C_{1}^{h}\right)  \cdots\left(
C_{r}^{1},\cdots,C_{r}^{h}\right)  =C|_{C_{1}^{1}}\cdots C|_{C_{r}^{1}}.
\]
The $k^{th}$ factor $C_{k}=C|_{C_{k}^{1}}$ represents the subframed matrix
$[C|_{C_{k}^{1}}]$ in $\left[  C\right]  $. Moreover, if $1\leq k<\ell\leq r,$
the formal product $\left[  C_{k}\right]  \cdots\left[  C_{\ell}\right]  $ is
a subframed matrix in $\left[  C\right]  $ as well. To see this, consider an
entry $AB$ of $C_{k}C_{k+1};$ then $(A,B)$ is a formal TP. Let $A=(A^{1}%
,\dots,A^{h})$ and $B=(B^{1},\dots,B^{h});$ then $\left(  A^{i},B^{i}\right)
$ is also a formal TP since $A^{i}$ and $B^{i}$ have the same respective
matrix dimensions and i/o sets as $A$ and $B$. Furthermore, since $C^{1}%
=C_{1}^{1}\cdots C_{r}^{1}$ is a bipartition matrix, so is $C_{k}^{1}%
C_{k+1}^{1}$, and it follows that $A^{1}B^{1}$ is a bipartition. Let
$T^{1}\left(  \mathbf{AB}\right)  =A^{1}B^{1};$ then $T^{i}\left(
\mathbf{AB}\right)  =A^{i}B^{i}$ for all $i$ and $AB=(A^{1}B^{1},\dots
,A^{h}B^{h})\ $represents the subframed element $\left[  AB\right]  =\left[
A\right]  \left[  B\right]  $ in $\left[  C\right]  .$ Thus $C_{k}C_{k+1}$
represents the subframed matrix $\left[  C_{k}C_{k+1}\right]  =\left[
C_{k}\right]  \left[  C_{k+1}\right]  $ in $\left[  C\right]  .$ Continue inductively.

\subsection{Coherent Framed Matrices}

Recall that if $\mathbf{C=}\left(  \mathbf{c}_{ij}\right)  $ is an elementary
matrix, then $T^{k}\left(  \mathbf{c}_{ij}\right)  $ is a $k$-formal
bipartition and $T^{k}\left(  \mathbf{C}\right)  =\left(  T^{k}\left(
\mathbf{c}_{ij}\right)  \right)  $ is a $k$-formal matrix.

\begin{definition}
	A $k$-formal matrix is
	
	\begin{itemize}

		\item \textbf{simple} if its indecomposable factors are bipartition matrices.
		
		\item \textbf{semi-simple} if its rows and columns are simple.
	\end{itemize}
\end{definition}

\noindent A $1$-formal matrix is a bipartition matrix, hence simple.

\begin{definition}
	A \textbf{(maximally) (totally) coherent simple matrix} has (maximally)
	(totally) coherent indecomposable factors.\textbf{ }A \textbf{(maximally)
		coherent semi-simple matrix} has (maximally) coherent rows, columns, and
	entries. A \textbf{(maximally) totally coherent semi-simple matrix} is
	(maximally) coherent with (maximally) totally coherent entries.
	
	A path $(C^{1},C^{2})$ is \textbf{(maximally) (totally) coherent} if $C^{1}$
	is (maximally) precoherent and $C^{2}$ is both semi-simple and (maximally)
	(totally) coherent.
\end{definition}

\begin{example}
	\label{22cframed}  Continuing the discussion in Example \ref{22framed},
	consider  the indecomposable maximally precoherent bipartition matrix
	\[
	C^{1}=\left(
	\begin{array}
		[c]{cc}%
		\frac{0|1}{1|0}\smallskip & \frac{0|1}{2|0}\\
		\frac{0|2}{1|0} & \frac{0|2|0}{0|2|0}%
	\end{array}
	\right)  =\left(
	\begin{array}
		[c]{ccc}%
		\left(
		\begin{array}
			[c]{c}%
			\frac{0}{1}\smallskip\\
			{\frac{0}{1}}%
		\end{array}
		\right)  \left(  \frac{1}{0}\text{ \ }\frac{1}{0}\right)  &  & \left(
		\begin{array}
			[c]{c}%
			\frac{0}{2}\smallskip\\
			{\frac{0}{2}}%
		\end{array}
		\right)  \left(  \frac{1}{0}\text{ \ }\frac{1}{0}\right) \\
		&  & \\
		\left(
		\begin{array}
			[c]{c}%
			\frac{0}{1}\smallskip\\
			{\frac{0}{1}}%
		\end{array}
		\right)  \left(  \frac{2}{0}\text{ \ }\frac{2}{0}\right)  &  & \frac
		{0|2|0}{0|2|0}%
	\end{array}
	\right)  .
	\]
	The first row and first column of $C^{1}$ have the respective incoherent
	factors
	\[
	\left(
	\begin{array}
		[c]{cc}%
		\frac{0}{1}\smallskip & \frac{0}{2}\smallskip\\
		{\frac{0}{1}} & {\frac{0}{2}}%
	\end{array}
	\right)  \text{ and }\left(
	\begin{array}
		[c]{cc}%
		\frac{1}{0}\smallskip & \frac{1}{0}\smallskip\\
		{\frac{2}{0}} & {\frac{2}{0}}%
	\end{array}
	\right)  ,
	\]
	each of which can be coheretized two ways:
	\[
	\left\{  \left(
	\begin{array}
		[c]{cc}%
		\frac{0}{1}\smallskip & \frac{0}{2}\smallskip\\
		{\frac{0|0}{1|0}} & {\frac{0|0}{0|2}}%
	\end{array}
	\right)  ,\text{ }\left(
	\begin{array}
		[c]{cc}%
		\frac{0|0}{0|1}\smallskip & \frac{0|0}{2|0}\smallskip\\
		{\frac{0}{1}} & {\frac{0}{2}}%
	\end{array}
	\right)  \right\}  \text{ and }\left\{  \left(
	\begin{array}
		[c]{cc}%
		\frac{1|0}{0|0}\smallskip & \frac{1}{0}\smallskip\\
		{\frac{0|2}{0|0}} & {\frac{2}{0}}%
	\end{array}
	\right)  ,\text{ }\left(
	\begin{array}
		[c]{cc}%
		\frac{1}{0}\smallskip & \frac{0|1}{0|0}\smallskip\\
		{\frac{2}{0}} & {\frac{2|0}{0|0}}
	\end{array}
	\right)  \right\}  .
	\]
	Thus, there are four coherent choices for $C^{2}$. Consider the
	indecomposable  choice
	\[
	C^{2}=\left(
	\begin{array}
		[c]{ccc}%
		\left(
		\begin{array}
			[c]{c}%
			\frac{0}{1}\smallskip\\
			\frac{0|0}{1|0}%
		\end{array}
		\right)  \left(
		\begin{array}
			[c]{cc}%
			\frac{1|0}{0|0}\smallskip & \frac{1}{0}\smallskip
		\end{array}
		\right)  &  & \left(
		\begin{array}
			[c]{c}%
			\frac{0}{2}\smallskip\\
			\frac{0|0}{0|2}%
		\end{array}
		\right)  \left(
		\begin{array}
			[c]{cc}%
			\frac{1}{0} & \frac{1}{0}%
		\end{array}
		\right) \\
		&  & \\
		\left(
		\begin{array}
			[c]{c}%
			\frac{0}{1}\smallskip\\
			\frac{0}{1}%
		\end{array}
		\right)  \left(
		\begin{array}
			[c]{cc}%
			{\frac{0|2}{0|0}} & {\frac{2}{0}}%
		\end{array}
		\right)  &  & \left(  \frac{0|2}{0|2}\right)  \left(  \frac{0}{0}\text{
			\ }\frac{0}{0}\right)  =\left(
		\begin{array}
			[c]{c}%
			\frac{0}{0}\smallskip\\
			\frac{0}{0}%
		\end{array}
		\right)  \left(  \tfrac{2|0}{2|0}\right)
	\end{array}
	\right)  .
	\]
	Its rows factor as
	\[
	\left(
	\begin{array}
		[c]{cc}%
		\frac{0}{1} & \frac{0}{2}\smallskip\\
		\frac{0|0}{1|0} & \frac{0|0}{0|2}%
	\end{array}
	\right)  \left(
	\begin{array}
		[c]{cccc}%
		\frac{1|0}{0|0}\smallskip & \frac{1}{0}\smallskip & \tfrac{1}{0} & \tfrac
		{1}{0}%
	\end{array}
	\right)  \text{ and }\left(
	\begin{array}
		[c]{cc}%
		\frac{0}{1} & \frac{0}{0}\smallskip\\
		\frac{0}{1} & \frac{0}{0}%
	\end{array}
	\right)  \left(
	\begin{array}
		[c]{ccc}%
		{\frac{0|2}{0|0}} & {\frac{2}{0}} & \tfrac{2|0}{2|0}%
	\end{array}
	\right)  ,
	\]
	and its columns factor as
	\[
	\left(
	\begin{array}
		[c]{c}%
		\frac{0}{1}\smallskip\\
		\frac{0|0}{1|0}\smallskip\\
		\frac{0}{1}\smallskip\\
		\frac{0}{1}%
	\end{array}
	\right)  \left(
	\begin{array}
		[c]{cc}%
		\frac{1|0}{0|0}\smallskip & \frac{1}{0}\smallskip\\
		{\frac{0|2}{0|0}} & {\frac{2}{0}}%
	\end{array}
	\right)  \text{ and }\left(
	\begin{array}
		[c]{c}%
		\frac{0}{2}\smallskip\\
		\frac{0|0}{0|2}\smallskip\\
		\tfrac{0|2}{0|2}%
	\end{array}
	\right)  \left(
	\begin{array}
		[c]{cc}%
		\frac{1}{0} & \frac{1}{0}\smallskip\\
		\frac{0}{0} & \frac{0}{0}%
	\end{array}
	\right)  .
	\]
	Thus $C^{2}$ is semi-simple with maximally coherent rows, columns, and
	entries. Since the entries of $C^{2}$ are totally coherent, the path
	$C=\left(  C^{1},C^{2}\right)  $ is maximally totally coherent. Note that
	\[
	|C|^{row}+|C|^{col}+|C|^{ent}=1+1+1=3=\#\mathbf{is}\left(  C\right)
	+\#\mathbf{os}\left(  C\right)  -1,
	\]
	which is a property of all \textquotedblleft top dimensionally coherent
	framed  matrices\textquotedblright\ (see Proposition \ref{tdimC}).
\end{example}

\begin{remark}
	Unless explicitly indicated otherwise, we henceforth identify a framed matrix
	$\left[  C\right]  $ with its canonical representative $\bar{C}$ and denote
	both by $C$ (without the bar).
\end{remark}

\begin{definition}
	\label{coherent-framed}Let $pq\geq1.$ A $q\times p$ framed matrix $C=\left(
	C^{1},\ldots,C^{h}\right)  $ is \textbf{(maximally) coherent} if
	
	\begin{enumerate}

		\item $h=1$ and $C^{1}$ is (maximally) totally coherent.
		
		\item $h\geq2,$ (maximal) coherence has been defined for all framed matrices
		of height less than $h,$ 			 the path $\left(  C^{1},C^{2}\right)  $
		is (maximally) coherent, and the entries of $C$ are (maximally) coherent.
		
	\end{enumerate}
\end{definition}

\noindent For example, the path $\left(  C^{1},C^{2}\right)  $ in Example
\ref{22cframed} represents a $2\times2$ maximally coherent framed matrix of
height $2$.

\begin{remark}
	In general, an empty biblock in a framed matrix $C$ is \emph{non unital} and deleting
	it may change the dimension and/or coherence properties of $C.$ Henceforth, we
	will only display an empty biblock in a (coherent) $C$ when deleting it
	changes the dimension (or coherence properties) of $C$ (cf. Examples 17 and 21).
\end{remark}

\begin{proposition}
	\label{totally}If $C$ is a (maximally) coherent $q\times p$ framed matrix and
	$pq\geq1,$ top level structure matrices within $C$ are (maximally) totally coherent.
\end{proposition}

\begin{proof}
	We proceed by induction on height.
	
	Let $C=\left(  C^{1}\right)  $ be a (maximally) coherent framed matrix of
	height $1.$ Since $C^{1}$ is (maximally) totally coherent by Definition
	\ref{coherent-framed}, part (1), the top level structure matrices within $C,$
	i.e., the indecomposable factors of $C^{1},$ are (maximally) totally coherent.
	
	Let $c$ be a (maximally) coherent framed element of height $2.$ Then $c$ has
	(maximally) coherent formally indecomposable factors $C_{k}$ by Definition
	\ref{coherent-framed}, part (2a). But $C_{k}$ is a (maximally) coherent
	framed  matrix of height $1.$ Hence $C_{k}^{1}=C_{k}$ is (maximally) totally
	coherent  by Definition \ref{coherent-framed}, part (1), and the top level
	structure  matrices within $c$ are the (maximally) totally coherent factors
	$C_{k}$.
	
	Let $C=\left(  c_{ij}\right)  $ be a $q\times p$ (maximally) coherent framed
	matrix of height $2$ with $p+q>2.$ Since $c_{ij}$ is a (maximally) coherent
	framed element of height $2,$ the top level structure matrices within $C$,
	which are top level structure matrices within some $c_{ij},$ are (maximally)
	totally coherent by the previous argument.
	
	Let $c$ be a (maximally) coherent framed element of height $3.$ Then $c$ has
	(maximally) coherent indecomposable factors $C_{k}$ by Definition
	\ref{coherent-framed}, part (2a). Since $C_{k}$ is a (maximally) coherent
	framed matrix of height $2,$ a top level structure matrix within $c$ is a top
	level structure matrix within some $C_{k},$ which is (maximally) totally
	coherent by the previous argument.
	
	Continue inductively.
\end{proof}

\subsection{The Coherent Framed Join $\mathfrak{m}{\circledast}_{cc}\mathfrak{n}$}

\begin{definition}
	Let $\mathbf{a}$ and $\mathbf{b}$ be ordered sets. The \textbf{coherent framed
		join of }$\mathbf{a}$ \textbf{with} $\mathbf{b}$ is defined and denoted by
	\[
	\mathbf{a}\circledast_{cc}\mathbf{b:}=\{\text{coherent framed elements in
	}\mathbf{a\circledast b\}}.
	\]
	The subset of length $r$ coherent framed partitions in $\mathbf{a}
	\circledast_{cc}\mathbf{b}$ is denoted by $\mathbf{a}\circledast_{cc}
	^{r}\mathbf{b.}$
\end{definition}

\begin{remark}
	\label{P(n)}Since $\mathfrak{0}\circledast_{cc}\mathfrak{n}=\mathfrak{0}
	\circledast\mathfrak{n}$ and $\mathfrak{n}\circledast_{cc}\mathfrak{0}
	=\mathfrak{n}\circledast\mathfrak{0},$ we identify $\mathfrak{0}
	\circledast_{cc}\mathfrak{n}$ and $\mathfrak{n}\circledast_{cc}\mathfrak{0}$
	with $P\left(  \mathfrak{n}\right)  .$
\end{remark}
\begin{example}
	\label{Ex23}Let us determine the elements\ of $\mathfrak{2\circledast}
	_{cc}\mathfrak{1}$ with initial bipartition
	\[
	c^{1}=\frac{\,\,0|1}{12|0}=\left(
	\begin{array}
		[c]{c}%
		\frac{0}{12}\smallskip\\
		\frac{0}{12}%
	\end{array}
	\right)  \left(  \frac{1}{0}\text{ }\frac{1}{0}\text{ }\frac{1}{0}\right)  .
	\]
	The replacements $\left\{  \frac{0}{12},\frac{0|0}{1|2} ,\frac{0|0}
	{2|1}\right\}  $ of $\frac{0}{12}$ give rise to the following seven indecomposables, the last three of which are coherent:
	\[
	\left\{  \left(
	\begin{array}
		[c]{c}%
		\frac{0}{12}\smallskip\\
		\frac{0}{12}%
	\end{array}
	\right)  ,\text{ }\left(
	\begin{array}
		[c]{c}%
		\frac{0}{12}\smallskip\\
		\frac{0|0}{2|1}%
	\end{array}
	\right)  ,\text{ }\left(
	\begin{array}
		[c]{c}%
		\frac{0|0}{1|2}\smallskip\\
		\frac{0}{12}%
	\end{array}
	\right)  ,\text{ }\left(
	\begin{array}
		[c]{c}%
		{\frac{0|0}{1|2}}\smallskip\\
		{\frac{0|0}{2|1}}%
	\end{array}
	\right)  ,\text{ }\left(
	\begin{array}
		[c]{c}%
		\frac{0}{12}\smallskip\\
		{\frac{0|0}{1|2}}%
	\end{array}
	\right)  ,\text{ }\left(
	\begin{array}
		[c]{c}%
		{\frac{0|0}{2|1}}\smallskip\\
		\frac{0}{12}%
	\end{array}
	\right)  ,\text{ }\left(
	\begin{array}
		[c]{c}%
		{\frac{0|0}{2|1}}\smallskip\\
		{\frac{0|0}{1|2}}%
	\end{array}
	\right)  \right\}  ,
	\]
	Thus
	\[
	c^{2}\in\left\{  \left(
	\begin{array}
		[c]{c}%
		\frac{0}{12}\smallskip\\
		{\frac{0|0}{1|2}}%
	\end{array}
	\right)  \left(  \tfrac{1}{0}\text{ }\tfrac{1}{0}\text{ }\tfrac{1}{0}\right)
	,\text{ }\left(
	\begin{array}
		[c]{c}%
		{\frac{0|0}{2|1}}\smallskip\\
		\frac{0}{12}%
	\end{array}
	\right)  \left(  \tfrac{1}{0}\text{ }\tfrac{1}{0}\text{ }\tfrac{1}{0}\right)
	,\text{ }\left(
	\begin{array}
		[c]{c}%
		{\frac{0|0}{2|1}}\smallskip\\
		{\frac{0|0}{1|2}}%
	\end{array}
	\right)  \left(  \tfrac{1}{0}\text{ }\tfrac{1}{0}\text{ }\tfrac{1}{0}\right)
	\right\}  .
	\]
	
\end{example}

\begin{example}
	Let us count the codimension $1$ elements in $\mathfrak{3}\circledast
	_{cc}\mathfrak{1}$. Let $c^{1}=\frac{B_{1}|B_{2}}{A_{1}|A_{2}}\in
	P_{2}^{\prime}(\mathfrak{3})\times P_{2}^{\prime}(\mathfrak{1}).$ Twelve
	bipartitions are produced when the six two-block partitionings of $123$ are
	paired with $0|1$ and $1|0.$ In addition, pre and post appending an empty
	block to $123$ produces an additional four, of which $\frac{\!\!\!\!\!0|1}{0|123}$ and
	$\frac{123|0}{\,\,\,\,\,1|0}$ are discarded for dimensional reasons. Thus the
	codimension $1$ elements in $\mathfrak{3}\circledast_{cc}\mathfrak{1}$ arise
	from $14$ initial bipartitions $c^{1}$. Routine calculations show that
	$\left\{  c^{2}\right\}  $ is a singleton set unless $c^{1}\in\left\{
	\frac{0|1}{12|3},\frac{0|1}{13|2},\frac{0|1}{23|1},\frac{0|1}{123|0}\right\}
	.$ Of these, the first three give rise to two codimension $1$ elements and
	the  last gives rise to eight (corresponding to the eight terms of $\Delta_P(P_{3})$).
	For example, if $c^{1}=\frac{0|1}{12|3},$ then
	\[
	c^{2}\in\left\{  \left(
	\begin{array}
		[c]{c}%
		\frac{0}{12}\smallskip\\
		\frac{0|0}{1|2}%
	\end{array}
	\right)  \left(  \tfrac{1}{0}\text{ }\tfrac{1}{0}\text{ }\tfrac{1}{3}\right)
	,\text{ }\left(
	\begin{array}
		[c]{c}%
		\frac{0|0}{2|1}\smallskip\\
		\frac{0}{12}%
	\end{array}
	\right)  \left(  \tfrac{1}{0}\text{ }\tfrac{1}{0}\text{ }\tfrac{1}{3}\right)
	\right\}  \text{.}
	\]
	Thus $\mathfrak{3}\circledast_{cc}\mathfrak{1}$ contains exactly $24$
	codimension $1$ elements.
\end{example}

\subsection{Top Dimensional Coherence}

Define the\emph{\ input and output vacuosities} of a $q\times p$ bipartition
matrix $C$ by $\overset{\wedge}{v}\left(  C\right)  :={\sum\nolimits_{i\in
		\mathfrak{q}}}v(\overset{\wedge}{\alpha}_{i}\left(  C\right)  )$ and
$\overset{\vee}{v}\left(  C\right)  :=\sum_{j\in\mathfrak{p}}v(\overset{\vee
}{\beta}_{j}\left(  C\right)  )$, respectively (see (\ref{dim-vac})). In
particular, $\overset{\wedge}{v}\left(  C\right)  =0$ if and only if all
blocks in the input partitions of $C$ are non-empty. This occurs, for example,
when the columns of $C$ are indecomposable, and dually for $\overset{\vee
}{v}\left(  C\right)  =0$. Of course, an elementary matrix has total vacuosity zero.

A framed matrix $B$ \emph{embeds} in a framed matrix $C$ if there exists
an injection $\iota:B\rightarrow C$ such that $\iota\left(  B\right)
=C|_{B^{1}}$.

\begin{definition}
	\label{dim-rigid}A $q\times p$ framed matrix $C$ with $pq>1$ is
	\textbf{Dimensionally Complete} (\textbf{DC}) if one of the following
	conditions holds:
	
	\begin{enumerate}

		\item $\mathbf{is}\left(  C\right)  =\varnothing$ or $\mathbf{os}\left(
		C\right)  =\varnothing.$
		
		\item If a $t\times s$ framed matrix $B$ embeds in $C,$ then
		\[
		\left\vert B\right\vert ^{row}=\overset{\vee}{v}\left(  B^{1}\right)
		+\sum_{j\in\mathfrak{s}}|B_{\ast j}|^{row}\ \ \text{and \ }\left\vert
		B\right\vert ^{col}=\overset{\wedge}{v}\left(  B^{1}\right)  +\sum
		_{i\in\mathfrak{t}}|B_{i\ast}|^{col}.
		\]
	\end{enumerate}
	\noindent A framed element $c$ is \textbf{Dimensionally Complete}
	(\textbf{DC}) if all formally indecomposable factors of $c$ are DC.
\end{definition}

By definition of dimension, a framed row matrix $C$  with
$\mathbf{is}\left(  C\right)  \neq\varnothing$ and $\mathbf{os}\left(
C\right)  \neq\varnothing$ is DC if and only if $\overset{\vee}{v}\left(
B^{1}\right)  =0$ for all subframed row matrices $B$ in $C$, and dually for
framed column matrices. Furthermore, elementary matrices,
and hence bipartitions, are DC.

\begin{definition}
	\label{TD-coherent}A maximally coherent DC matrix is \textbf{Top Dimensionally
		(TD) coherent}. A framed element $c$ with TD coherent formally indecomposable
	factors is \textbf{TD coherent}. Denote the subset of TD coherent elements in
	$\mathbf{a\circledast}_{cc}\mathbf{b}$ by $\mathbf{a\circledast}
	_{td}\mathbf{b}$ and the subset of TD coherent formal products in
	$\mathbf{a}\circledast_{cc}^{r}\mathbf{b}$ by $\mathbf{a}\circledast_{td}
	^{r}\mathbf{b}$.
\end{definition}

The definition of maximal coherence implies

\begin{proposition}
	\label{PropTD}Coherent elementary matrices are TD coherent, and the rows and
	columns of a TD coherent matrix are TD coherent.
\end{proposition}

\begin{example}
	\label{TD-coherent-ex}The indecomposable maximally coherent matrix
	\[
	C=\left(
	\begin{array}
		[c]{cc}%
		\frac{1|0|2}{0|1|0} & \frac{12}{0}%
	\end{array}
	\right)
	\]
	is not TD\ coherent because $\left\vert C\right\vert ^{row}=0\neq
	\overset{\vee}{v}\left(  C\right)  =1$, but the indecomposable maximally
	totally coherent matrix
	\[
	D=\left(
	\begin{array}
		[c]{cc}%
		\frac{1|2}{0|1} & \frac{12}{0}%
	\end{array}
	\right)
	\]
	is TD coherent because $\left\vert D\right\vert ^{row}=\overset{\vee
	}{v}\left(  D\right)  =\left\vert D\right\vert ^{col}=\overset{\wedge
	}{v}\left(  D\right)  =0$.
\end{example}

When verifying the TD coherence of a framed matrix $[C],$ it suffices to check
that its canonical representative $\bar{C}$ is TD coherent.

\subsection{The Dimension of a TD Coherent Matrix}

Unlike the dimension of a general framed matrix, the dimension of a
\emph{TD}\ \emph{coherent} framed matrix $c$ is completely determined by the
cardinalities of its i/o sets and its maximal factorization length $l(C)$.

\begin{proposition}
	\label{tdimC}If $C$ is a non-null TD coherent framed matrix, then
	\begin{equation}
		\left\vert C\right\vert =\#\mathbf{is}(C)+\#\mathbf{os}(C)-l\left(  C\right)
		. \label{fdimC}%
	\end{equation}
	
\end{proposition}

\begin{proof}
	Assume $C$ is an elementary matrix. Then $C$ is a $q\times p$ maximally
	coherent elementary matrix and $l(C) =1$. If $p=q=1,$ the conclusion follows
	by definition of dimension. If $p=1,$ $q\geq2,$ and $\mathbf{is}\left(
	C\right)  \neq\varnothing,$ maximal row coherence implies $\#\mathbf{is}
	\left(  C\right)  =1$ by Proposition \ref{el-coherence}. Since $C$ has equal
	singleton denominators, $\left\vert C\right\vert ^{col}=\left\vert
	C\right\vert ^{row}=0$ so that $\left\vert C\right\vert =\left\vert
	C\right\vert ^{ent}=\#\mathbf{os}\left(  C\right)  ,$ and dually for
	$p\geq2,$  $q=1,$ and $\mathbf{os}\left(  C\right)  \neq\varnothing.$ If
	$p,q\geq2,$ then  $\left\vert C\right\vert =1$ when $\mathbf{is}\left(
	C\right)  ,\mathbf{os} \left(  C\right)  \neq\varnothing;$ otherwise
	$\left\vert C\right\vert =0$  (see Example \ref{coh-elem}). The conclusion
	follows in each case.
	
	Let $r=l\left(  C \right)  .$ Since the initial bipartition matrix $C^{1}$ is
	maximally precoherent,
	\begin{equation}
		\left\vert \overset{\wedge}{\mathbf{e}}(C^{1})\right\vert =\left\vert
		\overset{\wedge}{eq}(C^{1})\right\vert =\#\mathbf{is}(C)-r\ \ \text{
			when}\ \ \mathbf{is}(C)\neq\varnothing\text{ \ and} \label{maxi1}%
	\end{equation}
	\begin{equation}
		\left\vert \overset{\vee}{\mathbf{e}}(C^{1})\right\vert =\left\vert
		\overset{\vee}{eq}(C^{1})\right\vert =\#\mathbf{os}(C)-r\ \ \text{ when
		}\ \ \mathbf{os}(C)\neq\varnothing. \label{maxi2}%
	\end{equation}

	\begin{case}
		$\mathbf{is}(C)=\varnothing$ or $\mathbf{os}(C)=\varnothing.$ Then $h\left(
		C\right)  =1$ and either $\left\vert C\right\vert =\left\vert C\right\vert
		^{col}=\mathbf{\#os}\left(  C\right)  -r$ or $\left\vert C\right\vert
		=\left\vert C\right\vert ^{row}=\mathbf{\#is}\left(  C\right)  -r$ so that
		Formula \ref{fdimC} holds.
	\end{case}
	
	\begin{case}
		$\left(  \#\mathbf{in}(C),\#\mathbf{os}(C)\right)  =\left(  m,n\right)
		\neq(0,0).$ Assume Formula \ref{fdimC} holds for all non-null TD coherent
		matrices $B$ such that $\#\mathbf{in}(B)+\#\mathbf{os}(B)<m+n$
		and $ (\#\mathbf{in}(B),$ $\#\mathbf{os}(B))\leq\left(  m,n\right)  .$
		Since  $C$ is DC,
		\begin{align}
			\hspace{-1.7in}|C| =|C|^{row}+|C|^{col}+|C|^{ent}\label{mdimC}
		\end{align}
			\[  =\overset{\vee}{v}\left(  C^{1}\right)  +\sum_{j\in\mathfrak{p}}|C_{\ast
				j}|^{row}+\sum_{j\in\mathfrak{p}}|C_{\ast j}|^{col}+\sum_{j\in\mathfrak{p}
			}|C_{\ast j}|^{ent}=\overset{\vee}{v}\left(  C^{1}\right)  +\sum
			_{j\in\mathfrak{p}}\left\vert C_{\ast j}\right\vert .
			\]
		For each $j,$ the framed column matrix $C_{\ast j}$ is TD coherent by
		Proposition \ref{PropTD}. Let $r_{j}:=l(\overset{\vee}{\beta}_{j}(C^{1}))$
		and  let $C_{\ast j}=C_{\ast j}^{1}\cdots C_{\ast j}^{r_{j}};$ then $|C_{\ast
			j}|=\#\mathbf{is}(C_{\ast j})+\#\mathbf{os}(C_{\ast j})-l(\overset{\vee
		}{\beta}_{j}(C^{1}))$ by the induction hypotheses. But $\#\mathbf{os}(C_{\ast
			j})=|\overset{\vee}{\beta}_{j}(C^{1})|+l(\pi(\overset{\vee}{\beta}_{j}
		(C^{1})))$ implies
		\begin{align}
			|C_{\ast j}|  &  =\#\mathbf{is}(C_{\ast j})+|\overset{\vee}{\beta}_{j}
			(C^{1})|-l(\overset{\vee}{\beta}_{j}(C^{1}))+l(\pi(\overset{\vee}{\beta}
			_{j}(C^{1})))\label{dim-col-2}\\
			&  =\#\mathbf{is}(C_{\ast j})+|\overset{\vee}{\beta}_{j}(C^{1})|-\overset{\vee
			}{v}_{j}(C^{1}),\nonumber
		\end{align}
		and by combining Formulas (\ref{mdimC}), (\ref{dim-col-2}), and (\ref{maxi2})
		we obtain
		\begin{align*}
			|C|  &  =\sum_{j\in\mathfrak{p}}\left\vert C_{\ast j}\right\vert
			+\overset{\vee}{v}\left(  C^{1}\right)  =\sum_{j\in\mathfrak{p}}\left[
			\#\mathbf{is}(C_{\ast j})+|\overset{\vee}{\beta}_{j}(C^{1})|-\overset{\vee
			}{v}_{j}(C^{1})\right]  +\overset{\vee}{v}\left(  C^{1}\right) \\
			&  =\sum_{j\in\mathfrak{p}}\left[  \#\mathbf{is}(C_{\ast j})+|\overset{\vee
			}{\beta}_{j}(C^{1})|\right]  =\#\mathbf{is}(C)+\left\vert \overset{\vee
			}{\mathbf{e}}(C^{1})\right\vert =\#\mathbf{is}(C)+\#\mathbf{os}(C)-r.
		\end{align*}
		
	\end{case}
\end{proof}

\begin{corollary}
	\label{tdimc}If $c$ is a non-null TD coherent framed element on $\frac
	{\mathbf{b}}{\mathbf{a}}$, then
	\begin{equation}
		\left\vert c\right\vert =\#\mathbf{a}+\#\mathbf{b}-l\left(  c \right)  .
		\label{fdimc}%
	\end{equation}
	
\end{corollary}

\begin{proof}
	Apply Proposition \ref{tdimC} to $c $.
\end{proof}

Since a coherent bipartition is a TD coherent framed element we immediately obtain

\begin{corollary}
	\label{coherent-bipartition}If a bipartition $\frac{\beta}{\alpha}\in
	P_{r}^{\prime}\left(  \mathbf{a}\right)  \times P_{r}^{\prime}\left(
	\mathbf{b}\right)  $ is coherent, then $\frac{\beta}{\alpha}$ is maximally
	totally coherent and
	\begin{equation}
		\left\vert \frac{\beta}{\alpha}\right\vert =\#\mathbf{a}+\#\mathbf{b}
		-r=|\alpha\Cup\beta|. \label{bi-dim}%
	\end{equation}
	
\end{corollary}

\begin{corollary}
	If $c\in\mathbf{a}\circledast_{cc}\mathbf{b}$ with initial bipartition $c^{1}
	$, then $\left\vert c\right\vert \leq\left\vert c^{1}\right\vert $.
\end{corollary}

\begin{proof}
	By Corollaries \ref{tdimc} and (\ref{bi-dim}), $\left\vert c\right\vert $
	attains its maximum value $\left\vert c^{1}\right\vert $ when $c$ is TD coherent.
\end{proof}

\section{Face Operators and Chain Complexes}

\subsection{The Face Operator $\tilde{\delta}$ on $\mathfrak{m}	{\circledast}_{cc}\mathfrak{n}$}

\hspace{.1in}

Let $m,n\geq0.$ In this subsection we define a face operator $\tilde{\delta}$
on $\mathfrak{m}\circledast_{cc}\mathfrak{n}$ in terms of partitioning actions on bipartition matrices. Incoherent matrices so obtained are either coheretized or discarded (cf. Example \ref{33obstruction}).

Recall Proposition \ref{totally}, which asserts that a top level structure matrix within a
coherent framed matrix is totally coherent.

\begin{definition}
	\label{delta}
	Given $c\in\mathfrak{m}\circledast_{cc}\mathfrak{n},$ let
	$\tilde{\delta}(c)$ denote the set of \textbf{coherent codimension} $1$
	\textbf{faces} \textbf{of} $c$. If $c=\frac{\mathfrak{n}}{\mathfrak{m}}$
	define
	\begin{equation}
		\tilde{\delta}\left(  \frac{\mathfrak{n}}{\mathfrak{m}}\right)  :=\left\{
		\begin{array}
			[c]{cl}%
			\varnothing, & m+n<2\\
			\mathfrak{m}\circledast_{td}^{2}\mathfrak{n}, & m+n\geq2.
		\end{array}
		\right.
	\end{equation}
	Otherwise, let $c=C_{1}\cdots C_{r}$ and define
	\[
	\tilde{\delta}(c):=\bigcup_{\substack{1\leq k\leq r\smallskip\\F\in
			\tilde{\delta}\left(  C_{k}\right)  }}\left\{  C_{1}\cdots C_{k-1}\cdot F\cdot
	C_{k+1}\cdots C_{r}\right\}  ,
	\]
	where $\tilde{\delta}\left(  C_{k}\right)  $ denotes the set of all
	$\left\vert C_{k}\right\vert -1$ dimensional coherent framed matrices that
	arise from $C_{k}$ in the following way: Given a positive dimensional top
	level structure matrix $E$ within $C_{k},$ let $B$ be either $E$ or a
	positive  dimensional row, column, or elementary bipartition entry of $E.$ If
	$B$ is a/an
	
	\begin{enumerate}

		\item elementary bipartition $\frac{\mathbf{b}}{\mathbf{a}},$ let $D\in
		\{B_{1}B_{2}\in\tilde{\delta}\left(  \frac{\mathbf{b}}{\mathbf{a}}\right)
		\}.$\smallskip
		
		\item $1\times p$ matrix, $p\geq2,$ let $D\in\{$coheretizations $B_{1}
		^{\prime}B_{2}$ of $B_{1}B_{2}=\partial_{\mathbf{M},\mathbf{N}}\left(
		B\right)  ,$ where $\partial_{\mathbf{M},\mathbf{N}}\left(  B\right)  $
		ranges  over all row actions$\}.$\smallskip
		
		\item $q\times1$ matrix, $q\geq2$, let $D\in\{$coheretizations $B_{1}
		B_{2}^{\prime}$ of $B_{1}B_{2}=\partial_{\mathbf{M},\mathbf{N}}\left(
		B\right)  ,$ where $\partial_{\mathbf{M},\mathbf{N}}\left(  B\right)  $
		ranges  over all column actions $\}.$\smallskip
		
		\item $q\times p$ matrix, $p,q\geq2,$ let $D\in\{ $coheretizations
		$B_{1}^{\prime}B_{2}^{\prime}$ of $B_{1}B_{2}=\partial_{\mathbf{M},\mathbf{N}
		}(B) ,$ where $\partial_{\mathbf{M},\mathbf{N}}(B) $ ranges over all row
		$i$/column $j$ actions $\}$.
	\end{enumerate}
	
	\noindent Given $D,$ let $t\times s$ be the matrix dimensions of $C_{k}$ and
	let $C_{k}^{D}$ be the $t\times s$ framed matrix representative obtained from
	$C_{k}$ via the replacement $B\leftarrow D$. If $C_{k}^{D}$ is coherent and
	formally decomposable, then $C_{k}^{D}\in\tilde{\delta}\left(  C_{k}\right)
	.$ If $C_{k}^{D}$ is formally indecomposable and $F$ is the canonical
	representative of $[C_{k}^{D}]$, then $F\in\tilde{\delta}\left(  C_{k}\right)
	$ if either $s,t\geq2$ or
	
	\begin{enumerate}

		\item[(a)] $s\geq2,$ $C_{k}=\left(  c_{1}\cdots c_{s}\right)  ,$ $F=(f_{1}\cdots
		f_{s}),$  and for some non-strongly extreme pair $\left(
		\mathbf{M}^{i},\mathbf{N}^{i}\right)  $ and some $j\in\mathfrak{s}$, the entry action $\partial_{\mathbf{M}^{i},\mathbf{N}^{i}}(c_{j}^{1})=f_{j}^{1}$.
		
		\item[(b)] $t\geq2,$ $C_{k}=\left(  c_{1}\cdots c_{t}\right)  ^{T}$, $F=(f_{1}\cdots
		f_{t})^{T},$ and for some non-strongly extreme pair $\left(
		\mathbf{M}^{j},\mathbf{N}^{j}\right)$ and some $i\in\mathfrak{t}$, the entry action $\partial_{\mathbf{M}^{j},\mathbf{N}^{j}}(c_{i}^{1})=f_{i}^{1}$.
	\end{enumerate}
\end{definition}

When $m,n\geq2$, the set $\tilde{\delta}\left(  \frac{\mathfrak{n}%
}{\mathfrak{m}}\right)  =\mathfrak{m}\circledast_{td}^{2}\mathfrak{n}$
consists of all framed partitions of $\frac{\mathfrak{n}}{\mathfrak{m}}$ of
length $2$.

\begin{example}
	Consider the $5$-dimensional coherent framed element
	\[
	c=C_{1}C_{2}=\left(
	\begin{array}
		[c]{c}%
		\frac{1}{1}\smallskip\\
		\frac{0}{1}\smallskip\\
		\frac{0}{1}\smallskip\\
		\frac{0}{1}%
	\end{array}
	\right)  \left(  \frac{2|34}{0|0}\text{ }\frac{24|3}{2|34}\right)  .
	\]
	Write $C_{2}=\left(  c_{1}\text{ }c_{2}\right)  $ and consider the
	$4$-dimensional $1\times1$ (totally) coherent top level element
	\[
	c_{2}^{1}=B_{1}B_{2}=\frac{24|3}{2|34}=\left(
	\begin{array}
		[c]{c}%
		\frac{2}{2}\smallskip\\
		\frac{4}{2}%
	\end{array}
	\right)  \left(  \frac{3}{0}\text{ }\frac{3}{34}\right)  .
	\]
	Six admissible formally indecomposable matrices $F$ arise from two column
	actions and four row actions on the indecomposable factors $B_{i}$. For
	example, the replacement $B_{1}\leftarrow D,$ where $D$ is the right column
	action
	\[
	D=\partial_{\mathbf{M,N}}\left(  B_{1}\right)  =\left(
	\begin{array}
		[c]{c}%
		\frac{0|2}{2|0}\smallskip\\
		\frac{4|0}{2|0}%
	\end{array}
	\right)  =\left(
	\begin{array}
		[c]{c}%
		\frac{0}{2}\smallskip\\
		\frac{0}{2}\smallskip\\
		\frac{4}{2}%
	\end{array}
	\right)  \left(
	\begin{array}
		[c]{cc}%
		\frac{2}{0}\smallskip & \frac{2}{0}\\
		\frac{0}{0} & \frac{0}{0}%
	\end{array}
	\right)  ,
	\]
	produces the $3$-dimensional coherent bipartition
	\[
	\frac{4|2|3}{2|0|34}=\left(
	\begin{array}
		[c]{c}%
		\frac{0}{2}\smallskip\\
		\frac{0}{2}\smallskip\\
		\frac{4}{2}%
	\end{array}
	\right)  \left(
	\begin{array}
		[c]{cc}%
		\frac{2}{0}\smallskip & \frac{2}{0}\\
		\frac{0}{0} & \frac{0}{0}%
	\end{array}
	\right)  \left(  \frac{3}{0}\text{ }\frac{3}{34}\right)  .
	\]
	Alternatively, $\frac{4|2|3}{2|0|34}$ is obtained by appropriately
	partitioning the biblock $\frac{24}{2}$. The canonical representative
	\[
	F=\left(  f_{1}\text{ }f_{2}\right)  =\left(  \left(  \frac{2|34}{0|0}\right)
	\text{ }\left(  \frac{4|2|3}{2|0|34}\right)  \right)
	\]
	given by the replacement $c_{2}^{1}\leftarrow\frac{4|2|3}{2|0|34}$ is
	coherent, indecomposable, and satisfies $f_{2}^{1}=\partial_{\left(  \left\{
		2\right\}  ,\left\{  4\right\}  \right)  }(c_{2}^{1}),$ Since $\left(
	\left\{  2\right\}  ,\left\{  4\right\}  \right)  $ is not strongly extreme,
	$F$ is admissible by Definition \ref{delta}, condition (a), and $C_{1}\cdot
	F\in\tilde{\delta}\left(  c\right)  .$
\end{example}

Our next example demonstrates that partitioning actions may fail to preserve coherence.
\begin{example}\label{33obstruction}
	Consider the dimension $4$ coherent framed element
	\[
	c=\left(
	\begin{array}
		[c]{c}
		\frac{0}{1}\smallskip\\
		\frac{0}{1}\smallskip\\
		\frac{0}{1}\smallskip\\
		\frac{0}{1}
	\end{array}
	\right)  \left(
	\begin{array}
		[c]{cc}
		\frac{1|2|3}{3|2|0} & \frac{123}{0}
	\end{array}
	\right)  .
	\]
	\noindent The partitioning action given by the replacement $\frac{123}{0}\leftarrow \frac {12|3}{0|0}$
	factors as
	\[
	\left(
	\begin{array}
		[c]{cc}
		\frac{1|2|3}{3|2|0} & \frac{12|3}{0|0}
	\end{array}
	\right)
	=
	\left(
	\begin{array}
		[c]{cc}
		\frac{1|2}{3|2}\smallskip & \frac{12}{0}\\
		\frac{0|0}{3|2} & \frac{0}{0}
	\end{array}
	\right)  \left(
	\begin{array}
		[c]{cccc}
		\frac{3}{0} & \frac{3}{0} & \frac{3}{0} & \frac{3}{0}
	\end{array}
	\right) .
	\]
	Since the $2\times 2$ (row) incoherent factor fails to admit a $3$-dimensional coheretization, the given partitioning action fails to preserve coherence, and consequently, fails to produce an element of $\tilde{\delta}(c).$
\end{example}

\begin{proposition}
	If $C$ is a formally indecomposable TD\ coherent matrix, then $F\in
	\tilde{\delta}\left(  C\right)  $ is TD coherent if and only if $F$ is
	formally decomposable.
\end{proposition}

\begin{proof}
	A formally indecomposable $F\in\tilde{\delta}\left(  C\right)  $ is not TD
	coherent for dimensional reasons by Proposition \ref{tdimC}, and a formally
	decomposable $F\in\tilde{\delta}\left(  C\right)  $ is TD coherent by a
	straightforward check.
\end{proof}

\begin{example}
	\label{action}The $3$-dimensional maximally coherent framed element
	\[
	c=\left(\mathbf{C}_{1}\mathbf{C}_{2}\mathbf{C}_{3}=\frac{0|13|2}{2|13|0}\,,\text{
	}\left(
	\begin{array}
		[c]{c}%
		\frac{0}{2}\smallskip\\
		\frac{0}{2}\smallskip\\
		\frac{0}{2}\smallskip\\
		\frac{0}{2}%
	\end{array}
	\right)
	\!\!
	\left(
	\begin{array}
		[c]{cc}%
		\frac{0|1}{0|1}\smallskip & \frac{0|1}{3|0}\\
		\frac{3}{1} & \frac{3|0}{3|0}%
	\end{array}
	\right)
	\!\!
	\left(\tfrac{2}{0}\text{ }\tfrac{2}{0}\text{ }\tfrac{2}{0}\text{
	}\tfrac{2}{0}\right)\right)
	\]
	has exactly fourteen (14) codimension $1$ (decomposable) faces. To see this,
	note that $c=C_{1}C_{2}C_{3},$ where $C_{1}=\mathbf{C}_{1},$ $C_{2}=\left(
	\mathbf{C}_{2},C_{2}^{1}\right)  ,$ and $C_{3}=\mathbf{C}_{3}.$ The only
	positive dimensional top level structure matrix within $c$ is
	\[
	C_{2}^{1}=\left(
	\begin{array}
		[c]{cc}%
		\frac{0|1}{0|1}\smallskip & \frac{0|1}{3|0}\\
		\frac{3}{1} & \frac{3|0}{3|0}%
	\end{array}
	\right)  .
	\]
	Thus a codimension $1$ face in $\tilde{\delta}\left(  c\right)  $ has the
	form  $C_{1}\cdot F\cdot C_{3},$ where $F=\left(  \mathbf{C}_{2},F^{1}\right)
	\in\tilde{\delta}\left(  C_{2}\right)  $ and $F^{1}$ is a totally coherent
	coheretization of some row $i/$column $j$ action $\partial_{\mathbf{M,N}
	}\left(  C_{2}^{1}\right)  $ given by Definition \ref{delta}, item (4).
	\medskip
	
	\noindent One codimension $1$ face arises from the row $1$/column $1$ incoherent action with $\mathcal{V}_{11}=\varnothing$:
	\smallskip
	
	\[
	E=\left(
	\begin{array}
		[c]{cc}%
		\frac{0|0|1}{0|1|0}\smallskip & \frac{0|1|0}{3|0|0}\\
		\frac{3|0}{1|0} & \frac{3|0}{3|0}%
	\end{array}
	\right)  =\left(
	\begin{array}
		[c]{cc}%
		\frac{0|0}{0|1}\smallskip & \frac{0}{3}\\
		\frac{0|0}{0|1}\smallskip & \frac{0}{3}\\
		\frac{3}{1} & \frac{3}{3}%
	\end{array}
	\right)  \left(
	\begin{array}
		[c]{cccc}%
		\frac{1}{0}\smallskip & \frac{1}{0} & \frac{1|0}{0|0} & \frac{1|0}{0|0}\\
		\frac{0}{0} & \frac{0}{0} & \frac{0}{0} & \frac{0}{0}%
	\end{array}
	\right)
	\]
	\medskip
	
	\noindent the first factor of which coheretizes one way and gives
	\smallskip
	
	\[
	F^{1}=\left(
	\begin{array}
		[c]{cc}%
		\frac{0|0}{0|1}\smallskip & \frac{0|0}{3|0}\\
		\frac{0|0}{0|1}\smallskip & \frac{0|0}{3|0}\\
		\frac{3}{1} & \frac{3}{3}%
	\end{array}
	\right)  \left(
	\begin{array}
		[c]{cccc}%
		\frac{1}{0}\smallskip & \frac{1}{0} & \frac{1|0}{0|0} & \frac{1|0}{0|0}\\
		\frac{0}{0} & \frac{0}{0} & \frac{0}{0} & \frac{0}{0}%
	\end{array}
	\right)  .
	\]
	\medskip
	
	\noindent Twelve codimension $1$ faces arise from row $2/$column $1$ actions.
	Four are given by coherent actions with $\mathcal{V}_{21}=\varnothing$:
	\smallskip
	
	\[
	F^{1}\in\left\{  \left(
	\begin{array}
		[c]{cc}%
		\frac{0|1|0}{0|1|0}\smallskip & \frac{0|1}{3|0}\\
		\frac{3|0}{0|1} & \frac{3|0|0}{3|0|0}%
	\end{array}
	\right)  ,\left(
	\begin{array}
		[c]{cc}%
		\frac{0|1|0}{0|1|0}\smallskip & \frac{0|1}{3|0}\\
		\frac{3|0}{0|1} & \frac{3|0|0}{0|3|0}%
	\end{array}
	\right)  ,\left(
	\begin{array}
		[c]{cc}%
		\frac{0|1|0}{0|0|1}\smallskip & \frac{0|1}{3|0}\\
		\frac{3|0}{0|1} & \frac{3|0|0}{3|0|0}%
	\end{array}
	\right)  ,\left(
	\begin{array}
		[c]{cc}%
		\frac{0|1|0}{0|0|1}\smallskip & \frac{0|1}{3|0}\\
		\frac{3|0}{0|1} & \frac{3|0|0}{0|3|0}%
	\end{array}
	\right)  \right\}  .
	\]
	\medskip
	
	\noindent
	Four arise from the incoherent action with $\mathcal{V}_{21}=\{(1,2)\}$:
	\smallskip
	
	\[
	E\in\left\{  \left(
	\begin{array}
		[c]{cc}%
		\frac{0|1|0}{0|1|0}\smallskip & \frac{0|1|0}{3|0|0}\\
		\frac{0|3}{1|0} & \frac{0|3|0}{3|0|0}%
	\end{array}
	\right)  ,\left(
	\begin{array}
		[c]{cc}%
		\frac{0|0|1}{0|1|0}\smallskip & \frac{0|0|1}{0|3|0}\\
		\frac{0|3}{1|0} & \frac{0|3|0}{0|3|0}%
	\end{array}
	\right)  \right\}  .
	\]
	\medskip
	
	\noindent The first factor of the first matrix coheretizes two ways and gives
	
	\smallskip
	
	\[
	F^{1}\in\left\{\!\!  \left(
	\begin{array}
		[c]{cc}%
		\frac{0|1}{0|1}\smallskip & \frac{0|1}{3|0}\\
		\frac{0|0}{0|1}\smallskip & \frac{0|0}{3|0}\\
		\frac{0}{1} & \frac{0}{3}%
	\end{array}
	\right) \!\! \left(
	\begin{array}
		[c]{cccc}%
		\frac{0}{0}\smallskip & \frac{0}{0} & \frac{0}{0}\smallskip & \frac{0}{0}\\
		\frac{3}{0} & \frac{3}{0} & \frac{3|0}{0|0} & \frac{3|0}{0|0}%
	\end{array}
	\right)  , \left(
	\begin{array}
		[c]{cc}%
		\frac{0|1}{0|1}\smallskip & \frac{0|1}{3|0}\\
		\frac{0}{1}\smallskip & \frac{0}{3}\\
		\frac{0|0}{1|0} & \frac{0|0}{0|3}%
	\end{array}
	\right) \!\! \left(
	\begin{array}
		[c]{cccc}%
		\frac{0}{0}\smallskip & \frac{0}{0} & \frac{0}{0}\smallskip & \frac{0}{0}\\
		\frac{3}{0} & \frac{3}{0} & \frac{3|0}{0|0} & \frac{3|0}{0|0}%
	\end{array}
	\right) \!\! \right\} ;
	\]
	\medskip
	
	\noindent the second factor of the second matrix also coheretizes two ways and
	gives
	\smallskip
	
	\[
	F^{1}\in\left\{  \left(
	\begin{array}
		[c]{cc}%
		\frac{0|0}{0|1}\smallskip & \frac{0}{0}\\
		\frac{0|0}{0|1}\smallskip & \frac{0}{0}\\
		\frac{0}{1}\smallskip & \frac{0}{0}\\
		\frac{0}{1} & \frac{0}{0}%
	\end{array}
	\right)  \left(
	\begin{array}
		[c]{ccc}%
		\frac{1}{0}\smallskip & \frac{0|1}{0|0} & \frac{0|1}{3|0}\\
		\frac{3}{0} & \frac{3|0}{0|0} & \frac{3|0}{0|0}%
	\end{array}
	\right)  ,\text{ }\left(
	\begin{array}
		[c]{cc}%
		\frac{0|0}{0|1}\smallskip & \frac{0}{0}\\
		\frac{0|0}{0|1}\smallskip & \frac{0}{0}\\
		\frac{0}{1}\smallskip & \frac{0}{0}\\
		\frac{0}{1} & \frac{0}{0}%
	\end{array}
	\right)  \left(
	\begin{array}
		[c]{ccc}%
		\frac{1|0}{0|0}\smallskip & \frac{1}{0} & \frac{0|1}{3|0}\\
		\frac{0|3}{0|0} & \frac{3}{0} & \frac{3|0}{0|0}%
	\end{array}
	\right)  \right\}  .
	\]
	\medskip
	
	\noindent And four arise from the incoherent action with $\mathcal{V}_{21}=\varnothing$:
	\smallskip
	
	\[
	E=\left(
	\begin{array}
		[c]{cc}%
		\frac{0|0|1}{0|1|0}\smallskip & \frac{0|1}{3|0}\\
		\frac{0|3}{1|0} & \frac{0|3|0}{3|0|0}%
	\end{array}
	\right)  =\left(
	\begin{array}
		[c]{cc}%
		\frac{0|0}{0|1}\smallskip & \frac{0}{3}\\
		\frac{0|0}{0|1}\smallskip & \frac{0}{3}\\
		\frac{0}{1}\smallskip & \frac{0}{3}\\
		\frac{0}{1} & \frac{0}{3}%
	\end{array}
	\right)  \left(
	\begin{array}
		[c]{cccc}%
		\frac{1}{0}\smallskip & \frac{1}{0} & \frac{1}{0} & \frac{1}{0}\\
		\frac{3}{0} & \frac{3}{0} & \frac{3|0}{0|0} & \frac{3|0}{0|0}%
	\end{array}
	\right)
	\]
	\medskip
	
	\noindent each factor of which coheretizes two ways and gives
	\smallskip
	
	\[
	F^{1}\in\left\{  \left(
	\begin{array}
		[c]{cc}%
		\frac{0|0}{0|1}\smallskip & \frac{0|0}{3|0}\\
		\frac{0|0}{0|1}\smallskip & \frac{0|0}{3|0}\\
		\frac{0}{1}\smallskip & \frac{0}{3}\\
		\frac{0|0}{1|0} & \frac{0|0}{0|3}%
	\end{array}
	\right) \!\! \left(\!\!
	\begin{array}
		[c]{cccc}%
		\frac{1|0}{0|0}\smallskip & \frac{1}{0} & \frac{0|1}{0|0} & \frac{0|1}{0|0}\\
		\frac{0|3}{0|0} & \frac{3}{0} & \frac{3|0}{0|0} & \frac{3|0}{0|0}%
	\end{array}\!\!
	\right)  ,\left(\!\!
	\begin{array}
		[c]{cc}%
		\frac{0|0}{0|1}\smallskip & \frac{0|0}{3|0}\\
		\frac{0|0}{0|1}\smallskip & \frac{0|0}{3|0}\\
		\frac{0|0}{0|1}\smallskip & \frac{0|0}{3|0}\\
		\frac{0}{1} & \frac{0}{3}%
	\end{array}\!\!
	\right) \!\! \left(\!\!
	\begin{array}
		[c]{cccc}%
		\frac{1|0}{0|0}\smallskip & \frac{1}{0} & \frac{0|1}{0|0} & \frac{0|1}{0|0}\\
		\frac{0|3}{0|0} & \frac{3}{0} & \frac{3|0}{0|0} & \frac{3|0}{0|0}%
	\end{array}
	\!\!
	\right)  ,\right.
	\]
	\[
	\left.  \hspace*{0.25in}\left(
	\begin{array}
		[c]{cc}%
		\frac{0|0}{0|1}\smallskip & \frac{0|0}{3|0}\\
		\frac{0|0}{0|1}\smallskip & \frac{0|0}{3|0}\\
		\frac{0}{1}\smallskip & \frac{0}{3}\\
		\frac{0|0}{1|0} & \frac{0|0}{0|3}%
	\end{array}
	\right) \!\! \left(\!\!
	\begin{array}
		[c]{cccc}%
		\frac{1}{0}\smallskip & \frac{0|1}{0|0} & \frac{0|1}{0|0} & \frac{0|1}{0|0}\\
		\frac{3}{0} & \frac{3|0}{0|0} & \frac{3|0}{0|0} & \frac{3|0}{0|0}%
	\end{array}\!\!
	\right)  ,\left(
	\begin{array}
		[c]{cc}%
		\frac{0|0}{0|1}\smallskip & \frac{0|0}{3|0}\\
		\frac{0|0}{0|1}\smallskip & \frac{0|0}{3|0}\\
		\frac{0|0}{0|1}\smallskip & \frac{0|0}{3|0}\\
		\frac{0}{1} & \frac{0}{3}%
	\end{array}
	\right) \!\! \left(\!\!
	\begin{array}
		[c]{cccc}%
		\frac{1}{0}\smallskip & \frac{0|1}{0|0} & \frac{0|1}{0|0} & \frac{0|1}{0|0}\\
		\frac{3}{0} & \frac{3|0}{0|0} & \frac{3|0}{0|0} & \frac{3|0}{0|0}%
	\end{array}\!\!
	\right)  \right\}  .
	\]
	\medskip
	
	\noindent Finally, one codimension $1$ face arises from the row $2$/column $2$ incoherent action with $\mathcal{V}_{22}=\varnothing$:
	\[
	E=\left(
	\begin{array}
		[c]{cc}%
		\frac{0|1}{0|1}\smallskip & \frac{0|0|1}{0|3|0}\\
		\frac{0|3}{0|1} & \frac{0|3|0}{3|0|0}%
	\end{array}
	\right)  =\left(
	\begin{array}
		[c]{cc}%
		\frac{0}{0}\smallskip & \frac{0|0}{0|3}\\
		\frac{0}{0}\smallskip & \frac{0|0}{0|3}\\
		\frac{0}{0}\smallskip & \frac{0}{3}\\
		\frac{0}{0} & \frac{0}{3}%
	\end{array}
	\right)  \left(
	\begin{array}
		[c]{ccc}%
		\frac{1}{1}\smallskip & \frac{1}{0} & \frac{1}{0}\\
		\frac{3}{1} & \frac{3|0}{0|0} & \frac{3|0}{0|0}%
	\end{array}
	\right)
	\]
	the second factor of which coheretizes one way and gives
	\[
	F^{1}=\left(
	\begin{array}
		[c]{cc}%
		\frac{0}{0}\smallskip & \frac{0|0}{0|3}\\
		\frac{0}{0}\smallskip & \frac{0|0}{0|3}\\
		\frac{0}{0}\smallskip & \frac{0}{3}\\
		\frac{0}{0} & \frac{0}{3}%
	\end{array}
	\right)  \left(
	\begin{array}
		[c]{ccc}%
		\frac{1}{1}\smallskip & \frac{0|1}{0|0} & \frac{0|1}{0|0}\\
		\frac{3}{1} & \frac{3|0}{0|0} & \frac{3|0}{0|0}%
	\end{array}
	\right)  .
	\]
	
\end{example}

\subsection{The Face Operator $\tilde{\partial}$ on $\mathfrak{m}
	\protect\widetilde{\circledast}_{pp}\mathfrak{n}$}

In this subsection we introduce three important concepts:\ the integrability
of certain framed matrices, the \textquotedblleft prebalanced framed
join\textquotedblright\ $\mathfrak{m}\widetilde{\circledast}_{pp}%
\mathfrak{n}\subset\mathfrak{m}{\circledast}_{cc}\mathfrak{n},$ and a face
operator $\tilde{\partial}$ on $\mathfrak{m}\widetilde{\circledast}%
_{pp}\mathfrak{n}$.
\begin{definition}
	\label{integrable}Let $\mathcal{M}$ denote the set of coherent framed
	matrices. A \textbf{TD} \textbf{quadratic} \textbf{matrix} has exactly two
	non-null TD coherent formally indecomposable factors\textbf{.} A TD quadratic
	matrix $F$ is \textbf{integrable }if there exists a unique TD coherent
	$C\in\mathcal{M}$ such that $F\in\mathcal{\tilde{\delta}}\left(  C\right)  .$
	When $F$ is integrable, we refer to $C$ as the\textbf{\ integral of }$F$ and
	write $\int F=C$. An indecomposable or integrable framed matrix whose TD
	quadratic subframed matrices are integrable is \textbf{totally integrable}.
\end{definition}

\noindent When $C=\int F$ we have $\left\vert C\right\vert =\left\vert
F\right\vert +1$ by the definition of $\mathcal{\tilde{\delta}}\left(
C\right)  .$ Furthermore, if $F=F_{1}F_{2},$ and its factors $F_{1}$ and
$F_{2}$ are indecomposable TD coherent column and row matrices, respectively,
then%
\[
\int F=\frac{\mathbf{os}(F_{1})}{\mathbf{is}(F_{2})}.
\]

\begin{example}
	
	\begin{enumerate}

		\item[(a)] Consider the TD quadratic matrix
		\[
		F=\left(
		\begin{array}
			[c]{cc}%
			\tfrac{1|2}{2|1} & \tfrac{1|2}{0|0}\vspace{1mm}\\
			\tfrac{0|0}{2|1} & \tfrac{0|0}{0|0}%
		\end{array}
		\right)  =\left(
		\begin{array}
			[c]{cc}%
			\frac{1}{2}\smallskip & \frac{1}{0}\\
			\frac{0}{2}\smallskip & \frac{0}{0}\\
			\frac{0}{2} & \frac{0}{0}%
		\end{array}
		\right)  \left(
		\begin{array}
			[c]{ccc}%
			\frac{2}{1}\smallskip & \frac{2}{0} & \frac{2}{0}\\
			\frac{0}{1} & \frac{0}{0} & \frac{0}{0}%
		\end{array}
		\right)
		\]
		and the matrix
		\[
		C_{1}=\left(
		\begin{array}
			[c]{cc}%
			\tfrac{1|2}{2|1} & \tfrac{1|2}{0|0}\vspace{1mm}\\
			\tfrac{0}{12} & \tfrac{0|0}{0|0}%
		\end{array}
		\right)  \in\mathcal{M}.
		\]
		Set $e=$ $\overset{\wedge}{eq}\left(  C_{1}\right)  =12$ and $\mathbf{x=}
		\overset{\wedge}{\mathbf{e}}\left(  F\right)  \sqsubseteq_{diag}\Delta
		_{P}^{\left(  1\right)  }\left(  \partial e\right)  .$ By Proposition
		\ref{ab}, part (2), $\mathbf{x}_{11}\mathbf{=}12\times2|1$ is the unique
		diagonal component of $e$ that can be obtained from $\mathbf{x}$ by a single
		factor replacement; consequently, $C_{1}$ is the unique maximally row
		coherent  framed bipartition matrix such that $\partial_{\mathbf{M,N}}\left(
		C_{1}\right)  =F$ for some $\left(  \mathbf{M,N}\right)  .$ By a dual
		argument,
		\[
		C_{2}=\left(
		\begin{array}
			[c]{cc}%
			\tfrac{1|2}{2|1} & \tfrac{12}{0}\vspace{1mm}\\
			\tfrac{0|0}{2|1} & \tfrac{0|0}{0|0}%
		\end{array}
		\right)  \in\mathcal{M}
		\]
		is the unique maximally column coherent framed matrix such that
		$\partial_{\mathbf{M,N}}\left(  C_{2}\right)
		=F $ for some $\left(  \mathbf{M,N}\right)  $. Since ${C}_{1}\neq{C}_{2},$ no $C\in\mathcal{M}$ with the
		required  integrability property exists and $F$ is not integrable.
		
		\item[(b)] On the other hand, consider the TD\ quadratic matrix
		\[
		F=\left(
		\begin{array}
			[c]{cc}%
			\tfrac{2|1}{1|2} & \tfrac{2|1}{0|0}\vspace{1mm}\\
			\tfrac{0|0}{1|2} & \tfrac{0|0}{0|0}%
		\end{array}
		\right)  =\left(
		\begin{array}
			[c]{cc}%
			\frac{0}{1}\smallskip & \frac{0}{0}\\
			\frac{2}{1}\smallskip & \frac{2}{0}\\
			\frac{0}{1} & \frac{0}{0}%
		\end{array}
		\right)  \left(
		\begin{array}
			[c]{ccc}%
			\frac{1}{0}\smallskip & \frac{1}{2} & \frac{1}{0}\\
			\frac{0}{0} & \frac{0}{2} & \frac{0}{0}%
		\end{array}
		\right)
		\]
		and the matrix
		\[
		C=\left(
		\begin{array}
			[c]{cc}%
			\tfrac{12}{12} & \tfrac{2|1}{0|0}\vspace{1mm}\\
			\tfrac{0|0}{1|2} & \tfrac{0|0}{0|0}%
		\end{array}
		\right)  \in\mathcal{M}.
		\]
		Then $\partial_{\{1\}_{11},\{2\}_{11}}(C)=F$, and by Proposition \ref{ab},
		part (2), $C$ is the unique element of $\mathcal{M}$ such that $\partial
		_{\mathbf{M,N}}\left(  C\right)  =F$ for some $\left(  \mathbf{M,N}\right)
		$.  Hence $
		{\textstyle\int}
		F=C$.
		
		\item[(c)] In Example \ref{not-totally-coherent} we considered the incoherent
		matrix
		\begin{align*}
			F  &  =\left(
			\begin{array}
				[c]{cc}%
				\frac{0|1}{12|0} & \frac{0|1}{45|0}%
			\end{array}
			\right)  =\left(
			\begin{array}
				[c]{cc}%
				\frac{0}{12}\smallskip & \frac{0}{45}\\
				\frac{0}{12} & \frac{0}{45}%
			\end{array}
			\right)  \left(
			\begin{array}
				[c]{cccccc}%
				\frac{1}{0} & \frac{1}{0} & \frac{1}{0} & \frac{1}{0} & \frac{1}{0} & \frac
				{1}{0}%
			\end{array}
			\right) \\
			&  =\partial_{(\{12\}_{11},\varnothing_{11}),(\{45\}_{12},\varnothing_{12}
				)}\left(
			\begin{array}
				[c]{cc}%
				\frac{1}{12} & \frac{1}{45}%
			\end{array}
			\right)  .
		\end{align*}
		Since $\overset{\wedge}{\mathbf{e}}\left(  F\right)  =1245\sqsubseteq
		\hspace*{-0.15in}\diagup_{diag}{\ \Delta}_{P}^{(0)}(\partial e)$ for any cell
		$e$, Proposition \ref{ab} does not apply. Nevertheless by inspection, exactly
		three bipartition matrices $C_{i}$ satisfy $\partial_{\mathbf{M},\mathbf{N}
		}(C_{i})=F$ for some $\left(  \mathbf{M},\mathbf{N}\right)  ,$ namely,
		\[
		C_{1}=\left(
		\begin{array}
			[c]{cc}%
			\frac{1}{12} & \frac{1}{45}%
		\end{array}
		\right)  ,\text{ }C_{2}=\left(
		\begin{array}
			[c]{cc}%
			\frac{0|1}{12|0} & \frac{1}{45}%
		\end{array}
		\right)  ,\text{ and }C_{3}=\left(
		\begin{array}
			[c]{cc}%
			\frac{1}{12} & \frac{0|1}{45|0}%
		\end{array}
		\right)  .
		\]

		Of these, only $C_{1}$ is TD coherent and the coheretization
		\[
		F^{\prime}=\left(
		\begin{array}
			[c]{cc}%
			\frac{0}{12}\smallskip & \frac{0}{45}\\
			\frac{0|0|0|0}{1|2|0|0} & \frac{0|0|0|0}{0|0|4|5}%
		\end{array}
		\right)  \left(
		\begin{array}
			[c]{cccccc}%
			\frac{1}{0} & \frac{1}{0} & \frac{1}{0} & \frac{1}{0} & \frac{1}{0} & \frac
			{1}{0}%
		\end{array}
		\right)  \in\mathcal{\tilde{\delta}}\left(  C_{1}\right)  .
		\]
		Thus $\int F^{\prime}=C_{1}$; in fact, $F^{\prime}$ is totally integrable via
		the coheretizations $\left(
		\begin{array}
			[c]{c}%
			\frac{0}{12}\smallskip\\
			\frac{0|0}{1|2}%
		\end{array}
		\right)  \left(
		\begin{array}
			[c]{ccc}%
			\frac{1}{0} & \frac{1}{0} & \frac{1}{0}%
		\end{array}
		\right)  $ and $\left(
		\begin{array}
			[c]{c}%
			\frac{0}{45}\\
			\frac{0|0}{4|5}%
		\end{array}
		\right)
		\left(
		\begin{array}
			[c]{ccc}%
			\frac{1}{0} & \frac{1}{0} & \frac{1}{0}%
		\end{array}
		\right)
		$ of the incoherent subframed elements $\left(  \frac{0|1}{12|0}\right)  $
		and  $\left(  \frac{0|1}{45|0}\right)  $.
	\end{enumerate}
\end{example}

Let\emph{ }$\mathcal{CM}$ denote the subset of all totally integrable matrices
in $\mathcal{M}$. Given\emph{ }$C\in\mathcal{M},$\ denote the subsets of
formally indecomposable and formally decomposable elements of $\tilde{\delta
}(C)\cap\mathcal{CM}$\emph{\ }by\emph{ }$\tilde{\delta}_{b}(C)$ and
$\tilde{\delta}_{d}(C),$\emph{\ }respectively\emph{.} Let $x\in\left\{
b,d\right\}  ;$ if $F_{1}F_{2}\in\mathcal{\tilde{\delta}}(C)$, extend
$\mathcal{\tilde{\delta}}_{x}$ as a derivation so that
\[
\mathcal{\tilde{\delta}}_{x}\left(  F_{1}F_{2}\right)  :=\mathcal{\tilde
	{\delta}}_{x}\left(  F_{1}\right)  F_{2}\cup F_{1}\mathcal{\tilde{\delta}}%
_{x}\left(  F_{2}\right)  .
\]

\begin{definition}
	\label{extend}Let $m,n\geq1$ and define $\mathcal{CM}_{1}\left(  m,n\right)
	:=\{\frac{\mathfrak{n}}{\mathfrak{m}}\}.$ For $r\geq2$ define
	\[
	\mathcal{CM}_{r}(m,n):=\bigcup\limits_{\substack{x_{i}\in\left\{  b,d\right\}
			\\1\leq i\leq s}}\left\{  F\in\mathcal{\tilde{\delta}}_{x_{s}}\cdots
	\mathcal{\tilde{\delta}}_{x_{1}}\left(  \frac{\mathfrak{n}}{\mathfrak{m}
	}\right)  :l\left(  F\right)  =r\right\}  .
	\]
	The\textbf{\ prebalanced framed join of} $\mathfrak{m}$ \textbf{and}
	$\mathfrak{n}$ is the positively graded set
	\[
	\mathfrak{m}\widetilde{\circledast}_{pp}\mathfrak{n}:=\mathcal{CM}_{\ast
	}(m,n)
	\]
	with \textbf{face operator }$\tilde{\partial}:=\mathcal{\tilde{\delta}}
	_{d}\cup\mathcal{\tilde{\delta}}_{b}$.
\end{definition}

TD coherent row and column matrices of positive dimension in $\mathcal{CM}$
have the following key property:

\begin{lemma}
	\label{basic}Given a positive dimensional TD coherent framed row (respt.
	column) matrix $C\in\mathcal{CM}$ with at least $2$ entries and $F\in{\delta
	}_{b}\left(  C\right)  $, there exists a unique row (respt. column) matrix
	$D\in\mathcal{M\smallsetminus}\left\{  C\right\}  $ such that $F\in\delta
	_{b}\left(  D\right)  $.
\end{lemma}

\begin{proof}
	Let $C=(c_{1}\cdots c_{p})\in\mathcal{CM}$ be a positive dimensional TD
	coherent row matrix and let $F=(f_{1}\cdots f_{p})\in\tilde{\delta}_{b}\left(
	C\right)  ;$ then $\left\vert F\right\vert =\left\vert C\right\vert -1.$
	Since  $F$ is formally indecomposable, so is $F^{1}.$ If $h=1,$ the
	underlying  structure matrix $B\neq C^{1}$ (otherwise, the row action on
	$C^{1}$ implies  that $F^{1}$ is decomposable); hence $B$ is an elementary
	bipartition entry of  $C^{1}.$ If $h>1,$ then $B$ lies in some top level entry
	$c_{j}^{h}.$ In  either case, there is a $j^{th}$ entry action $f_{j}%
	^{1}=\partial_{M^{i} ,N^{i}}(c_{j}^{1})$ such that $\left(  M^{i}%
	,N^{i}\right)  $ is not strongly  extreme by Definition \ref{delta}.
	Furthermore, the indecomposability of  $F^{1}$ implies $\#\mathbf{os}%
	(C)\geq2.$ Let $c_{j}^{1}=\left(  A_{j} ^{1}|\cdots|A_{j}^{r_{j}},B_{j}%
	^{1}|\cdots|B_{j}^{r_{j}}\right)  ;$ then  $\left(  M^{i},N^{i}\right)
	\subseteq(A_{j}^{i},B_{j}^{i}).$ Let $f_{j} =C_{1}\cdots C_{r_{j}+1}.$ There
	are two cases:
	
	\begin{case} $N^{i}\subseteq$ $B_{j}^{i}$ is extreme. Since $F^{1}$ is
		indecomposable and $\left(  M^{i},N^{i}\right)  $ is not strongly extreme,
		either $\left(  M^{i},N^{i}\right)  =(A_{j}^{i},\varnothing)$ for some $i>1$
		or $\left(  M^{i},N^{i}\right)  =(\varnothing,B_{j}^{i})$ for some $i<r_{j}$.
		If $N^{i}=B_{j}^{i},$ set $\ell=i.$ Since elements of $\mathcal{CM}$ are
		totally integrable, all TD quadratic subframed matrices of $F$ are
		integrable,  and in particular, $C_{\ell}C_{\ell+1}$ is integrable. Let
		$D^{\prime}:=\int C_{\ell}C_{\ell+1}$ and set
		\[
		d_{j}=C_{1}\cdots C_{\ell-1}\cdot D^{\prime}\cdot C_{\ell+2}\cdots C_{r_{j}
			+1},
		\]
		where ${d}_{j}^{1}={f}_{j}^{1}[\ell];$ then $f_{j}\in\tilde{\delta}\left(
		d_{j}\right)  .$ Obtain $D$ from $F$ via the replacement $f_{j}\leftarrow
		d_{j}$; then $F\in\delta_{b}\left(  D\right)  $. If $N^{i}=\varnothing,$ set
		$\ell=i-1$ and obtain $D$ from $F$ by a similar argument. In either case,
		$D\neq C$ since the denominators of $d_{j}^{1}$ and ${c}_{j}^{1}$ are
		distinct. On the other hand, the numerators of ${d}_{j}^{1}$ and $c_{j}^{1}$
		(hence the numerators of all entries) are equal. But $\overset{\vee
		}{\mathbf{e}}(D^{1})=\overset{\vee}{\mathbf{e}}(C^{1})$ implies $D^{1}$ is
		coherent so that $D\in\mathcal{M\smallsetminus}\left\{  C\right\}  $ as claimed.
	\end{case}
	
	\begin{case} $N^{i}\subseteq$ $B_{j}^{i}$ is not extreme. First,
		$l\pi(\overset{\vee}{\beta}_{j}(F^{1}))=l\pi(\overset{\vee}{\beta}_{j}
		(C^{1}))+1$ by the indecomposability of $F^{1}$. Set $e=\overset{\vee
		}{eq}(C^{1}),$ $\mathbf{e}=\overset{\vee}{\mathbf{e}}(C^{1}),$ and
		$\partial_{N^{i}}^{j}(\mathbf{e})=\overset{\vee}{\mathbf{e}}(F^{1});$ then by
		Proposition \ref{ab}, part (1), there is a unique $\mathbf{x}_{j\ell}
		\neq\mathbf{e}$ such that $\mathbf{x}_{j\ell}\sqsubseteq\Delta_{P}
		^{(q-1)}(e).$ As in Case 5.10, let $D^{\prime}:=\int C_{\ell}C_{\ell+1}$ and set
		$d_{j}=C_{1}\cdots C_{\ell-1}\cdot D^{\prime}\cdot C_{\ell+2}\cdots
		C_{r_{j}+1},$ where$\ {d}_{j}^{1}={f}_{j}^{1}[\ell];$ then $f_{j}\in
		\tilde{\delta}(d_{j}).$ Let $D$ be the matrix obtained from $F$ via the
		replacement $f_{j}\leftarrow d_{j}$; then $F\in\delta_{b}\left(  D\right)  $.
		But the numerator of ${d}_{j}^{1}={f}_{j}^{1}[\ell]$ is the $j^{th}$ factor
		of  $\mathbf{x}_{j\ell}.$ Hence $\overset{\vee}{\mathbf{e}}(D^{1})=\mathbf{x}
		_{j\ell},$ the bipartition matrix $D^{1}$ is coherent, and $D\in
		\mathcal{M\smallsetminus}\left\{  C\right\}  $ as desired.
	\end{case}
	
	The proof for column matrices is entirely dual.
\end{proof}

\begin{example}
	Consider the totally coherent indecomposable $1\times2$ matrix
	\[
	C=\left(
	\begin{array}
		[c]{cc}%
		\frac{1|2}{1|2} & \frac{12}{0}%
	\end{array}
	\right)  .
	\]
	
\end{example}

(a) Set $j=1$ and $\left(  M^{1},N^{1}\right)  =\left(  \varnothing
,\{1\}\right)  .$ Then $N^{1}$ is extreme,
\[
F=\left(
\begin{array}
	[c]{cc}%
	\frac{1|0|2}{0|1|2} & \frac{12}{0}%
\end{array}
\right)  \in\partial_{\left(  M^{1},N^{1}\right)  }(C), \text{ and }D=\left(
\begin{array}
	[c]{cc}%
	\frac{1|2}{0|12} & \frac{12}{0}%
\end{array}
\right)  .
\]

(b) Set $j=2$ and $\left(  M^{2},N^{2}\right)  =\left(  \varnothing
,\{2\}\right)  $. Then $N^{2}$ is not extreme,
\[
F=\left(
\begin{array}
	[c]{cc}%
	\frac{1|2}{1|2} & \frac{2|1}{0|0}%
\end{array}
\right)  \in\partial_{\left(  M^{2},N^{2}\right)  }(C),\text{ and }D=\left(
\begin{array}
	[c]{cc}%
	\frac{12}{12} & \frac{2|1}{0|0}%
\end{array}
\right)  .
\]
\begin{example}
	For the $2$-dimensional TD quadratic matrix
	\[
	C=\left(
	\begin{array}
		[c]{c}%
		\frac{0|1}{1|2}\smallskip\\
		\frac{2|0}{1|2}%
	\end{array}
	\right)  =\left(
	\begin{array}
		[c]{c}%
		\frac{0}{1}\smallskip\\
		\frac{0}{1}\smallskip\\
		\frac{2}{1}%
	\end{array}
	\right)  \left(
	\begin{array}
		[c]{cc}%
		\frac{1}{0}\smallskip & \frac{1}{2}\\
		\frac{0}{0} & \frac{0}{2}
	\end{array}
	\right)  =C_{1}C_{2}
	\]
	we have
	\[
	\tilde{\delta}\left(  C_{1}\right)  =
	\left\{
	\left(
	\begin{array}
		[c]{c}
		\frac{0|0}{1|0}\smallskip\\
		\frac{0|0}{1|0}\smallskip\\
		\frac{0|2}{1|0}
	\end{array}
	\right),
	\left(
	\begin{array}
		[c]{c}
		\frac{0|0}{0|1}\smallskip\\
		\frac{0|0}{0|1}\smallskip\\
		\frac{2|0}{0|1}
	\end{array}
	\right)
	\right\}=\left\{
	\left(
	\begin{array}
		[c]{c}
		\frac{0}{1}\smallskip\\
		\frac{0}{1}\smallskip\\
		\frac{0}{1}\smallskip\\
		\frac{0}{1}
	\end{array}
	\right)\left(
	\begin{array}
		[c]{ccc}
		\frac{0}{0}&  \frac{0}{0} \smallskip\\
		\frac{0}{0}&   \frac{0}{0} \smallskip\\
		\frac{2}{0}& \frac{2}{0}
	\end{array}
	\right),
	\left(
	\begin{array}
		[c]{c}
		\frac{0}{0}\smallskip\\
		\frac{0}{0}\smallskip\\
		\frac{2}{0}
	\end{array}
	\right)\left(
	\begin{array}
		[c]{c}
		\frac{0}{1}\smallskip\\
		\frac{0}{1}\smallskip\\
		\frac{0}{1}
	\end{array}
	\right)
	\right\}
	\]
	(cf. Example \ref{TD-coherent-ex}).
\end{example}

\begin{example}
	\label{framed-31}(Cf. Example \ref{matrad-42}). Let
	\[
	\rho=\frac{0|1}{2|13}=\left(
	\begin{array}
		[c]{c}%
		\frac{0}{2}\vspace{1mm}\\
		\frac{0}{2}%
	\end{array}
	\right)  \left(  \frac{1}{1}\,\,\,\frac{1}{3}\right)  =C_{1}C_{2}
	\in\mathfrak{3}\widetilde{\circledast}_{pp}\mathfrak{1};
	\]
	then $\tilde{\partial}(\rho)=\tilde{\partial}(C_{1})C_{2}\cup C_{1}
	\tilde{\partial}(C_{2})=C_{1}\tilde{\partial}(C_{2}).$ In turn, $\tilde
	{\partial}(C_{2}):=\{\tilde{\partial}_{k}(C_{2})\}_{1\leq k\leq7}$ is given
	by  the following row actions $\partial_{\mathbf{M,N}}\left(
	C_{2}\right)  :$
	
	\begin{itemize}

		\item  $(\mathbf{M},\mathbf{N})=\left\{  \left(  \{1\},\varnothing\right)
		,\left(  \varnothing,\varnothing\right)  )\right\}  :$
		\[
		\tilde{\partial}_{1}(C_{2})=\left(  \frac{0|1}{1|0}\ \ \frac{0|1}{0|3}\right)
		=\left(
		\begin{array}
			[c]{cc}%
			\frac{0}{1} & \frac{0}{0}\vspace{1mm}\\
			\frac{0}{1} & \frac{0}{0}%
		\end{array}
		\right)  \left(  \frac{1}{0}\ \ \frac{1}{0}\ \ \frac{1}{3}\right).
		\]

		\item  $(\mathbf{M},\mathbf{N})=\left\{  \left(  \varnothing,\varnothing
		\right)  ,\left(  \{3\},\varnothing\right)  \right\}  :$
		\[
		\tilde{\partial}_{2}(C_{2})=\left(  \frac{0|1}{0|1}\ \ \frac{0|1}{3|0}\right)
		=\left(
		\begin{array}
			[c]{cc}%
			\frac{0}{0} & \frac{0}{3}\vspace{1mm}\\
			\frac{0}{0} & \frac{0}{3}%
		\end{array}
		\right)  \left(  \frac{1}{1}\ \ \frac{1}{0}\ \ \frac{1}{0}\right).
		\]

		\item  $(\mathbf{M},\mathbf{N})=\left\{  \left(  \left\{  1\right\}
		,\varnothing\right)  ,\left(  \{3\},\varnothing\right)  \right\}  $ (cf.
		Example \ref{bi-22})$:$
		\[
		\text{\ \ \ }\partial_{\mathbf{M},\mathbf{N}}(C_{2})=\left(  \frac{0|1}
		{1|0}\ \ \frac{0|1}{3|0}\right)  =\left(
		\begin{array}
			[c]{cc}
			\frac{0}{1} & \frac{0}{3}\vspace{1mm}\\
			\frac{0}{1} & \frac{0}{3}%
		\end{array}
		\right)  \left(  \frac{1}{0}\ \ \frac{1}{0}\ \ \frac{1}{0}\ \ \frac{1}
		{0}\right)
		\]

		$\hspace*{1.3in}\xrightarrow{(\eta_1,\eta_2)}\text{\ \ \ }\tilde{\partial}
		_{3}(C_{2})=\left(
		\begin{array}
			[c]{cc}%
			\frac{0|0}{0|1} & \frac{0|0}{3|0}\vspace{1mm}\\
			\frac{0}{1} & \frac{0}{3}%
		\end{array}
		\right)  \left(  \frac{1}{0}\ \ \frac{1}{0}\ \ \frac{1}{0}\ \ \frac{1}
		{0}\right)  .$\medskip
		
		$\hspace*{1.3in}\xrightarrow{(\eta_2,\eta_1)}\text{\ \ \ }\tilde{\partial}
		_{4}(C_{2})=\left(
		\begin{array}
			[c]{cc}%
			\frac{0}{1} & \frac{0}{3}\vspace{1mm}\\
			\frac{0|0}{1|0} & \frac{0|0}{0|3}%
		\end{array}
		\right)  \left(  \frac{1}{0}\ \ \frac{1}{0}\ \ \frac{1}{0}\ \ \frac{1}
		{0}\right)  .$\medskip
		
		\item  $(\mathbf{M},\mathbf{N})=\left\{  \left(  \varnothing,\left\{
		1\right\}  \right)  ,\left(  \{3\},\left\{  1\right\}  \right)  \right\}  :$
		\[
		\tilde{\partial}_{5}(C_{2})=\left(  \frac{1|0}{0|1}\ \ \frac{1|0}{3|0}\right)
		=\left(  \frac{1}{0}\ \ \frac{1}{3}\right)  \left(  \frac{0}{1}\ \ \frac{0}
		{0}\ \ \frac{0}{0}\right)  .
		\]

		\item  $(\mathbf{M},\mathbf{N})=\left\{  \left(  \left\{  1\right\}
		,\left\{  1\right\}  \right)  ,\left(  \varnothing,\left\{  1\right\}
		\right)  \right\}  :$
		\[
		\tilde{\partial}_{6}(C_{2})=\left(  \frac{1|0}{1|0}\ \ \frac{1|0}{0|3}\right)
		=\left(  \frac{1}{1}\ \ \frac{1}{0}\right)  \left(  \frac{0}{0}\ \ \frac{0}
		{0}\ \ \frac{0}{3}\right)  .
		\]

		\item  $(\mathbf{M},\mathbf{N})=\left\{  \left(  \varnothing,\left\{
		1\right\}  \right)  ,\left(  \varnothing,\left\{  1\right\}  \right)
		\right\}  :$
		\[
		\tilde{\partial}_{7}(C_{2})=\left(  \frac{1|0}{0|1}\ \ \frac{1|0}{0|3}\right)
		=\left(  \frac{1}{0}\ \ \frac{1}{0}\right)  \left(  \frac{0}{1}\ \ \frac{0}
		{3}\right)  .
		\]
		
	\end{itemize}
\end{example}

\subsection{The Chain Complex $(  \mathfrak{m}	\widetilde{\circledast}_{pp}\mathfrak{n},\tilde{\partial})  $}

Let $R$ be a commutative ring with unity, and let
\[
\tilde{C}_{n,m}:=\langle\mathfrak{m}\widetilde{\circledast}_{pp}%
\mathfrak{n}\rangle
\]
denote the free $R$-module generated by the set $\mathfrak{m}%
\widetilde{\circledast}_{pp}\mathfrak{n}$. Define a degree $-1$ endomorphism
$\tilde{\partial}:\tilde{C}_{n,m}\rightarrow\tilde{C}_{n,m}$ on a generator
$\rho=C_{1}\cdots C_{r}\in\tilde{C}_{n,m}$ by
\begin{equation}
	\tilde{\partial}(\rho):=\sum_{1\leq s\leq r}(-1)^{|C_{1}|+\cdots+|C_{s-1}
		|}\ C_{1}\cdots\tilde{\partial}(C_{s})\cdots C_{r}, \label{differential}%
\end{equation}
where $\tilde{\partial}(C_{s}):=\sum\nolimits_{F\in\tilde{\delta}
	(C)}(-1)^{\epsilon_{F}}F$ and $\epsilon_{F}$ is defined as follows:\ Let
$q\times p$ be the matrix dimensions of $C_{s}$ and let $C=(c_{ij}):=C_{s}$.

\begin{enumerate}
	\item[(a)] If $F=(f_{ij})\in\delta_{b}\left(  C\right)  $, let $(i,j)\in
	\mathfrak{q\times p}$ be the smallest pair of positive integers such that
	$|f_{ij}|=|c_{ij}|-1;$ then
	\[
	\epsilon_{F}:=\sum_{\substack{1\leq k<i\\j\in\mathfrak{p}}}\left[
	\#\mathbf{is}(C_{\ast j})-l(c_{kj})\right]  +\sum_{1\leq\ell<j}\left[
	\#\mathbf{os}(C_{i\ell})-l(f_{i\ell})\right]  .
	\]
	\
	
	\item[(b)] Let $F=F_{1}F_{2}\in\delta_{d}(C).$ There are two cases.
	
	\begin{case} $C$ is TD coherent. Let
		\[
		D_{1}|D_{2}=\mathbf{is}(F_{1})|\mathbf{is}(F_{2})\Cup\mathbf{os}%
		(F_{1})|\mathbf{os}(F_{2})\in P(\mathbf{is}(C)\Cup\mathbf{os}(C)),
		\]
		and let ${\overset{\diamond}{\varepsilon}_{_{F}}}:=(-1)^{\#D_{1}}\cdot
		\textit{shuff\hspace{.02in}}(D_{1};D_{2}).$ Denote the signs of $\overset{\wedge}{\mathbf{e}}(C)$ in\linebreak
		$\Delta^{(q-1)}(P_{\#\mathbf{is}(C)})$ and $\overset{\vee}{\mathbf{e}}(C)$ in
		$\Delta^{(p-1)}(P_{\#\mathbf{os}(C)})$ by $\overset{\wedge}{\epsilon}_{_{C}}$
		and $\overset{\vee}{\epsilon}_{_{C}},$ respectively; then%
		\[
		\epsilon_{_{F}}:=\left\{
		\begin{array}
			[c]{cl}%
			\overset{\diamond}{\varepsilon}_{_{F}}+\overset{\vee}{\epsilon}_{_{F_{1}}
			}+\overset{\wedge}{\epsilon}_{_{F_{2}}}, & C=\frac{\mathfrak{m}}{\mathfrak{n}
			}\vspace{1mm}\\
			\overset{\diamond}{\varepsilon}_{_{F}}, & \text{otherwise}.
		\end{array}
		\right.
		\]
	\end{case}
	
	\begin{case} $C$ is not TD coherent. Let $C\in\delta_{b}(D)$ for
		some TD coherent $D$. Then $F\in\delta_{b}\delta_{d}(D);$ furthermore, if
		$E=E_{1}E_{2}\in\delta_{d}(D)$ and $F_{i}\in\delta_{b}(E_{i}),$ then
		$\epsilon_{_{F}}=\epsilon_{_{C}}\epsilon_{_{E}}(-1)^{|E_{i-1}|}\epsilon
		_{_{F_{i}}}+1.$
	\end{case}
	
\end{enumerate}

\begin{proposition}
	\label{d-tildepartial} The map $\tilde{\partial}$ is a differential on
	$\tilde{ C}_{n,m},$ i.e., $\tilde{\partial}^{2}=0.$
\end{proposition}

\begin{proof}
	We only prove $\tilde{\partial}^{2}(\frac{\mathfrak{n}}{\mathfrak{m}})=0$,
	which illustrates the basic arguments. Non-trivial cases assume $m+n\geq3$ and
	apply Lemma \ref{basic}. Consider a component $\tilde{\partial}_{E_{1}E_{2}
	}(\frac{\mathfrak{n}}{\mathfrak{m}}):=E_{1}E_{2}\in\tilde{\partial}
	(\frac{\mathfrak{n}}{\mathfrak{m}})$ and a face component $F\in\tilde
	{\partial}(E_{1}E_{2}).$	
	It suffices to show that there exists a unique component $E_{1}^{\prime}
	E_{2}^{\prime}$ of $\tilde{\partial}(\frac{\mathfrak{n}}{\mathfrak{m}})$
	distinct from $E_{1}E_{2}$ such that $F$ is a codimension $1$ face of
	$E_{1}^{\prime}E_{2}^{\prime}$. There are two cases.
	
	\begin{case}$F=\rho\!\cdot \! E_{2}\in\tilde{\partial}(E_{1})E_{2}$.
		Then $\rho$ is coherent by definition and may or may not be decomposable.
		There are two subcases.\medskip
		
		\noindent\textbf{Subcase 5.18.a}. $\rho=E_{1}^{1}E_{1}^{2}$ is decomposable. Then
		by definition there is a unique coherent matrix $C$ such that $E_{1}^{2}
		E_{2}\in\tilde{\partial}(C).$ Thus $E_{1}^{1}C\in\tilde{\partial}\left(
		\frac{\mathfrak{n}}{\mathfrak{m}}\right)  $ and $E_{1}^{\prime}E_{2}^{\prime
		}:=E_{1}^{1}C$ has the required properties.\medskip
		
		\noindent\textbf{Subcase 5.18.b}. $\rho$ is indecomposable. Let $(B,C):=(\rho
		,E_{2});$ by Lemma \ref{basic}, there is a unique coherent matrix $D$ such
		that $\rho\in\tilde{\partial}(D)$. Then $DE_{2}\in\tilde{\partial}\left(
		\frac{\mathfrak{n}}{\mathfrak{m}}\right)  $ and $E_{1}^{\prime}E_{2}^{\prime
		}:=DE_{2}$ has the required properties. \medskip
	\end{case}
	
	\begin{case} $F=E_{1}\!\cdot\! \rho\in E_{1}\tilde{\partial}(E_{2})$.
		Then $\rho$ is coherent and may or may not be decomposable. There are two
		subcases. \medskip
		
		\noindent\textbf{Subcase 5.19.a.} $\rho=E_{1}^{2}E_{2}^{2}$ is decomposable. Then
		by definition there is a unique coherent matrix $C$ such that $E_{1}E_{2}
		^{1}\in\tilde{\partial}(C).$ Thus $CE_{2}^{2}\in\tilde{\partial}\left(
		\frac{\mathfrak{n}}{\mathfrak{m}}\right)  $ and $E_{1}^{\prime}E_{2}^{\prime
		}:=CE_{2}^{2}$ has the required properties.\medskip
		
		\noindent\textbf{Subcase 5.19.b.} $\rho$ is indecomposable. Let $(C,B):=(E_{1}
		,\rho);$ by Lemma \ref{basic}, there is a unique coherent matrix $D$ such
		that  $\rho\in\tilde{\partial}(D)$. Then $E_{1}D\in\tilde{\partial}\left(
		\frac{\mathfrak{n}}{\mathfrak{m}}\right)  $ and $E_{1}^{\prime}E_{2}^{\prime
		}:=E_{1}$ has the required properties.
	\end{case}
\end{proof}

\subsection{The Balanced Framed Join $\mathfrak{m}{\circledast}_{pp}\mathfrak{n}$}

\begin{definition}
	\label{balancedextend} A coherent framed element $c$ is \textbf{balanced} if
	each top level (totally coherent) structure matrix in $c$ is primitively
	coherent (cf. Subsection \ref{coheretizing}). The \textbf{balanced framed join
		of} $\mathfrak{m}$ \textbf{and} $\mathfrak{n}$ is the subset
	\[
	(\mathfrak{m}\circledast_{pp}\mathfrak{n},\tilde{\partial})\subseteq
	(\mathfrak{m}\widetilde{\circledast}_{pp}\mathfrak{n},\tilde{\partial})
	\]
	of balanced framed elements of height $h\leq3.$
\end{definition}

In fact, if $\mathbf{c}=\mathfrak{n}/\mathfrak{m,}$ a balanced farmed element
$c\in\mathfrak{m}\circledast_{pp}\mathfrak{n}$ is characterized by a $2$-level
path of generalized bipartitions $c=(T^{1}(\mathbf{c}),T^{2}(\mathbf{c}))$ and
indexes some face of the polytope $PP_{n,m}$ defined in Section \ref{PP-KK} below.

If $0\leq m\leq2$ or $0\leq n\leq2,$ an element $c\in\mathfrak{m}%
\circledast_{cc}\mathfrak{n}$ is balanced and uniquely determined by a
$2$-level path for dimensional reasons. Thus we immediately obtain (also see
Section \ref{PP-KK})

\begin{proposition}
	\label{equality-ranges2}If $0\leq m\leq2$ or $0\leq n\leq2,$ then
	\begin{equation}
		\mathfrak{m}\circledast_{pp}\mathfrak{n}=\mathfrak{m}\circledast
		_{cc}\mathfrak{n}. \label{equality2}%
	\end{equation}
	
\end{proposition}

\begin{remark}
	Our earlier result in \cite{SU4} agrees with Equality (\ref{equality2}) in the
	indicated ranges, and Definition \ref{balancedextend} extends that result
	to all $m,n\geq 0$.
\end{remark}

\subsection{The Reduced Balanced Framed Join $\mathfrak{m}{\circledast}_{kk}\mathfrak{n}$}

\begin{definition}
	A bipartition matrix $C$ is \textbf{reduced } if it represents a class of bipartition
	matrices in $\mathfrak{m}\circledast\mathfrak{n}$ such that $C\sim C^{\prime}$ if and only if $C$ and $C^{\prime}$
	differ only in the number of empty biblocks in their entries. The
	\textbf{reduced} framed join and the \textbf{reduced balanced} framed join
	are  denoted  by
	\[
	\mathfrak{m}\circledast_{red}\mathfrak{n:=m}\circledast\mathfrak{n/}\sim\ \text{
		and }\ \ \mathfrak{m}\circledast_{kk}\mathfrak{n:=m}\circledast_{pp}
	\mathfrak{n/}\sim,
	\]
	respectively. The canonical projections (denoted by the same symbol) are
	\begin{equation} \label{bivartheta}		\vartheta\vartheta:\mathfrak{m}\circledast\mathfrak{n}\longrightarrow
		\mathfrak{m}\circledast_{red}\mathfrak{n}
	\end{equation}
	and
	\begin{equation}\label{co-bivartheta}		\vartheta\vartheta:\mathfrak{m}\circledast_{pp}\mathfrak{n}\longrightarrow
		\mathfrak{m}\circledast_{kk}\mathfrak{n}.
	\end{equation}

\end{definition}

If $C=(c_{ij})$ is a $q\times p$ reduced framed matrix, the dimension formula in
Definition \ref{dim-GBPM} reduces to
\[
|C|=\sum_{\left(  i,j\right)  \in\mathfrak{q}\times\mathfrak{p}}|c_{ij}|.
\]
Consequently, if $\rho\in\mathfrak{m}\circledast_{kk}\mathfrak{n,}$ then
\[
|\rho|=\sum\left\vert \frac{\mathbf{b}}{\mathbf{a}}\right\vert =\sum
(\#\mathbf{a}+\#\mathbf{b}-1),
\]
where $\frac{\mathbf{b}}{\mathbf{a}}$ ranges over all top level elementary
non-null bipartitions in $\rho.$

The face operator $\tilde{\partial}$ induces a face operator
\[
\partial:\mathfrak{m}\circledast_{kk}\mathfrak{n}\rightarrow\mathfrak{m}%
\circledast_{kk}\mathfrak{n}%
\]
of degree $-1$, which acts on a
reduced balanced framed matrix $C=\left(  c_{ij}\right)  $ as a derivation
with respect to the entries $c_{ij}$.
Identify $(\mathfrak{0}\circledast_{pp}\mathfrak{n}
,\tilde{\partial})$ and $(\mathfrak{m}\circledast_{pp}\mathfrak{0}
,\tilde{\partial})$ with the permutahedra $P_{n}$ and $P_{m},$ identify
$(\mathfrak{0}\circledast_{kk}\mathfrak{n},{\partial})$ and $(\mathfrak{m}
\circledast_{kk}\mathfrak{0},{\partial})$ with the associahedra
$K_{n+1}$ and $K_{m+1}$, and denote the induced projections by
\[
\vartheta:P_{n}\rightarrow K_{n+1}\ \ \text{and}\ \ \vartheta:P_{m}\rightarrow
K_{m+1}.
\]
\begin{example}
	\label{P-K}
	In $P_3$ we have
	\[    2|1|3=\left(\frac{0}{2} \right)\!\! \left(\frac{0}{1}\, \, \frac{0}{0} \right)\!\!\left(\frac{0}{0}\,\,\frac{0}{0}\,\,\frac{0}{3}\right)=
	\left(\frac{0}{2} \right) \!\! \left(\left(\frac{0}{1}\right)\!\!\left(\frac{0}{0} \,\, \frac{0}{0}\right)\,\,
	\left(\frac{0}{0}\right)\!\!\left(\frac{0}{3}\right)\right)  \]
	and
	
	\[    2|3|1=\left(\frac{0}{2} \right)\!\! \left(\frac{0}{0}\, \, \frac{0}{3} \right)\!\!\left(\frac{0}{1}\,\,\frac{0}{0}\,\,\frac{0}{0}\right)=
	\left(\frac{0}{2} \right) \!\! \left(\left(\frac{0}{0}\right)\!\!\left(\frac{0}{1}\right)\,\,
	\left(\frac{0}{3}\right)\!\!\left(\frac{0}{0} \,\, \frac{0}{0}\right)
	\right) . \]
	Thus, in $K_4 $ we obtain
	\[
	\vartheta(2|13)=\vartheta(2|1|3)=\vartheta(2|3|1).
	\]
	In Example \ref{framed-31}, $|\vartheta\vartheta(\tilde{\partial}_{3}
	(\rho))|=|\vartheta\vartheta(\tilde{\partial}_{4}(\rho))|=|\vartheta
	\vartheta(\tilde{\partial}_{7}(\rho))|=0.$ Consequently, the $2$-dimensional
	element $\vartheta\vartheta(\rho)\in\mathfrak{3}\circledast_{kk}
	\,\mathfrak{1}\,$has four $1$-dimensional faces induced by $\tilde{\partial
	}_{1}(C_{2}),$ $\tilde{\partial}_{2}(C_{2}),$ $\tilde{\partial}_{5}(C_{2}),$
	and $\tilde{\partial}_{6}(C_{2}).$
\end{example}

\subsection{The Chain Complex $(\mathfrak{m}{\circledast}_{kk}\mathfrak{n},\partial)$}

Let $R$ be a commutative ring with unity, and let $C_{\ast}(\mathfrak{m}%
\circledast_{kk}\mathfrak{n})$ denote the free $R$-module generated by the set
$\mathfrak{m}\circledast_{kk}\mathfrak{n}.$ The map $\partial$ on
$\mathfrak{m}\circledast_{kk}\mathfrak{n}$ induces a degree $-1$ operator
\[
\partial:C_{\ast}(\mathfrak{m}\circledast_{kk}\mathfrak{n})\rightarrow
C_{\ast}(\mathfrak{m}\circledast_{kk}\mathfrak{n})
\]
$\ $defined for $\rho=C_{1}\cdots C_{r}\in C_{\ast}(\mathfrak{m}%
\circledast_{kk}\mathfrak{n})$ by
\begin{equation}
	\partial(\rho):=\sum_{1\leq s\leq r}(-1)^{\epsilon_{_{C_{s}}}}\ C_{1}
	\cdots{\partial}(C_{s})\cdots C_{r}. \label{reduced-differential}%
\end{equation}

By Proposition \ref{d-tildepartial} we immediately obtain

\begin{theorem}
	\label{df-partial}The map ${\partial}$ is a differential on $C_{\ast
	}(\mathfrak{m}\circledast_{kk}\mathfrak{n}),$ i.e., ${\partial}^{2}=0.$
\end{theorem}

\section{Prematrads}

\subsection{Prematrads Defined}

In this subsection we review the notion of a \emph{prematrad} introduced in \cite{SU4}. Let
$R$ be a (graded or ungraded) commutative ring with unity $1_{R}$ and let
$M=\left\{  M_{n,m}\right\}  _{mn\geq 1}$ be a bigraded module over
$R.$ Fix a set of bihomogeneous module generators $G=\left\{  \alpha_{x}%
^{y}\in M_{y,x}:x,y\in\mathbb{N}\right\}  ;$ then $M_{y,x}=\left\langle
\alpha_{x}^{y}\right\rangle \ $is the $R$-module generated by $\alpha_{x}%
^{y}.$ Identify a monomial $\alpha_{x_{1}}^{y_{1}}\cdots\alpha_{x_{p}}^{y_{p}%
}\in G^{\otimes p}$ with the $1\times p$ matrix $\left[  \alpha_{x_{1}}%
^{y_{1}}\text{ }\cdots\text{ }\alpha_{x_{p}}^{y_{p}}\right]  ,$ and identify
the monomial \[(\alpha_{x_{1,1}}^{y_{1,1}}\cdots\alpha_{x_{1,p}}^{y_{1,p}%
})\otimes\cdots\otimes(\alpha_{x_{q,1}}^{y_{q,1}}\cdots\alpha_{x_{q,p}
}^{y_{q,p}})\in\left(  G^{\otimes p}\right)  ^{\otimes q}\]
with the $q\times
p$ matrix
\[
\left[
\begin{array}
	[c]{ccc}%
	\alpha_{x_{1,1}}^{y_{1,1}} & \cdots & \alpha_{x_{1,p}}^{y_{1,p}}\\
	\vdots &  & \vdots\\
	\alpha_{x_{q,1}}^{y_{q,1}} & \cdots & \alpha_{x_{q,p}}^{y_{q,p}}%
\end{array}
\right]  .
\]

Denote the set of $q\times p$ matrices over $\mathbb{N}$ by $\mathbb{N}%
^{q\times p}\mathbb{,}$ and denote the double tensor module of $M$ by $TTM$.
Given $X=\left[  x_{i,j}\right]  ,\ Y=\left[  y_{i,j}\right]  \in\mathbb{N}%
^{q\times p},$ let ${M}_{Y,X}=\left(  M_{y_{1,1},x_{1,1}}\otimes\cdots\otimes
M_{y_{1,p},x_{1,p}}\right)  \otimes\cdots\otimes\left(  M_{y_{q,1},x_{q,1}%
}\otimes\cdots\otimes M_{y_{q,p},x_{q,p}}\right)  $ and consider the
\emph{matrix submodule}%
\[
\overline{\mathbf{M}}:=\bigoplus_{pq\geq1}\left(  M^{\otimes p}\right)
^{\otimes q}=\bigoplus_{\substack{X,Y\in\mathbb{N}^{q\times p}\\pq\geq
		1}}M_{Y,X}\subset TTM.
\]
A $q\times p$ \emph{bisequence matrix }$B\in\overline{\mathbf{M}}$ has the
form
\[
B=\left[
\begin{array}
	[c]{ccc}%
	\beta_{x_{1}}^{y_{1}} & \cdots & \beta_{x_{p}}^{y_{1}}\\
	\vdots &  & \vdots\\
	\beta_{x_{1}}^{y_{q}} & \cdots & \beta_{x_{p}}^{y_{q}}%
\end{array}
\right]  ;
\]
$B$ is \emph{elementary }if $\beta_{x_{j}}^{y_{i}}\in G$ for all $\left(
i,j\right)  $. Define the \emph{input and output leaf sequences} \emph{of} $B$ by
${ils}\left(  B\right)  :=  \mathbf{x}= \left(  x_{1},\ldots,x_{p}\right)  \ $and
${ols}\left(  B\right)  :=\mathbf{y}=\left(  y_{1},\ldots,y_{q}\right)
^{T}$; define the {$indeg$}$\left(  B\right)  :=\Sigma x_{j},$ the {$outdeg$%
}$\left(  B\right)  :=\Sigma y_{i},$ and the ${bideg}(B):=\left(
{indeg}\left(  B\right)  ,\right.  $ $\left.  {outdeg}\left(  B\right)
\right)  .$ Let
\[
\mathbf{M}_{\mathbf{x}}^{\mathbf{y}}:=\left\{  \text{bisequence matrices }%
B\in\overline{\mathbf{M}}:\mathbf{x}\times\mathbf{y}={ils}\left(
B\right)  \times{ols}\left(  B\right)  \in\mathbb{N}^{1\times p}%
\times\mathbb{N}^{q\times1}\right\}  ;
\]
the \emph{bisequence submodule}
\[
\mathbf{M:}=\bigoplus_{\substack{\mathbf{x\times y\in}\mathbb{N}^{1\times
			p}\times\mathbb{N}^{q\times1}\\pq\geq1}}\mathbf{M}_{\mathbf{x}}^{\mathbf{y}%
}\subset TTM.
\]

A \emph{Transverse Pair} (TP) of bisequence matrices is a pair $A\times
B\in\mathbf{M}_{p}^{\mathbf{y}}\times\mathbf{M}_{\mathbf{x}}^{q}$ such that
$\mathbf{x}\times\mathbf{y}\in\mathbb{N}^{1\times p}\times\mathbb{N}%
^{q\times1}.$ A pair of matrices $A\times B\in\overline{\mathbf{M}}%
\times\overline{\mathbf{M}}$ is a\ \emph{Block Transverse Pair} (BTP) if there
exist block decompositions $A=\left[  A_{ij}\right]  $ and $B=\left[
B_{ij}\right]  $ such that $A_{ij}$ is a $q_{i}\times1$ block for each $j$,
$B_{ij}$ is a $1\times p_{j}$ block for each $i$, and $A_{ij}\times B_{ij}$ is
a TP for all $(i,j)$. When a BTP decomposition of $A\times B$ exists, it is
unique. Note that a pair of bisequence matrices $A^{q\times s}\times
B^{t\times p}$ is a BTP if and only if
\begin{equation}
	indeg(A)=p\ \ \text{and}\ \ outdeg(B)=q. \label{btp}%
\end{equation}

\begin{example}
	\label{example1}A $\left(  4\times2,2\times3\right)  $ monomial pair $A\times
	B\in\mathbf{M}_{_{21}}^{\overset{\overset{\overset{\overset{1}{5}}{4}}{3}}{}
	}\times\mathbf{M}_{_{123}}^{\overset{\overset{3}{1}}{}}$ is a $2\times2$ BTP
	per the block decompositions
	
	\vspace{.2in}
	\hspace{.6in}
	\setlength{\unitlength}{.06in}\linethickness{0.4pt}
	\begin{picture}(118.66,29.34)
		\put(-5,16.77){\makebox(0,0)[cc]{$A=$}} \
		\put(5.23,25.00){\makebox(0,0)[cc]{$\alpha^{1}_{2}$}}
		\put(5.23,18.00){\makebox(0,0)[cc]{$\alpha^{5}_{2}$}}
		\put(5.23,11.67){\makebox(0,0)[cc]{$\alpha^{4}_{2}$}}
		\put(5.23,4.67){\makebox(0,0)[cc]{$\alpha^{3}_{2}$}}
		\put(1.66,8.67){\dashbox{0.67}(6.33,20.00)[cc]{}}
		\put(1.66,2.34){\dashbox{0.67}(6.33,4.67)[cc]{}}
		\put(14.00,25.00){\makebox(0,0)[cc]{$\alpha^{1}_{1}$}}
		\put(14.00,18.00){\makebox(0,0)[cc]{$\alpha^{5}_{1}$}}
		\put(14.00,11.67){\makebox(0,0)[cc]{$\alpha^{4}_{1}$}}
		\put(14.00,4.67){\makebox(0,0)[cc]{$\alpha^{3}_{1}$}}
		\put(10.33,8.67){\dashbox{0.67}(6.33,20.00)[cc]{}}
		\put(10.33,2.34){\dashbox{0.67}(6.33,4.67)[cc]{}}
		\put(-1.0,1.0){\line(0,1){29}}
		\put(-1.0,1.0){\line(1,0){2}}
		\put(-1.0,30.0){\line(1,0){2}}
		\put(19.0,1.0){\line(0,1){29}}
		\put(19.0,1.0){\line(-1,0){2}}
		\put(19.0,30.0){\line(-1,0){2}}
		\put(24.83,16.77){\makebox(0,0)[cc]{\hspace*{0.2in}\text{and\ \ }$B=$}} \
		\put(39.00,20.01){\makebox(0,0)[cc]{$\beta^{3}_{1}$}}
		\put(45.58,20.01){\makebox(0,0)[cc]{$\beta^{3}_{2}$}}
		\put(53.53,20.01){\makebox(0,0)[cc]{$\beta^{3}_{3}$}}
		\put(35.67,17.67){\dashbox{0.67}(12.67,5.67)[cc]{}}
		\put(50.33,17.67){\dashbox{0.67}(6.67,5.67)[cc]{}}
		\put(39.00,12.34){\makebox(0,0)[cc]{$\beta^{1}_{1}$}}
		\put(45.58,12.34){\makebox(0,0)[cc]{$\beta^{1}_{2}$}}
		\put(53.53,12.34){\makebox(0,0)[cc]{$\beta^{1}_{3}$}}
		\put(35.67,10.01){\dashbox{0.67}(12.67,5.67)[cc]{}}
		\put(50.53,10.01){\dashbox{0.67}(6.67,5.67)[cc]{}}
		\put(33.0,7.0){\line(0,1){19}}
		\put(33.0,7.0){\line(1,0){2}}
		\put(33.0,26.0){\line(1,0){2}}
		\put(59.5,7.0){\line(0,1){19}}
		\put(59.5,7.0){\line(-1,0){2}}
		\put(59.5,26.0){\line(-1,0){2}}
		\put(62.00,16.00){\makebox(0,0)[cc]{$.$}}
	\end{picture}
	
\end{example}
Given $\mathbf{x}=\left(  x_{1},\ldots,x_{p}\right)  ,$ define $\left\Vert
\mathbf{x}\right\Vert :=\Sigma x_{j}.$ A family of maps $\gamma=\{\gamma
_{\mathbf{x}}^{\mathbf{y}}:\mathbf{M}_{p}^{\mathbf{y}}\otimes\mathbf{M}%
_{\mathbf{x}}^{q}\rightarrow\mathbf{M}_{\left\Vert \mathbf{x}\right\Vert
}^{\left\Vert \mathbf{y}\right\Vert },$ $\mathbf{x}\times\mathbf{y}%
\in\mathbb{N}^{1\times p}\times\mathbb{N}^{q\times1},$ $pq\geq1\}$
extends to
a global product $\Upsilon:\overline{\mathbf{M}}\otimes\overline{\mathbf{M}%
}\rightarrow\overline{\mathbf{M}}$ via
\begin{equation}
	\Upsilon\left(  A\otimes B\right)  =\left\{
	\begin{array}
		[c]{cc}%
		\left[  \gamma\left(  A_{ij}\otimes B_{ij}\right)  \right]  , & A\times B=\left[
		A_{ij}\right]  \times\left[  B_{ij}\right]  \ \text{is a\ BTP}\\
		0, & \text{otherwise.}%
	\end{array}
	\right.  \label{upsilon}%
\end{equation}
Juxtaposition $AB$ denotes the product $\Upsilon\left(  A\otimes B\right)  ;$
it is easy to check that $\Upsilon$ is closed in $\mathbf{M}$. Let
$\eta:R\rightarrow M_{1,1}$ be a map and let $\mathbf{1}:=\eta(1_{R}).$ The
map $\eta$ is a $\gamma$-\emph{unit}\textbf{\ }if $\mathbf{1}\alpha
=\alpha\mathbf{1}=\alpha$ for all $\alpha\in M.$ A constant matrix with unital
entries is a \emph{unital matrix}. Note that if $\beta_{p}^{q}\in M_{q,p},$
the products
\[
\underset{q}{[\underbrace{\mathbf{1}\text{ }\mathbf{1}\text{ }\cdots\text{
		}\mathbf{1}}]^{T}}\times\beta_{p}^{q}\ \
\text{and}\ \
\beta_{p}^{q}\times
\underset{p}{[\underbrace{\mathbf{1}\text{ }\mathbf{1}\text{ }\cdots\text{
		}\mathbf{1}}]}
\]
are TPs. More generally, if $A$ is a matrix in $\overline{\mathbf{M}},$ there
exist unital matrices whose sizes are determined by $A$ such that
$\mathbf{1}\times A$ and $A\times\mathbf{1}$ are BTPs. Thus if $\eta$ is a
$\gamma$-unit and $\mathbf{1}=\eta(1_{R})$, then $\mathbf{1}A=A\mathbf{1}=A$
for all $A\in\overline{\mathbf{M}}$.

\begin{definition}
	\label{prematrad}Let $M$ be a bigraded $R$-module together with a family of
	structure maps $\tilde{\gamma}=\{\tilde{\gamma}_{\mathbf{x}}^{\mathbf{y}%
	}:\mathbf{M}_{p}^{\mathbf{y}}\otimes\mathbf{M}_{\mathbf{x}}^{q}\rightarrow
	\mathbf{M}_{\left\Vert \mathbf{x}\right\Vert }^{\left\Vert \mathbf{y}%
		\right\Vert },$ $\mathbf{x}\times\mathbf{y}\in\mathbb{N}^{1\times p}%
	\times\mathbb{N}^{q\times1},$ $pq\geq1\}$ such that $\Upsilon$ is associative
	in $\mathbf{M.}$ Then $(M,\tilde{\gamma})$ is a \textbf{non-unital prematrad.}
	Let $\eta$ be a $\tilde{\gamma}$-unit and let $\gamma$ be the family of
	structure maps induced by the unital action $\tilde{\gamma}\left(
	\mathbf{1}\otimes m\right)  =\tilde{\gamma}\left(  m\otimes\mathbf{1}\right)
	=m.$ Then $\left(  M,\gamma,\eta\right)  $ is a \textbf{prematrad}. Given
	prematrads $(M,\gamma,\eta)$ and $(M^{\prime},\gamma^{\prime},\eta^{\prime})$,
	a map $f:{M}\rightarrow{M}^{\prime}$ is a \textbf{map of prematrads} if
	$f\gamma_{\mathbf{x}}^{\mathbf{y}}=\gamma{^{\prime}}_{\mathbf{x}}^{\mathbf{y}%
	}(f^{\otimes q}\otimes f^{\otimes p})$ for all $\mathbf{x\times y}$ and
	$\eta^{\prime}=f\eta.$
\end{definition}

\begin{example}\label{ENDA}
	Let $H$ be a free $R$-module of finite type. View $M=\mathcal{E}nd_{TH}\ $as the
	bigraded $R$-module $M=\left\{  M_{n,m}={Hom}\left(  H^{\otimes m},H^{\otimes n}\right)  \right\}  _{mn\geq 1}$
	and define $\eta\left(1_{R}\right)  :={\mathbf{Id}}_{H}.$
	In \cite{Markl2}, M. Markl
	defined the submodule\emph{\ }of \emph{special elements} in $M$ whose additive
	generators are monomials expressed as \emph{elementary fractions}\ of the
	form
	\[
	\frac{\alpha_{p}^{y_{1}}\cdots\alpha_{p}^{y_{q}}}{\alpha_{x_{1}}^{q}%
		\cdots\alpha_{x_{p}}^{q}}\in\mathbf{M}_{\left\Vert \mathbf{x}\right\Vert
	}^{\left\Vert \mathbf{y}\right\Vert },
	\]
	where juxtaposition in the numerator and denominator denotes tensor product
	and the $j^{th}$ output of $\alpha_{x_{i}}^{q}$ is linked to the $i^{th}$
	input of $\alpha_{p}^{y_{j}}.$ The \emph{fraction product} $\gamma
	:\mathbf{M}_{p}^{\mathbf{y}}\otimes\mathbf{M}_{\mathbf{x}}^{q}\rightarrow
	\mathbf{M}_{\left\Vert \mathbf{x}\right\Vert }^{\left\Vert \mathbf{y}%
		\right\Vert }$ given by $\gamma\left(  \alpha_{p}^{y_{1}}\cdots\alpha
	_{p}^{y_{q}}\otimes\alpha_{x_{1}}^{q}\cdots\alpha_{x_{p}}^{q}\right)
	=\frac{\alpha_{p}^{y_{1}}\cdots\alpha_{p}^{y_{q}}}{\alpha_{x_{1}}^{q}
		\cdots\alpha_{x_{p}}^{q}}$ extends to an associative product $\Upsilon$ on
	$\mathbf{M,}$ $\eta$ is a $\gamma$-unit, and $\left(  \mathcal{E}nd_{TH},\gamma
	,\eta\right)  $ is a prematrad.
\end{example}

If $A$ is a matrix in $\overline{\mathbf{M}},$ the transpose $A\mapsto A^{T}$
induces an automorphism $\sigma=\sigma_{\ast,\ast}$ of $\overline{\mathbf{M}}$
whose component $\sigma_{p,q}:\left(  M^{\otimes p}\right)  ^{\otimes
	q}\overset{\approx}{\rightarrow}\left(  M^{\otimes q}\right)  ^{\otimes p}$ is
the linear extension of the canonical permutation of tensor factors in
$\left(  G^{\otimes p}\right)  ^{\otimes q},$\ i.e.,
\[%
\begin{array}
	[c]{c}%
	\sigma_{p,q}:\left(  \alpha_{x_{1,1}}^{y_{1,1}}\otimes\cdots\otimes
	\alpha_{x_{1,p}}^{y_{1,p}}\right)  \otimes\cdots\otimes\left(  \alpha
	_{x_{q,1}}^{y_{q,1}}\otimes\cdots\otimes\alpha_{x_{q,p}}^{y_{q,p}}\right)
	\mapsto\hspace{1in}\medskip\\
	\hspace{1in}\left(  \alpha_{x_{1,1}}^{y_{1,1}}\otimes\cdots\otimes
	\alpha_{x_{q,1}}^{y_{q,1}}\right)  \otimes\cdots\otimes\left(  \alpha
	_{x_{1,p}}^{y_{1,p}}\otimes\cdots\otimes\alpha_{x_{q,p}}^{y_{q,p}}\right)
	.\smallskip
\end{array}
\]
The identification $\left(  G^{\otimes p}\right)  ^{\otimes q}\leftrightarrow
\left(  p,q\right)  $ expresses a $q\times p$ bisequence matrix $A\in
\mathbf{M}_{\mathbf{x}}^{\mathbf{y}}$ as an operator $A:\left(  G^{\otimes
	\left\Vert \mathbf{x}\right\Vert }\right)  ^{\otimes q}\rightarrow\left(
G^{\otimes p}\right)  ^{\otimes\left\Vert \mathbf{y}\right\Vert }$ on
$\mathbb{N}^{2}$ whose action is given by the composition
\[
\hspace{-.7in}
\left(  G^{\otimes\left\Vert \mathbf{x}\right\Vert }\right)  ^{\otimes
	q}\approx\left(  G^{\otimes x_{1}}\otimes\cdots\otimes G^{\otimes x_{p}%
}\right)  ^{\otimes q}\rightarrow\left(  G^{\otimes y_{1}}\right)  ^{\otimes
	p}\otimes\cdots\otimes\left(  G^{\otimes y_{q}}\right)  ^{\otimes p}
\]
\[
\hspace{1.8in}\overset{\sigma_{y_{1},p}\otimes\cdots\otimes\sigma_{y_{q},p}}{\longrightarrow
}\left(  G^{\otimes p}\right)  ^{\otimes y_{1}}\otimes\cdots\otimes\left(
G^{\otimes p}\right)  ^{\otimes y_{q}}\approx\left(  G^{\otimes p}\right)
^{\otimes\left\Vert \mathbf{y}\right\Vert }.
\]
Thus $A$ can be thought of as an arrow $\left(  \left\Vert \mathbf{x}\right\Vert
,q\right)  \mapsto\left(  p,\left\Vert \mathbf{y}\right\Vert \right)  $ in
$\mathbb{N}^{2}$ (see Figure 5).
\vspace*{.6in}

$\hspace*{.5in}
A=\left[
\begin{array}
	[c]{ccc}%
	\alpha_{2}^{2} &  & \alpha_{3}^{2}\\
	&  & \\
	\alpha_{2}^{4} &  & \alpha_{3}^{4}%
\end{array}
\right]  \in\mathbf{M}_{_{23}}^{\overset{\overset{2}{4}}{}}\ \mapsto$
\setlength{\unitlength}{0.001in}
\begin{picture}
	(152,-140)(202,-100) \thicklines
	\put(600,-650){\line(1,0){1100}}
	\put(600,-650){\line(0,1){1300}}
	\put(820,-683){\line(0,1){75}}
	\put(1360,-683){\line(0,1){75}}
	\put(563,-430){\line(1,0){75}}
	\put(563,350){\line(1,0){75}}
	\put(820,350){\makebox(0,0){$\bullet$}}
	\put(1360,-430){\makebox(0,0){$\bullet$}}
	\put(1290,-330){\vector(-2,3){400}}
	\put(1300,0){\makebox(0,0){$A$}}
	\put(450,-430){\makebox(0,0){$2$}}
	\put(450,350){\makebox(0,0){$6$}}
	\put(820,-800){\makebox(0,0){$2$}}
	\put(1360,-800){\makebox(0,0){$5$}}
\end{picture}
\vspace*{0.6in}

\begin{center}
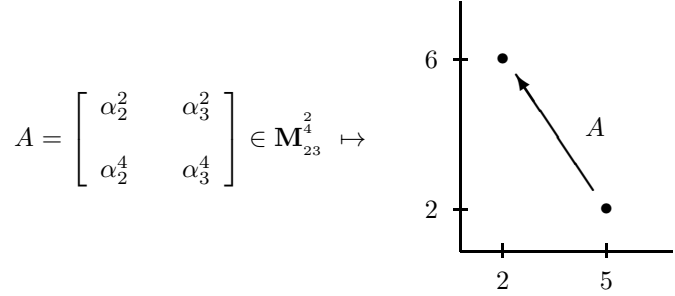
	Figure 5. A $2\times2$ bisequence matrix $A$ thought of as an arrow in $\mathbb{N}$.\bigskip
\end{center}

A TP $\alpha\times\beta\in\mathbf{M}_{p}^{\mathbf{y}}\times\mathbf{M}%
_{\mathbf{x}}^{q}$ can be represented by a 2-step path $\left(  \left\Vert
\mathbf{x}\right\Vert ,1\right)  \overset{\beta}{\mapsto}\left(  p,q\right)
\overset{\alpha}{\mapsto}\left(  1,\left\Vert \mathbf{y}\right\Vert \right)
$. If $p=1,$ $\alpha$ is represented by a vertical arrow $\left(
1,q\right)  \mapsto\left(  1,\left\Vert \mathbf{x}\right\Vert \right)  .$ If
$q=1,$ $\beta$ is represented by a horizontal arrow $\left(  \left\Vert
\mathbf{x}\right\Vert ,1\right)  \mapsto\left(  p,1\right)  .$ If
$\mathbf{x=1,}$ $\beta$ is represented by a vertical arrow $\left(
\left\Vert \mathbf{x}\right\Vert ,1\right)  \mapsto\left(  \left\Vert
\mathbf{x}\right\Vert ,q\right)  .$ If $\mathbf{y=1,}$ $\alpha$ is represented by a horizontal arrow $\left(  p,\left\Vert \mathbf{y}\right\Vert \right)
\mapsto\left(  1,\left\Vert \mathbf{y}\right\Vert \right)  .$ If
$\mathbf{x\neq1}$ and $q\neq1\mathbf{,}$ $\beta$ is represented by a
left-leaning\emph{ }arrow $\left(  \left\Vert \mathbf{x}\right\Vert ,1\right)
\mapsto\left(  p,q\right)  .$ And if $\mathbf{y}\neq\mathbf{1}$ and
$p\neq1\mathbf{,}$ $\alpha$ is represented by a left-leaning\emph{
}arrow\emph{ }$\left(  p,q\right)  \mapsto\left(  1,\left\Vert \mathbf{y}%
\right\Vert \right)  $.
A non-zero monomial $A=A_{1}\cdots A_{s}$ with $A_{i}\in\mathbf{M}$ is represented by
a directed $k$-step piece-wise linear path from the $x$-axis to the
$y$-axis in $\mathbb{N}^{2},$ where $s-k$ is the number of evaluated
subproducts in an association. Of course, each such path can represent multiple monomials.

\begin{example}
	The unevaluated product
	\begin{equation}
		ABC=\left[
		\begin{array}
			[c]{r}%
			\alpha_{2}^{2}\smallskip\\
			\alpha_{2}^{2}\smallskip\\
			\alpha_{2}^{2}%
		\end{array}
		\right]  \left[
		\begin{array}
			[c]{cc}%
			\beta_{2}^{2}\smallskip & \beta_{1}^{2}\smallskip\\
			\beta_{2}^{1} & \beta_{1}^{1}%
		\end{array}
		\right]  \left[
		\begin{array}
			[c]{rrr}%
			\gamma_{2}^{2} & \gamma_{2}^{2} & \gamma_{2}^{2}%
		\end{array}
		\right]  \in\mathbf{M}_{_{2}}^{\overset{\overset{\overset{2}{2}}{2}}{}%
		}\mathbf{M}_{_{21}}^{\overset{\overset{2}{1}}{}}\mathbf{M}_{_{222}%
		}^{\overset{2}{}} \label{assoc}%
	\end{equation}
	expressed as the composition
	\[
	G^{\otimes6}\overset{C}{\overrightarrow{\longrightarrow\overset{}{\left(
				G^{\otimes2}\right)  ^{\otimes3}}\overset{\text{ \ }\sigma_{2,3}
			}{\longrightarrow}}}\left(  G^{\otimes3}\right)  ^{\otimes2}
	\overset{B}{\overrightarrow{\longrightarrow\overset{}{\left(  G^{\otimes
					2}\right)  ^{\otimes3}}\overset{\text{ \ }\sigma_{2,2}\otimes\mathbf{1}%
			}{\longrightarrow}}}\left(  G^{\otimes2}\right)  ^{\otimes3}
	\overset{A}{\longrightarrow}G^{\otimes6}
	\]
	is represented by the 3-step path in Figure 6. If we evaluate $AB$ or $BC$,
	the associations
	\[
	\left(  AB\right)  C=\left[
	\begin{array}
		[c]{r}%
		\left[
		\begin{array}
			[c]{c}%
			\alpha_{2}^{2}\smallskip\\
			\alpha_{2}^{2}%
		\end{array}
		\right]  \left[
		\begin{array}
			[c]{rr}%
			\beta_{2}^{2} & \beta_{1}^{2}%
		\end{array}
		\right] \\
		\\
		\left[
		\begin{array}
			[c]{c}%
			\alpha_{2}^{2}%
		\end{array}
		\right]  \left[
		\begin{array}
			[c]{cc}%
			\beta_{2}^{1} & \beta_{1}^{1}%
		\end{array}
		\right]
	\end{array}
	\right]  \left[
	\begin{array}
		[c]{rrr}%
		\gamma_{2}^{2} & \gamma_{2}^{2} & \gamma_{2}^{2}%
	\end{array}
	\right]  \in\mathbf{M}_{_{3}}^{\overset{\overset{4}{2}}{}}\mathbf{M}_{_{222}%
	}^{\overset{2}{}}%
	\]
	and
	\[
	A\left(  BC\right)  =\left[
	\begin{array}
		[c]{r}%
		\alpha_{2}^{2}\smallskip\\
		\alpha_{2}^{2}\smallskip\\
		\alpha_{2}^{2}%
	\end{array}
	\right]  \left[
	\begin{array}
		[c]{ccc}%
		\left[
		\begin{array}
			[c]{c}%
			\beta_{2}^{2}\smallskip\\
			\beta_{2}^{1}%
		\end{array}
		\right]  \left[
		\begin{array}
			[c]{rr}%
			\gamma_{2}^{2} & \gamma_{2}^{2}%
		\end{array}
		\right]  &  & \left[
		\begin{array}
			[c]{c}%
			\beta_{1}^{2}\smallskip\\
			\beta_{1}^{1}%
		\end{array}
		\right]  \left[
		\begin{array}
			[c]{c}%
			\gamma_{2}^{2}%
		\end{array}
		\right]
	\end{array}
	\right]  \in\mathbf{M}_{_{2}}^{\overset{\overset{\overset{2}{2}}{2}}{}%
	}\mathbf{M}_{_{42}}^{\overset{3}{}}%
	\]
	expressed as the compositions
	\[
	G^{\otimes6}\overset{C}{\overrightarrow{\longrightarrow\overset{}{\left(
				G^{\otimes2}\right)  ^{\otimes3}}\overset{\text{ \ }\sigma_{2,3}%
			}{\longrightarrow}}}\left(  G^{\otimes3}\right)  ^{\otimes2}%
	\overset{AB}{\longrightarrow}G^{\otimes6},
	\]
	and
	\[
	G^{\otimes6}\overset{BC}{\overrightarrow{\overset{}{\rightarrow\left(
				G^{\otimes3}\right)  ^{\otimes2}\overset{\text{ \ }\sigma_{3,2}%
				}{\longrightarrow}}}}\left(  G^{\otimes2}\right)  ^{\otimes3}%
	\overset{A}{\longrightarrow}G^{\otimes6}%
	\]
	are represented by the 2-step paths in Figure 6. Finally, the evaluated
	associations $\left(  AB\right)  C=A\left(  BC\right)  \in\mathbf{M}_{_{6}%
	}^{\overset{6}{}}$ are represented by the 1-step path $G^{\otimes6}\rightarrow
	G^{\otimes6}$. \vspace{0.1in}
\end{example}

\unitlength3mm \linethickness{1.0 pt}
\ifx\plotpoint\undefined\newsavebox{\plotpoint}\fi
\begin{picture}(45,19)(0,0)
	\put(9,2){\line(0,1){17}}
	\put(9,2){\line(1,0){21}}
	
	\put(24.241,2.418){\line(0,-1){.42}}
	\put(27.709,2.523){\line(0,-1){.42}}
	\put(9.0,8.199){\line(1,0){.42}}
	\put(9.0,11.247){\line(1,0){.42}}
	\put(9.0,14.4){\line(1,0){.42}}
	\put(16.988,2.0){\line(0,1){.42}}
	\put(13.325,2.155){\line(0,1){.42}}
	\put(20.683, 2.0){\line(0,1){.42}}
	\put(9.0,5.277){\line(1,0){.42}}
	\put(9.0,17.021){\line(1,0){.42}}
	
	\put(13.06,7.9){\vector(-4,3){.03}}
	\multiput(16.893,5.203)(-.04713321,.03353709){82}{\line(-1,0){.04713321}}
	\multiput(20.635,8.93)(-.0333183,-.0358812){29}{\line(0,-1){.0358812}}
	\multiput(19.668,8.94)(.0359639,-.0335663){31}{\line(1,0){.0359639}}
	\multiput(15.85,16.426)(-.0333183,-.0358812){29}{\line(0,-1){.0358812}}
	\multiput(17.561,16.426)(-.0333183,-.0358812){29}{\line(0,-1){.0358812}}
	\multiput(14.9,16.44)(.0359639,-.0335663){31}{\line(1,0){.0359639}}
	\multiput(16.595,16.426)(.0359639,-.0335663){31}{\line(1,0){.0359639}}
	
	\put(19.4,7.581){\line(1,0){3.3}}
	\put(24.028,7.581){\line(1,0){3.047}}
	\multiput(21.278,7.953)(.0371627,.0325174){16}{\line(1,0){.0371627}}
	\multiput(21.865,8.473)(.0330335,-.0330335){16}{\line(1,0){.0330335}}
	\put(21.872,8.473){\line(0,1){.53}}
	
	\multiput(24.325,8.87)(.0371627,-.0330335){18}{\line(1,0){.0371627}}
	\multiput(24.994,8.276)(.0371627,.0334465){20}{\line(1,0){.0371627}}
	\put(25.03,8.25){\line(0,-1){.4}}
	\put(26.4,8.95){\line(0,-1){1.1}}
	
	\multiput(19.717,7.25)(.0323154,-.0323154){23}{\line(0,-1){.0323154}}
	\multiput(16.595,14.716)(.0323154,-.0323154){23}{\line(0,-1){.0323154}}
	\multiput(20.46,6.5)(.0330335,.0454211){17}{\line(0,1){.038}}
	\multiput(17.339,13.973)(.0330335,.0454211){17}{\line(0,1){.03}}
	\put(20.46,6.615){\line(0,-1){.595}}
	\put(17.339,13.973){\line(0,-1){.5}}
	\multiput(21.352,7.3)(.03325086,-.03325086){38}{\line(0,-1){.03325086}}
	\multiput(14.95,14.691)(.03325086,-.03325086){36}{\line(0,-1){.03325086}}
	\multiput(22.541,7.25)(-.03303353,-.03303353){36}{\line(0,-1){.03303353}}
	\multiput(16.075,14.716)(-.03303353,-.03303353){36}{\line(0,-1){.03303353}}
	\multiput(24.9,7.21)(.03539307,-.03362342){42}{\line(1,0){.03539307}}
	\multiput(26.15,7.284)(-.03279064,-.04372085){33}{\line(0,-1){.04372085}}
	
	\multiput(15.0,5.9)(-.0330335,-.0371627){15}{\line(0,-1){.05}}
	\multiput(14.5,5.3)(-.0330335,.0454211){14}{\line(0,1){.04}}
	\put(14.5,5.45){\line(0,-1){.6}}
	

	\multiput(13.55,5.946)(-.0330335,-.0412919){25}{\line(0,-1){.0412919}}
	\multiput(12.75,5.95)(.0335663,-.0359639){28}{\line(0,-1){.0359639}}

	\put(13.15,4.25){\line(0,-1){.5}}
	
	\multiput(13.15,3.73)(-.0330335,-.0330335){14}{\line(0,-1){.0330335}}
	\multiput(13.15,3.716)(.0318538,-.0318538){14}{\line(1,0){.0318538}}
	
	\put(14.5,4.385){\line(0,-1){1.115}}
	
	\multiput(21.0,3.38)(.0416222,-.0327032){23}{\line(1,0){.0416222}}
	\multiput(21.8,3.4)(-.0327032,-.0356762){23}{\line(0,-1){.0356762}}
	\multiput(19.9,3.38)(.0416222,-.0327032){23}{\line(1,0){.0416222}}	
	\multiput(20.7,3.4)(-.0327032,-.0356762){23}{\line(0,-1){.0356762}}
	\multiput(18.8,3.38)(.0416222,-.0327032){23}{\line(1,0){.0416222}}
	\multiput(19.6,3.4)(-.0327032,-.0356762){23}{\line(0,-1){.0356762}}
	
	\multiput(9.857,11.92)(.0416222,-.0327032){23}{\line(1,0){.0416222}}
	\multiput(10.7,11.92)(-.0327032,-.0356762){23}{\line(0,-1){.0356762}}
	\multiput(9.857,10.743)(.0416222,-.0327032){23}{\line(1,0){.0416222}}
	\multiput(10.7,10.743)(-.0327032,-.0356762){23}{\line(0,-1){.0356762}}
	\multiput(9.91,9.566)(.0416222,-.0327032){23}{\line(1,0){.0416222}}
	\multiput(10.753,9.566)(-.0356762,-.0327032){23}{\line(-1,0){.0356762}}
	
	\multiput(17.2,12.561)(-.0431567,-.0335663){31}{\line(-1,0){.0431567}}
	\multiput(15.936,12.71)(.03303353,-.03303353){36}{\line(0,-1){.03303353}}
	\put(15.234,11.247){\line(1,0){2.453}}
	\multiput(15.301,9.514)(.0334465,.040879){20}{\line(0,1){.040879}}
	\multiput(15.97,10.331)(.0330335,-.0454211){18}{\line(0,-1){.0454211}}
	\put(15.949,10.333){\line(0,1){.5}}	
	\put(17.4,10.8){\line(0,-1){1.3}}
	
	\put(14.811,15.088){\line(1,0){3.1}}
	\put(9.5,0.5){\makebox(0,0)[cc]{1}}
	\put(13.325,0.5){\makebox(0,0)[cc]{2}}
	\put(16.967,0.5){\makebox(0,0)[cc]{3}}
	\put(20.609,0.5){\makebox(0,0)[cc]{4}}
	\put(24.251,.5){\makebox(0,0)[cc]{5}}
	\put(27.967,0.5){\makebox(0,0)[cc]{6}}
	\put(17.041,5.128){\vector(-4,1){.035}}
	\multiput(27.744,2.155)(-.1202569,.03340469){89}{\line(-1,0){.1202569}}
	\put(9.237,16.946){\vector(-1,2){.035}}
	\multiput(13.102,7.953)(-.033608027,.078203294){115}{\line(0,1){.078203294}}
	\multiput(14.15,7.408)(.082349,-.032095){4}{\line(1,0){.082349}}
	\multiput(14.992,7.099)(.082349,-.032095){11}{\line(1,0){.082349}}
	\multiput(16.804,6.396)(.082349,-.032095){11}{\line(1,0){.082349}}
	\multiput(18.616,5.69)(.082349,-.032095){11}{\line(1,0){.082349}}
	\multiput(20.428,4.987)(.082349,-.032095){11}{\line(1,0){.082349}}
	\multiput(22.24,4.284)(.082349,-.032095){11}{\line(1,0){.082349}}
	\multiput(24.052,3.581)(.082349,-.032095){11}{\line(1,0){.082349}}
	\multiput(25.864,2.878)(.082349,-.032095){11}{\line(1,0){.082349}}
	
	\multiput(17.12,5.132)(-.0331368,.0492406){16}{\line(0,1){.0492406}}
	\multiput(16.059,6.708)(-.0331368,.0492406){16}{\line(0,1){.0492406}}
	\multiput(14.998,8.284)(-.0331368,.0492406){16}{\line(0,1){.0492406}}
	\multiput(13.937,9.86)(-.0331368,.0492406){16}{\line(0,1){.0492406}}
	\multiput(12.876,11.436)(-.0331368,.0492406){16}{\line(0,1){.0492406}}
	\multiput(11.815,13.012)(-.0331368,.0492406){16}{\line(0,1){.0492406}}
	\multiput(10.754,14.588)(-.0331368,.0492406){16}{\line(0,1){.0492406}}
	
	\put(7.3,2){\makebox(0,0)[cc]{1}}
	\put(7.3,5){\makebox(0,0)[cc]{2}}
	\put(7.3,8){\makebox(0,0)[cc]{3}}
	\put(7.3,11){\makebox(0,0)[cc]{4}}
	\put(7.3,14){\makebox(0,0)[cc]{5}}
	\put(7.3,17){\makebox(0,0)[cc]{6}}
\end{picture}

\begin{center}
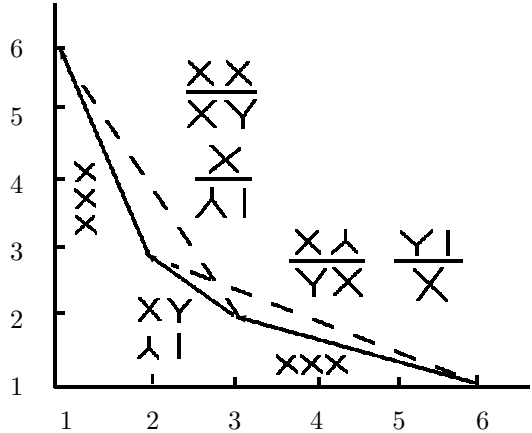
	Figure 6. The polygonal paths of $ABC.$
\end{center}

\subsection{Free Prematrads}

In this subsection we review the fundamental notion of a \emph{free prematrad}, but the
development here differs somewhat from that in \cite{SU4}. To begin, we
introduce an auxiliary object, the free \emph{non-unital prematrad}.

The notion of a bisequence matrix and its related constructions extend to
matrices whose entries are general bigraded objects; we refer to such matrices
as \emph{Generalized BiSequence Matrices (GBSMs).} When $A$ is a GBSM and
\[(indeg(A)),outdeg(A))
=(m,n),\]
we refer to $A$ as an $(m,n)$\emph{-matrix}.

Let $\xi=A_{1}\cdots A_{r}$ be a monomial of GBSMs. The \emph{bidegree of
}$\xi$ is defined and denoted by $bideg(\xi):=(outdeg(A_{1}),indeg(A_{r})).$
When $bideg(\xi)=(m,n)$, we refer to $\xi$ as an $(m,n)$-\emph{monomial}.

\begin{definition}
	\label{free-monomials} Given a bigraded set $Q=\left\{  Q_{n,m}\right\}
	_{mn\geq 1}$ with based element $\mathbf{1}\in Q_{1,1}$, construct the
	set $\tilde{G}\left(  Q\right)  :=\{\tilde{G}_{n,m}\left(  Q\right)  \}$
	of\textbf{ free bigraded monomials generated by }$Q$ inductively as follows:
	Define $\tilde{G}_{1,1}\left(  Q\right)  :=Q_{1,1}.$ If $m+n\geq3,$ and
	$\tilde{G} _{j,i}\left(  Q\right)  $ has been constructed for all $\left(
	i,j\right)  \leq\left(  m,n\right)  $ such that $i+j<m+n,$ define
	\begin{equation}
		\tilde{G}_{n,m}\left(  Q\right)  :=Q_{n,m}\cup\left\{  A_{1}\cdots A_{s}
		:s\geq2\right\}  , \label{entries}%
	\end{equation}
	where for all $k$,
	
	\begin{enumerate}
		\item[\textit{(i)}] $A_{k}$ is a GBSM over $\{\tilde{G}_{j,i}\left(  Q\right)
		:\left(  i,j\right)  \leq\left(  m,n\right)  \text{and } i+j<m+n\}$,
		
		\item[\textit{(ii)}] $A_{1}$ is a column matrix with {$outdeg$}$(A_{1})=n,$
		
		\item[\textit{(iii)}] $A_{s}$ is a row matrix with {$indeg$}$(A_{s})=m,$
		
		\item[\textit{(iv)}] if $B=\left(  \xi_{uv}\right)  $ is a $q\times p$ GBSM
		within $A_{k}$ and $\xi_{uv}=B_{1}\cdots B_{\ell-1}\cdot\mathbf{1}\cdot
		B_{\ell+1}\cdots B_{r}\mathbf{\ }$for some $\ell$, define $\xi_{uv}=B_{1}\cdots
		B_{\ell-1}B_{\ell+1}\cdots B_{r}$ when either $pq=1$ and $B=A_{k}$ or $pq>1$ and
		$B_{u\ast}$ and $B_{\ast v}$ are indecomposable,
		
		\item[\textit{(v)}] $A_{k}\times A_{k+1}$ is a BTP,
		
		\item[\textit{(vi)}] $A_{k}A_{k+1}$ is the formal product, and
		
		\item[\textit{(vii)}] $(A_{k}A_{k+1})A_{k+2}=A_{k}(A_{k+1}A_{k+2}).$\medskip
	\end{enumerate}
	A $Q$-\textbf{monomial} is an elementary monomial in $\tilde{G}\left(  Q\right) $.
\end{definition}

\begin{remark}
	Monomials in $\tilde{G}\left(  Q\right)  $ are generated by the component sets
	$Q_{n,m}$, and not by particular elements thereof. The double corolla
	$\Gamma_{m}^{n}$ represents the set $Q_{n,m}$.
\end{remark}

In view of item (iv), we henceforth assume that an $(m,n)$-monomial $\xi
=A_{1}\cdots A_{r}$ is a formal product of non-unital indecomposables unless
explicitly indicated otherwise.

Let $\Theta=\left\{  \Theta_{n,m}\right\}  _{mn\geq 1}$, where
$\Theta_{n,m}=\{\theta_{m}^{n}\}$ is a singleton set when $\Theta_{n,m}%
\neq\varnothing$. When this occurs, we identify the set $\Theta_{n,m}$ and its
element $\theta_{m}^{n}$.

\begin{example}
	The monomial
	\[
	\left[
	\begin{array}
		[c]{c}%
		\theta_{2}^{1}\smallskip\\
		\theta_{2}^{1}%
	\end{array}
	\right]  \left[
	\begin{array}
		[c]{cc}%
		\theta_{2}^{1} & \mathbf{1}\smallskip\\
		\theta_{2}^{1} & \mathbf{1}%
	\end{array}
	\right]  \left[
	\begin{array}
		[c]{cc}%
		\theta_{1}^{2} & \theta_{1}^{2}%
	\end{array}
	\right]  \in\tilde{G}(\Theta)
	\]
	is a $\Theta$-monomial, but
	\[
	\text{ }\left[\!
	\begin{array}
		[c]{c}%
		\left[  \theta_{2}^{1}\right] \!\! \left[\!
		\begin{array}
			[c]{cc}
			\theta_{2}^{1} & \mathbf{1}
		\end{array}
		\!\right]\medskip\\
		
		\left[  \theta_{2}^{1}\right]\!\!  \left[\!
		\begin{array}
			[c]{cc}
			\mathbf{1} & \theta_{2}^{1}
		\end{array}
		\!\right]
	\end{array}
	\!\right]  \left[
	\begin{array}
		[c]{cc}%
		\theta_{1}^{2} & \theta_{1}^{2}%
	\end{array}
	\right]  \in\tilde{G}(\Theta)
	\]
	is not.
\end{example}

\begin{proposition}
	\label{associativity}Let $M=\langle\overset{\sim}{{G}}\left(  \Theta\right)
	\rangle$ and let $\tilde{\gamma}=\{\tilde{\gamma}_{\mathbf{x}}^{\mathbf{y}
	}:\mathbf{M}_{p}^{\mathbf{y}}\otimes\mathbf{M}_{\mathbf{x}}^{q}\rightarrow
	\mathbf{M}_{\left\Vert \mathbf{x}\right\Vert }^{\left\Vert \mathbf{y}
		\right\Vert }\}$ be the linear extension of juxtaposition. Then the induced
	global product $\widetilde{\Upsilon}$ acts associatively on $\mathbf{M.}$
\end{proposition}

\begin{proof}
	Let $A\times B$ and $B\times C$ be BTPs of bisequence matrices over $\tilde
	{G}(\Theta)$ and consider the associated block decompositions $\left[
	A_{k^{\prime},l}\right]  \times\left[  B_{i,j^{\prime}}\right]  $ and $\left[
	B_{i^{\prime},j}\right]  \times\left[  C_{u,v^{\prime} }\right]  .$ Let
	$B=[\xi_{i,j}]^{q\times p},$ let $B_{i,j^{\prime} }=[\xi_{i,m_{j^{\prime
			}-1}+1}\cdots\xi_{i,m_{j^{\prime}}}], $ and let $B_{i^{\prime},j}%
	=[\xi_{n_{i^{\prime}-1}+1,j}\cdots\xi_{n_{i^{\prime} },j}]^{T},$ where $1\leq
	j^{\prime}\leq s$ indexes the blocks in the partition
	\[
	m_{0}+1<\cdots<m_{1}|\cdots|m_{j^{\prime}-1}+1<\cdots<m_{j^{\prime}}
	|\cdots|m_{s-1}+1<\cdots<m_{s}\in P(\mathfrak{p})
	\]
	and $1\leq i^{\prime}\leq t$ indexes the
	blocks in the partition
	\[
	n_{0}+1<\cdots<n_{1}|\cdots|n_{i^{\prime}-1}+1<\cdots<n_{i^{\prime}}
	|\cdots|n_{t-1}+1<\cdots<n_{t}\in P(\mathfrak{q}).
	\]
	Let $u_{j^{\prime}}=m_{j^{\prime}
	}-m_{j^{\prime}-1}$ and $v_{i^{\prime}}=n_{i^{\prime}}-n_{i^{\prime}-1}, $ and
	consider the block decomposition $\left[  B_{i^{\prime},j^{\prime} }\right]  $
	of $B\ $given by
	\[
	B_{i^{\prime},j^{\prime}}=[\xi_{i,j}]_{n_{i^{\prime}-1}\leq i\leq
		n_{i^{\prime}};\text{ }m_{j^{\prime}-1}\leq j\leq m_{j^{\prime}}
	}^{v_{i^{\prime}}\times u_{j^{\prime}}}.
	\]
	Now consider the juxtaposition of bisequence matrices
	\[	A_{i^{\prime},j}B_{i^{\prime},j^{\prime}}C_{i,j^{\prime}}:=\left[
	\begin{array}
		[c]{c}%
		A_{n_{i^{\prime}-1}+1,j}\\
		\vdots\\
		A_{n_{i^{\prime}},j}%
	\end{array}
	\right]  ^{v_{i}^{\prime}\times1}\cdot B_{i^{\prime}j^{\prime}}^{v_{i^{\prime
		}}\times u_{j^{\prime}}}\cdot\lbrack C_{i,m_{j^{\prime}-1}+1}\text{ }
	\cdots\text{ }C_{i,m_{j^{\prime}}}]^{1\times u_{j^{\prime}}}.
	\]
	It is easy to check that $A_{i^{\prime},j}\times B_{i^{\prime},j^{\prime}} $
	and $B_{i^{\prime},j^{\prime}}\times C_{i,j^{\prime}}$ are BTPs so that
	$A_{i^{\prime},j}B_{i^{\prime},j^{\prime}}C_{i,j^{\prime}}\in\tilde{G}
	(\Theta)$ and axiom (vii) applies. Therefore $(A_{i^{\prime},j}B_{i^{\prime
		},j^{\prime}})C_{i,j^{\prime}}=A_{i^{\prime},j}(B_{i^{\prime},j^{\prime}
	}C_{i,j^{\prime}}),\ $and it follows that $(AB)C=A(BC)$ as claimed.
\end{proof}

\begin{definition}
	\label{tildefreeprematrad}Let $\tilde{F}^{{pre}}(\Theta)=\langle
	\overset{\sim}{{G}}\left(  \Theta\right)  \rangle$ and let $\tilde{\gamma}$ be
	the product in Proposition \ref{associativity}. Then $(\tilde{F}^{{pre}
	}(\Theta),\tilde{\gamma})$ is the \textbf{free non-unital prematrad generated
		by} $\Theta.$
	
	\label{freeprematrad}Define ${G}_{n,m}(\Theta):=\overset{\sim}{G}_{n,m}
	(\Theta)/\left(  A\sim\mathbf{1}A\sim A\mathbf{1}\right)  $ and let $F^{{pre}
	}\left(  \Theta\right)  =\left\langle G\left(  \Theta\right)  \right\rangle .$
	Let $\gamma$ be the product induced by $\tilde{\gamma},$ and let
	$\eta:R\rightarrow F_{1,1}^{{pre}}\left(  \Theta\right)  $ be the unit given
	by $\eta\left(  1_{R}\right)  =\mathbf{1}$. Then $(F^{{pre}}(\Theta
	),\gamma,\eta)$\ is the \textbf{free prematrad generated by} $\Theta$.
\end{definition}

\noindent There is the obvious projection
\begin{equation}
	\varsigma\varsigma_{n,m}^{{pre}}:\tilde{F}_{n,m}^{{pre}}(\Theta)\rightarrow
	{F}_{n,m}^{{pre}}(\Theta). \label{2varsigma}%
\end{equation}
In the operadic cases, we denote $\varsigma\varsigma_{1,n}^{{pre}}$ and
$\varsigma\varsigma_{n,1}^{{pre}}\ $by $\varsigma_{n};$ which we mean will be
clear from context.

\begin{example}\label{bialg}
	\underline{\textbf{The Bialgebra Prematrad}.} Let $\Theta=\left\{  \theta
	_{1}^{1}=\mathbf{1},\text{ }\theta_{2}^{1},\text{ }\theta_{1}^{2}\right\}  $
	and consider $\mathcal{H}^{{pre}}=F^{{pre}}\left(  \Theta\right)  /\sim,$
	where $A\sim B$ if ${bideg}\left(  A\right)  ={bideg}\left(  B\right)  .$ Let
	$\bar{\gamma}$ be the product on $\mathcal{H}^{pre}$ induced by the projection $F^{{pre}}\left(
	\Theta\right)  \rightarrow\mathcal{H}^{{pre}}\mathbf{.}$ The prematrad
	$\left(  \mathcal{H}^{{pre}},\bar{\gamma},\eta\right)  $ is the
	\textbf{bialgebra prematrad} \textbf{generated by} $\Theta$ if\smallskip
	
	\begin{enumerate}
		\item[\textit{(i)}] $\theta_{2}^{1}$ is associative: $\left[  \theta_{2}
		^{1}\right]  \left[
		\begin{array}
			[c]{cc}%
			\theta_{2}^{1} & \mathbf{1}%
		\end{array}
		\right]  =\left[  \theta_{2}^{1}\right]  \left[
		\begin{array}
			[c]{cc}%
			\mathbf{1} & \theta_{2}^{1}%
		\end{array}
		\right]  $,\smallskip
		
		\item[\textit{(ii)}] $\theta_{1}^{2}$ is coassociative: $\left[
		\begin{array}
			[c]{c}%
			\theta_{1}^{2}\\
			\mathbf{1}%
		\end{array}
		\right]  \left[  \theta_{1}^{2}\right]  =\left[
		\begin{array}
			[c]{c}%
			\mathbf{1}\\
			\theta_{1}^{2}%
		\end{array}
		\right]  \left[  \theta_{1}^{2}\right]  $, and\smallskip
		
		\item[\textit{(iii)}] $\theta_{1}^{2}$ and $\theta_{2}^{1}$ are Hopf
		compatible: $\left[  \theta_{1}^{2}\right]  \left[  \theta_{2}^{1}\right]
		=\left[
		\begin{array}
			[c]{c}%
			\theta_{2}^{1}\\
			\theta_{2}^{1}%
		\end{array}
		\right]  \left[
		\begin{array}
			[c]{cc}%
			\theta_{1}^{2} & \theta_{1}^{2}%
		\end{array}
		\right]  $.\smallskip
	\end{enumerate}
	
	\noindent The component ${\mathcal{H}}_{n,m}^{{pre}}$ of $\mathcal{H}^{{pre}}
	$ is generated by a class $c_{n,m},$ where
	\begin{multline*}
		c_{1,1}=\left\{  \mathbf{1}
		\right\},\,
		c_{1,2}=\left\{  \theta_{2}^{1}\right\}  ,\,
		c_{2,1}=\left\{  \theta_{1}^{2}\right\}  ,\,
		c_{1,3}=\left\{ \gamma
		(\theta_{2}^{1};\theta_{2}^{1},\mathbf{1}),\gamma(\theta_{2}^{1}
		;\mathbf{1,}\theta_{2}^{1})\right\}  ,\\
		c_{3,1}=\{ \gamma(\theta_{1}
		^{2},\mathbf{1};\theta_{1}^{2}),\gamma(\mathbf{1},\theta_{1}^{2}
		;\theta_{1}^{2})\} ,\,
		c_{2,2}=
		\left\{ \gamma(\theta_{1}^{2} ;\theta_{2}
		^{1}),\gamma(\theta_{2}^{1},
		\theta_{2}^{1};\theta_{1}^{2} ,\theta_{1}^{2})\right\}  ,
	\end{multline*}
	and so on. Note that ${\mathcal{H}}_{1,\ast}^{{\ pre}}$ is a
	non-sigma operad generated by $c_{1,2}$ subject to axiom (\textit{i}) and
	${\mathcal{H}}_{\ast,1}^{{pre}}$ is a non-sigma operad generated by $c_{2,1}$
	subject to axiom (\textit{ii}); both are isomorphic to the associativity
	operad $\underline{{\mathcal{A}}ss}$ \cite{MSS}. Given a graded $R$-module
	$H,$ a map of prematrads ${\mathcal{H}}^{{\ {pre}} }\rightarrow \mathcal{E}nd_{TH}$
	defines a bialgebra structure on $H$. In this case, each representative of
	$c_{n,m}$ determines a directed piece-wise linear path from $(m,1)$ to $(1,n)$ in
	$\mathbb{N}^{2}$, and each such path represents the class $c_{n,m} $ (see
	Figure 6).
\end{example}

\subsection{The Dimension of $\tilde{G}_{n,m}(\Theta)$}

Given a bigraded set $Q=\{Q_{m,n}\}_{mn\ge 1}$, let $\tilde{G}(Q)=\{\tilde{G}_{m,n}(Q)\}$ be the set of free bigraded monomials generated by $Q$. There is a canonical map
\begin{equation}\label{digamma}
	\digamma:\tilde{G}(Q_{m+1,n+1})\rightarrow \mathfrak{m}\circledast\mathfrak{n},
\end{equation}
which is a bijection if $Q_{m,n}$ is a singleton for all $m$ and $n$.
The definition of $\digamma$ requires some preliminaries.

\begin{definition}
	Let $B=\left(  b_{ij}\right)  $ be a $q\times p$ GBSM. The \emph{input and
		output leaf sequences of }$B$ are defined and denoted by $ils\left(  B\right)
	:=\left(  indeg\left(  b_{11}\right)  ,\ldots,indeg\left(  b_{1p}\right)
	\right)$ and $ols\left(  B\right)  :=\left(  outdeg\left(  b_{11}\right)
	,\ldots,outdeg\left(  b_{q1}\right)  \right) . $	
	Let $\xi=B_{1}\cdots B_{k}$ be a monomial in $\tilde{G}_{m,n}(Q).$
	
	\begin{enumerate}
		\item[(\textit{i)}] The \textbf{Input Leaf} \textbf{Decomposition of} $\xi$
		is the sequence
		\[
		ild(\xi):=\left( ils(B_{1}),\ldots,ils(B_{k})\right);
		\]
		when $B_{1}\cdots
		B_{k}$ is indecomposable we write $ILD(\xi)$. Given $ILD(\xi),$ the \textbf{Reduced Input Leaf Decomposition of} $\xi$, denoted by $RILD(\xi),$ is
		obtained from $ILD\left(\xi\right)$ by suppressing all unital components.
		\medskip
		
		\item[\textit{(ii)}] The \textbf{Output Leaf Decomposition of} $\xi$
		is the sequence
		\[
		old(\xi)  :=\left( ols(B_{1}),\ldots,ols(B_{k})\right)  ;
		\]
		when
		$B_{1}\cdots B_{k} $ is indecomposable we write $OLD\left(  \xi\right)  $.
		Given $OLD(\xi),$ the \textbf{Reduced Output Leaf Decomposition of}
		$\xi$, denoted by $ROLD(\xi),$ is obtained from $OLD(\xi)$ by suppressing all unital components.
	\end{enumerate}
\end{definition}

Given an indecomposable monomial $\xi=B_{1}\cdots B_{k}\in \tilde{G}_{m,n}(Q)$ and its i/o leaf decompositions $(\alpha, \beta):=(( \mathbf{m}_{1},\ldots,\mathbf{m}_{k}) ,( \mathbf{n}_{1},\ldots,\mathbf{n}_{k}))$, consider the  corresponding reduced i/o leaf decompositions $(R\alpha, R\beta)$  and
form the bipartition $\bar{R}\beta/\bar{R}\alpha,$
where $\bar{R}\alpha$ and $\bar{R}\beta$ are obtained  from
$\overset{\wedge}{\epsilon}\hspace{.01in}^{-1}(T_{R\alpha})$ and $\overset{\vee}{\epsilon}\hspace{.01in}^{-1}(T^{R\beta})$
by inserting empty blocks
in positions corresponding to the unital components of $\alpha$ and $\beta$.

Now define the map $\digamma$ in (\ref{digamma}) on a generator $Q_{m+1,n+1}\in Q$ by $\digamma(Q_{m+1,n+1}):=
\frac{\mathfrak{n}}{\mathfrak{m}}$ and on a monomial $\xi\in\tilde
{G}_{n+1,m+1}(Q)$ as follows: For each monomial $\zeta$ within $\xi$  (including $\zeta=\xi$), define $\digamma(\zeta):=\bar{R}OLD(\zeta)/\bar{R}ILD(\zeta);$
then $\digamma(\xi)$   is the top  level component of an element in $\mathfrak{m}\circledast\mathfrak{n}.$ Since the top level component of a framed element uniquely determines its canonical representative, we obtain the desired map $\digamma$. Note that  $\xi$ is  a  $\Theta$-monomial if and only if $h(\digamma(\xi))=1.$

\begin{example}
	\label{colrow}Consider the $\Theta$-monomial
	\[
	\xi=B_{1}B_{2}B_{3}B_{4}=\left[
	\begin{array}
		[c]{c}%
		\theta_{2}^{1}\smallskip\\
		\theta_{2}^{2}\smallskip\\
		\theta_{2}^{2}%
	\end{array}
	\right]  \left[
	\begin{array}
		[c]{cc}%
		\theta_{1}^{2} & \theta_{2}^{2}\smallskip\\
		\mathbf{1} & \theta_{2}^{1}%
	\end{array}
	\right]  \left[  \theta_{1}^{2}\text{ }\theta_{2}^{2}\text{ }\theta_{1}
	^{2}\right]  \left[  \theta_{2}^{1}\text{ }\mathbf{1}\text{ }\mathbf{1}\text{
	}\mathbf{1}\right]  \in\tilde{G}_{5,5}\left(  \Theta\right).
	\]
	The reduced leaf decompositions $RILD(\xi)=\left( (2),(12),(121),(2111)\right)$  and $ROLD(\xi)=\left((122),(21)\right.$, $\left. (2)\right)$ are
	the leaf decompositions of $T_{{RILD}\left(  \xi\right)  }$ and
	$T^{{ROLD}\left(  \xi\right)  }$ pictured in Figure 7 (reading from top-down).   But ${ILD}(\xi)=RILD(\xi)$ and $OLD(\xi)$ $=\left(  (122),(21),(2),(1)\right) $
	imply	$\overset{\wedge}{\epsilon}\hspace{.01in}^{-1}(T_{RILD(\xi)})=1|3|4|2$ and $\overset{\vee}{\epsilon} \hspace{.01in}^{-1}(T^{ROLD(\xi)})=24|1|3$. Hence
	\[
	\digamma(\xi)=\frac{\bar{R}OLD(\xi)}{\overset{\ }{\bar{R}ILD(\xi)}}=\frac{24|1|3|0}{1|3|4|2}\in \mathfrak{4}\circledast\mathfrak{4}.
	\]	
	
	\bigskip
	\unitlength 1mm 
	\linethickness{0.8pt}
	\ifx\plotpoint\undefined\newsavebox{\plotpoint}\fi 
	\begin{picture}(77.572,16.187)(0,0)
		\multiput(35.7,0.6)(.086102036,.067092495){128}{\line(1,0){.086102036}}
		\multiput(47.09,9.25)(.071565328,-.067092495){128}{\line(1,0){.071565328}}
		\put(46.985,9.04){\line(0,1){7.042}}
		\put(44.5,0.6){\line(0,1){2.042}}
		\put(50.3,0.6){\line(0,1){2.042}}
		\multiput(47.721,4.94)(.07207651,-.06607013){35}{\line(1,0){.07207651}}
		\multiput(38.721,2.94)(.07207651,-.06607013){35}{\line(1,0){.07207651}}
		\multiput(49.613,6.727)(-.08058554,-.06657066){64}{\line(-1,0){.08058554}}
		
		\multiput(60.019,15.977)(.070074384,-.067379215){117}{\line(1,0){.070074384}}
		\multiput(68.217,8.094)(.078612859,.067130082){119}{\line(1,0){.078612859}}
		\put(68.112,8.199){\line(0,-1){5.361}}
		\multiput(64.539,11.772)(.08122258,.06688918){66}{\line(1,0){.08122258}}
		\multiput(64.433,16.082)(.06607013,-.06907332){35}{\line(0,-1){.06907332}}
		\multiput(72.317,15.977)(.06638626,-.06638626){38}{\line(0,-1){.06638626}}
	\end{picture}
	\begin{center}
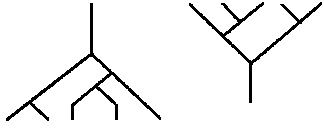
		Figure 7. The PLTs $T_{{RILD}\left(  \xi\right)  }$ and
		$T^{{ROLD}\left(  \xi\right)  }$.
	\end{center}
\end{example}
\smallskip

\begin{definition}
	The \textbf{dimension} of a monomial  $\xi\in\tilde{G}(\Theta)$ is defined and denoted by
	\[
	\left\vert \xi\right\vert :=\left\vert \digamma\left(  \xi\right)  \right\vert
	.
	\]
\end{definition}

\vspace{0.1in}
\noindent Then $|\theta^{t}_{s}|=s+t-3,$ $|(\theta^{1}_{x_{1}}\cdots\theta
^{1}_{x_{p}})|=x_{1}+\cdots+x_{p}-p-1$ and $|(\theta_{1}^{y_{1}}\cdots
\theta_{1}^{y_{q}})^{T}|=y_{1}+\cdots+y_{q}-q-1.$

\begin{example}
	Consider the $2\times2$ matrix
	\[
	B=\left[
	\begin{array}
		[c]{ccc}%
		\left[
		\begin{array}
			[c]{c}%
			\theta_{2}^{1}\smallskip\\
			\left[  \theta_{2}^{1}\right]  \left[  \mathbf{1}\text{ \ }\mathbf{1}\right]
		\end{array}
		\right]  \left[  \theta_{1}^{2}\,\,\,\,\left[
		\begin{array}
			[c]{c}%
			\mathbf{1}\\
			\mathbf{1}%
		\end{array}
		\right]  \left[  \theta_{1}^{2}\right]  \right]  &  & \left[
		\begin{array}
			[c]{c}%
			\theta_{2}^{1}\smallskip\\
			\left[  \mathbf{1}\right]  \left[  \theta_{2}^{1}\right]
		\end{array}
		\right]  \left[  \theta_{1}^{2}\,\,\theta_{1}^{2}\right]\bigskip \\
		\left[
		\begin{array}
			[c]{c}%
			\theta_{2}^{1}\smallskip\\
			\theta_{2}^{1}%
		\end{array}
		\right]  \left[  \theta_{1}^{2}\,\,\,\,\left[  \theta_{1}^{2}\right]  \left[
		\mathbf{1}\right]  \right]  &  & \left[
		\begin{array}
			[c]{c}%
			\mathbf{1}\\
			\mathbf{1}%
		\end{array}
		\right]  \left[  \theta_{2}^{2}\right]  \left[  \mathbf{1}\text{ \ }
		\mathbf{1}\right]
	\end{array}
	\right]  ,
	\]
	whose first row and first column are decomposable (cf. Example \ref{22cframed}%
	). To evaluate  $\left\vert \xi_{1,\ast}\right\vert $ for example, we
	decompose and obtain
	\[
	\left\vert \xi_{1,\ast}\right\vert =\left\vert \left[
	\begin{array}
		[c]{ccc}%
		\theta_{2}^{1} &  & \theta_{2}^{1}\medskip\\
		\left[  \theta_{2}^{1}\right]  \left[  \mathbf{1}\text{ \ }\mathbf{1}\right]
		&  & \left[  \mathbf{1}\right]  \left[  \theta_{2}^{1}\right]
	\end{array}
	\right]  \right\vert +\left\vert \left[  \theta_{1}^{2}\ \ \left[
	\begin{array}
		[c]{c}%
		\mathbf{1}\\
		\mathbf{1}%
	\end{array}
	\right]  \!\!\left[  \theta_{1}^{2}\right]  \ \ \theta_{1}^{2}\ \ \theta
	_{1}^{2}\right]  \right\vert =1+0=1.
	\]
	Completing the dimension calculation we have
	\[
	|B|=|B|^{row}+|B|^{col}+|B|^{ent}=1+1+1=3.
	\]
	
\end{example}
\pagebreak

\section{Free Matrads}

In this section we construct the free non-unital matrad $\tilde{F}(\Theta),$ which is an
auxiliary object, then obtain the free matrad ${F}(\Theta)$ by defining the
unit relation.

\subsection{Free Non-unital Matrads}

\label{bb}
\begin{definition}
	Given the map $\digamma$ defined in (\ref{digamma}), let
	\[
	\tilde{\mathcal{B}}_{n,m}(\Theta):=\digamma^{-1}(\mathfrak{m} \circledast
	_{pp}\mathfrak{n})\subset\tilde{G}_{n,m}(\Theta).
	\]
	The \textbf{free non-unital matrad generated by}
	$\Theta$ is the bigraded module
	\[
	\tilde{F}(\Theta):=\langle\tilde{\mathcal{B}}(\Theta)\rangle.
	\]
	
\end{definition}

\noindent Then by definition, $\tilde{\mathcal{B}}(\Theta)$ inherits the face operator
$\tilde{\partial}$ from $(\mathfrak{m}\circledast_{pp}\mathfrak{n}%
,\tilde{\partial})$. For example,
\begin{equation}
	\label{theta-22}\tilde{\partial}(\theta_{2}^{2})=\left[  \theta_{1}
	^{2}\right]  \left[  \theta_{2}^{1}\right]  -\left[
	\begin{array}
		[c]{c}%
		\theta_{2}^{1}\vspace{1mm}\\
		\theta_{2}^{1}%
	\end{array}
	\right]  \left[
	\begin{array}
		[c]{cc}%
		\theta_{1}^{2} & \theta_{1}^{2}%
	\end{array}
	\right]  .
\end{equation}

\noindent Let $\tilde{F}(\Theta)_{n,m}:=\langle \tilde{\mathcal{B}}_{n,m}(\Theta)\rangle$. By Proposition \ref{d-tildepartial} we immediately obtain

\begin{proposition}
	\label{dif-non}The face operator $\tilde{\partial}$ on $\mathfrak{m}
	\circledast_{pp}\mathfrak{n}$ induces a degree $-1$ differential $\tilde{\partial
	}:\tilde{F}(\Theta)_{n,m}\rightarrow\tilde{F}(\Theta)_{n,m}$ and a
	degree $-1$ differential on $\tilde{F}(\Theta)$.
\end{proposition}

Given $\xi\in\tilde{\mathcal{B}}(\Theta),$ the analogs of the operators
$\eta_{1}$ and $\eta_{2}$ are denoted by the same symbols and given by
$\eta_{1}(\xi)=\mathbf{1}\cdot\xi$ and $\eta_{2}(\xi)=\xi\cdot\mathbf{1}.$

\begin{example}
	Let $B=\left[
	\begin{array}
		[c]{cc}%
		\theta_{2}^{2} & \mathbf{1}\cdot\theta_{1}^{2}\vspace{1mm}\\
		\theta_{2}^{2} & \theta_{1}^{2}\cdot\mathbf{1}%
	\end{array}
	\right]  \in\tilde{\mathbf{B}};$ then $\tilde{\partial}(B)$ has two
	components  $\tilde{\partial}_{1}(B)$ and $\tilde{\partial}_{2}(B).$ The first
	is given by applying (\ref{theta-22}) to the $\left(  1,1\right)  $ entry:
	\[
	\tilde{\partial}_{1}(B)=\left[
	\begin{array}
		[c]{cc}%
		\left[
		\begin{array}
			[c]{c}%
			\theta_{2}^{1}\vspace{1mm}\\
			\theta_{2}^{1}%
		\end{array}
		\right]  \!\!\left[  \theta_{1}^{2}\,\,\,\theta_{1}^{2}\right]  &
		\mathbf{1}\cdot\theta_{1}^{2}\vspace{1mm}\\
		\theta_{2}^{2} & \theta_{1}^{2}\cdot\mathbf{1}%
	\end{array}
	\right]
	=\left[
	\begin{array}
		[c]{cc}%
		\left[
		\begin{array}
			[c]{c}%
			\theta_{2}^{1}\vspace{1mm}\\
			\theta_{2}^{1}%
		\end{array}
		\right]  \!\!\left[  \theta_{1}^{2}\,\,\,\theta_{1}^{2}\right]  & \left[
		\begin{array}
			[c]{c}%
			\mathbf{1}\\
			\mathbf{1}%
		\end{array}
		\right]\! [\theta_{1}^{2}]\vspace{1mm}\\
		\left[  \theta_{2}^{2}\right]\!  [\mathbf{1}\ \ \mathbf{1}] & \theta_{1}
		^{2}\cdot\mathbf{1}%
	\end{array}
	\right]
	\]
	\[
	=\left[
	\begin{array}
		[c]{cc}%
		\theta_{2}^{1}\vspace{1mm} & \mathbf{1}\\
		\theta_{2}^{1}\vspace{1mm} & \mathbf{1}\\
		\theta_{2}^{2} & \theta_{1}^{2}%
	\end{array}
	\right]  \left[
	\begin{array}
		[c]{ccc}%
		\theta_{1}^{2}\vspace{1mm} & \theta_{1}^{2} & \theta_{1}^{2}\\
		\mathbf{1} & \mathbf{1} & \mathbf{1}%
	\end{array}
	\right]  .
	\]
	The second arises by applying (\ref{theta-22}) to the entries in the first
	column, then coheretizing:
	
	\vspace{-.2in}
	
	\begin{multline*}
		\hspace{-.13in}
		B\rightarrow\left[
		\begin{array}
			[c]{cc}%
			\left[
			\begin{array}
				[c]{c}%
				\theta_{2}^{1}\vspace{1mm}\\
				\theta_{2}^{1}%
			\end{array}
			\right]  \!\!\left[  \theta_{1}^{2}\,\,\,\theta_{1}^{2}\right]  &
			\mathbf{1}\cdot\theta_{1}^{2}\vspace{1mm}\\
			\left[
			\begin{array}
				[c]{c}%
				\theta_{2}^{1}\vspace{1mm}\\
				\theta_{2}^{1}%
			\end{array}
			\right]  \!\!\left[  \theta_{1}^{2}\,\,\,\theta_{1}^{2}\right]  & \theta
			_{1}^{2}\cdot\mathbf{1}%
		\end{array}
		\right]
		\xrightarrow{coheretize}
		\\
		\tilde{\partial}_{2}(B)=\left[\!\!
		\begin{array}
			[c]{cc}%
			\left[
			\begin{array}
				[c]{c}%
				\theta_{2}^{1}\vspace{1mm}\\
				\theta_{2}^{1}%
			\end{array}
			\right]  \!\!\left[  \theta_{1}^{2}\,\,\,\,\mathbf{1}\cdot\theta_{1}
			^{2}\right]  & \!\!\left[
			\begin{array}
				[c]{c}%
				\mathbf{1}\\
				\mathbf{1}%
			\end{array}
			\right]  \!\!\left[  \mathbf{1}\cdot\theta_{1}^{2}\right]  \vspace{1mm}\\
			\left[
			\begin{array}
				[c]{c}%
				\theta_{2}^{1}\vspace{1mm}\\
				\theta_{2}^{1}%
			\end{array}
			\right]  \!\!\left[  \theta_{1}^{2}\,\,\,\,\theta_{1}^{2}\cdot\mathbf{1}
			\right]  & \left[
			\begin{array}
				[c]{c}%
				\mathbf{1}\\
				\mathbf{1}%
			\end{array}
			\right]  \!\!\left[  \theta_{1}^{2}\cdot\mathbf{1}\right]
		\end{array}\!\!
		\right] \!\! =\!\!
		\left[
		\begin{array}
			[c]{cc}%
			\theta_{2}^{1}\vspace{1mm} & \mathbf{1}\\
			\theta_{2}^{1}\vspace{1mm} & \mathbf{1}\\
			\theta_{2}^{1}\vspace{1mm} & \mathbf{1}\\
			\theta_{2}^{1} & \mathbf{1}%
		\end{array}
		\right] \!\! \left[
		\begin{array}
			[c]{ccc}%
			\theta_{1}^{2} & \mathbf{1}\cdot\theta_{1}^{2} & \mathbf{1}\cdot\theta_{1}
			^{2}\vspace{1mm}\\
			\theta_{1}^{2} & \theta_{1}^{2}\cdot\mathbf{1} & \theta_{1}^{2}\cdot\mathbf{1}
		\end{array}
		\right]\!  .
	\end{multline*}
	
\end{example}

\begin{example}
	\label{matrad-42} Consider the $2$-dimensional monomial $\xi=AB=[\theta
	_{2}^{1}\theta_{2}^{1}]^{T}[\theta_{2}^{2}\theta_{2}^{2}]\in\tilde
	{\mathcal{B}}_{2,4}$, which corresponds to the element $\rho$ in Example
	\ref{framed-31}; then $\tilde{\partial}(\xi)=\tilde{\partial}(A)B\cup
	A\tilde{\partial}(B)=A\tilde{\partial}(B).$ The $1$-dimensional components in
	$\tilde{\partial}(B):=\{\tilde{\partial}_{k}(B)\}_{1\leq k\leq7}$ are
	\[
	\hspace{-.55in}
	\begin{array}
		[c]{l}%
		\tilde{\partial}_{1}(B)=\left[  \,\left[
		\begin{array}
			[c]{c}%
			\theta_{2}^{1}\smallskip\\
			\theta_{2}^{1}%
		\end{array}
		\right]\!\!  \left[  \theta_{1}^{2}\text{ \ }\theta_{1}^{2}\right]  \,\,\,\left[
		\begin{array}
			[c]{c}%
			\mathbf{1}\smallskip\\
			\mathbf{1}%
		\end{array}
		\right] \!\! \left[  \theta_{2}^{2}\right]  \,\right]  =\left[
		\begin{array}
			[c]{cc}%
			\theta_{2}^{1}\smallskip & \mathbf{1}\\
			\theta_{2}^{1} & \mathbf{1}%
		\end{array}
		\right] \! \left[  \theta_{1}^{2}\text{ \ }\theta_{1}^{2}\text{ \ }\theta
		_{2}^{2}\right]  ;
		\medskip\\
		\tilde{\partial}_{2}(B)=\left[  \,\left[
		\begin{array}
			[c]{c}%
			\mathbf{1}\smallskip\\
			\mathbf{1}%
		\end{array}
		\right] \!\! \left[  \theta_{2}^{2}\right]  \,\,\,\left[
		\begin{array}
			[c]{c}%
			\theta_{2}^{1}\smallskip\\
			\theta_{2}^{1}%
		\end{array}
		\right] \!\! \left[  \theta_{1}^{2}\text{ \ }\theta_{1}^{2}\right]  \,\right]
		=\left[
		\begin{array}
			[c]{cc}%
			\mathbf{1} & \theta_{2}^{1}\smallskip\\
			\mathbf{1} & \theta_{2}^{1}%
		\end{array}
		\right]\!  [\theta_{2}^{2}\text{ \ }\theta_{1}^{2}\text{ \ }\theta_{1}^{2}];
	\end{array}
	\]
	\[
	\begin{array}
		[c]{l}%
		\tilde{\partial}_{3}(B)=\left[  \,\left[
		\begin{array}
			[c]{c}%
			\lbrack\mathbf{1}][\theta_{2}^{1}]\\
			\theta_{2}^{1}%
		\end{array}
		\right] \! [\theta_{1}^{2}\text{ \ }\theta_{1}^{2}]\,\,\ \,\left[
		\begin{array}
			[c]{c}%
			\lbrack\theta_{2}^{1}][\mathbf{1}\text{ \ }\mathbf{1}]\smallskip\\
			\theta_{2}^{1}%
		\end{array}
		\right] \! [\theta_{1}^{2}\text{ \ }\theta_{1}^{2}]\,\,\right]  =\medskip\\
		\hspace{2in}\left[\!\!\!
		\begin{array}
			[c]{cc}%
			\begin{array}
				[c]{c}%
				\lbrack\mathbf{1}][\theta_{2}^{1}]\smallskip\\
				\theta_{2}^{1}
			\end{array}
			&
			\begin{array}
				[c]{c}%
				\lbrack\theta_{2}^{1}][\mathbf{1}\text{ \ }\mathbf{1}]\smallskip\\
				\theta_{2}^{1}
			\end{array}
		\end{array}
		\!\!\!\right] \! \left[  \theta_{1}^{2}\text{ \ }\theta_{1}^{2}\text{ \ }
		\theta_{1}^{2}\text{ \ }\theta_{1}^{2}\right]  ;\medskip\\
		\tilde{\partial}_{4}(B)=\left[  \,\left[
		\begin{array}
			[c]{c}%
			\theta_{2}^{1}\smallskip\\
			\left[  \theta_{2}^{1}\right] \! \left[  \mathbf{1}\text{ \ }\mathbf{1}\right]
		\end{array}
		\right] \! [\theta_{1}^{2}\text{ \ }\theta_{1}^{2}]\,\,\,\left[
		\begin{array}
			[c]{c}%
			\theta_{2}^{1}\smallskip\\
			\left[  \mathbf{1}\right]\!  \left[  \theta_{2}^{1}\right]
		\end{array}
		\right] \! [\theta_{1}^{2}\text{ \ }\theta_{1}^{2}]\,\,\right]  =\medskip\\
		\hspace{2in}\left[  \!\!\!
		\begin{array}
			[c]{cc}%
			\begin{array}
				[c]{c}%
				\theta_{2}^{1}\smallskip\\
				\left[  \theta_{2}^{1}\right]\!  \left[  \mathbf{1}\text{\ }\mathbf{1}\right]
			\end{array}
			&
			\begin{array}
				[c]{c}%
				\theta_{2}^{1}\smallskip\\
				\left[  \mathbf{1}\right] \! \left[  \theta_{2}^{1}\right]
			\end{array}
		\end{array}
		\!\!\!\right] \! [\theta_{1}^{2}\text{ \ }\theta_{1}^{2}\text{ \ }\theta_{1}%
		^{2}\text{ \ }\theta_{1}^{2}];\medskip\\
		\tilde{\partial}_{5}(B)=[\,[\theta_{1}^{2}][\theta_{2}^{1}]\,\,\,[\theta
		_{2}^{2}][\mathbf{1}\mathbf{1}]\,]=[\theta_{1}^{2}\text{ \ }\theta_{2}%
		^{2}][\theta_{2}^{1}\text{ \ }\mathbf{1}\text{ \ }\mathbf{1}];\medskip\\
		\tilde{\partial}_{6}(B)=[\,[\theta_{2}^{2}][\mathbf{1}\text{ \ }%
		\mathbf{1}]\,\,\,[\theta_{1}^{2}][\theta_{2}^{1}]\,]=[\theta_{2}^{2}\text{
			\ }\theta_{1}^{2}][\mathbf{1}\text{ \ }\mathbf{1}\text{ \ }\theta_{2}%
		^{1}];\medskip\\
		\tilde{\partial}_{7}(B)=[\,[\theta_{1}^{2}][\theta_{2}^{1}]\,\,\,[\theta
		_{1}^{2}][\theta_{2}^{1}]\,]=[\theta_{1}^{2}\,\ \theta_{1}^{2}][\theta_{2}%
		^{1}\,\ \theta_{2}^{1}].
	\end{array}
	\]
	
\end{example}

\subsection{Free Matrads Defined}

\begin{definition}
	\label{freematrad}For $m, n\geq 0$, define $\mathcal{B}_{n+1,m+1}%
	(\Theta):=\tilde{\mathcal{B}}_{n,m}(\Theta)/  \left(  B\sim\mathbf{1}B\sim
	B\mathbf{1}\right)  $. Let $F( \Theta): =\left\langle \mathcal{B}( \Theta)
	\right\rangle ,$ let $\gamma$ be the product induced by $\tilde{\gamma}$ on
	$\tilde{F}( \Theta),$ and let $\eta:R\rightarrow F_{1,1}(\Theta) $ be the
	unit  given by $\eta( 1_{R}) =\mathbf{1}.$ Then $(F(\Theta),\gamma,\eta)$ is
	the  \textbf{free matrad generated by} $\Theta$.
\end{definition}

\noindent The map
\begin{equation}
	\varsigma\varsigma:=\left.  \varsigma\varsigma^{{pre}}\right\vert _{\tilde
		{F}(\Theta)}:\tilde{F}_{m,n}(\Theta)\rightarrow{F}_{m+1,n+1}(\Theta)
	\label{tildevarsigmas}%
\end{equation}
is the canonical projection and the dimension of a balanced monomial $\xi
\in\mathcal{B}(\Theta)$ is
\[
|\xi|:=\sum_{\theta_{s}^{t}\text{ in }\xi}|\theta_{s}^{t}|.
\]
From (\ref{digamma}) we immediately establish the following commutative diagram:

\[%
\begin{array}
	[c]{ccc}%
	\ \ \ \ \tilde{\mathcal{B}}_{n,m}(\Theta) & \overset{\approx}{\longrightarrow}
	& \mathfrak{m}\circledast_{pp}\mathfrak{n}\vspace*{0.1in}\\
	\varsigma\varsigma\downarrow\text{ \ \ \ } &  & \ \ \ \ \ \downarrow
	\vartheta\vartheta\vspace*{0.1in}\\
	{\mathcal{B}}_{n+1,m+1}(\Theta) & \overset{\approx}{\longrightarrow} &
	\mathfrak{m}\circledast_{kk}\mathfrak{n}.
\end{array}
\]
We have (cf. Proposition \ref{dif-non})

\begin{theorem}
	\label{dif-unital} The differential $\tilde{\partial}$ induces a differential
	${\partial}$ on $F(\Theta)$ such that
	\[
	(F(\Theta)),\partial)\overset{\approx}{\longrightarrow}(C_{\ast}
	(\mathfrak{m}\circledast_{kk}\mathfrak{n}),\partial)
	\]
	is a canonical isomorphism.
\end{theorem}

\begin{example}
	Regarding Example \ref{matrad-42}, let $\zeta=\varsigma\varsigma(\xi)$ and
	$C=\varsigma\varsigma(B).$ Once again, $\zeta\in\mathcal{B}_{4,2}(\Theta)$ is
	$2$-dimensional and has four $1$-dimensional faces $\partial_{1}(C),$
	$\partial_{2}(C),$ $\partial_{5}(C),$ and $\partial_{6}(C)$ respectively induced by
	$\tilde{\partial}_{1}(B),$ $\tilde{\partial}_{2}(B),$ $\tilde{\partial}
	_{5}(B),$ and $\tilde{\partial}_{6}(B),$ because $|\varsigma\varsigma
	(\tilde{\partial}_{3}(\xi))|=|\varsigma\varsigma(\tilde{\partial}_{4}
	(\xi))|=|\varsigma\varsigma(\tilde{\partial}_{7}(\xi))|=0$ in $\mathcal{B}
	_{4,2}(\Theta).$
\end{example}

We denote the chain complex $(F(\Theta)),\partial)$ by $(\mathcal{H}_{\infty},\partial)$ and refer to it as the $A_{\infty}$\emph{-bialgebra matrad} (cf. Theorem \ref{kkh} below). Thus an $A_\infty$-bialgebra structure on a DG module $A$ is defined by a prematrad map $\mathcal{H}_{\infty}\rightarrow \mathcal{E}nd_{TA}$ (cf. Example \ref{ENDA}).

\section{Constructions of PP and KK}

\label{PP-KK}In this section we construct the polytopes $PP$ and $KK$ as
geometric realizations of the balanced and reduced balanced framed joins. We
subsequently identify the free matrad $\mathcal{H}_{\infty}$ with the cellular
chains of $KK.$

\subsection{The Bipermutahedron $PP_{n,m}$}

Let $w\in P_{m+n}$. By equality $P_{m}\ast_{c}P_{n}=P_{m+n},$ which follows from (\ref{join}), the decomposition
$w=\alpha\Cup\beta\in P_{m}\ast_{c}P_{n}$ uniquely
determines a bipartition $\beta/\alpha\in P_{r}^{\prime}\left(  m\right)
\times P_{r}^{\prime}\left(  n\right).$
Define the \emph{combinatorial join}
\[
\mathfrak{m}\ast_{c}\mathfrak{n:}=\left\{  \beta/\alpha\in
P_{r}^{\prime}\left(  m\right)  \times P_{r}^{\prime}\left(  n\right)
:\alpha\Cup\beta\in P_{m}\ast_{c}P_{n},\  1\leq r\leq m+n\right\};
\]
then the bijection
$
g_{1}:\mathfrak{m}\ast_{c}\,\mathfrak{n}{\rightarrow
}P_{m+n}
$
given by $\beta/\alpha\mapsto\alpha\Cup\beta$ identifies $\beta/\alpha\in \mathfrak{m}\ast_{c}\,\mathfrak{n}$ with
the face $C_{1}|\cdots|C_{r}:=\alpha\Cup\beta$ of $P_{m+n}$, and $\partial_{M^{k},N^{k}}({\beta}/
{\alpha})$ with the boundary component $\partial_{X^{k}}(C_{1}|\cdots|C_{r}),$
where $X^{k}:=M^{k}\cup(N^{k}+m)\subseteq C_{k}$ (cf. (2.6)). Furthermore, $g_1 $ extends to a
bijection
\begin{equation}
	g:\mathfrak{m}\circledast_{pp}\,\mathfrak{n}\overset{\approx}{\longrightarrow
	}P_{m}\ast_{pp}P_{n}=PP_{m,n}\label{PP}%
\end{equation}
via $(c_{1},c_{2},c_{3})\overset{g}{\mapsto}(g_{1}(c_{1}),c_{2}),$ where
$(g_{1}(c_{1}),c_{2})$ ranges over the faces of $PP_{n,m}$ to be constructed.

Given $m,n\geq 0,$ construct the polytope $PP_{n,m}$ by subdividing $P_{m+n}$
in the following way: A face $e\sqsubseteq P_{m+n}$ is indexed by a partition
$C_{1}|\cdots|C_{r}\in P(\mathfrak{m\Cup n})=P_{m+n},$ and in particular,
$\mathfrak{m\Cup n}$ indexes the top dimensional cell. While the top
dimensional cell of $PP_{n,m}$ is also indexed by $\mathfrak{m\Cup n}%
=g_{1}(\mathfrak{n}/\mathfrak{m})$ thought of as the elementary bipartition
$\mathfrak{n}/\mathfrak{m}$, a proper face $e\sqsubset PP_{n,m}$ is indexed
by a pair $\left(  w,T(w)\right)  ,$ where $w=\alpha\Cup\beta$ is expressed as
the bipartition $\beta/\alpha=\mathbf{C}_{1}\cdots\mathbf{C}_{r}\in
P_{r}^{\prime}(\mathfrak{m})\times P_{r}^{\prime}(\mathfrak{n)}$ and
$T(w)=C_{1}\cdots C_{r}:=T^{1}(\mathbf{C}_{1})\cdots T^{1}(\mathbf{C}_{r})$ is
a coherent formal product of bipartition matrices. When $\beta/\alpha$ is
coherent (when $|w|=0$ for example), we identify $\left(  w,T(w)\right)  $
with $w.$

Express $w$ as a product $w=E_{1}\times\cdots\times E_{r}$ of permutahedra. A
proper cell $e=(w,T(w))\sqsubset PP_{n,m}$ is a subdivision cell of the form
$e=e_{1}\times\cdots\times e_{r}\sqsubseteq w$ such that $e_{k}\subseteq
E_{k}$ and no proper face of $w$ contains $e$. The dimension $|e|:=|e_{1}%
|+\cdots+|e_{r}|,$ where
\[
|e_{k}|:=|\overset{\wedge}{\mathbf{e}}(C_{k})|+|\overset{\vee}{\mathbf{e}%
}(C_{k})|-\overset{\wedge}{v}(C_{k})-\overset{\vee}{v}(C_{k})+1.
\]
Thus, $|e|=|w|$ if and only if $C_{k}$ is TD coherent for all $k.$

The action of $\tilde{\partial}$ on the top dimensional cell $\mathfrak{m\Cup
	n}\sqsubset PP_{n,m}$ produces the set of all codimension $1$ faces
\[
\tilde{\partial}(\mathfrak{m\Cup n})=\left\{  (C_{1}|C_{2},T(C_{1}%
|C_{2})):C_{1}|C_{2}\in P_{2}(\mathfrak{m\Cup n})\right\}  ,
\]
and its action on a proper cell $e=(w,T(w))$ is defined by the set
\[
\tilde{\partial}(e):=\bigcup\limits_{k\in\mathfrak{r}} \left\{  \left(  w,\partial_{\mathbf{M}_{1}		,\mathbf{N}_{1}}^{k}T(w)\right), \left(  \partial_{M^{k}}
(w),\text{coheretizations\ of}\ \partial_{\mathbf{M}_{2},\mathbf{N}_{2}}
^{k}T(w)\right)\right\}
\]
where $\partial_{\mathbf{M}_{1},\mathbf{N}_{1}}^{k}T(w)$ is indecomposable and
$\partial_{\mathbf{M}_{2},\mathbf{N}_{2}}^{k}T(w)$ is decomposable with
coheretizable matrix factors.

A cell of codimension $2$ or more is given by the iterated action of
$\tilde{\partial}$ on some cell of codimension $1$. Furthermore, a careful
analysis of Lemma \ref{basic} reveals that the face structure of $e_{k}$
agrees with that of $\overset{\wedge}{\mathbf{e}}(C_{k})$ or $\overset{\vee
}{\mathbf{e}}(C_{k})$ thought of as cells in the $(q_{k}-1)^{st}$ or
$(p_{k}-1)^{st}$ subdivision complexes of $P_{\#\mathbf{is}(C_{k})}$ or
$P_{\#\mathbf{os}(C_{k})},$ respectively (see Example \ref{subdivision} and
Figure 8 below). It is straightforward to check that $g$ is compatible with
face operators.
\begin{example}
	\label{subdivision}To obtain the subdivision cells of $P_{2,2}$ contained in
	$w=1|234\subset P_{4},$ construct the decomposition $1|234=1|2$ $\Cup$ $0|34$
	and obtain the bipartition
	\[
	\frac{\,\,0|34}{1|2}=\left(
	\begin{array}
		[c]{c}%
		\frac{0}{1}\vspace{1mm}\\
		\frac{0}{1}\vspace{1mm}\\
		\frac{0}{1}%
	\end{array}
	\right) \! \left(
	\begin{array}
		[c]{cc}
		\frac{34}{0} & \frac{34}{2}
	\end{array}
	\right)  .
	\]
	Then $w$ subdivides as $a\cup b\cup c,$ where $a$ and $b$ are squares and $c$
	is a heptagon. The squares are labeled by
	\[
	a=\left(  \frac{0|34}{1|2},\left(
	\begin{array}
		[c]{c}%
		\frac{0}{1}\vspace{1mm}\\
		\frac{0}{1}\vspace{1mm}\\
		\frac{0}{1}%
	\end{array}
	\right) \!\! \left(
	\begin{array}
		[c]{cc}%
		\frac{34}{0} & \frac{4|3}{2|0}%
	\end{array}
	\right)  \right)  \text{ \ and \ \ }
	b=\left(  \frac{0|34}{1|2},\left(
	\begin{array}
		[c]{c}%
		\frac{0}{1}\vspace{1mm}\\
		\frac{0}{1}\vspace{1mm}\\
		\frac{0}{1}%
	\end{array}
	\right)\!\!  \left(
	\begin{array}
		[c]{cc}%
		\frac{34}{0} & \frac{4|3}{0|2}%
	\end{array}
	\right)  \right)  ;
	\]
	the heptagon $c$ is labeled by
	\[
	c=\left(  \frac{0|34}{1|2},\left(
	\begin{array}
		[c]{c}%
		\frac{0}{1}\vspace{1mm}\\
		\frac{0}{1}\vspace{1mm}\\
		\frac{0}{1}\vspace{1mm}\\
		\frac{0}{1}%
	\end{array}
	\right) \!\! \left(
	\begin{array}
		[c]{cc}%
		\frac{3|4}{0|0} & \frac{34}{2}%
	\end{array}
	\right)  \right)  .
	\]
	One subdivision vertex  with initial bipartition $\frac{0|34}{1|2}$
	lies in the interior of $w$:
	\[
	z_{1}=\left(  \frac{0|34}{1|2}\,,\left(
	\begin{array}
		[c]{c}%
		\frac{0}{1}\vspace{1mm}\\
		\frac{0}{1}\vspace{1mm}\\
		\frac{0}{1}%
	\end{array}
	\right)\!\!  \left(
	\begin{array}
		[c]{cc}%
		\frac{3|4}{0|0} & \frac{4|0|3}{0|2|0}%
	\end{array}
	\right)  \right)  .
	\]
	Three subdivision vertices lie in the boundary of $w:$
	\[
	z_{2}=\left(  \frac{0|34|0}{1|0|2},\left(
	\begin{array}
		[c]{c}%
		\frac{0}{1}\vspace{1mm}\\
		\frac{0}{1}\vspace{1mm}\\
		\frac{0}{1}%
	\end{array}
	\right)\!\!  \left(
	\begin{array}
		[c]{cc}%
		\frac{3|4}{0|0} & \frac{4|3}{0|0}%
	\end{array}
	\right) \! \left(
	\begin{array}
		[c]{cc}%
		\frac{0}{0} & \frac{0}{2}%
	\end{array}
	\right)  \right)  ,
	\]%
	\[
	u_{1}=\left(  \frac{0|0|34}{1|2|0}\,,\left(
	\begin{array}
		[c]{c}%
		\frac{0}{1}\vspace{1mm}\\
		\frac{0}{1}\vspace{1mm}\\
		\frac{0}{1}%
	\end{array}
	\right) \!\! \left(
	\begin{array}
		[c]{cc}%
		\frac{0}{0} & \frac{0}{2}\vspace{1mm}\\
		\frac{0}{0} & \frac{0}{2}\vspace{1mm}\\
		\frac{0}{0} & \frac{0}{2}%
	\end{array}
	\right) \!\! \left(
	\begin{array}
		[c]{ccc}%
		\frac{3|4}{0|0} & \frac{3|4}{0|0} & \frac{4|3}{0|0}%
	\end{array}
	\right)  \right)  ,
	\]
	and
	\[
	u_{2}=\left(  \frac{0|0|34}{1|2|0},\left(
	\begin{array}
		[c]{c}%
		\frac{0}{1}\vspace{1mm}\\
		\frac{0}{1}\vspace{1mm}\\
		\frac{0}{1}%
	\end{array}
	\right) \!\! \left(
	\begin{array}
		[c]{cc}%
		\frac{0}{0} & \frac{0}{2}\vspace{1mm}\\
		\frac{0}{0} & \frac{0}{2}\vspace{1mm}\\
		\frac{0}{0} & \frac{0}{2}%
	\end{array}
	\right)\!\!  \left(
	\begin{array}
		[c]{ccc}%
		\frac{3|4}{0|0} & \frac{4|3}{0|0} & \frac{4|3}{0|0}%
	\end{array}
	\right)  \right)  .
	\]
	\vspace{0.2in}
\end{example}

\unitlength 1mm
\linethickness{0.4pt} \ifx\plotpoint\undefined\newsavebox{\plotpoint}\fi
\begin{picture}(82.829,76.248)(0,0)
	\put(66.67,5.33){\line(0,1){49.67}}
	\put(66.67,5.33){\line(-6,5){28.67}}
	\put(66.67,55){\line(-4,3){28.33}}
	\put(66.67,45){\line(-6,5){20.67}}
	\put(46,22.66){\line(0,1){48}}
	\put(46,70.66){\circle*{1.6}}
	\put(46,62.33){\circle*{1.6}}
	\put(46,70.66){\circle*{1.6}}
	\put(66.67,55){\circle*{1.6}}
	\put(66.67,45.33){\circle*{1.6}}
	\put(66.67,35.33){\circle*{1.6}}
	\put(66.67,5.33){\circle*{1.6}}
	\put(46,45.66){\circle*{1.6}}
	\put(46,22.66){\circle*{1.6}}
	\put(66.67,25.33){\circle*{1.0}}
	\put(66.67,15.33){\circle*{1.0}}
	\put(77,35.33){\circle*{0}}
	\put(77,15.33){\circle*{0}}
	\put(46,53){\circle*{1.0}}
	\multiput(45.75,53.25)(.01686746988,-.02269076305){1245}{\line(0,-1){.02269076305}}
	\multiput(56,39.5)(.0301966292,.0168539326){356}{\line(1,0){.0301966292}}
	\put(62.75,37.25){\makebox(0,0)[cc]{$a$}}
	\put(55.25,47.75){\makebox(0,0)[cc]{$b$}}
	\put(55.75,39.75){\circle*{1}}
	\put(66.893,25.419){\line(1,0){10.554}}
	\put(67.042,54.852){\line(1,0){13.973}}
	\put(66.893,15.46){\line(1,0){10.257}}
	\put(66.596,5.649){\line(1,0){14.865}}
	\put(66.744,35.23){\line(1,0){13.676}}
	\multiput(46.082,45.784)(-.020298966,.0168637564){476}{\line(-1,0){.020298966}}
	\put(55.71,58.443){\makebox(0,0)[cc]{${14|23}$}}
	\put(66.641,45.619){\line(1,0){9.67}}
	\put(82.829,44.988){\makebox(0,0)[cc]{${124|3}$}}
	\put(80.516,21.023){\makebox(0,0)[cc]{${12|34}$}}
	\put(31.954,67.272){\makebox(0,0)[cc]{${134|2}$}}
	\put(32.954,37.272){\makebox(0,0)[cc]{${13|24}$}}
	\put(54.448,28.591){\makebox(0,0)[cc]{${c}$}}
	\put(53.7,38.471){\makebox(0,0)[cc]{$z_1$}}
	\put(43.5,53.5){\makebox(0,0)[cc]{$z_2$}}
	\put(64.0,15.346){\makebox(0,0)[cc]{$u_1$}}
	\put(64.0,24.386){\makebox(0,0)[cc]{$u_2$}}
	\put(66.431,2.943){\makebox(0,0)[cc]{$_{1|2|3|4}$}}
	\put(43.517,20){\makebox(0,0)[cc]{$_{1|3|2|4}$}}
	\put(71.2,37){\makebox(0,0)[cc]{$_{1|2|4|3}$}}
	\put(71.2,47.352){\makebox(0,0)[cc]{$_{1|4|2|3}$}}
	\put(41,44.568){\makebox(0,0)[cc]{$_{1|3|4|2}$}}
	\put(40.573,62.437){\makebox(0,0)[cc]{$_{1|4|3|2}$}}
\end{picture}

\begin{center}
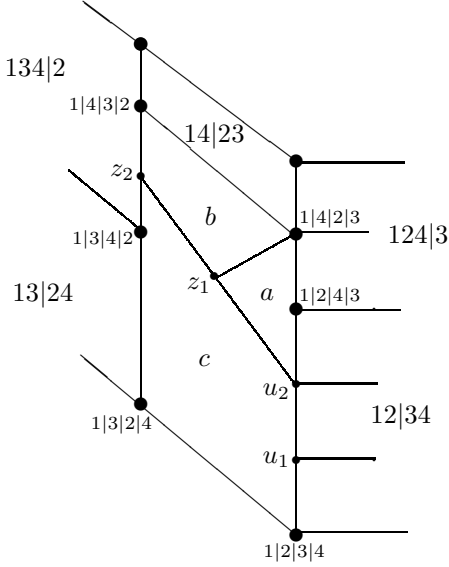
	Figure 8. The subdivision cells of $1|234$ in $PP_{2,2}.$
\end{center}

\vspace{0.2in}

\subsection{The Biassociahedron $KK_{n+1,m+1}$}

The\emph{ reduced combinatorial framed join of $\mathfrak{m}$ and $\mathfrak{n}$}, denoted by $\mathfrak{m}\ast_{ck}\,\mathfrak{n}$, is  the image of the subset $\mathfrak{m}\ast_{c}\,\mathfrak{n}\subseteq \mathfrak{m}\circledast\mathfrak{n} $ under the projection $\vartheta\vartheta$, i.e., $\mathfrak{m}\ast_{ck}\,\mathfrak{n}:=\vartheta\vartheta(\mathfrak{m}\ast_{c}\,\mathfrak{n}).$
In particular, the equivalence relation on $\mathfrak{m}\circledast\mathfrak{n}$    restricted to  $0$-dimensional bipartitions $\frac{\beta_{1}}{\alpha_{1}
}=\mathbf{C}_{1,1}\cdots\mathbf{C}_{m+n,1}$ and $\frac{\beta_{2}}{\alpha_{2}
}=\mathbf{C}_{1,2} \cdots
\mathbf{C}_{m+n,2}$ in $\mathfrak{m}\ast_{c}\,\mathfrak{n}$ implies
$\frac{\beta_{1}}{\alpha_{1}}\sim\frac{\beta_{2}
}{\alpha_{2}}$ if and only if there exist associations $A_{1}\cdots
A_{s}=\mathbf{C}_{1,1}\cdots\mathbf{C}_{m+n,1},$ $s\leq m+n,$ and $B_{1}\cdots
B_{t}=\mathbf{C}_{1,2}\cdots\mathbf{C}_{m+n,2},$ $t\leq m+n,$ such that

\begin{itemize}
	\item the number of semi-null matrix factors $A_{k}$ with $\#\mathbf{is}%
	(A_{k})\neq0$ equals the number of semi-null matrix factors $B_{k}$ with
	$\#\mathbf{is}(B_{k})\neq0$,
	
	\item corresponding semi-null matrix factors $A_{k}$ and $B_{k}$ differ only
	in the empty biblocks in their corresponding entries,
	
	\item the number of semi-null matrix factors $A_{\ell}$ with $\#\mathbf{os}%
	(A_{\ell})\neq0$ equals the number of semi-null matrix factors $B_{\ell}$ with
	$\#\mathbf{os}(B_{\ell})\neq0$, and
	
	\item corresponding semi-null matrix factors $A_{\ell}$ and $B_{\ell}$ differ
	only in the empty biblocks in their corresponding entries.
\end{itemize}
(See Example \ref{P-K}   and Remark \ref{ck}.)

Thus we obtain the quotient polytopes $K_{n+1,m+1}:=P_{m+n}/\sim$ of dimension
$m+n-1$ determined by the set $\mathfrak{m}\ast_{ck}\,\mathfrak{n}$ together
with the cellular projection
\[
\vartheta_{n,m}:P_{m+n}\rightarrow K_{n+1,m+1}.
\]
While $K_{n+1,1}=K_{1,n+1}$ is the associahedron $K_{n+1}$ for all $n\geq1$,
$K_{n+1,m+1}$ is the permutahedron $P_{m+n}$ for $1\leq m,n\leq2$ and
$K_{n,2}=K_{2,n}$ is the multiplihedron $J_{n}$ for all $n$ (cf. Remark
\ref{ck}).

The canonical projection $\vartheta\vartheta$ defined in (\ref{co-bivartheta})
and the bijection $g$ defined in (\ref{PP}) induce the subdivision of
$K_{n+1,m+1}$ that produces the biassociahedron $KK_{n+1,m+1}$ and commutes
the following diagram:
\[%
\begin{array}
	[c]{ccc}%
	\ \ \ \ PP_{n,m} & \overset{\approx}{\longrightarrow} & P_{m+n}\vspace
	*{0.1in}\\
	\vartheta\vartheta\downarrow\text{ \ \ \ } &  & \ \ \ \ \ \downarrow
	\vartheta_{n,m}\vspace*{0.1in}\\
	KK_{n+1,m+1} & \underset{\approx}{\longrightarrow} & K_{n+1,m+1},
\end{array}
\]
where the horizontal maps are non-cellular homeomorphisms induced by the
subdivision process. We also refer to $KK_{m+1,n+1}$ as the \textit{reduced
	balanced framed join of the associahedra} $K_{m+1}$ and $K_{n+1},$ and denote
$K_{m+1}\ast_{kk}K_{n+1}:=KK_{m+1,n+1}.$

\begin{remark}
	\label{ck}
	Regarding Example \ref{P-K}, the associahedron $K_{n+1}$ is
	identified with $\mathfrak{n}\circledast_{kk}\mathfrak{0}$ but not with
	$K_{n,2}=\mathfrak{n}\smallsetminus\{n\}\circledast_{kk}\mathfrak{1}$. In the
	simplest case $n=4,$ the vertices $\frac{0|0|0|0}{2|3|4|1}$ and $\frac
	{0|0|0|0}{2|1|4|3}$ of $P_{4}$ are identified in $K_{5}$ but not when viewed
	as bipartitions $\frac{0|0|4|0}{2|3|0|1}$ and $\frac{0|0|4|0}{2|1|0|3}$ in
	$P_{4}=P_{3}\ast_{c}P_{1}$ in $K_{4,2}.$ Thus, $K_{4,2}=J_4\neq K_{5}$ while $|K_{4,2}|=|K_{5}|=3$. Verification is an exercise in bipartition matrix factorizations and left to the reader.
\end{remark}

Let $(C_{\ast}(PP_{n,m}),\tilde{\partial})$ and $(C_{\ast}(KK_{n+1,m+1}%
),\partial)$ be the cellular chain complexes of $PP_{n,m}$ and $KK_{n+1,m+1},$
respectively. There exist isomorphisms
\[
(C_{\ast}(\mathfrak{m}\circledast_{pp}\mathfrak{n}),\tilde{\partial
})\overset{\approx}{\longrightarrow}(C_{\ast}(PP_{n,m}),\tilde{\partial})
\]
and
\[
(C_{\ast}(\mathfrak{m}\circledast_{kk}\mathfrak{n}),\partial)\overset{\approx
}{\longrightarrow}(C_{\ast}(KK_{n+1,m+1}),\partial).
\]
These facts, together with Proposition \ref{dif-non} and Theorem
\ref{dif-unital}, immediately imply

\begin{theorem}
	\label{pph}There is a canonical isomorphism of chain complexes%
	\begin{equation}
		\tilde{\iota}_{\ast}:(\tilde{F}(\Theta)_{n,m},\tilde{\partial}%
		)\overset{\approx}{\longrightarrow}(C_{\ast}(PP_{n,m}),\tilde{\partial})
	\end{equation}

	extending the isomorphisms%
	\[
	\tilde{F}(\Theta)_{n+1,0}\overset{\approx}{\longrightarrow}C_{\ast}%
	(PP_{n,0})=C_{\ast}(P_{n})
	\]
	and
	\[
	\tilde{F}(\Theta)_{0,m+1}\overset{\approx}{\longrightarrow}C_{\ast}%
	(PP_{0,m})=C_{\ast}(P_{m}).
	\]
	
\end{theorem}

\begin{theorem}
	\label{kkh}There is a canonical isomorphism of chain complexes
	\begin{equation}
		\iota_{\ast}:({\mathcal{H}}_{\infty})_{n,m}\overset{\approx}{\longrightarrow
		}C_{\ast}(KK_{n,m})\label{inftymatrad}%
	\end{equation}
	extending the standard isomorphisms
	\[
	\mathcal{A}_{\infty}(n)=({\mathcal{H}}_{\infty})_{n,1}\overset{\approx
	}{\longrightarrow}C_{\ast}(KK_{n,1})=C_{\ast}(K_{n})
	\]
	and
	\[
	\mathcal{A}_{\infty}(m)=({\mathcal{H}}_{\infty})_{1,m}\overset{\approx
	}{\longrightarrow}C_{\ast}(KK_{1,m})=C_{\ast}(K_{m}).
	\]
	
\end{theorem}

The contractibility of $KK_{n,m}$ implies that the canonical projection
$\varrho:\left(  \mathcal{H}_{\infty},\partial\right)  \rightarrow\left(
{\mathcal{H}}^{pre},0\right)  $ is the minimal resolution mentioned in
the introduction, and an $A_{\infty}$-bialgebra structure on a DGM $H$ is
given by a morphism of matrads $\mathcal{H}_{\infty}\rightarrow U_{H};$ thus, an $A_\infty$-bialgebra is an algebra over $\mathcal{H}_\infty$.

Pictures of $KK_{n,m}$ with codimension 1 faces labeled by bipartitions in
$\mathfrak{m}\circledast_{kk}\mathfrak{n}$ for $2\leq m+n\leq6$ appear below.
These pictures first appeared in \cite{Markl2} and subsequently in \cite{SU4}.
\pagebreak

\noindent\underline{For $KK_{1,3}$:}

$
\begin{array}
	[c]{lllll}%
	1|2 & \leftrightarrow & \frac{0|0}{1|2} & \leftrightarrow & \gamma(\theta
	_{2}^{1}\,; \theta_{1}^{1}\theta_{2}^{1})\vspace{1mm}\newline\\
	2|1 & \leftrightarrow & \frac{0|0}{2|1} & \leftrightarrow & \gamma(\theta
	_{2}^{1}\,; \theta_{2}^{1}\theta^{1}_{1})\\
	&  &  &  &
\end{array}
$

\noindent\underline{For $KK_{3,1}$:}

$
\begin{array}
	[c]{lllll}%
	1|2 & \leftrightarrow & \frac{1|2}{0|0} & \leftrightarrow & \gamma(\theta
	_{1}^{2}\theta_{1}^{1} \,;\theta_{1}^{2})\vspace{1mm}\newline\\
	2|1 & \leftrightarrow & \frac{2|1}{0|0} & \leftrightarrow & \gamma(\theta
	_{1}^{1}\theta_{1}^{2} \,; \theta_{1}^{2})\\
	&  &  &  &
\end{array}
$

\vspace{0.1in}

\hspace{0.8in}\unitlength=1.00mm
\linethickness{0.4pt} \ifx\plotpoint\undefined\newsavebox{\plotpoint}\fi
\begin{picture}(66.625,9.75)(0,0)
	\put(9.25,6.625){\line(1,0){18.875}}
	\put(67.25,6.5){\line(1,0){18.875}}
	\put(9.25,6.5){\circle*{.75}}
	\put(9.25,6.625){\circle*{.75}}
	\put(28.125,6.625){\circle*{1}}
	\put(86.125,6.5){\circle*{1}}
	\put(9.25,6.625){\circle*{.75}}
	\put(9.25,6.625){\circle*{1}}
	\put(67.25,6.5){\circle*{1}}
	\put(9.125,9.75){\makebox(0,0)[cc]{${1|2}$}}
	\put(67.125,9.625){\makebox(0,0)[cc]{${1|2}$}}
	\put(28.125,9.5){\makebox(0,0)[cc]{$2|1$}}
	\put(86.125,9.375){\makebox(0,0)[cc]{$2|1$}}
	\put(9.25,    2.1){\makebox(0,0)[cc]{$\frac{0|0}{1|2}$}}
	\put(28.125,  2.1){\makebox(0,0)[cc]{$\frac{0|0}{2|1}$}}
	\put(67.25, 2.1){\makebox(0,0)[cc]{${\frac{1|2}{0|0}}$}}
	\put(86.125,2.1){\makebox(0,0)[cc]{$\frac{2|1}{0|0}$}}
\end{picture}
\begin{center}
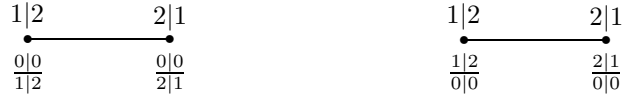
	Figure 9. The biassociahedra $KK_{1,3}=K_{3}$ and $KK_{3,1}=K_{3}$
	(intervals). \vspace{0.15in}
\end{center}
\vspace{-0.1in}

\noindent\underline{For $KK_{2,2}$:}

$
\begin{array}
	[c]{lllll}%
	1|2 & \leftrightarrow & \frac{0|1}{1|0} & \leftrightarrow & \gamma(\theta
	_{2}^{1}\theta_{2}^{1}\,;\theta_{1}^{2}\theta_{1}^{2})\vspace{1mm}\newline\\
	2|1 & \leftrightarrow & \frac{1|0}{0|1} & \leftrightarrow & \gamma(\theta
	_{1}^{2}\,;\theta_{2}^{1})\\
	&  &  &  &
\end{array}
$

\hspace{0.1in}
\unitlength=1.00mm\special{em:linewidth 0.4pt}
\linethickness{0.4pt}
\begin{picture}(86.33,13.33)
	\put(43.00,8.67){\line(1,0){43.00}} \put(43.33,8.67){\circle*{1.33}}
	\put(85.67,8.67){\circle*{1.33}} \put(43.33,13.33){\makebox(0,0)[cc]{$1|2$}}
	\put(43.33,03.33){\makebox(0,0)[cc]{$\frac{0|1}{1|0}$}}
	\put(85.67,13.00){\makebox(0,0)[cc]{$2|1$}}
	\put(85.67,03.00){\makebox(0,0)[cc]{$\frac{1|0}{0|1}$}}
\end{picture}

\begin{center}
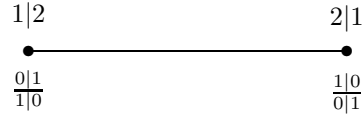
	Figure 10. The biassociahedron $KK_{2,2}$ (an interval).\vspace{0.1in}
\end{center}
\bigskip

\noindent\underline{For $KK_{3,2}$:}

$
\begin{array}
	[c]{lllll}
	1|23 & \leftrightarrow & \frac{0|12}{1|0} &  \leftrightarrow & \gamma(\theta_{2}^{1}\theta_{2}
	^{1}\theta_{2}^{1}\,;\gamma(\theta_{1}^{2}{\theta_{1}^{1}};\theta_{1}
	^{2})\theta_{1}^{3}+\theta_{1}^{3}\gamma({\theta_{1}^{1}}\theta_{1}^{2}
	;\theta_{1}^{2}))\vspace{1mm}\newline\\
	13|2 & \leftrightarrow & \frac{2|1}{1|0} &  \leftrightarrow& \gamma(\theta_{2}^{1}\theta_{2}
	^{2}\,;\theta_{1}^{2}\theta_{1}^{2})\vspace{1mm}\newline\\
	3|12 & \leftrightarrow & \frac{2|1}{0|1}&  \leftrightarrow& \gamma(\theta_{1}^{1}\theta_{1}
	^{2}\,;\theta_{2}^{2})\vspace{1mm}\newline\\
	12|3 & \leftrightarrow & \frac{1|2}{1|0} &  \leftrightarrow& \gamma(\theta_{2}^{2}\theta_{2}
	^{1}\,;\theta_{1}^{2}\theta_{1}^{2})\vspace{1mm}\newline\\
	2|13 & \leftrightarrow & \frac{1|2}{0|1} &  \leftrightarrow & \gamma(\theta_{1}^{2}\theta_{1}
	^{1}\,;\theta_{2}^{2})\vspace{1mm}\newline\\
	23|1 & \leftrightarrow & \frac{12|0}{0|1}&  \leftrightarrow & \gamma(\theta_{1}^{3}\,;\theta_{2}
	^{1})\\
	&  &  &  &
\end{array}
$
\newline \noindent\underline{For $KK_{2,3}$:}

$
\begin{array}
	[c]{lllll}
	1|23 & \leftrightarrow & \frac{0|1}{1|2} &  \leftrightarrow & \gamma(\theta_{2}^{1}\theta_{2}
	^{1}\,;\theta_{1}^{2}\theta_{2}^{2})\vspace{1mm}\newline\\
	13|2 & \leftrightarrow & \frac{1|0}{1|2} &  \leftrightarrow & \gamma(\theta_{2}^{2}\,;\theta
	_{1}^{1}\theta_{2}^{1})\vspace{1mm}\newline\\
	3|12 & \leftrightarrow & \frac{1|0}{0|12} &  \leftrightarrow& \gamma(\theta_{1}^{2}\,;\theta_{3}
	^{1})\vspace{1mm}\newline\\
	12|3 & \leftrightarrow & \frac{0|1}{12|0} &  \leftrightarrow & \gamma(\gamma(\theta_{2}
	^{1}\,;\theta_{1}^{1}\theta_{2}^{1})\theta_{3}^{1}+ \theta_{3}^{1}
	\gamma(\theta_{2}^{1}\,;\theta_{2}^{1}\theta_{1}^{1})\,;\theta_{1}^{2}
	\theta_{1}^{2}\theta_{1}^{2})\vspace{1mm}\newline\\
	2|13 & \leftrightarrow & \frac{0|1}{2|1}&  \leftrightarrow & \gamma(\theta_{2}^{1}\theta_{2}
	^{1}\,;\theta_{2}^{2}\theta_{1}^{2})\vspace{1mm}\newline\\
	23|1 & \leftrightarrow & \frac{1|0}{2|1} &  \leftrightarrow & \gamma(\theta_{2}^{2}\,;\theta
	_{2}^{1}\theta_{1}^{1})\\
	&  &  &  &
\end{array}
$
\pagebreak
\
\vspace*{0.2in}

\hspace*{-0.5in}
\unitlength=1.00mm\special{em:linewidth 0.4pt}
\linethickness{0.4pt}
\begin{picture}(88.33,43.67)
	\put(6.67,40.67){\line(1,0){48.33}}
	\put(55.00,40.67){\line(0,-1){31.67}}
	\put(55.00,9.00){\line(-1,0){48.33}}
	\put(6.67,9.00){\line(0,1){31.67}}
	\put(6.67,40.67){\circle*{1.33}}
	\put(6.67,25.00){\circle*{1.33}}
	\put(6.67,9.00){\circle*{1.33}}
	\put(55.00,40.67){\circle*{1.33}}
	\put(55.00,9.00){\circle*{1.33}}
	\put(55.00,25.00){\circle*{1.33}}
	\put(6.67,17.00){\circle*{0.67}}
	\put(1.33,33.33){\makebox(0,0)[cc]{$13|2$}}
	\put(10.33,33.33){\makebox(0,0)[cc]{$\frac{2|1}{1|0}$}}
	\put(1.33,17.00){\makebox(0,0)[cc]{$1|23$}}
	\put(10.80,17.00){\makebox(0,0)[cc]{$\frac{0|12}{1|0}$}}
	\put(29.00,5.00){\makebox(0,0)[cc]{$12|3$}}
	\put(29.00,13.00){\makebox(0,0)[cc]{$\frac{1|2}{1|0}$}}
	\put(30.00,43.67){\makebox(0,0)[cc]{$3|12$}}
	\put(30.00,35.67){\makebox(0,0)[cc]{$\frac{2|1}{0|1}$}}
	\put(60.0,17.00){\makebox(0,0)[cc]{$2|13$}}
	\put(51.0,17.00){\makebox(0,0)[cc]{$\frac{1|2}{0|1}$}}
	\put(60.0,32.67){\makebox(0,0)[cc]{$23|1$}}
	\put(51.0,32.67){\makebox(0,0)[cc]{$\frac{12|0}{0|1}$}}
	\put(80.67,40.67){\line(1,0){48.33}}
	\put(129.00,40.67){\line(0,-1){31.67}}
	\put(129.00,9.00){\line(-1,0){48.33}}
	\put(80.67,9.00){\line(0,1){31.67}}
	\put(80.67,40.67){\circle*{1.33}}
	\put(80.67,25.00){\circle*{1.33}}
	\put(80.67,9.00){\circle*{1.33}}
	\put(129.00,40.67){\circle*{1.33}}
	\put(129.00,9.00){\circle*{1.33}}
	\put(129.00,25.00){\circle*{1.33}}
	\put(75.5,33.33){\makebox(0,0)[cc]{$13|2$}}
	\put(84.33,33.33){\makebox(0,0)[cc]{$\frac{1|0}{1|2}$}}
	\put(75.5,17.00){\makebox(0,0)[cc]{$1|23$}}
	\put(84.80,17.00){\makebox(0,0)[cc]{$\frac{0|1}{1|2}$}}
	\put(104.50,5.00){\makebox(0,0)[cc]{$12|3$}}
	\put(104.50,13.00){\makebox(0,0)[cc]{$\frac{0|1}{12|0}$}}
	\put(104.00,43.67){\makebox(0,0)[cc]{$3|12$}}
	\put(104.00,35.67){\makebox(0,0)[cc]{$\frac{1|0}{0|12}$}}
	\put(134.0,17.00){\makebox(0,0)[cc]{$2|13$}}
	\put(125.33,17.00){\makebox(0,0)[cc]{$\frac{0|1}{2|1}$}}
	\put(134.0,32.67){\makebox(0,0)[cc]{$23|1$}}
	\put(126.33,32.67){\makebox(0,0)[cc]{$\frac{1|0}{2|1}$}}
	\put(104.33,9.00){\circle*{0.67}}
\end{picture}
\vspace{-.2in}
\begin{center}
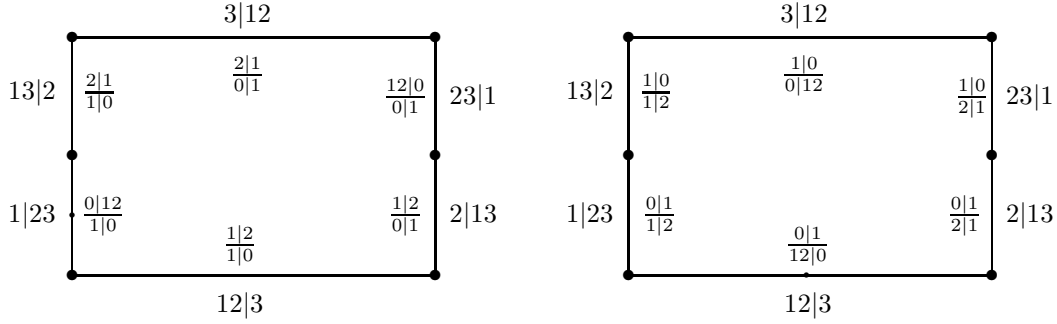
	Figure 11. The biassociahedra $KK_{3,2}$ and $KK_{2,3}$ (heptagons).
\end{center}

\noindent\underline{For $KK_{3,3}$:}

$
\begin{array}
	[c]{lllll}
	1|234 & \leftrightarrow & \frac{0|12}{1|2} & \leftrightarrow & \gamma(\theta_{2}^{1}\theta
	_{2}^{1}\theta_{2}^{1}\,; \gamma(\theta_{1}^{2}\theta_{1}^{1}\,;\theta_{1}
	^{2}) \theta_{2}^{3} +a+b), \text{ where}\\
	&  &  &  & \ \ \ \ a+b=\theta_{1}^{3} \gamma(\theta_{2}^{1}\theta_{2}
	^{2}\,;\theta_{1}^{2}\theta_{1}^{2})+ \theta_{1}^{3} \gamma(\theta_{1}
	^{1}\theta_{1}^{2} \,;\theta_{2}^{2}) \vspace{1mm}\newline\\
	123|4 & \leftrightarrow & \frac{1|2}{12|0}  & \leftrightarrow & \gamma(c+d+\theta_{3}^{2}
	\gamma(\theta_{2}^{1}\,;\theta_{2}^{1}\theta_{1}^{1})\,;\theta_{1}^{2}
	\theta_{1}^{2}\theta_{1}^{2}),\text{ where}\\
	&  &  &  & \ \ \ \ c+d=\gamma(\theta_{2}^{1}\theta_{2}^{1}\,; \theta_{1}
	^{2}\theta_{2}^{2})\theta_{3}^{1}+\gamma(\theta_{2}^{2}\,;\theta_{1}^{1}
	\theta_{2}^{1})\theta_{3}^{1}\vspace{1mm}\newline\\
	2|134 & \leftrightarrow & \frac{0|12}{2|1}  & \leftrightarrow& \gamma(\theta_{2}^{1}\theta_{2}^{1}\theta_{2}^{1}\,; \theta_{2}^{3}\gamma(\theta_{1}^{1}\theta_{1}
	^{2}\,;\theta_{1}^{2}) +e+f),\text{ where}\\
	&  &  &  & \ \ \ \ e+f= \gamma(\theta_{2}^{2}\theta_{2}^{1}\,;\theta_{1}
	^{2}\theta_{1}^{2})\theta_{1}^{3} + \gamma(\theta_{1}^{2}\theta_{1}
	^{1}\,;\theta_{2}^{2}) \theta_{1}^{3}
\end{array}
$	

$
\begin{array}
	[c]{lllll}
	124|3 & \leftrightarrow & \frac{2|1}{12|0}  & \leftrightarrow& \gamma(g+h+\gamma(\theta_{2}
	^{1}\,;\theta_{1}^{1}\theta_{2}^{1})\theta_{3}^{2}\,;\theta_{1}^{2}\theta
	_{1}^{2}\theta_{1}^{2}),\text{ where}\\
	&  &  &  & \ \ \ \ g+h=\theta_{3}^{1}\gamma(\theta_{2}^{1}\theta_{2}
	^{1}\,;\theta_{2}^{2}\theta_{1}^{2})+\theta_{3}^{1}\gamma(\theta_{2}
	^{2}\,;\theta_{2}^{1}\theta_{1}^{1})
	\vspace{1mm}\newline\\
	134|2 & \leftrightarrow & \frac{12|0}{1|2}  & \leftrightarrow& \gamma(\theta_{2}^{3}\,;\theta
	_{1}^{1}\theta_{2}^{1})\vspace{1mm}\newline\\
	234|1 & \leftrightarrow & \frac{12|0}{2|1}  & \leftrightarrow & \gamma(\theta_{2}^{3}\,;\theta
	_{2}^{1}\theta_{1}^{1})\vspace{1mm}\newline\\
	3|124 & \leftrightarrow & \frac{1|2}{0|12} & \leftrightarrow& \gamma(\theta_{1}^{2}\theta_{1}
	^{1}\,;\theta_{3}^{2})\vspace{1mm}\newline\\
	4|123 & \leftrightarrow & \frac{2|1}{0|12}  & \leftrightarrow & \gamma(\theta_{1}^{1}\theta_{1}
	^{2}\,;\theta_{3}^{2})\vspace{1mm}\newline\\
	23|14 & \leftrightarrow & \frac{1|2}{2|1}  & \leftrightarrow & \gamma(\theta_{2}^{2}\theta_{2}
	^{1}\,;\theta_{2}^{2}\theta_{1}^{2})\vspace{1mm}\newline\\
	14|23 & \leftrightarrow & \frac{2|1}{1|2}  & \leftrightarrow& \gamma(\theta_{2}^{1}\theta_{2}
	^{2}\,;\theta_{1}^{2}\theta_{2}^{2})\vspace{1mm}\newline\\
	24|13 & \leftrightarrow & \frac{2|1}{2|1}  &\leftrightarrow & \gamma(\theta_{2}^{1}\theta_{2}
	^{2}\,;\theta_{2}^{2}\theta_{1}^{2})\vspace{1mm}\newline\\
	13|24 & \leftrightarrow & \frac{1|2}{1|2}  & \leftrightarrow & \gamma(\theta_{2}^{2}\theta_{2}
	^{1}\,;\theta_{1}^{2}\theta_{2}^{2})\vspace{1mm}\newline\\
\end{array}
$	

$
\begin{array}
	[c]{lllll}
	34|12 & \leftrightarrow & \frac{12|0}{0|12}  & \leftrightarrow & \gamma(\theta_{1}^{3}\,;\theta
	_{3}^{1})\vspace{1mm}\newline\\
	12|34 & \leftrightarrow & \frac{0|12}{12|0}  & \leftrightarrow & \gamma\lbrack\theta_{3}^{1}
	\gamma(\theta_{2}^{1}\,;\theta_{2}^{1}\theta_{1}^{1})\gamma(\theta_{2}
	^{1}\,;\theta_{2}^{1}\theta_{1}^{1})+\gamma(\theta_{2}^{1}\,; \theta_{1}
	^{1}\theta_{2}^{1})\theta_{3}^{1}\gamma(\theta_{2}^{1}\,;\theta_{2}^{1}
	\theta_{1}^{1})\\
	&  &  &  & \ \ \ \ \ \ +\gamma(\theta_{2}^{1}\,;\theta_{2}^{1}\theta_{1}
	^{1})\gamma(\theta_{2}^{1}\,;\theta_{2}^{1}\theta_{1}^{1})\theta_{3}
	^{1}\,;\vspace{1mm}\newline\\
	&  &  &  & \ \ \ \theta_{1}^{3}\gamma(\theta_{1}^{1}\theta_{1}^{2}
	\,;\theta_{1}^{2})\gamma(\theta_{1}^{1}\theta_{1}^{2}\,;\theta_{1}^{2}
	)+\gamma(\theta_{1}^{2}\theta_{1}^{1}\,;\theta_{1}^{2})\theta_{1}^{3}
	\gamma(\theta_{1}^{1}\theta_{1}^{2}\,;\theta_{1}^{2})+\\
	&  &  &  & \ \ \ \ \ \ +\gamma(\theta_{1}^{2}\theta_{1}^{1}\,;\theta_{1}
	^{2})\gamma(\theta_{1}^{2} \theta_{1}^{1}\,;\theta_{1}^{2})\theta_{1}^{3}]\\
	&  &  &  &
\end{array}
$

\
\vspace{0.2in}
\hspace*{-.5in}
\unitlength 1mm 
\linethickness{0.4pt}
\ifx\plotpoint\undefined\newsavebox{\plotpoint}\fi 
\begin{picture}(98.5,76.995)(0,0)
	\put(67,5.33){\line(1,0){30.33}}
	\put(97.33,15.33){\line(-1,0){30.33}}
	\put(97.33,25.33){\line(-1,0){30.33}}
	\put(97.33,35.33){\line(-1,0){30.33}}
	\put(77,5.33){\line(0,1){30}}
	\put(67.33,29.33){\line(1,0){4.67}}
	\put(72,35){\line(0,-1){5.67}}
	\put(97.67,55.33){\line(0,-1){50}}
	\put(66.67,5.33){\line(0,1){49.67}}
	\put(66.67,55){\line(1,0){31}}
	\put(92.33,40){\line(0,1){0}}
	\put(97.67,55){\line(-6,5){25.67}}
	\put(66.67,5.33){\line(-6,5){28.67}}
	\put(38.33,29.33){\line(0,1){47}}
	\put(66.67,55){\line(-4,3){28.33}}
	\put(38.33,76.33){\line(1,0){33.67}}
	\put(72,55.33){\line(0,1){21}}
	\put(79.33,70.66){\line(0,-1){15.33}}
	\put(79.33,54.66){\line(0,-1){19}}
	\put(79.33,35){\line(0,-1){9.33}}
	\put(79.33,25){\line(0,-1){2.33}}
	\put(66.67,45){\line(-6,5){20.67}}
	\put(46,45.33){\line(-1,1){7.67}}
	\put(38.33,29.33){\line(1,0){7.33}}
	\put(38.33,70.66){\line(1,0){6.33}}
	\put(47.67,70.66){\line(1,0){24.33}}
	\put(46,22.66){\line(0,1){48}}
	\put(46,70.66){\circle*{.67}}
	\put(46,62.33){\circle*{1.33}}
	\put(46,70.66){\circle*{1.33}}
	\put(38.33,76.33){\circle*{1.33}}
	\put(38.33,70.66){\circle*{1.33}}
	\put(72,76.33){\circle*{1.33}}
	\put(72,70.66){\circle*{1.33}}
	\put(79.33,70.33){\circle*{1.33}}
	\put(79.33,61){\circle*{1.33}}
	\put(97.67,55){\circle*{1.33}}
	\put(97.67,35.33){\circle*{1.33}}
	\put(66.67,55){\circle*{1.33}}
	\put(66.67,45.33){\circle*{1.33}}
	\put(72,51.33){\circle*{1.33}}
	\put(79.33,44.66){\circle*{1.33}}
	\put(72,29.33){\circle*{.67}}
	\put(72,29.33){\circle*{1.33}}
	\put(79.33,23){\circle*{1.33}}
	\put(66.67,35.33){\circle*{1.33}}
	\put(66.67,5.33){\circle*{1.33}}
	\put(97.67,5.33){\circle*{1.33}}
	\put(46,45.66){\circle*{1.33}}
	\put(38.33,52.66){\circle*{1.33}}
	\put(38.33,29.33){\circle*{1.33}}
	\put(46,22.66){\circle*{1.33}}
	\put(66.67,25.33){\circle*{.67}}
	\put(66.67,15.33){\circle*{.67}}
	\put(77,35.33){\circle*{0}}
	\put(77,25.33){\circle*{.67}}
	\put(77,15.33){\circle*{0}}
	\put(77,5.33){\circle*{.67}}
	\put(87,25.33){\circle*{.67}}
	\put(87,15.33){\circle*{.67}}
	\put(87,5.33){\circle*{.67}}
	\put(97.67,25.33){\circle*{.67}}
	\put(97.67,15.33){\circle*{.67}}
	\put(79.33,52){\circle*{.67}}
	\put(77,35.33){\circle*{.67}}
	\put(77,15.33){\circle*{.67}}
	\put(46,53){\circle*{.67}}
	\put(88,14.33){\line(1,-1){9.67}}
	\put(86.33,16){\line(-1,1){7}}
	\put(79.33,23){\line(-6,5){7.33}}
	\put(87,5.75){\line(0,1){29.75}}
	\multiput(79.5,61)(.044207317,-.033536585){164}{\line(1,0){.044207317}}
	\put(87.5,24.5){\line(0,-1){.25}}
	\multiput(87.75,24.75)(.0354545455,-.0336363636){275}{\line(1,0){.0354545455}}
	\multiput(86.25,26)(-.03353659,.03963415){82}{\line(0,1){.03963415}}
	\multiput(79.5,23)(.03289474,.04605263){38}{\line(0,1){.04605263}}
	\multiput(81.5,25.75)(.03358209,.04850746){67}{\line(0,1){.04850746}}
	\multiput(72,21.5)(.1956522,.0326087){23}{\line(1,0){.1956522}}
	\multiput(77.5,22.5)(.125,.03125){8}{\line(1,0){.125}}
	\multiput(68,29.25)(.03365385,-.0625){52}{\line(0,-1){.0625}}
	\multiput(70.5,24.75)(.03365385,-.05769231){52}{\line(0,-1){.05769231}}
	\multiput(72.25,21.25)(.038461538,-.033653846){104}{\line(1,0){.038461538}}
	\multiput(77.75,17)(.0416667,-.0333333){30}{\line(1,0){.0416667}}
	\multiput(80.75,14.5)(.040780142,-.033687943){141}{\line(1,0){.040780142}}
	\multiput(45.75,53.25)(.0337078652,-.0453451043){623}{\line(0,-1){.0453451043}}
	\multiput(56,39.5)(.060393258,.033707865){178}{\line(1,0){.060393258}}
	\put(46.5,29.25){\line(1,0){16.25}}
	\put(64.25,29.25){\line(1,0){2}}
	\multiput(87.5,9.25)(.046218487,-.033613445){119}{\line(1,0){.046218487}}
	\put(76.75,35.5){\line(0,1){19.75}}
	\put(66.75,45.5){\line(1,0){10}}
	\put(72,54.5){\line(0,-1){8.25}}
	\put(72,45){\line(0,-1){9.25}}
	\multiput(71.75,51.25)(.035447761,-.03358209){134}{\line(1,0){.035447761}}
	\multiput(77.25,46.25)(.03365385,-.03365385){52}{\line(0,-1){.03365385}}
	\multiput(79.25,52.25)(.033482143,-.147321429){112}{\line(0,-1){.147321429}}
	\multiput(83.25,34.75)(.0333333,-.3666667){15}{\line(0,-1){.3666667}}
	\put(62.75,37.25){\makebox(0,0)[cc]{$a$}}
	\put(55.25,47.75){\makebox(0,0)[cc]{$b$}}
	\put(69.5,40){\makebox(0,0)[cc]{$c$}}
	\put(69.5,48.25){\makebox(0,0)[cc]{$d$}}
	\put(80.5,18.5){\makebox(0,0)[cc]{$g$}}
	\put(72.0,27.0){\makebox(0,0)[cc]{$h$}}
	\put(94.25,11.5){\makebox(0,0)[cc]{$e$}}
	\put(81.5,30.5){\makebox(0,0)[cc]{$f$}}
	\put(87,35.25){\circle*{1}}
	\put(76.75,45.5){\circle*{1}}
	\put(76.75,55){\circle*{1}}
	\put(55.75,39.75){\circle*{1}}
	\put(72.25,21.5){\circle*{1}}
	\put(83.75,29.25){\circle*{1}}
	\put(68.25,29.25){\circle*{1}}
	\put(92.75,5.5){\circle*{1}}
	\put(88,54.5){\line(5,-4){10}}
	\put(97.75,46.75){\circle*{1.5}}
\end{picture}
\begin{center}
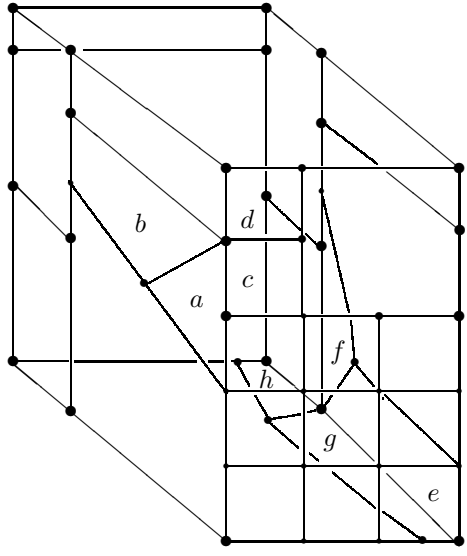
	Figure 12. The biassociahedron $KK_{3,3}$ as a subdivision of $P_{4}$.
\end{center}

\noindent\underline{For $KK_{4,2}$:}

$
\begin{array}
	[c]{lllll}
	1|234 & \leftrightarrow & \frac{0|123}{1|0}  & \leftrightarrow& \gamma(\theta_{2}^{1}\theta
	_{2}^{1}\theta_{2}^{1}\theta_{2}^{1}\,;a+b-c+d+e+f),\text{ where}\\
	&  &  &  & \ \ \ a=\theta_{1}^{4}\gamma(\theta_{1}^{1}\gamma(\theta_{1}
	^{1}\theta_{1}^{2}\,;\theta_{1}^{2})\,;\theta_{1}^{2})\\
	&  &  &  & \ \ \ b=\gamma(\gamma(\theta_{1}^{2}\theta_{1}^{1}\,;\theta_{1}
	^{2})\theta_{1}^{1}\,;\theta_{1}^{2})\theta_{1}^{4}\\
	&  &  &  & \ \ \ c=\gamma(\theta_{1}^{2}\theta_{1}^{1}\theta_{1}^{1}
	\,;\theta_{1}^{3})\,\gamma(\theta_{1}^{1}\theta_{1}^{1}\theta_{1}^{2}
	\,;\theta_{1}^{3})\\
	&  &  &  & \ \ \ d=\gamma(\theta_{1}^{3}\theta_{1}^{1}\,;\theta_{1}
	^{2})\,\gamma(\theta_{1}^{1}\theta_{1}^{2}\theta_{1}^{1}\,;\theta_{1}^{3})\\
	&  &  &  & \ \ \ e=\gamma(\theta_{1}^{3}\theta_{1}^{1}\,;\theta_{1}
	^{2})\,\gamma(\theta_{1}^{1}\theta_{1}^{3}\,;\theta_{1}^{2})\\
	&  &  &  & \ \ \ f=\gamma(\theta_{1}^{1}\theta_{1}^{2}\theta_{1}^{1}
	\,;\theta_{1}^{3})\,\gamma(\theta_{1}^{1}\theta_{1}^{3}\,;\theta_{1}^{2})\\
	\\
	123|4 & \leftrightarrow & \frac{12|3}{1|0} & \leftrightarrow & \gamma(\theta_{2}^{3}\theta_{2}
	^{1}\,;\theta_{1}^{2}\theta_{1}^{2})\vspace{1mm}\newline\\
	2|134 & \leftrightarrow & \frac{1|23}{0|1}  & \leftrightarrow & \gamma(\theta_{1}^{2}\theta_{1}
	^{1}\theta_{1}^{1}\,;\theta_{2}^{3})\vspace{1mm}\newline\\
	124|3 & \leftrightarrow & \frac{13|2}{1|0}  & \leftrightarrow & \gamma(\theta_{2}^{2}\theta_{2}
	^{2}\,;\theta_{1}^{2}\theta_{1}^{2})\vspace{1mm}\newline\\
	134|2 & \leftrightarrow & \frac{23|1}{1|0}  & \leftrightarrow & \gamma(\theta_{2}^{1}\theta_{2}
	^{3}\,;\theta_{1}^{2}\theta_{1}^{2})\vspace{1mm}\newline\\
	234|1 & \leftrightarrow & \frac{123|0}{0|1}  & \leftrightarrow & \gamma(\theta_{1}^{4}\,;\theta
	_{2}^{1})\vspace{1mm}\newline\\
	3|124 & \leftrightarrow & \frac{2|13}{0|1}  & \leftrightarrow & \gamma(\theta_{1}^{1}\theta_{1}
	^{2}\theta_{1}^{1}\,;\theta_{2}^{3})\vspace{1mm}\newline\\
	4|123 & \leftrightarrow & \frac{3|12}{0|1} & \leftrightarrow & \gamma(\theta_{1}^{1}\theta_{1}
	^{1}\theta_{1}^{2}\,;\theta_{2}^{3})\vspace{1mm}\newline\\
\end{array}
$	

$
\begin{array}
	[c]{lllll}
	13|24 & \leftrightarrow & \frac{2|13}{1|0}  & \leftrightarrow& \gamma\lbrack\theta_{2}^{1}
	\theta_{2}^{2}\theta_{2}^{1}\,;\theta_{1}^{3}\gamma(\theta_{1}^{1}\theta
	_{1}^{2}\,;\theta_{1}^{2})+\gamma(\theta_{1}^{2}\theta_{1}^{1}\,;\theta
	_{1}^{2})\theta_{1}^{3}]\vspace{1mm}\newline\\
	14|23 & \leftrightarrow & \frac{3|12}{1|0} &\leftrightarrow & \gamma\lbrack\theta_{2}^{1}
	\theta_{2}^{1}\theta_{2}^{2}\,;\theta_{1}^{3}\gamma(\theta_{1}^{1}\theta
	_{1}^{2}\,;\theta_{1}^{2})+\gamma(\theta_{1}^{2}\theta_{1}^{1}\,;\theta
	_{1}^{2})\theta_{1}^{3}]\vspace{1mm}\newline\\
	23|14 & \leftrightarrow & \frac{12|3}{0|1}  &\leftrightarrow & \gamma(\theta_{1}^{3}\theta_{1}
	^{1}\,;\theta_{2}^{2})\vspace{1mm}\newline\\
	34|12 & \leftrightarrow & \frac{23|1}{0|1}
	& = & \gamma(\theta_{1}^{1}\theta_{1}
	^{3}\,;\theta_{2}^{2})\vspace{1mm}\newline\\
	12|34 & \leftrightarrow & \frac{1|23}{1|0}  & \leftrightarrow& \gamma\lbrack\theta_{2}^{2}
	\theta_{2}^{1}\theta_{2}^{1}\,;\theta_{1}^{3}\gamma(\theta_{1}^{1}\theta
	_{1}^{2}\,;\theta_{1}^{2})+\gamma(\theta_{1}^{2}\theta_{1}^{1}\,;\theta
	_{1}^{2})\theta_{1}^{3}]\vspace{1mm}\newline\\
\end{array}
$
\bigskip

\unitlength=0.75mm\special{em:linewidth 0.4pt} \linethickness{0.4pt}
\begin{picture}(98.33,76.99)
	\put(67.00,5.33){\line(1,0){30.33}} \put(67.33,29.33){\line(1,0){4.67}}
	\put(97.67,55.33){\line(0,-1){50.00}} \put(66.67,5.33){\line(0,1){49.67}}
	\put(66.67,55.00){\line(1,0){31.00}} \put(97.67,55.00){\line(-6,5){25.67}}
	\put(66.67,5.33){\line(-6,5){28.67}} \put(38.33,29.33){\line(0,1){47.00}}
	\put(66.67,55.00){\line(-4,3){28.33}} \put(38.33,76.33){\line(1,0){33.67}}
	\put(72.00,55.33){\line(0,1){21.00}} \put(79.33,70.66){\line(0,-1){15.33}}
	\put(79.33,44.33){\line(-1,1){7.33}} \put(46.00,45.33){\line(-1,1){7.67}}
	\put(38.33,29.33){\line(1,0){7.33}} \put(46.33,29.33){\line(1,0){20.00}}
	\put(38.33,70.66){\line(1,0){6.33}} \put(47.67,70.66){\line(1,0){24.33}}
	\put(46.00,22.66){\line(0,1){48.00}} \put(46.00,70.66){\circle*{0.67}}
	\put(46.00,70.66){\circle*{1.33}} \put(38.33,76.33){\circle*{1.33}}
	\put(38.33,70.66){\circle*{1.33}} \put(72.00,76.33){\circle*{1.33}}
	\put(72.00,70.66){\circle*{1.33}} \put(79.33,70.33){\circle*{1.33}}
	\put(97.67,55.00){\circle*{1.33}} \put(66.67,55.00){\circle*{1.33}}
	\put(72.00,51.33){\circle*{1.33}} \put(79.33,44.66){\circle*{1.33}}
	\put(72.00,29.33){\circle*{0.67}} \put(72.00,29.33){\circle*{1.33}}
	\put(79.33,23.00){\circle*{1.33}} \put(66.67,5.33){\circle*{1.33}}
	\put(97.67,5.33){\circle*{1.33}} \put(46.00,45.66){\circle*{1.33}}
	\put(38.33,52.66){\circle*{1.33}} \put(38.33,29.33){\circle*{1.33}}
	\put(46.00,22.66){\circle*{1.33}} \put(66.67,20.33){\line(1,0){31.00}}
	\put(56.00,14.33){\line(0,1){48.67}} \put(66.67,20.33){\line(-4,3){20.67}}
	\put(82.33,20.00){\line(1,-1){15.33}} \put(81.67,20.67){\line(-6,5){9.67}}
	\put(56.00,14.00){\circle*{0.67}} \put(56.00,20.67){\circle*{0.67}}
	\put(56.00,28.00){\circle*{0.67}} \put(56.00,62.67){\circle*{0.67}}
	\put(46.00,35.67){\line(-6,5){7.67}} \put(46.00,54.33){\line(1,-4){5.67}}
	\put(97.67,20.33){\circle*{0.67}} \put(66.67,20.33){\circle*{0.67}}
	\put(38.33,41.67){\circle*{0.67}} \put(51.33,32.00){\circle*{0.67}}
	\put(46.00,36.00){\line(2,-3){10.00}} \put(46.00,36.00){\circle*{0.67}}
	\put(46.00,53.67){\circle*{0.67}} \put(66.67,42.67){\line(1,0){31.00}}
	\put(66.67,42.67){\line(-6,5){20.67}} \put(46.00,59.67){\circle*{1.33}}
	\put(66.67,42.67){\circle*{1.33}} \put(97.67,42.67){\circle*{1.33}}
	\put(56.00,51.67){\circle*{0.67}} \put(79.33,43.00){\line(0,1){11.67}}
	\put(79.33,22.67){\line(0,1){19.67}} \put(72.00,29.33){\line(0,1){13.00}}
	\put(72.00,43.00){\line(0,1){11.67}}
\end{picture}

\begin{center}
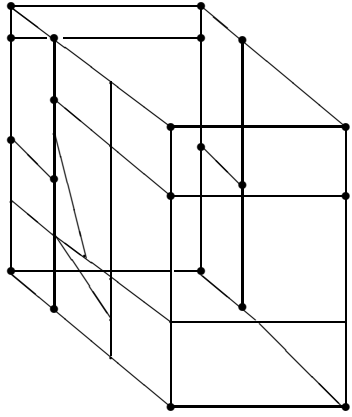
	Figure 13. The biassociahedron $KK_{4,2}$ as a subdivision of $J_{4}
	=K_{4,2}=\vartheta_{3,1}(P_{4})$.
\end{center}

\vspace{0.3in}
\unitlength=0.80mm \special{em:linewidth 0.4pt} \linethickness{0.4pt}
\begin{picture}(98.33,76.99)
	\put(67.00,5.33){\line(1,0){30.33}} \put(97.67,55.33){\line(0,-1){50.00}}
	\put(66.67,5.33){\line(0,1){49.67}} \put(66.67,55.00){\line(1,0){31.00}}
	\put(66.67,5.33){\line(-6,5){28.67}} \put(38.33,29.33){\line(0,1){47.00}}
	\put(66.67,55.00){\line(-4,3){28.33}} \put(38.33,76.33){\line(1,0){33.67}}
	\put(72.00,55.33){\line(0,1){21.00}} \put(38.33,76.33){\circle*{1.33}}
	\put(72.00,76.33){\circle*{1.33}} \put(97.67,55.00){\circle*{1.33}}
	\put(66.67,55.00){\circle*{1.33}} \put(66.67,5.33){\circle*{1.33}}
	\put(97.67,5.33){\circle*{1.33}} \put(38.33,29.33){\circle*{1.33}}
	\put(57.33,62.33){\line(0,-1){49.33}} \put(38.33,29.00){\line(1,0){18.33}}
	\put(58.00,29.00){\line(1,0){7.67}} \put(67.33,29.00){\line(1,0){4.67}}
	\put(72.00,29.00){\line(0,1){25.33}} \put(66.67,19.00){\line(1,0){31.00}}
	\put(72.33,29.00){\line(5,-4){12.00}} \put(85.67,18.33){\line(1,-1){12.33}}
	\put(57.33,52.67){\line(4,-3){9.33}} \put(59.67,61.00){\line(1,0){12.33}}
	\put(56.67,61.00){\line(-1,0){18.33}} \put(57.33,62.00){\circle*{1.33}}
	\put(57.33,52.33){\circle*{1.33}} \put(66.67,46.00){\circle*{1.33}}
	\put(97.67,19.00){\circle*{1.33}} \put(66.67,19.00){\circle*{1.33}}
	\put(57.33,13.33){\circle*{1.33}} \put(72.00,29.00){\circle*{1.33}}
	\put(38.33,61.00){\circle*{0.67}} \put(72.00,61.00){\circle*{0.67}}
	\put(72.00,61.00){\circle*{0.67}} \put(82.67,20.33){\line(-1,0){15.33}} \
	\put(66.00,20.33){\line(-1,0){8.00}} \put(56.67,20.33){\line(-1,0){7.67}}
	\put(59.33,29.00){\line(5,-4){6.67}} \put(67.33,22.33){\line(3,-2){4.67}}
	\put(57.33,33.00){\line(-3,2){19.00}} \ \put(49.00,20.33){\line(0,1){18.33}}
	\put(86.33,5.33){\line(0,1){13.67}} \ \put(62.00,9.00){\line(1,5){2.33}}
	\put(49.00,20.33){\line(2,1){8.00}} \put(58.00,24.33){\line(2,1){4.67}}
	\put(88.67,34.33){\line(-3,2){16.67}} \ \put(88.33,34.33){\line(1,-3){5.00}} \ \
	\put(93.67,18.67){\line(1,-3){4.33}} \put(83.33,20.00){\line(0,1){18.00}} \ \
	\put(57.33,33.67){\circle*{1.33}} \put(38.33,45.33){\circle*{1.33}}
	\put(38.33,61.00){\circle*{1.33}} \put(72.00,61.00){\circle*{1.33}}
	\put(72.00,45.00){\circle*{1.33}} \put(88.33,34.33){\circle*{1.33}}
	\put(86.33,5.33){\circle*{0.67}} \put(49.00,20.33){\circle*{0.67}}
	\put(49.00,38.67){\circle*{0.67}} \put(83.33,37.67){\circle*{0.67}} \
	\put(83.33,20.33){\circle*{0.67}} \ \put(70.30,20.33){\circle*{0.67}}
	\put(64.33,20.33){\circle*{0.67}} \put(62.00,9.33){\circle*{0.67}}
	\put(62.33,26.67){\circle*{0.67}} \put(59.67,29.00){\circle*{0.67}}
	\put(72.33,76.00){\line(6,-5){25.33}} \put(86.33,5.33){\line(-1,1){13.33}}
	\put(86.33,19.00){\circle*{0.67}} \put(97.67,46.00){\circle*{1.33}}
	\put(88.33,34.33){\line(4,5){9.33}}
\end{picture}\vspace*{-0.1in}

\bigskip
\begin{center}
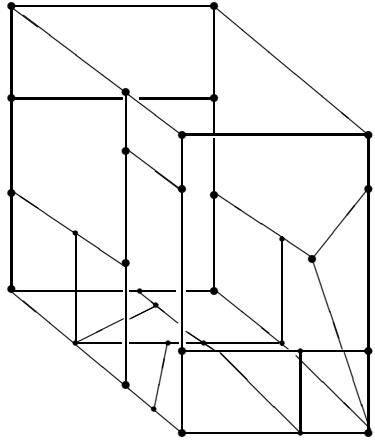
	Figure 14. The biassociahedron $KK_{2,4}$ as a subdivision of $J_{4}
	=K_{2,4}=\vartheta_{1,3}(P_{4})$.
\end{center}

\section{The $A_{\infty}$-Bialgebra Morphism Matrad}

\subsection{The Combinatorial Cylinder $\mathcal{ZJ}_{n+1,m+1}$}

Let $m, n\geq0.$ In this subsection we define the combinatorial cylinder $\mathcal{ZJ}_{n+1,m+1}$ on the reduced balanced framed join $\mathfrak{m}
\circledast_{kk}\mathfrak{n}$ as a quotient of the combinatorial
cylinder $\mathcal{ZP}_{n,m}$ on the balanced framed join $\mathfrak{m}%
\circledast_{pp}\mathfrak{n}.$

Given an integer $z>\max\{m,n\},$ define $\mathfrak{m}_{z}:=\mathfrak{m}%
\cup\{z\}$ and $\mathfrak{n}_{z}:=\mathfrak{n}\cup\{z\}$, and consider the disjoint union   $\left(  \mathfrak{m}_{z}\circledast\mathfrak{n}\right)
\sqcup\left(  \,\mathfrak{m}\circledast\mathfrak{n}_{z}\right)$. A framed element $c\in \left(  \mathfrak{m}_{z}\circledast\mathfrak{n}\right)
\sqcup\left(  \,\mathfrak{m}\circledast\mathfrak{n}_{z}\right)$ has the form $c=C_{1}\cdots C_{k-1}\cdot F\cdot C_{k+1}\cdots C_{r}$, where $F$ is the matrix factor containing $z$ and $z\in \mathbf{is}(F)$ if and only if $c\in  \mathfrak{m}_{z}\circledast\mathfrak{n}$.    A \emph{structure bipartition within} $c$ is a bipartition within $c$ of the form $\frac{B_{1}|\cdots|B_{r}}{A_{1}|\cdots|A_{r}}=
\mathbf{C}_{1}
\cdots\mathbf{C}_{k-1}\cdot \mathbf{F}\cdot \mathbf{C}_{k+1}\cdots\mathbf{C}_{r}$.

Let $ZF_{n,m}$ denote the subset of $\left(  \mathfrak{m}_{z}\circledast\mathfrak{n}\right)
\sqcup\left(  \,\mathfrak{m}\circledast\mathfrak{n}_{z}\right)  $ with the following property: If $c\in ZF_{n,m}$
and $\frac{B_{1}|\cdots|B_{r}}{A_{1}|\cdots|A_{r}}$ is a structure bipartition within $c$, then
$z$ lies in ${A_i}$ or ${B_i}$ but not in  $A_{r}$ or $B_{1}$ when $r>1.$
Given $c\in ZF_{n,m}$, let $c_{m,n}$ denote the framed element
obtained from $c$ by discarding $z\in A_k\cup B_k$ when $r=1$ or $1<k<r.$

\begin{definition}
	\label{cylinder}Given $m, n\geq 0$, choose an integer $z>\max\{m,n\}.$ The
	\textbf{combinatorial cylinder} \textbf{on} \textbf{the framed join}
	$\mathfrak{m}\,\circledast\,\mathfrak{n}$  is the set $\widetilde{\mathcal{ZF}
	}_{n,m}:=$
	\begin{itemize}
		\item $\{\frac{0}{z}\sim \frac{z}{0}\}$ when $m=n=0,$
		
		\item $\mathfrak{m}\ast_c\{z\}$ when $m>0$ and $n=0,$
		
		\item $\{z\}\ast_c\mathfrak{n}$ when $m=0$ and $n>0,$ and
		\item $ZF_{n,m} /\sim$ when $mn>0$, where $c^1\sim c^2$ if and only if $c^1_{m,n}=c^2_{m,n}$.
	\end{itemize}
	The \textbf{combinatorial cylinder on the prebalanced framed join}
	$\mathfrak{m}\,\widetilde{\circledast}_{pp}\,\mathfrak{n}$ is the set $\widetilde{\mathcal{ZP}}_{n,m}:=$
	
	\begin{itemize}
		
		\item $\widetilde{\mathcal{ZF}}_{n,0}$ for all $n\geq0,$
		
		\item $\widetilde{\mathcal{ZF}}_{0,m}$ for all $m\geq0,$ and
		
		\item
		$\{c\in \widetilde{\mathcal{ZF}}_{n,m}:
		c\in \left(m_{z}\widetilde{\circledast}_{pp}\mathfrak{n}\right)\sqcup
		\left(\mathfrak{m}\widetilde{\circledast}_{pp} n_{z}\right)\}$
		for all $mn>0.$

	\end{itemize}
	The \textbf{combinatorial cylinder on the reduced prebalanced framed join}
	$\mathfrak{m}\,\widetilde{\circledast}_{pp}\,\mathfrak{n}$ is the set
	\[
	\widetilde{\mathcal{ZJ}}_{n+1,m+1}:=\widetilde{\mathcal{ZP}}_{n,m}/\sim,
	\]
	where $c^{\prime}\sim c\in\widetilde{\mathcal{ZP}}_{n,m}$ if and only if $c$
	and $c^{\prime}$ have the same number of indecomposable matrix factors and
	corresponding indecomposable matrix factors differ only in the number of empty
	blocks in their entries. As in the definitions of $ \mathfrak{m}\circledast_{pp}\mathfrak{n}\subseteq \mathfrak{m} \widetilde\circledast_{pp}\mathfrak{n}$ and $\mathfrak{m}\circledast_{kk}\mathfrak{n}=
	\mathfrak{m}\circledast_{pp}\mathfrak{n}/\sim$, the \textbf{combinatorial cylinder on the reduced balanced framed join}
	$\mathfrak{m}\,{\circledast}_{pp}\,\mathfrak{n}$ is the set
	\[
	\mathcal{ZJ}_{n+1,m+1}:=\mathcal{ZP}_{n,m}/\sim.
	\]
	
\end{definition}

\noindent Then in particular, $\mathcal{ZJ}_{n+1,m+1}=\mathcal{ZP}_{n,m}$ when $0\leq
m+n\leq1.$

As in the absolute case, there is the face operator $\tilde{\partial}_{n,m}$
on $\widetilde{\mathcal{ZP}}_{n,m},$ the projection of chain complexes
$
C_{\ast}(\widetilde{\mathcal{ZP}}_{n,m})\rightarrow C_{\ast}
(\widetilde{\mathcal{ZJ}}_{n+1,m+1}),
$
the induced face operator $\partial_{n+1,m+1}$ on $\widetilde{\mathcal{ZJ}%
}_{n+1,m+1},$ and the restricted projection
\[
C_{\ast}({\mathcal{ZP}}_{n,m})\rightarrow C_{\ast}(\mathcal{ZJ}_{n+1,m+1}).
\]
Thus, denoting
\[
\mathfrak{g}_{m+1}^{n+1}:=\left\{
\begin{array}
	[c]{cl}%
	\lbrack\frac{z}{0}]=[\frac{0}{z}],\smallskip & m=n=0\\
	\lbrack\frac{\mathfrak{n}}{z}],\smallskip & m=0,\text{ }n>0\\
	\lbrack\frac{z}{\mathfrak{m}}],\smallskip & m>0,\text{ }n=0\\
	\lbrack\frac{n_{z}}{\mathfrak{m}}]=[\frac{\mathfrak{n}}{m_{z}}], & mn>0
\end{array}
\right.
\]
we have
\begin{equation}
	\tilde{\partial}_{n,m}\left(  \mathfrak{g}_{m+1}^{n+1}\right)  :=\left\{
	\begin{array}
		[c]{cc}%
		\varnothing, & m=n=0\\
		\left\{  F_{1}C_{2},C_{1}F_{2}\right\}  , & m+n>0.
	\end{array}
	\right.
\end{equation}

\subsection{Relative Prematrads}
Let $\left(  M,\gamma_{_{M}},\eta_{_{M}}\right)  $ and $\left(  N,\gamma
_{_{N}},\eta_{_{N}}\right)  $ be $R$-prematrads, where $\gamma_{M}=\{
\gamma_{\mathbf{x}}^{\mathbf{y}}:\mathbf{M}_{p}^{\mathbf{y}}\otimes
\mathbf{M}_{\mathbf{x}}^{q}$ $\rightarrow\mathbf{M}_{\left\Vert \mathbf{x}\right\Vert }^{\left\Vert \mathbf{y}\right\Vert}\}  $ and $\gamma_{N}=\{  \gamma_{\mathbf{x}}^{\mathbf{y}%
}:\mathbf{N}_{p}^{\mathbf{y}}\otimes\mathbf{N}_{\mathbf{x}}^{q}\rightarrow
\mathbf{N}_{\left\Vert \mathbf{x}\right\Vert }^{\left\Vert \mathbf{y}%
	\right\Vert }\}  ,$ let $E=\{E_{n,m}\}_{mn\geq1}$ be a bigraded
$R$-module, and let $\lambda=\{  \lambda_{\mathbf{x}}^{\mathbf{y}%
}:\mathbf{M}_{p}^{\mathbf{y}}\otimes\mathbf{E}_{\mathbf{x}}^{q}\rightarrow
\mathbf{E}_{\left\Vert \mathbf{x}\right\Vert }^{\left\Vert \mathbf{y}%
	\right\Vert }\}  $ and $\rho=\{  \rho_{\mathbf{x}}^{\mathbf{y}%
}:\mathbf{E}_{p}^{\mathbf{y}}\otimes\mathbf{N}_{\mathbf{x}}^{q}\rightarrow
\mathbf{E}_{\left\Vert \mathbf{x}\right\Vert }^{\left\Vert \mathbf{y}%
	\right\Vert }\}  $ be left and right $R$-module actions ($\gamma_M,\gamma_N, \lambda, \text{ and } \rho$ are subject to $\mathbf{x\times y}\in\mathbb{N}^{1\times p}\times
\mathbb{N}^{q\times1}$ and $pq\geq1$). Then $\lambda$ and $\rho$ induce left
and right global products $\Upsilon_{\lambda}:{\mathbf{M}}\otimes{\mathbf{E}%
}\rightarrow{\mathbf{E}}$ and $\Upsilon_{\rho}:{\mathbf{E}}\otimes{\mathbf{N}%
}\rightarrow{\mathbf{E}}$ in the same way that $\gamma_{M}$ induces the
associative global product $\Upsilon_{M}:{\mathbf{M}}\otimes{\mathbf{M}%
}\rightarrow{\mathbf{M}}$.

\begin{definition}
	\label{reldefn}A tuple $\left(  M,E,N,\lambda,\rho\right)  $ is a
	\textbf{relative prematrad} if
	
	\begin{enumerate}
		\item[\textit{(i)}] associativity holds, i.e.,
		
		\begin{enumerate}
			\item[\textit{(a)}] $\Upsilon_{\rho}(\Upsilon_{\lambda}\otimes${$\mathbf{1}$
			}$)=\Upsilon_{\lambda}(${$\mathbf{1}$}$\otimes\Upsilon_{\rho}),$
			
			\item[\textit{(b)}] $\Upsilon_{\lambda}(\Upsilon_{M}\otimes${$\mathbf{1}$
			}$)=\Upsilon_{\lambda}(${$\mathbf{1}$}$\otimes\Upsilon_{\lambda}),$
			
			\item[\textit{(c)}] $\Upsilon_{\rho}(\Upsilon_{\rho}\otimes${$\mathbf{1}$%
			}$)=\Upsilon_{\rho}(${$\mathbf{1}$}$\otimes\Upsilon_{N}),$ and
		\end{enumerate}
		
		\item[\textit{(ii)}] for all $a,b\in\mathbb{N,}$ the units $\eta_{M}$ and
		$\eta_{N}$ induce the canonical isomorphisms
		\[
		R^{\otimes b}\otimes\mathbf{E}_{a}^{b}\overset{\eta_{_{M}}^{\otimes b}%
			\otimes{\mathbf{1}}}{\longrightarrow}\mathbf{M}_{1}^{\mathbf{1}^{b\times1}%
		}\otimes\mathbf{E}_{a}^{b}\overset{\lambda_{a}^{\mathbf{1}^{b\times1}%
		}}{\longrightarrow}\mathbf{E}_{a}^{b}\text{ and }\mathbf{E}_{a}^{b}\otimes
		R^{\otimes a}\overset{{\mathbf{1}}\otimes\eta_{N}^{\otimes a}}{\longrightarrow
		}\mathbf{E}_{a}^{b}\otimes\mathbf{N}_{\mathbf{1}^{1\times a}}^{1}%
		\overset{\rho_{\mathbf{1}^{1\times a}}^{b}}{\longrightarrow}\mathbf{E}_{a}%
		^{b}.
		\]
		
	\end{enumerate}
	
	\noindent Under conditions (i) and (ii), $E$ is as an $M$-$N$-\textbf{bimodule
	}and an $M$\textbf{-bimodule }when $M=N$.
\end{definition}

\begin{definition}\label{remorphism}
	A \textbf{morphism} $f:(M,E,N,\lambda,\rho)\rightarrow(M^{\prime},E^{\prime
	},N^{\prime},\lambda^{\prime},\rho^{\prime})$ of relative prematrads is a
	triple $\left(  f_{M}:{M}\rightarrow{M}^{\prime},\text{ }f_{E}:E\rightarrow
	E^{\prime},\text{ }f_{N}:N\rightarrow N^{\prime}\right)  $ such that
	
	\begin{enumerate}
		\item[\textit{(i)}] $f_{M}$ and $f_{N}$ are maps of prematrads, and
		
		\item[\textit{(ii)}] $f_{E}$ commutes with left and right actions, i.e., for
		all $\mathbf{x\times y\in}\mathbb{N}^{p\times1}\times\mathbb{N}^{1\times q}$
		and $pq\geq1$ we have $f_{E}\hspace{.02in} \lambda_{\mathbf{x}}^{\mathbf{y}}%
		=\lambda{^{\prime}}_{\mathbf{x}}^{\mathbf{y}}\hspace{.02in} (f_{M}^{\otimes q}\otimes
		f_{E}^{\otimes p})$ and $f_{E}\hspace{.02in} \rho_{\mathbf{x}}^{\mathbf{y}}%
		=\rho{^{\prime}}_{\mathbf{x}}^{\mathbf{y}}\hspace{.02in} (f_{E}^{\otimes q}\otimes
		f_{N}^{\otimes p}).$
	\end{enumerate}
\end{definition}

Tree representations of $\lambda_{1}^{\mathbf{x}}$ and $\rho_{1}^{\mathbf{x}}
$ are related to those of $\lambda_{\mathbf{x}}^{1}$ and $\rho_{\mathbf{x}%
}^{1}$ by a reflection in some horizontal axis. Although $\rho_{\mathbf{x}%
}^{1}$ agrees with Markl, Shnider, and Stasheff's right module action over an
operad \cite{MSS}, $\lambda_{\mathbf{x}}^{1}$ differs fundamentally from their
left module action, and our definition of an \textquotedblleft
operadic\ bimodule\textquotedblright\ is consistent with their definition of
an operadic ideal.

Given graded $R$-modules $A$ and $B,$ let
\[%
\begin{array}
	[c]{lllll}%
	U_{A} & := & \mathcal{E}nd_{TA} & = & \left\{  N_{s,p}=Hom\left(  A^{\otimes
		p},A^{\otimes s}\right)  \right\}  _{p,s\geq1},\\
	U_{A,B} & := & Hom\left(  TA,TB\right)  & = & \left\{  E_{t,q}=Hom\left(
	A^{\otimes q},B^{\otimes t}\right)  \right\}  _{q,t\geq1},\\
	U_{B} & := & \mathcal{E}nd_{TB} & = & \left\{  M_{u,r}=Hom\left(  B^{\otimes
		r},B^{\otimes u}\right)  \right\}  _{r,u\geq1},
\end{array}
\]
and define left and right actions $\lambda$ and $\rho$ in terms of the
horizontal and vertical operations $\times$ and $\circ$ analogous to those in
the prematrad structures on $U_{A}$ and $U_{B}$ (see Section 2 above and \cite{SU4}). Then the\emph{\ }relative PROP\texttt{ }$\left(  U_{A}%
,U_{A,B},U_{B},\lambda,\rho\right)  $ is the universal example of a relative prematrad.

\begin{definition}
	Given bigraded sets $F=\left\{  F_{n,m}\right\}  _{mn\geq1}$ and $Q=\left\{
	Q_{n,m}\right\}  _{mn\geq1}$ with based element $\mathbf{1}\in Q_{1,1},$  construct the
	bigraded set $\tilde{E}\left(  F;Q\right)  =\{\tilde{E}_{n,m}\left(
	F;Q\right)  \}_{mn\geq1}$ of \textbf{relative free monomials generated by }$F$
	\textbf{and }$Q$ inductively as follows: Define $\tilde{E}_{1,1}\left(
	F;Q\right)  :=F_{1,1}.$ If $m+n\geq3$ and $\tilde{E}_{j,i}\left(  F;Q\right)
	$ has been constructed for all $\left(  i,j\right)  \leq\left(  m,n\right)  $
	such that $i+j<m+n,$ define
	\[
	\tilde{E}_{n,m}(F;Q):={F}_{n,m}\cup\left\{  C_{1}\cdots C_{r+s+1}=A_{1}\cdots
	A_{r}\cdot F\cdot B_{1}\cdots B_{s}:r+s\geq 1  \right\}  ,
	\]
	where
	
	\begin{enumerate}
		\item[\textit{(i)}] $F$ is a GBSM over $\{\tilde{E}_{j,i}\left(  F;Q\right)
		:\left(  i,j\right)  \leq\left(  m,n\right)  ,$ $i+j<m+n\}$,\vspace*{0.05in}
		
		\item[\textit{(ii)}] $F$ is a column matrix with $outdeg(F)=n$ when $r=0$,
		\vspace*{0.05in}
		
		\item[\textit{(iii)}] $F$ is a row matrix with $indeg(F)=m$ when $s=0$,
		\vspace*{0.05in}
		
		\item[\textit{(iv)}] $A_{k}$ and $B_{\ell}$ are GBSMs\ over $\{\tilde{G}_{j,i}\left(
		Q\right)  :\left(  i,j\right)  \leq\left(  m,n\right)  ,$ $i+j<m+n\}$,\vspace*{0.05in}
		
		\item[\textit{(v)}] $A_{1}$ is a column matrix with $outdeg(A_{1})=n$ when $r>0$,
		\vspace*{0.05in}
		
		\item[\textit{(vi)}] $B_{s}$ is a row matrix with $indeg(B_{s})=m$ when $s>0$,
		\vspace*{0.05in}
		
		\item[\textit{(vii)}] if $B=\left(  \xi_{tu}\right)  $ is a $q\times p$ GBSM over $\tilde{G}_{\ast,\ast}\left(
		Q\right)$
		within $C_{i}$ and $\xi_{tu}=B_{1}\cdots B_{k-1}\cdot\mathbf{1}\cdot
		B_{k+1}\cdots B_{r}$ for some $k,$ define $\xi_{tu}:=B_{1}\cdots B_{k-1}B_{k+1}\cdots B_{r}$ when either $pq=1$ and $B=C_{i}$ or $pq>1$ and
		$B_{t\ast}$ and $B_{\ast u}$ are indecomposable,

		\vspace*{0.05in}
		
		\item[\textit{(viii)}]$\mathfrak{f}_{s}^{t}=[\mathbf{1}\cdots\mathbf{1}]^T[ \mathfrak{f}_{s}^{t}]=[ \mathfrak{f}_{s}^{t}][\mathbf{1}\cdots\mathbf{1}],$
		
		\vspace*{0.05in}
		
		\item[\textit{(ix)}] $C_{i}\times C_{i+1}$ is a BTP for all $i$,\vspace
		*{0.05in}
		
		\item[\textit{(x)}] $C_{i}C_{i+1}$ is the formal product, and\vspace
		*{0.05in}
		
		\item[\textit{(xi)}] $(C_{i}C_{i+1})C_{i+2}=C_{i}(C_{i+1}C_{i+2}).$
		\vspace*{0.05in}
		
	\end{enumerate}
	
	\noindent An $(F,Q)$-\textbf{monomial }is an elementary monomial in $\tilde{E}(F;Q)$.	
	
\end{definition}
\noindent Note that in view of item (vii), the sets $\tilde{E}_{n,1}(F;Q)$ and $\tilde{E}_{1,m}(F;Q)$ consist of $(F,Q)$-monomials exclusively.

Let $\mathfrak{F}=\left\{  \mathfrak{f}_{m}^{n}\right\}  _{mn\geq1}$ and
$\Theta=\left\{  \theta_{m}^{n}:\theta_{1}^{1}=\mathbf{1}\right\}  _{mn\geq1}$
be bigraded sets with at most one element of bidegree $\left(  m,n\right) .$

\begin{definition}
	Let $\tilde{F}^{pre}\left(  \Theta;\mathfrak{F};\Theta\right)  =\left\langle
	\tilde{E}(\mathfrak{F};\Theta)\right\rangle $ and define left and right
	actions $\tilde{\lambda}^{pre}:$ $\tilde{F}^{pre}(\Theta)\otimes\tilde
	{F}^{pre}\left(  \Theta,\mathfrak{F},\Theta\right)  \rightarrow\tilde{F}%
	^{pre}\left(  \Theta,\mathfrak{F},\Theta\right)  $ and $\tilde{\rho}%
	^{pre}:\tilde{F}^{pre}\left(  \Theta,\mathfrak{F},\Theta\right)  \otimes
	\tilde{F}^{pre}(\Theta)\rightarrow\tilde{F}^{pre}\left(  \Theta,\mathfrak{F}%
	,\Theta\right)  $ by juxtaposition. Then
	\[
	(\tilde{F}^{pre}(\Theta),\tilde{F}^{pre}\left(  \Theta,\mathfrak{F}%
	,\Theta\right)  ,\tilde{F}^{pre}(\Theta),\tilde{\lambda}^{pre},\tilde{\rho
	}^{pre})
	\]
	is \textbf{relative free non-unital prematrad generated by }$\mathfrak{F}
	$\textbf{ \textbf{and} }$\Theta$.
	
	Define ${E}_{n,m}(\mathfrak{F};\Theta):=\overset{\sim}{E}_{n,m}
	(\mathfrak{F};\Theta)/\left(  A\sim\mathbf{1}A\sim A\mathbf{1}\right)  $ and
	let \[{F}^{pre}\left(  \Theta,\mathfrak{F},\Theta\right)  =\left\langle
	{E}(\mathfrak{F};\Theta)\right\rangle .\]
	The \textbf{relative free prematrad
		generated by }$\mathfrak{F}$ \textbf{and} $\Theta$ is the relative prematrad
	\[
	({F}^{pre}(\Theta),{F}^{pre}\left(  \Theta,\mathfrak{F},\Theta\right)
	,{F}^{pre}(\Theta),{\lambda}^{pre},{\rho}^{pre}).
	\]
	
\end{definition}

\begin{example}
	\label{22components}The bimodule $F_{2,2}^{pre}(\Theta,\mathfrak{F},\Theta)$
	contains 25 generators, namely, the indecomposable $\mathfrak{f}_{2}^{2}$ and
	the following twenty-four $\left(  \lambda^{pre},\rho^{pre}\right)
	$-decomposables:
	
	\medskip
	\noindent two of the form $AFB$:
	\begin{equation}
		\left[  \theta_{1}^{2}\right]  \left[  \mathfrak{f}_{1}^{1}\right]  \left[
		\theta_{2}^{1}\right]  \text{ and }\left[
		\begin{array}
			[c]{r}%
			\theta_{2}^{1}\medskip\\
			\theta_{2}^{1}%
		\end{array}
		\right]  \left[
		\begin{array}
			[c]{cc}%
			\mathfrak{f}_{1}^{1}\medskip & \mathfrak{f}_{1}^{1}\\
			\mathfrak{f}_{1}^{1} & \mathfrak{f}_{1}^{1}%
		\end{array}
		\right]  \left[
		\begin{array}
			[c]{ll}%
			\theta_{1}^{2} & \theta_{1}^{2}%
		\end{array}
		\right]  ; \label{relinner}%
	\end{equation}
	eleven of the form $FB_{1}\cdots B_{s}$:
	\[
	\left[  \mathfrak{f}_{1}^{2}\right]  \left[  \theta_{2}^{1}\right]  ,\text{
	}\left[
	\begin{array}
		[c]{c}%
		\mathfrak{f}_{1}^{1}\medskip\\
		\mathfrak{f}_{1}^{1}%
	\end{array}
	\right]  \left[  \theta_{2}^{2}\right]  ,\text{ }\left[
	\begin{array}
		[c]{c}%
		\mathfrak{f}_{1}^{1}\medskip\\
		\mathfrak{f}_{1}^{1}%
	\end{array}
	\right]  \left[  \theta_{1}^{2}\right]  \left[  \theta_{2}^{1}\right]  ,\text{
	}\left[
	\begin{array}
		[c]{c}%
		\mathfrak{f}_{1}^{1}\medskip\\
		\mathfrak{f}_{1}^{1}%
	\end{array}
	\right]  \left[
	\begin{array}
		[c]{r}%
		\theta_{2}^{1}\medskip\\
		\theta_{2}^{1}%
	\end{array}
	\right]  \left[
	\begin{array}
		[c]{ll}%
		\theta_{1}^{2} & \theta_{1}^{2}%
	\end{array}
	\right]  ,\bigskip
	\]%
	\[
	\left[
	\begin{array}
		[c]{c}%
		\mathfrak{f}_{2}^{1}\medskip\\
		\mathfrak{f}_{2}^{1}%
	\end{array}
	\right]  \left[
	\begin{array}
		[c]{ll}%
		\theta_{1}^{2} & \theta_{1}^{2}%
	\end{array}
	\right]  ,\text{ }\left[
	\begin{array}
		[c]{c}%
		\mathfrak{f}_{2}^{1}\medskip\\
		\left[  \mathfrak{f}_{1}^{1}\right]  \left[  \theta_{2}^{1}\right]
	\end{array}
	\right]  \left[
	\begin{array}
		[c]{ll}%
		\theta_{1}^{2} & \theta_{1}^{2}%
	\end{array}
	\right]  ,\text{ }\left[
	\begin{array}
		[c]{c}%
		\left[  \mathfrak{f}_{1}^{1}\right]  \left[  \theta_{2}^{1}\right]  \medskip\\
		\mathfrak{f}_{2}^{1}%
	\end{array}
	\right]  \left[
	\begin{array}
		[c]{ll}%
		\theta_{1}^{2} & \theta_{1}^{2}%
	\end{array}
	\right]  ,\bigskip\text{ }%
	\]%
	\[
	\left[
	\begin{array}
		[c]{c}%
		\left[  \theta_{2}^{1}\right]  \left[
		\begin{array}
			[c]{cc}%
			\mathfrak{f}_{1}^{1} & \mathfrak{f}_{1}^{1}%
		\end{array}
		\right]  \medskip\\
		\left[  \mathfrak{f}_{1}^{1}\right]  \left[  \theta_{2}^{1}\right]
	\end{array}
	\right]  \left[
	\begin{array}
		[c]{ll}%
		\theta_{1}^{2} & \theta_{1}^{2}%
	\end{array}
	\right]  ,\text{ }\left[
	\begin{array}
		[c]{c}%
		\left[  \mathfrak{f}_{1}^{1}\right]  \left[  \theta_{2}^{1}\right]  \medskip\\
		\left[  \theta_{2}^{1}\right]  \left[
		\begin{array}
			[c]{cc}%
			\mathfrak{f}_{1}^{1} & \mathfrak{f}_{1}^{1}%
		\end{array}
		\right]
	\end{array}
	\right]  \left[
	\begin{array}
		[c]{ll}%
		\theta_{1}^{2} & \theta_{1}^{2}%
	\end{array}
	\right]  ,\text{ }\bigskip
	\]%
	\[
	\left[
	\begin{array}
		[c]{c}%
		\left[  \theta_{2}^{1}\right]  \left[
		\begin{array}
			[c]{cc}%
			\mathfrak{f}_{1}^{1} & \mathfrak{f}_{1}^{1}%
		\end{array}
		\right]  \medskip\\
		\mathfrak{f}_{2}^{1}%
	\end{array}
	\right]  \left[
	\begin{array}
		[c]{ll}%
		\theta_{1}^{2} & \theta_{1}^{2}%
	\end{array}
	\right]  ,\text{ }\left[
	\begin{array}
		[c]{c}%
		\mathfrak{f}_{2}^{1}\medskip\\
		\left[  \theta_{2}^{1}\right]  \left[
		\begin{array}
			[c]{cc}%
			\mathfrak{f}_{1}^{1} & \mathfrak{f}_{1}^{1}%
		\end{array}
		\right]
	\end{array}
	\right]  \left[
	\begin{array}
		[c]{ll}%
		\theta_{1}^{2} & \theta_{1}^{2}%
	\end{array}
	\right]  ;\bigskip
	\]
	and their eleven respective duals of the form $A_{1}\cdots A_{r}F.$
\end{example}

\begin{example}
	\label{JJprime}Recall that the bialgebra prematrad $\mathcal{H}^{pre}$ has two
	prematrad generators $c_{1,2}$ and $c_{2,1},$ and a single module generator
	$c_{n,m}$ of bidegree $\left(  m,n\right)  $ for $m+n\geq 4$. Consequently,
	the $\mathcal{H}^{pre}$-bimodule $\mathcal{JJ}^{pre}$ has a single bimodule
	generator $\mathfrak{f}$ of bidegree $(1,1)$, and a single module generator $c_{n,m}$ of
	bidegree $(m,n)$ satisfying the structure relation
	\[
	\lambda(c_{n,m};\underset{m}{\underbrace{\mathfrak{f},\ldots,\mathfrak{f}}%
	})=\rho(\underset{n}{\underbrace{\mathfrak{f},\ldots,\mathfrak{f}}}%
	\,;c_{n,m}).
	\]
	More precisely, if $\Theta=\left\langle \theta_{1}^{1}=\mathbf{1,}\,\theta
	_{2}^{1},\,\theta_{1}^{2}\right\rangle $ and $\mathfrak{F}=\left\langle
	\mathfrak{f=f}_{1}^{1}\right\rangle ,$ then
	\[
	\mathcal{JJ}^{pre}=F^{pre}\left(  \Theta,\mathfrak{F},\Theta\right)  /\sim,
	\]
	where $u\sim u^{\prime}$ if and only if $bideg(u)=bideg(u^{\prime}).$ A
	bialgebra morphism $f:A\rightarrow B$ is the image of $\mathfrak{f}$ under a
	map $\mathcal{JJ}^{pre}\rightarrow U_{A,B}$ of relative prematrads. In this
	case, we omit the superscript and simply write $\mathcal{JJ}.$
\end{example}

\begin{example}
	Whereas $F_{1,\ast}^{pre}(\Theta)$ with bimodule generators
	$\{{\theta}_{m}^{1}\}_{m\geq1}$ and $F_{\ast,1}^{pre}(\Theta)$ with bimodule generators
	$\{{\theta}_{1}^{n}\}_{n\geq1}$ can be
	identified with the $A_{\infty}$-operad $\mathcal{A}_{\infty}$,\linebreak $F_{1,\ast}(\Theta,{\mathfrak{F}},\Theta)$ and
	$F_{\ast,1}(\Theta,{\mathfrak{F}},\Theta)$   can be identified with the $\mathcal{A}_{\infty}%
	$-bimodule $\mathcal{J}_{\infty}$. Thus an $A_{\infty}$-(co)algebra morphism
	$f:A\rightarrow B$ is the image of the $A_{\infty}$-bimodule generators under
	a map $\mathcal{J}_{\infty}\rightarrow Hom\left(  TA,B\right)  $ $($or
	$\mathcal{J}_{\infty}\rightarrow Hom\left(  A,TB\right)  )$ of relative prematrads.
\end{example}

When $\Theta=\left\langle \theta_{m}^{n}:\theta_{1}^{1}=\mathbf{1}%
\right\rangle _{mn\geq1}$ and $\mathfrak{F}=\left\langle \mathfrak{f}_{m}%
^{n}\right\rangle _{mn\geq1},$ the canonical projections $\varrho_{\Theta
}^{pre}:F^{pre}(\Theta)\rightarrow\mathcal{H}^{pre}$ and $\varrho
^{pre}:F^{pre}(\Theta,\mathfrak{F},\Theta)\rightarrow\mathcal{JJ}^{pre}$
define a map $\left(\varrho_{\Theta}^{pre},\varrho^{pre},\right.$
$\left.\varrho_{\Theta}^{pre}\right)$ of relative prematrads. If $\partial^{pre}$ is a
differential on $F^{pre}(\Theta,\mathfrak{F},\Theta)$ such that $\varrho
^{pre}$ is a free resolution in the category of relative prematrads, the
induced isomorphism $\varrho^{pre}:H_{\ast}\left(  F^{pre}(\Theta
,\mathfrak{F},\Theta),\partial^{pre}\right)  \approx\mathcal{JJ}^{pre}$
implies
\[%
\begin{array}
	[c]{c}%
	\partial^{pre}(\mathfrak{f}_{1}^{1})=0\\
	\\
	\partial^{pre}(\mathfrak{f}_{2}^{1})=\rho(\mathfrak{f}_{1}^{1};\theta_{2}%
	^{1})-\lambda(\theta_{2}^{1};\mathfrak{f}_{1}^{1},\mathfrak{f}_{1}^{1})\\
	\\
	\partial^{pre}(\mathfrak{f}_{1}^{2})=\rho(\mathfrak{f}_{1}^{1},\mathfrak{f}%
	_{1}^{1};\theta_{1}^{2})-\lambda(\theta_{1}^{2};\mathfrak{f}_{1}^{1}).
\end{array}
\]
This gives rise to the standard isomorphisms
\[%
\begin{array}
	[c]{ccccccc}%
	F_{1,2}^{pre}(\Theta,\mathfrak{F},\Theta)\smallskip & = & \left\langle \left[
	\theta_{2}^{1}\right]  \right. \!\!
	\left[
	\begin{array}
		[c]{cc}%
		\mathfrak{f}_{1}^{1} & \mathfrak{f}_{1}^{1}%
	\end{array}
	\right]  & , & \left[  \mathfrak{f}_{2}^{1}\right]  & , & \left.  \left[
	\mathfrak{f}_{1}^{1}\right]\!  \left[  \theta_{2}^{1}\right]  \right\rangle \\ \\
	\approx\text{ }\updownarrow\text{ \ }\smallskip &  & \updownarrow &  &
	\updownarrow &  & \updownarrow\\ \\
	C_{\ast}(J_{2})\smallskip & = & \left\langle 1|2\right.  & , & 12 & , &
	\left.  2|1\right\rangle \\ \\
	\approx\text{ }\updownarrow\text{ \ }\medskip &  & \updownarrow &  &
	\updownarrow &  & \updownarrow\\
	F_{2,1}^{pre}(\Theta,\mathfrak{F},\Theta) & = & \left\langle \left[
	\theta_{1}^{2}\right] \! \left[  \mathfrak{f}_{1}^{1}\right]  \right.  & , &
	\left[  \mathfrak{f}_{1}^{2}\right]  & , & \left[
	\begin{array}
		[c]{c}%
		\mathfrak{f}_{1}^{1}\medskip\\
		\mathfrak{f}_{1}^{1}%
	\end{array}
	\right]  \left. \!\! \left[  \theta_{1}^{2}\right]  \right\rangle
\end{array}
\]
(see Figure 17). A similar application of $\partial^{pre}$ to $\mathfrak{f}
_{1}^{n}$ and $\mathfrak{f}_{m}^{1}$ gives the isomorphisms
\begin{equation}
	F_{n,1}^{pre}(\Theta,\mathfrak{F},\Theta)\overset{\approx}{\longrightarrow
	}\mathcal{J}_{\infty}(n)=C_{\ast}(J_{n}) \label{iso1}%
\end{equation}
and
\begin{equation}
	F_{1,m}^{pre}(\Theta,\mathfrak{F},\Theta)\overset{\approx}{\longrightarrow
	}\mathcal{J}_{\infty}(m)=C_{\ast}(J_{m}) \label{iso2}%
\end{equation}
(see \cite{Stasheff}, \cite{Stasheff2}, \cite{SU2}).
Combinatorially, $J_{n+1}$ is the cylinder $K_{n+1}\times I$ and the
projection
\begin{equation}\label{multip}
	\pi_{n+1}:P_{n+1}\rightarrow J_{n+1}
\end{equation}
in (\ref{JJdiagram}) below is given by $\pi_{n+1}:=\pi_{0,n}$ (cf. \cite{SU2}).

\subsection{Relative Matrads.}

Let $f:=A_{1}\cdots A_{k-1}\cdot F\cdot A_{k+1}\cdots A_{r}\in\tilde
{E}_{n+1,m+1}(\mathfrak{F;}\Theta).$ As in the absolute case, the
\emph{dimension} of $f$ is the sum of the dimensions of its matrix factors,
i.e.,
\[
\left\vert f\right\vert =\left\vert F\right\vert +\sum_{1\leq k\leq
	r}\left\vert A_{k}\right\vert ,
\]
and in particular, $|\mathfrak{f}_{m}^{n}|=m+n-2.$

Given a positive integer $n$ and integer sequences $\mathbf{x}=(x_{1}%
,\ldots,x_{k})$ and $\mathbf{y}=(y_{1},\ldots,y_{k}),$ define $\mathbf{x}\smile
n:=(x_{1},\ldots,x_{k},n)$ and $\mathbf{x}+\mathbf{y:}=(x_{1}+y_{1}%
,\ldots,x_{k}+y_{k}).$

\begin{definition}
	\label{rleaf} Let $f=A_{1}\cdots A_{k-1}\cdot F \cdot A_{k+1}\cdots A_{r}\in \tilde{E}_{n+1,m+1}(\mathfrak{F;}\Theta).$
	
	\begin{enumerate}
		\item[(\textit{i)}] The \textbf{Input Leaf Decomposition of} $f$ is the
		sequence
		\[
		ILD(f):=\! \left\{  ils(A_{1}),\ldots,ils(A_{k-1}),ils(F),ils(A_{k+1}),\ldots,ils(A_{r}%
		)\right\}  ;
		\]
		the \textbf{Augmented Input Leaf Decomposition of} $f$ is the sequence
		\[
		AILD(f):=\left\{  ils(A_{1}),\ldots,ils(A_{k-1}),ils(F)+(0,\ldots,0,1),\right.
		\]
		\[		
		\hspace{2.5in}\left.ils(A_{k+1}\smile1),\ldots,ils(A_{r})\smile1\right\}.
		\]

		\item[\textit{(ii)}] The \textbf{Output Leaf Decomposition of} $f$ is the
		sequence
		\[
		\ \ OLD\left(  f\right)  :=\left\{  ols(A_{1}),\ldots,ols(A_{k-1}),ols(F),ols(A_{k+1}%
		),...,ols(A_{r})\right\};
		\]
		the \textbf{Augmented Output Leaf Decomposition of} $f$ is the sequence
		\[
		AOLD\left(  f\right)  := \! \left\{  ols(A_{1})\smile \!1,\ldots,ols(A_{k-1}
		)\smile \!1,\right.
		\]
		\[		
		\hspace{1.5in}\left. ols(F)+(0,\ldots,0,1),ols(A_{k+1}),\ldots,ols(A_{r})\right\}.
		\]
		
	\end{enumerate}
\end{definition}

Now define the map
\begin{equation}
	\digamma_{\!\!z}:\tilde{E}_{n+1,m+1}(\mathfrak{F;}\Theta)\rightarrow
	\widetilde{\mathcal{ZF}}_{n,m}\label{redigamma}%
\end{equation}
in a manner similar to the definition of $\digamma$ in (\ref{digamma}), but this time define
$\digamma_{\!\!z}$ on \emph{three} types of monomials: Let $f=A_{1}\cdots A_{k-1}\cdot F\cdot A_{k+1}\cdots A_{r}%
\in\tilde{E}_{n+1,m+1}(\mathfrak{F};\Theta),$ let $\left[  c\right]
:=\digamma_{\!\!z}(f)$, and let $c=(c^{1},c^{2},\ldots,c^{h})$ be the
canonical representative. If $r=1,$ then $f=\mathfrak{f}_{m+1}^{n+1}$ and $\digamma
_{\!\!z}\left(  f\right) : =\mathfrak{g}_{m+1}^{n+1}.$ Otherwise,

\begin{enumerate}
	\item if $m=0,$ then $\digamma_{\!\!z}(f)=\frac{\bar{R}OLD(f)}{\overset{}{\bar{R}AILD(f)}
	}=c^1=c.\smallskip$
	
	\item if $n=0,$ then $\digamma_{\!\!z}(f)=\frac{\bar{R}AOLD(f)}{\overset{}{\bar{R}ILD(f)}
	}=c^1=c.\smallskip$
	
	\item if $mn>0$ and $g$ is a monomial within $f$ (the case $g=f$ included)
	\begin{enumerate}

		\item $1<k=r,$ then $\digamma_{\!\!z}(g)=\frac{\bar{R}AOLD(f)}{\overset{}{\bar{R}ILD(f)}}.\smallskip$
		
		\item $1=k<r,$ then $\digamma_{\!\!z}(g)=\frac{\bar{R}OLD(f)}{\overset{}{\bar{R}AILD(f)}}.\smallskip$
		
		\item $1<k<r,$ then $\digamma_{\!\!z}(g)$ is given by either $3(a)$ or $3(b)$\smallskip.

	\end{enumerate}
\end{enumerate}

\noindent Then, for example,
\[
\digamma_{\!\!z}\left(  \left[  \theta_{1}^{n+1}\right] \!\! \left[
\mathfrak{f}_{1}^{1}\right]  \right)  = \frac{\mathfrak{n}|0}{0|z},\ \ \
\digamma_{\!\!z}\left(  \left[
\begin{array}
	[c]{c}%
	\mathfrak{f}_{1}^{1}\\
	\vdots\\
	\mathfrak{f}_{1}^{1}%
\end{array}
\right]\!\!  \left[  \theta_{1}^{n+1}\right]  \right)  =\frac{0|\mathfrak{n}}%
{z|0},\
\]
\[
\digamma_{\!\!z}\left(  \left[  \theta_{m+1}^{1}\right]\!\!  \left[
\mathfrak{f}_{1}^{1}\cdots\mathfrak{f}_{1}^{1}\right]  \right)  =\frac
{0|z}{\mathfrak{m}|0},\ \text{ and }\ \
\digamma_{\!\!z}\left(  \left[  \mathfrak{f}_{1}^{1}\right] \!\! \left[
\theta_{m+1}^{1}\right]  \right)  = \frac{z|0}{0|\mathfrak{m}}\ .
\]

\noindent
Note that $f$ is  an  $(\mathfrak{F},\Theta)$-monomial if and only if $h(\digamma_{\!\!z}(f))=1$.

\begin{remark}
	If $g$ is a monomial of bidegree $(t,1)$ or $(1,t)$, $t> 1$, its value $\digamma_{\!\!z}(g)$ given by (1) or (2) may differ from its value as a monomial within but distinct from $f$ in (3) (see Example \ref{exrule}).
\end{remark}

\begin{example}
	\label{exrule} Consider the monomial
	\[
	f=\left[
	\begin{array}
		[c]{c}%
		\mathfrak{f}_{2}^{1}\medskip\\
		\left[  \mathfrak{f}_{1}^{1}\right]\!\!  \left[  \theta_{2}^{1}\right]
	\end{array}
	\right] \! [\theta_{1}^{2}\,\,\theta_{1}^{2}]\in\tilde{E}_{2,2}(\mathfrak{F;}%
	\Theta).
	\]
	Apply item 3(b) in the definition of $	\digamma_{\!\!z}$ and obtain
	\[
	\digamma_{\!\!z}(f)=(c^{1},c^{2})=\left(  \frac{\,\,\,0|1}{1z|0}=\left(
	\begin{array}
		[c]{c}%
		\frac{0}{1z}\vspace{1mm}\\
		\frac{0}{1z}%
	\end{array}
	\right)  \!\!\left(  \frac{1}{0}\,\,\frac{1}{0}\,\,\frac{1}{0}\right)
	\,,\left(
	\begin{array}
		[c]{c}
		\frac{0}{1z}\vspace{1mm}\\
		\frac{0|0}{z|1}
	\end{array}
	\right)  \!\!\left(  \frac{1}{0}\,\,\frac{1}{0}\,\,\frac{1}{0}\right)
	\right)  \in\widetilde{\mathcal{ZF}}_{1,1}.
	\]
	Note that the incoherent left-hand factor of $c^2$ implies $\digamma_{\!\!z}(f)\notin \widetilde{\mathcal{ZP}}_{1,1}$. Also note that $\digamma_{\!\!z}(\mathfrak{f}_{2}^{1})=\frac{z}{1}\neq\frac{0}{1z}$, which is the value of $\mathfrak{f}_{2}^{1}$ as an $(\mathfrak{F},\Theta)$-monomial within $f$.
\end{example}

Apply (\ref{redigamma}) to define the subset $\tilde{\mathcal{B}}_{n,m}
(\Theta,\mathfrak{F},\Theta):=\digamma_{\!\!z}^{-1}({\mathcal{ZP}}_{n,m})\subset
\tilde{E}_{n,m}(\mathfrak{F;\Theta});$ then define
\[
{\mathcal{B}}_{n,m}(\Theta,\mathfrak{F},\Theta):=\tilde{\mathcal{B}}
_{n,m}(\Theta,\mathfrak{F},\Theta)/\left(  A\sim\mathbf{1}A\sim A\mathbf{1}
\right)  .
\]
\begin{definition}
	The \textbf{relative free non-unital balanced matrad}
	\textbf{generated by }$\mathfrak{F}$ and $\Theta$ is the bigraded module
	\[
	\tilde{F}(\Theta,\mathfrak{F},\Theta):=\langle\tilde{\mathcal{B}}%
	(\Theta,\mathfrak{F},\Theta)\rangle,
	\]
	and the \textbf{relative free balanced matrad}\emph{\ }\textbf{generated by
	}$\mathfrak{F}$ and $\Theta$ is the bigraded module
	\[
	{F}(\Theta,\mathfrak{F},\Theta):=\langle{\mathcal{B}}(\Theta,\mathfrak{F}%
	,\Theta)\rangle.
	\]
	
\end{definition}

\begin{example}
	\label{mult4}
	Consider the map $\pi_{n+1}:P_{n+1}\rightarrow J_{n+1}$
	for $n=3$ (cf. (\ref{multip}) and Remark \ref{ck}).
	Then  bijection (\ref{iso2}) is represented on the vertices of
	$J_{4}$ in $\pi_4(2|134\cup24|13)$ as follows:
	\[
	\begin{array}
		[c]{ccc}
		\left[
		\theta_{2}^{1}\right]\!\!  \left[  \theta_{2}^{1}\text{ }\mathbf{1}\right]\!\!
		\left[  \mathbf{1}\text{ }\mathbf{1}\text{ }\theta_{2}^{1}\right] \!\! \left[
		\mathfrak{f}_{1}^{1}\text{ }\mathfrak{f}_{1}^{1}\text{ }\mathfrak{f}_{1}
		^{1}\text{ }\mathfrak{f}_{1}^{1}\right]
		&\sim &
		\left[  \theta_{2}^{1}\right]\!\!  \left[  \mathbf{1}\text{ }\theta_{2}
		^{1}\right] \!\! \left[  \theta_{2}^{1}\text{ }\mathbf{1}\text{ }\mathbf{1}
		\right] \!\! \left[  \mathfrak{f}_{1}^{1}\text{ }\mathfrak{f}_{1}^{1}\text{
		}\mathfrak{f}_{1}^{1}\text{ }\mathfrak{f}_{1}^{1}\right]
		
		\\
		\updownarrow &  & \updownarrow\\
		\frac{0|0|0|z}{2|1|3|0}& \sim &
		\frac{0|0|0|z}{2|3|1|0} \\
		\updownarrow &  & \updownarrow\\
		2|1|3|4 & \overset{\pi_4}{\sim} & 2|3|1|4\\
		&  & \\
		\\

		\left[  \theta_{2}^{1}\right] \!\! \left[  \theta_{2}
		^{1}\text{ }\mathbf{1}\right] \!\! \left[  \mathfrak{f}_{1}^{1}\text{
		}\mathfrak{f}_{1}^{1}\text{ }\mathfrak{f}_{1}^{1}\right]
		\!\! \left[
		\mathbf{1} \text{ } \mathbf{1} \text{ }\theta_{2}^{1}\right]
		& &
		\left[  \theta_{2}^{1}\right] \!\! \left[  \mathbf{1}\text{ }\theta_{2}
		^{1}\right] \!\! \left[  \mathfrak{f}_{1}^{1}\text{ }\mathfrak{f}_{1}^{1}\text{
		}\mathfrak{f}_{1}^{1}\right]\!\!
		\left[  \theta_{2}^{1}\text{ }\mathbf{1}\text{ }\mathbf{1}\right]

		\\
		\updownarrow &  & \updownarrow\\
		
		\frac{0|0|z|0}{2|1|0|3}
		&  &
		
		\frac{0|0|z|0}{2|3|0|1}
		\\
		\updownarrow &  & \updownarrow\\
		2|1|4|3 &  & 2|3|4|1\\
		&  & \\
	\\

		\left[  \theta_{2}^{1}\right] \!\! \left[  \mathfrak{f}_{1}^{1}\text{
		}\mathfrak{f}_{1}^{1}\right] \!\! \left[  \theta_{2}^{1}\text{ }\mathbf{1}\right]
		\!\!\left[  \mathbf{1}\text{ }\mathbf{1}\text{ }\theta_{2}^{1}\right]

		& \sim &
		\left[  \theta_{2}^{1}\right]\!\!  \left[  \mathfrak{f}_{1}^{1}\text{
		}\mathfrak{f}_{1}^{1}\right] \!\! \left[  \mathbf{1}\text{ }\theta_{2}^{1}\right]\!\!
		\left[  \theta_{2}^{1}\text{ }\mathbf{1}\text{ }\mathbf{1}\right]

		\\
		\updownarrow &  & \updownarrow\\
		\frac{0|z|0|0}{2|0|1|3}
		&   \sim  &
		\frac{0|z|0|0}{2|0|3|1}
		
		\\
		\updownarrow &  & \updownarrow\\
		2|4|1|3 & \overset{\pi_4}{\sim} & 2|4|3|1\\
		&  & \\
		\\

		\left[  \mathfrak{f}_{1}
		^{1}\right] \!\! \left[  \theta_{2}^{1}\right]\!\!  \left[  \theta_{2}^{1}\text{
		}\mathbf{1}\right] \!\! \left[  \mathbf{1}\text{ }\mathbf{1}\text{ }\theta_{2}
		^{1}\right]
		&\sim &
		\left[  \mathfrak{f}_{1}^{1}\right] \!\! \left[  \theta_{2}^{1}\right] \!\! \left[
		\mathbf{1}\text{ }\theta_{2}^{1}\right]\!\!  \left[  \theta_{2}^{1}\text{
		}\mathbf{1}\text{ }\mathbf{1}\right]

		\\
		\updownarrow &  & \updownarrow\\
		
		\frac{z|0|0|0}{0|2|1|3}
		&   \sim  &
		\frac{z|0|0|0}{0|2|3|1}
		\\
		\updownarrow &  & \updownarrow\\
		4|2|1|3 & \overset{\pi_4}{\sim} & 4|2|3|1.
	\end{array}
	\]
	Note that the $2$-cell $24|13$ and the edge $2|13|4$ of $P_{4}$ degenerate
	under $\pi_4$ above, while  the $2$-cell $24|13$ and the edge $1|24|3$ of $P_{4}$ degenerate under the map $\pi_4$ defined in
	\cite{SU2}.
\end{example}

\begin{proposition}
	\label{r-dif-non}The map $\tilde{\partial}_{n,m}$ on the cellular chains
	$C_{\ast}({\mathcal{ZP}}_{n,m})$ induces a map
	\[
	\tilde{\partial}_{n,m}:\tilde{F}_{n,m}(\Theta,\mathfrak{F},\Theta
	)\rightarrow\tilde{F}_{n,m}(\Theta,\mathfrak{F},\Theta)
	\]
	of degree $-1$ and a differential $\tilde{\partial}=\sum\nolimits_{mn\geq1}\tilde{\partial}_{n,m}$ on $\tilde{F}(\Theta,\mathfrak{F},\Theta)$. Consequently, there
	is the identification of chain complexes
	\[
	{F}(\Theta,\mathfrak{F},\Theta)=C_{\ast}({\mathcal{ZJ}}_{n,m}).
	\]
	
\end{proposition}

\begin{example}
	The set $F_{2,2}(\Theta,\mathfrak{F},\Theta)$	consists of the following coherent dimension $1$ generators and their duals in Example \ref{22components} (cf. Figure 18):
	\bigskip
	\[
	\left[  \mathfrak{f}_{1}^{2}\right]\!\!  \left[  \theta_{2}^{1}\right]  ,\text{
	}\left[
	\begin{array}
		[c]{c}%
		\mathfrak{f}_{1}^{1}\medskip\\
		\mathfrak{f}_{1}^{1}%
	\end{array}
	\right] \!\! \left[  \theta_{2}^{2}\right]  ,\text{ }\left[
	\begin{array}
		[c]{c}%
		\left[  \mathfrak{f}_{1}^{1}\right] \!\! \left[  \theta_{2}^{1}\right]  \medskip\\
		\mathfrak{f}_{2}^{1}
	\end{array}
	\right]\!\!  \left[
	\begin{array}
		[c]{ll}
		\theta_{1}^{2} & \theta_{1}^{2}
	\end{array}
	\right]  ,\
	\left[
	\begin{array}
		[c]{c}
		\mathfrak{f}_{2}^{1}
		\medskip\\
		\left[  \theta_{2}^{1}\right]\!\!
		\left[
		\begin{array}
			[c]{cc}
			\mathfrak{f}_{1}^{1} & \mathfrak{f}_{1}^{1}
		\end{array}
		\right]
	\end{array}
	\right]\!\!
	\left[
	\begin{array}
		[c]{ll}
		\theta_{1}^{2} & \theta_{1}^{2}
	\end{array}
	\right].
	\]
	
\end{example}

\begin{definition}
	The $\mathcal{A}_{\infty}$- \textbf{bialgebra morphism matrad is the} DG
	$\mathcal{H}_{\infty}$-bimodule
	\[
	r\mathcal{H}_{\infty}:=\left(  F\left(  \Theta,\mathfrak{F},\Theta\right)
	,\partial\right)  .
	\]
\end{definition}
\noindent Thus a map $A\Rightarrow B$ of $A_{\infty}$-bialgebras is defined by a relative prematrad map
$r\mathcal{H}_{\infty}\rightarrow U_{A,B}$ (cf. Definitions \ref{remorphism} and \ref{gmorphism}).

\section{The Bimultiplihedron $JJ_{n,m}$}

Let $m, n\geq1.$ First, we construct the polytopes $JP_{n,m}$ and a face-preserving bijection
\begin{equation}
	\mathcal{ZP}_{n,m}\longrightarrow\{\text{cells of }\ JP_{n,m}\}, \label{JP}%
\end{equation}
which induces the structural
bijection
\[
\mathcal{ZJ}_{n+1,m+1}\longrightarrow\{\text{cells of }\ JJ_{n+1,m+1}\}
\]
and the commutative diagram
\begin{equation}%
	\begin{array}
		[c]{ccc}%
		\mathcal{ZP}_{n,m} & \overset{\approx}{\longrightarrow} & \{\text{cells of
		}\ JP_{n,m}\}\vspace*{0.1in}\\
		\hspace{.1in}\downarrow\text{ \ \ \ } &  & \ \ \ \ \ \downarrow\pi_{n,m}\vspace*{0.1in}\\
		\mathcal{ZJ}_{n+1,m+1} & \overset{\approx}{\longrightarrow} & \{\text{cells of
		}\ JJ_{n+1,m+1}\}.
	\end{array}
	\label{JJdiagram}%
\end{equation}
We refer to the polytopes $JJ=\left\{  JJ_{n+1,m+1}\right\}$ as
\emph{bimultiplihedra}. As in the absolute case,
apply the bijections $m_{z}\ast_{c}\,\mathfrak{n}\rightarrow P_{m+n+1}$ and
$\mathfrak{m}\ast_{c}\,n_{z}\rightarrow P_{m+n+1}$ with $z=m+n+1$ and fix the bijection (\ref{JP}). Then
$JP_{n,m}$ forms a subdivision of the cylinder $PP_{n,m}\times I$ in which
codimension 1 cells of the form $C_{1}F_{2}:=\left(  \frac{B_{1}|(B_{2}\cup
	z)}{\!\!\!\!\!\!\!\!\!\!\!\!A_{1}|A_{2}}\,,\dots\right)  $ and $F_{1}%
C_{2}:=\left(  \frac{\ \ \ \ \,B_{1}|B_{2}}{(A_{1}\cup z)|A_{2}}\,,\dots\right)
$ correspond to the subdivision cells of the codimension 1 cells $A|(B\cup
m+n+1)$ and $(A\cup m+n+1)|B$ for $A|B=A_{1}|A_{2}\Cup B_{1}|B_{2}\subset
P_{m+n+1},$ respectively. Then the bimultiplihedron $JJ_{n+1,m+1}$ is a
subdivision of the cylinder $KK_{n+1,m+1}\times I$; in particular, the octagon $JJ_{2,2}$ is a subdivision of
a hexagon (see Figure 18) and $JJ_{n+1,1}=JJ_{1,n+1}$ can be identified with
multiplihedra $J_{n+1}$ (cf. Example \ref{mult4}).

Thus, there is an isomorphism of chain complexes
\begin{equation}
	C_{\ast}(\mathcal{ZJ}_{n+1,m+1})\overset{\approx}{\longrightarrow}C_{\ast
	}(JJ_{n+1,m+1}), \label{rel-iso-jj}%
\end{equation}
and $C_{\ast}(JJ)$ realizes the $\mathcal{A}_{\infty}$-bialgebra morphism
matrad $r\mathcal{H}_{\infty}.$ Furthermore, the standard isomorphisms
(\ref{iso1}) and (\ref{iso2}) extend to isomorphisms
\begin{equation}
	(r\mathcal{H}_{\infty})_{n+1,m+1}\overset{\approx}{\longrightarrow}C_{\ast
	}(JJ_{n+1,m+1}), \label{JJiso}%
\end{equation}
and one recovers $\mathcal{J}_{\infty}$ by restricting the differential
$\partial$ to $(r\mathcal{H}_{\infty})_{1,\ast}$ or $(r\mathcal{H}_{\infty
})_{\ast,1}.$

Since $\rho\in\mathcal{ZJ}_{n,m}$ can be expressed as a matrix product
$
\rho=A_{1}\cdots A_{r}\cdot F\cdot B_{1}\cdots B_{s}
$
in which only $F$ contains $z,$ the element $\rho$ can be represented by a
piecewise linear path in $\mathbb{N}^{3}$ with $r+s+1$ directed components
from $(m+1,1,0)\in\mathbb{N}^{2}\times0$ to $(1,n+1,1)\in\mathbb{N}^{2}%
\times1$, where $B_{1}\cdots B_{s}$ is represented by a path in $\mathbb{N}%
^{2}\times0$ with directed components ${B}_{s},{B}_{s-1}%
,\ldots,{B}_{1}$, $F$ is represented by a directed component $\vec{F}$ from
$\mathbb{N}^{2}\times0$ to $\mathbb{N}^{2}\times1,$ and $A_{1}\cdots A_{r}$ is
represented by a path in $\mathbb{N}^{2}\times1$ with directed components
${A}_{r},{A}_{r-1},\ldots,{A}_{1}$ (see Figure 15).

\hspace{-.4in}
\unitlength=1.6mm \linethickness{0.4pt}
\begin{picture}(92.00,47.00)
	\put(32.00,9.00){\line(1,0){40.00}}
	\put(32.00,9.00){\circle*{1.00}}
	\put(52.00,9.00){\circle*{1.00}}
	\put(42.00,9.00){\circle*{1.00}}
	\put(32.00,5.00){\makebox(0,0)[cc]{$1$}}
	\put(42.00,5.00){\makebox(0,0)[cc]{$2$}}
	\put(52.00,5.00){\makebox(0,0)[cc]{$3$}} \ \
	\put(42.00,9.00){\vector(0,1){10.00}}
	\put(62.00,9.00){\vector(-1,0){10.00}}
	\put(52.00,9.00){\vector(-1,0){10.00}}
	\put(42.00,9.00){\vector(-1,0){10.00}}
	\put(32.00,9.00){\vector(0,1){10.00}}
	\put(14.33,22.00){\line(0,1){25.00}}
	\put(14.33,22.00){\circle*{1.00}}
	\put(24.33,22.00){\circle*{1.00}}
	\put(14.33,32.00){\circle*{1.00}}
	\put(14.33,42.00){\circle*{1.00}}
	\put(10.66,22.00){\makebox(0,0)[cc]{$1$}}
	\put(10.66,32.00){\makebox(0,0)[cc]{$2$}}
	\put(24.33,22.00){\vector(0,1){10.00}}
	\put(24.33,32.00){\vector(-1,0){10.00}}
	\put(24.33,22.00){\vector(-1,0){10.00}}
	\put(14.33,22.00){\vector(0,1){10.00}}
	\put(14.33,32.00){\vector(0,1){10.00}}
	\put(42.00,9.00){\line(-4,3){9.33}} \
	\put(31.33,16.67){\vector(-4,3){7.00}}
	\put(31.60,9.00){\vector(-4,3){17.27}}
	\put(31.60,19.00){\vector(-4,3){17.27}}
	\put(32.00,19.00){\circle*{1.33}}
	\put(42.00,19.00){\circle*{1.33}}
	\put(24.33,32.00){\circle*{1.33}}
	\put(32.00,19.00){\line(0,1){16.67}}
	\put(42.00,19.00){\line(-3,2){9.33}}
	\put(42.00,19.00){\vector(-1,0){10.00}}
	\put(52.00,19.00){\vector(-1,0){10.00}}
	\put(34.33,22.00){\vector(-1,0){10.00}}
	\put(24.33,22.00){\line(1,0){11.67}}
	\put(39.00,22.00){\line(1,0){16.33}}
	\put(52.00,19.00){\circle*{1.33}}
	\put(10.33,42.00){\makebox(0,0)[cc]{$3$}} \
	\put(39.00,24.20){\makebox(0,0)[cc]{$_{F^{2\times 2}}$}}
	\put(38.50,15.00){\makebox(0,0)[cc]{$_{F^{1\times 2}}$}}
	\put(20.33,24.20){\makebox(0,0)[cc]{$_{F^{2\times1}}$}}
	\put(20.67,13.67){\makebox(0,0)[cc]{$_{F^{1\times1}}$}} \
	\put(36.67,17.67){\makebox(0,0)[cc]{$_{A}$}}
	\put(44.00,13.67){\makebox(0,0)[cc]{$_{B}$}}
	\put(37.20,7.00){\makebox(0,0)[cc]{$_{A^{'}}$}}
	\put(30.33,14.33){\makebox(0,0)[cc]{$_{B^{'}}$}}
	\put(20.00,33.33){\makebox(0,0)[cc]{$_{C}$}} \
	\put(26.50,26.67){\makebox(0,0)[cc]{$_{D}$}}
	\put(12.50,27.00){\makebox(0,0)[cc]{$_{D^{'}}$}} \
	\put(21.90,20.60){\makebox(0,0)[cc]{$_{C^{'}}$}} \
	\put(34.33,22.00){\circle*{1.33}}
	\put(52.00,9.00){\vector(0,1){9.67}}
	\put(31.33,26.43){\vector(-4,3){7.00}}
\end{picture}
\vspace{-.6in}

\begin{center}
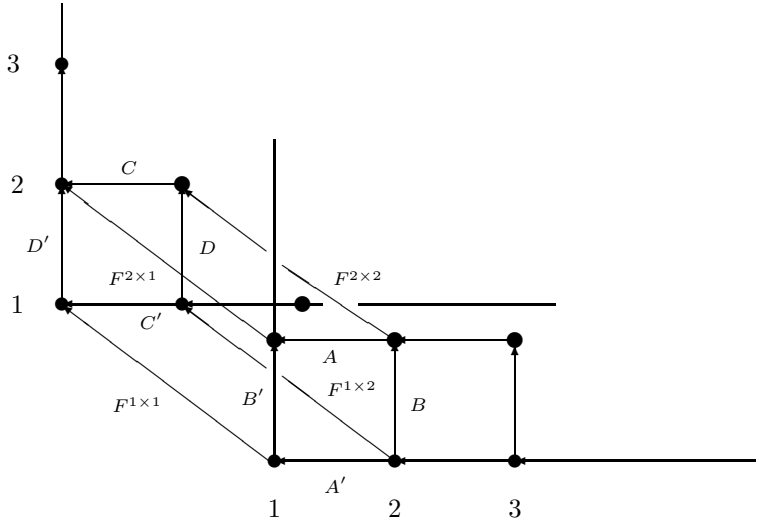
	Figure 15. The paths in $\mathbb{N}^{3}$ with directed 3 components
	corresponding to common vertices of $JJ_{2,2}$ and $P_{3}$.
\end{center}

Here are some pictures of $JJ_{n,m}$ for small values of $m$ and $n:$
\vspace{0.2in}

\unitlength=1.00mm \linethickness{0.4pt}
\ifx\plotpoint\undefined\newsavebox{\plotpoint}\fi
\begin{picture}(52.5,20)(0,0)
	\put(58.3, 14.5){\makebox(0,0)[cc]{$\bullet$}}
	\put(58.3, 18){\makebox(0,0)[cc]{1}}
	\put(58.3, 10.8){\makebox(0,0)[cc]{$\frac{0}{z}=\frac{z}{0}$}}
\end{picture}\vspace*{-0.6in}

\begin{center}
	Figure 16. The point $JJ_{1,1}.$\vspace{0.3in}
\end{center}
\unitlength=1.00mm \linethickness{0.4pt}
\begin{picture}(-40.33,13.33)
	\put(11.00,8.67){\line(1,0){43.00}}
	\put(11.33,8.67){\circle*{1.33}}
	\put(53.67,8.67){\circle*{1.33}}
	\put(11.33,13.33){\makebox(0,0)[cc]{$1|2$}}
	\put(11.33,03.33){\makebox(0,0)[cc]{$\frac{1|0}{0|z}$}}
	\put(53.67,13.00){\makebox(0,0)[cc]{$2|1$}}
	\put(53.67,03.00){\makebox(0,0)[cc]{$\frac{0|1}{z|0}$}}
	\put(71.00,8.67){\line(1,0){43.00}}
	\put(71.33,8.67){\circle*{1.33}}
	\put(113.67,8.67){\circle*{1.33}}
	\put(71.33,13.33){\makebox(0,0)[cc]{$1|2$}}
	\put(71.33,03.33){\makebox(0,0)[cc]{$\frac{0|z}{1|0}$}}
	\put(113.67,13.00){\makebox(0,0)[cc]{$2|1$}}
	\put(113.67,03.00){\makebox(0,0)[cc]{$\frac{z|0}{0|1}$}}
\end{picture}
\vspace{-0.1in}
\begin{center}
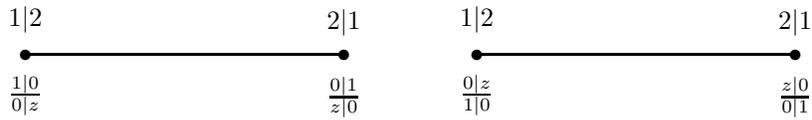
	Figure 17. The intervals $JJ_{2,1}$ and $JJ_{1,2}$.
\end{center}

\noindent\underline{For $JJ_{2,2}$}:

$
\begin{array}
	[c]{lllllllll}%
	1|23 & \leftrightarrow & \frac{0|2z}{\!\!\!1|0} & \leftrightarrow & \left[
	\begin{array}
		[c]{c}%
		\theta_{2}^{1}\smallskip\\
		\theta_{2}^{1}
	\end{array}
	\right]
	\left(  \left[  \!\!
	\begin{array}
		[c]{cc}
		\left[  \theta_{1}^{2}\right] \! \!\left[  \mathfrak{f}
		_{1}^{1}\right]	 & \mathfrak{f}_{1}^{2}
	\end{array}
	\!\!\right]
	+
	\left[  \begin{array}
		[c]{cc}
		\mathfrak{f}_{1}^{2}&
		\left[
		\begin{array}
			[c]{cc}
			\mathfrak{f}_{1}^{1}\smallskip  \\
			\mathfrak{f}_{1}^{1}
		\end{array}
		\right]  \!\!\left[  \theta_{1}^{2}\right]
	\end{array}
	\!\!\right]  \right)  \smallskip &  &  &  & \\
	13|2 & \leftrightarrow & \frac{\,\,\,0|2}{1z|0} & \leftrightarrow & \left(
	\left[
	\begin{array}
		[c]{c}
		\mathfrak{f}_{2}^{1}  \smallskip\\
		\left[  \theta_{2}^{1}\right]\!\!  \left[  \mathfrak{f}_{1}^{1}\,\,\mathfrak{f}
		_{1}^{1}\right]
	\end{array}
	\right]  +
	\left[
	\begin{array}
		[c]{c}
		\left[  \mathfrak{f}_{1}^{1}\right] \!\! \left[  \theta_{2}^{1}\right]
		\smallskip\\
		\mathfrak{f}_{2}^{1}
	\end{array}
	\right]  \right)
	\left[
	\begin{array}
		[c]{cc}%
		\theta_{1}^{2} & \theta_{1}^{2}%
	\end{array}
	\right]  \smallskip &  &  &  & \\
\end{array}
$

$
\begin{array}
	[c]{lllllllll}
	3|12 & \leftrightarrow & \frac{0|2}{z|1} & \leftrightarrow & \left[\!\!
	\begin{array}
		[c]{c}%
		\mathfrak{f}_{1}^{1}\smallskip\\
		\mathfrak{f}_{1}^{1}%
	\end{array}
	\!\!\right] \!\! [\theta_{2}^{2}]\smallskip &  &  &  & \\
	12|3 & \leftrightarrow & \frac{2|z}{1|0} & \leftrightarrow & \left[
	\theta_{2}^{2}\right] \!\! \left[  \mathfrak{f}_{1}^{1}\,\,\mathfrak{f}_{1}%
	^{1}\right]  \smallskip &  &  &  & \\
	2|13 & \leftrightarrow & \frac{2|z}{0|1} & \leftrightarrow & \left[
	\theta_{1}^{2}\right]\!\!  \left[  \mathfrak{f}_{2}^{1}\right]  \smallskip &  &  &
	& \\
	23|1 & \leftrightarrow & \frac{2|0}{z|1} & \leftrightarrow & \left[
	\mathfrak{f}_{1}^{2}\right] \!\! \left[  \theta_{2}^{1}\right]  \smallskip &  &  &
	& \\
	1|2|3 & \leftrightarrow & \frac{0|2|z}{1|0|0} & \leftrightarrow & \left[\!\!
	\begin{array}
		[c]{c}%
		\theta_{2}^{1}\smallskip\\
		\theta_{2}^{1}%
	\end{array}
	\!\!\right] \!\! \left[\!\!
	\begin{array}
		[c]{cc}%
		\theta_{1}^{2} & \theta_{1}^{2}%
	\end{array}
	\!\!	\right] \!\! \left[\!\!
	\begin{array}
		[c]{cc}%
		\mathfrak{f}_{1}^{1} & \mathfrak{f}_{1}^{1}%
	\end{array}
	\!\!\right]  \smallskip &  &  &  & \\
	1|3|2 & \leftrightarrow & \frac{0|z|2}{1|0|0} & \leftrightarrow & \left[\!\!
	\begin{array}
		[c]{c}%
		\theta_{2}^{1}\smallskip\\
		\theta_{2}^{1}%
	\end{array}
	\!\!\right]  \!\!\left[\!\!
	\begin{array}
		[c]{cc}%
		\mathfrak{f}_{1}^{1}\smallskip & \mathfrak{f}_{1}^{1}\\
		\mathfrak{f}_{1}^{1} & \mathfrak{f}_{1}^{1}%
	\end{array}
	\!\!\right] \!\! \left[\!\!
	\begin{array}
		[c]{cc}%
		\theta_{1}^{2} & \theta_{1}^{2}%
	\end{array}
	\!\!\right]  \smallskip &  &  &  & \\
	3|1|2 & \leftrightarrow & \frac{0|0|2}{z|1|0} & \leftrightarrow & \left[\!\!
	\begin{array}
		[c]{c}%
		\mathfrak{f}_{1}^{1}\smallskip\\
		\mathfrak{f}_{1}^{1}%
	\end{array}
	\!\!\right] \!\! \left[\!\!
	\begin{array}
		[c]{c}%
		\theta_{2}^{1}\smallskip\\
		\theta_{2}^{1}%
	\end{array}
	\!\!\right]\!\!  \left[\!\!
	\begin{array}
		[c]{cc}%
		\theta_{1}^{2} & \theta_{1}^{2}%
	\end{array}
	\!\!\right]  \smallskip &  &  &  & \\
	2|1|3 & \leftrightarrow & \frac{2|0|z}{0|1|0} & \leftrightarrow & \left[
	\theta_{1}^{2}\right]\!\!  \left[  \theta_{2}^{1}\right] \!\! \left[\!\!
	\begin{array}
		[c]{cc}%
		\mathfrak{f}_{1}^{1} & \mathfrak{f}_{1}^{1}%
	\end{array}
	\!\!\right]  \smallskip &  &  &  & \\
	2|3|1 & \leftrightarrow & \frac{2|z|0}{0|0|1} & \leftrightarrow & \left[
	\theta_{1}^{2}\right]\!\!  \left[  \mathfrak{f}_{1}^{1}\right]
	\!\!  \left[  \theta
	_{2}^{1}\right]  \smallskip &  &  &  & \\
	3|2|1 & \leftrightarrow & \frac{z|2|0}{0|0|1} & \leftrightarrow & \left[\!\!
	\begin{array}
		[c]{c}%
		\mathfrak{f}_{1}^{1}\smallskip\\
		\mathfrak{f}_{1}^{1}%
	\end{array}
	\!\!\right] \!\! \left[  \theta_{1}^{2}\right] \!\! \left[  \theta_{2}^{1}\right]   &  &
	&  &
\end{array}
$

$
\begin{array}
	[c]{rllll}%
	v_{1} & \leftrightarrow & \left(  \frac{0|2z}{\!\!\!1|0},\left(
	\begin{array}
		[c]{c}%
		\frac{0}{1}\vspace{1mm}\\
		\frac{0}{1}\vspace{1mm}\\
		\frac{0}{1}%
	\end{array}
	\right) \!\! \left(  \frac{2|z}{0|0}\,\frac{z|2}{0|0}\right)  \right)   &
	\leftrightarrow & \left[\!\!
	\begin{array}
		[c]{c}%
		\theta_{2}^{1}\smallskip\\
		\theta_{2}^{1}%
	\end{array}
	\!\!\right] \!\! \left[\!\!
	\begin{array}
		[c]{cc}%
		\left[  \theta_{1}^{2}\right]\!\!  \left[  \mathfrak{f}_{1}^{1}\right]   & \left[\!\!
		\begin{array}
			[c]{c}%
			\mathfrak{f}_{1}^{1}\smallskip\\
			\mathfrak{f}_{1}^{1}%
		\end{array}
		\!\!\right] \!\! \left[  \theta_{1}^{2} \right]
	\end{array}
	\!\!\right]  \bigskip\\
	v_{2} & \leftrightarrow & \left(  \frac{\,\,\,0|2}{1z|0},\left(
	\begin{array}
		[c]{c}%
		\frac{0|0}{z|1}\vspace{1mm}\\
		\frac{0|0}{1|z}%
	\end{array}
	\right) \!\! \left(  \frac{2}{0}\,\frac{2}{0}\,\frac{2}{0}\right)  \right)   &
	\leftrightarrow & \left[\!\!
	\begin{array}
		[c]{c}
		\left[  \mathfrak{f}_{1}^{1}\right] \!\! \left[  \theta_{2}^{1}\right]\vspace{1mm}\\
		\left[  \theta_{2}^{1}\right] \!\! \left[
		\begin{array}
			[c]{cc}
			\mathfrak{f}_{1}^{1} & \mathfrak{f}_{1}^{1}
		\end{array}
		\right]
	\end{array}
	\!\!\right]\!\!
	\left[\!\!
	\begin{array}
		[c]{cc}%
		\theta_{1}^{2} & \theta_{1}^{2}%
	\end{array}
	\!\!\right]  ,
\end{array}
\hspace*{0.45in}%
$
\vspace{.05in}

where $v_{1}$ and $v_{2}$ are the respective midpoints of the edges $1|23$
and $13|2$ in $P_{3}$.
\vspace{.4in}

\hspace{-.3in}
\unitlength=1.00mm \linethickness{0.4pt}
\begin{picture}(88.33,43.67)
	\put(38.67,40.67){\line(1,0){48.33}} \put(87.00,40.67){\line(0,-1){31.67}}
	\put(87.00,9.00){\line(-1,0){48.33}} \put(38.67,9.00){\line(0,1){31.67}}
	\put(38.67,40.67){\circle*{1.33}} \put(38.67,25.00){\circle*{1.33}}
	\put(38.67,9.00){\circle*{1.33}} \put(87.00,40.67){\circle*{1.33}}
	\put(87.00,9.00){\circle*{1.33}} \put(87.00,25.00){\circle*{1.33}}
	\put(32.00,33.33){\makebox(0,0)[cc]{$13|2$}}
	\put(43.3,33.33){\makebox(0,0)[cc]{$\frac{\,\,\,0|2}{1z|0}$}}
	\put(32.00,17.00){\makebox(0,0)[cc]{$1|23$}}
	\put(43.3,17.00){\makebox(0,0)[cc]{$\frac{0|2z}{\!\!\!1|0}$}}
	\put(61.00,5.00){\makebox(0,0)[cc]{$12|3 $}}
	\put(61.00,13.00){\makebox(0,0)[cc]{$\frac{2|z}{1|0}$}}
	\put(62.00,43.67){\makebox(0,0)[cc]{$ 3|12 $}}
	\put(62.00,35.67){\makebox(0,0)[cc]{$\frac{0|2}{z|1}$}}
	\put(93.33,17.00){\makebox(0,0)[cc]{$2|13  $}}
	\put(83.0,17.00){\makebox(0,0)[cc]{$\frac{2|z}{0|1}$}}
	\put(93.33,32.67){\makebox(0,0)[cc]{$23|1  $}}
	\put(83.0,32.67){\makebox(0,0)[cc]{$\frac{2|0}{ z|1}$}}
	\put(38.67,32.33){\circle*{0.67}} \put(38.67,17.33){\circle*{0.67}}
\end{picture}

\vspace*{-0.5in}

\begin{center}
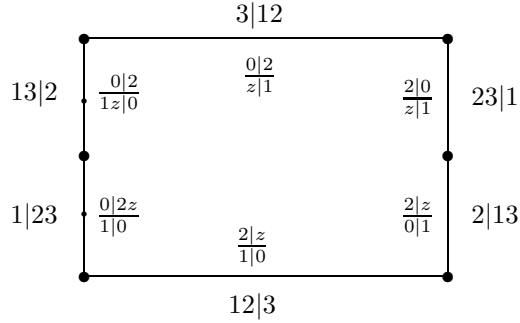
	Figure 18. The octagon $JJ_{2,2}$ as a subdivision of $P_{3}$.
\end{center}
\pagebreak

\noindent\underline{For $JJ_{2,3}$:}

$
\begin{array}
	[c]{lllll}%
	123|4 & \leftrightarrow & \frac{\,\,\,3|z}{12|0} & \leftrightarrow &
	\theta_{3}^{2}/\mathfrak{f}_{1}^{1}\mathfrak{f}_{1}^{1}\mathfrak{f}_{1}%
	^{1}\vspace{1mm}\\
	4|123 & \leftrightarrow & \frac{\!\!\!0|3}{z|12} & \leftrightarrow &
	\mathfrak{f}_{1}^{1}\mathfrak{f}_{1}^{1}/\theta_{3}^{2}\vspace{1mm}\\
	13|24 & \leftrightarrow & \frac{\,\,\,3|z}{1|2} & \leftrightarrow & \theta
	_{2}^{2}/\mathfrak{f}_{1}^{1}\mathfrak{f}_{2}^{1}\vspace{1mm}\\
	134|2 & \leftrightarrow & \frac{\,\,\,3|0}{1z|2} & \leftrightarrow &
	\mathfrak{f}_{2}^{2}/\theta_{1}^{1}\theta_{2}^{1}\vspace{1mm}\\
	3|124 & \leftrightarrow & \frac{\!\!\!3|z}{0|12} & \leftrightarrow &
	\theta_{1}^{2}/\mathfrak{f}_{3}^{1}\vspace{1mm}\\
	34|12 & \leftrightarrow & \frac{3|0}{\,\,\,z|12} & \leftrightarrow &
	\mathfrak{f}_{1}^{2}/\theta_{3}^{1}\vspace{1mm}\\
	1|234 & \leftrightarrow & \frac{\,\,\,0|3z}{1|2} & \leftrightarrow &
	\theta_{2}^{1}\theta_{2}^{1}/\left(  \theta_{1}^{2}\mathfrak{f}_{1}%
	^{1}\right)  \mathfrak{f}_{2}^{2}\\
	14|23 & \leftrightarrow & \frac{\,\,\,0|3}{1z|2} & \leftrightarrow &
	[(\mathfrak{f}_{1}^{1}/\theta_{2}^{1})\mathfrak{f}_{2}^{1}+\mathfrak{f}%
	_{2}^{1}\left(  \theta_{2}^{1}/\mathfrak{f}_{1}^{1}\mathfrak{f}_{1}%
	^{1}\right)  ]/\theta_{1}^{2}\theta_{2}^{2}\\
	2|134 & \leftrightarrow & \frac{\,\,\,0|3z}{2|1} & \leftrightarrow &
	\theta_{2}^{1}\theta_{2}^{1}/\mathfrak{f}_{2}^{2}(\mathfrak{f}_{1}%
	^{1}\mathfrak{f}_{1}^{1}/\theta_{1}^{2})\\
	24|13 & \leftrightarrow & \frac{\,\,\,0|3}{2z|1} & \leftrightarrow &
	[(\mathfrak{f}_{1}^{1}/\theta_{2}^{1})\mathfrak{f}_{2}^{1}+\mathfrak{f}%
	_{2}^{1}\left(  \theta_{2}^{1}/\mathfrak{f}_{1}^{1}\mathfrak{f}_{1}%
	^{1}\right)  ]/\theta_{2}^{2}\theta_{1}^{2}\\
	23|14 & \leftrightarrow & \frac{3|z}{2|1} & \leftrightarrow & \theta_{2}%
	^{2}/\mathfrak{f}_{2}^{1}\mathfrak{f}_{1}^{1}\\
	234|1 & \leftrightarrow & \frac{\,\,\,3|0}{2z|1} & \leftrightarrow &
	\mathfrak{f}_{2}^{2}/\theta_{2}^{1}\theta_{1}^{1}%
\end{array}
$

\hspace{0.01in}$%
\begin{array}
	[c]{lllll}%
	12|34 & \leftrightarrow & \frac{\,\,\,0|3z}{\!\!\!12|0} & \leftrightarrow &
	[\left(  \theta_{2}^{1}/\theta_{2}^{1}\theta_{1}^{1}\right)  \theta_{3}%
	^{1}+\theta_{3}^{1}\left(  \theta_{2}^{1}/\theta_{1}^{1}\theta_{2}^{1}\right)
	]/\\
	&  &  &  & [(\theta_{1}^{2}/\mathfrak{f}_{1}^{1})(\theta_{1}^{2}%
	/\mathfrak{f}_{1}^{1})\mathfrak{f}_{1}^{2}+\,(\theta_{1}^{2}/\mathfrak{f}%
	_{1}^{1})\mathfrak{f}_{1}^{2}(\mathfrak{f}_{1}^{1}\mathfrak{f}_{1}^{1}%
	/\theta_{1}^{2})\vspace{1mm}\\
	&  &  &  & +\mathfrak{f}_{1}^{2}(\mathfrak{f}_{1}^{1}\mathfrak{f}_{1}%
	^{1}/\theta_{1}^{2})(\mathfrak{f}_{1}^{1}\mathfrak{f}_{1}^{1}/\theta_{1}%
	^{2})]\vspace{1mm}\\
	124|3 & \leftrightarrow & \frac{\,\,\,0|3}{\!\!\!12z|0} & \leftrightarrow &
	[(\theta_{2}^{1}/\theta_{1}^{1}\theta_{2}^{1}/\mathfrak{f}_{1}^{1}%
	\mathfrak{f}_{1}^{1}\mathfrak{f}_{1}^{1})\mathfrak{f}_{3}^{1}+\mathfrak{f}%
	_{3}^{1}(\mathfrak{f}_{1}^{1}/\theta_{2}^{1}/\theta_{2}^{1}\theta_{1}^{1})\\
	&  &  &  & +(\theta_{3}^{1}/\mathfrak{f}_{1}^{1}\mathfrak{f}_{1}%
	^{1}\mathfrak{f}_{1}^{1})(\theta_{2}^{1}/\mathfrak{f}_{2}^{1}\mathfrak{f}%
	_{1}^{1})+(\theta_{3}^{1}/\mathfrak{f}_{1}^{1}\mathfrak{f}_{1}^{1}%
	\mathfrak{f}_{1}^{1})(\mathfrak{f}_{2}^{1}/\theta_{2}^{1}\theta_{1}^{1})\\
	&  &  &  & +(\theta_{2}^{1}/\mathfrak{f}_{2}^{1}\mathfrak{f}_{1}%
	^{1})(\mathfrak{f}_{2}^{1}/\theta_{2}^{1}\theta_{1}^{1})-(\theta_{2}%
	^{1}/\mathfrak{f}_{1}^{1}\mathfrak{f}_{2}^{1})(\mathfrak{f}_{2}^{1}/\theta
	_{1}^{1}\theta_{2}^{1})\\
	&  &  &  & -(\theta_{2}^{1}/\mathfrak{f}_{1}^{1}\mathfrak{f}_{2}%
	^{1})(\mathfrak{f}_{1}^{1}/\theta_{3}^{1})-(\mathfrak{f}_{2}^{1}/\theta
	_{1}^{1}\theta_{2}^{1})(\mathfrak{f}_{1}^{1}/\theta_{3}^{1})]/\,\theta_{1}%
	^{2}\theta_{1}^{2}\theta_{1}^{2}%
\end{array}
$
\vspace{.5in}

\hspace{-.27in}\unitlength 1mm 
\linethickness{0.5pt}
\ifx\plotpoint\undefined\newsavebox{\plotpoint}\fi 
\begin{picture}(99.547,71.641)(0,0)
	\thicklines
	\put(19.547,70.891){\line(1,0){44.75}}
	\put(54.297,51.891){\line(1,0){44.75}}
	\put(54.297,17.141){\line(1,0){44.75}}
	\multiput(64.047,70.891)(.0765864333,-.0421225383){457}{\line(1,0){.0765864333}}
	\multiput(19.547,70.891)(.0765864333,-.0421225383){457}{\line(1,0){.0765864333}}
	\multiput(19.547,36.141)(.0765864333,-.0421225383){457}{\line(1,0){.0765864333}}
	\put(19.547,70.891){\line(0,-1){34.25}}
	\put(98.797,51.641){\line(0,-1){34.75}}
	\multiput(37.297,49.641)(.076136364,-.042045455){220}{\line(1,0){.076136364}}
	\put(54.047,34.391){\line(1,0){44.75}}
	\thinlines
	\put(78.047,51.641){\line(0,-1){34.25}}
	\put(98.797,43.141){\line(-1,0){44.75}}
	\multiput(88.047,43.141)(.076923077,.041958042){143}{\line(1,0){.076923077}}
	\thicklines
	\put(86.047,51.391){\line(2,-1){8.5}}
	\multiput(95.547,46.641)(.09027778,-.04166667){36}{\line(1,0){.09027778}}
	\multiput(84.547,52.391)(-.0694444,.0416667){18}{\line(-1,0){.0694444}}
	\thinlines
	\multiput(54.047,43.391)(.185114504,.041984733){131}{\line(1,0){.185114504}}
	\put(54.047,17.391){\circle*{1.414}}
	\put(66.047,43.141){\circle*{.5}}
	\put(78.047,43.141){\circle*{.707}}
	\put(54.047,43.391){\circle*{1}}
	\put(66.047,43.141){\circle*{1}}
	\put(78.047,43.141){\circle*{1}}
	\put(78.047,51.891){\circle*{1}}
	\put(78.047,48.641){\circle*{1}}
	\put(98.797,48.891){\circle*{1}}
	\put(83.297,56.391){\circle*{1}}
	\thicklines
	\put(37.297,61.141){\line(0,-1){35.25}}
	\multiput(19.797,48.891)(.079545455,-.042045455){220}{\line(1,0){.079545455}}
	\put(37.297,44.391){\circle*{1}}
	\put(19.547,62.641){\line(1,0){14}}
	\put(35.547,62.641){\line(1,0){28.75}}
	\put(64.047,70.891){\line(0,-1){18.25}}
	\put(64.297,41.891){\line(0,-1){5.5}}
	\put(83.297,25.391){\line(0,1){8.75}}
	\put(83.297,39.891){\line(0,1){3}}
	\put(83.297,43.641){\line(0,1){8}}
	\put(83.297,52.391){\line(0,1){8}}
	\put(54.297,36.141){\line(1,0){10.5}}
	\put(53.547,36.141){\line(-1,0){16}}
	\put(36.797,36.141){\line(-1,0){17.25}}
	\put(64.297,36.641){\circle*{1.581}}
	\put(19.547,36.141){\circle*{1.5}}
	\put(37.297,39.641){\circle*{1.5}}
	\put(19.547,49.141){\circle*{1.5}}
	\put(19.547,62.641){\circle*{1.581}}
	\put(19.547,70.641){\circle*{1.5}}
	\put(64.047,70.891){\circle*{1.5}}
	\put(98.797,34.391){\circle*{1.5}}
	\put(54.047,34.391){\circle*{1.5}}
	\put(98.547,51.641){\circle*{1.5}}
	\put(37.297,60.891){\circle*{1.5}}
	\put(37.297,49.391){\circle*{1.5}}
	\put(98.797,45.141){\circle*{1.5}}
	\put(83.297,53.141){\circle*{1.5}}
	\put(98.797,17.391){\circle*{1.5}}
	\put(54.051,51.771){\line(0,-1){34.477}}
	\thinlines
	\put(54.109,27.946){\line(1,0){44.595}}
	\put(54.109,22.298){\circle*{1.189}}
	\multiput(53.811,37.46)(.090970149,.042156716){134}{\line(1,0){.090970149}}
	\put(54.109,37.757){\circle*{1.189}}
	\multiput(37.163,55.001)(.0605381818,-.0421636364){275}{\line(1,0){.0605381818}}
	\multiput(98.406,48.757)(-.08825,.0418125){64}{\line(-1,0){.08825}}
	\put(64.217,51.433){\line(0,-1){5.5}}
	\put(64.068,62.582){\circle*{1.189}}
	\put(54.109,51.879){\circle*{1.189}}
	\put(83.244,60.203){\circle*{1.189}}
	\put(83.393,25.865){\circle*{1.189}}
	\put(54.109,40.284){\circle*{1.487}}
	\put(98.852,42.96){\circle*{.892}}
	\put(37.311,54.852){\circle*{.892}}
	\put(54.109,27.946){\circle*{.892}}
	\put(98.852,27.798){\circle*{.892}}
	\put(78.041,37.46){\circle*{.892}}
	\put(78.041,34.338){\circle*{.892}}
	\put(78.041,27.649){\circle*{.892}}
	\put(78.041,17.243){\circle*{.892}}
	\put(54.109,22.595){\line(-1,0){.149}}
	\multiput(54.109,22.446)(.1115,-.037){4}{\line(1,0){.1115}}
	\put(98.852,22.298){\line(-1,0){44.892}}
	\put(78.041,22.298){\circle*{.892}}
	\put(98.852,22.298){\circle*{.892}}
	\put(37.311,26.311){\circle*{1.487}}
	\put(83.244,46.825){\circle*{.892}}
	\multiput(83.414,56.402)(.080738248,-.042051171){101}{\line(1,0){.080738248}}
	\multiput(78.148,37.29)(.15647462,.042152347){133}{\line(1,0){.15647462}}
	\put(83.329,38.309){\line(0,-1){3.483}}
	\put(83.329,39.244){\line(0,1){.595}}
	\put(64.302,45.445){\line(0,-1){2.039}}
	\put(83.329,39.838){\circle*{1.189}}
	\put(88.171,43.151){\circle*{.849}}
	\put(88.171,43.066){\circle*{.537}}
	\put(83.329,38.989){\line(0,1){.7645}}
	\multiput(90.72,21.915)(.076853846,-.042067368){105}{\line(1,0){.076853846}}
	\multiput(68.55,33.977)(.072512936,-.042126563){123}{\line(1,0){.072512936}}
	\multiput(79.507,27.522)(.084943725,-.042102542){115}{\line(1,0){.084943725}}
	\multiput(67.53,34.657)(-.06984262,.04152804){45}{\line(-1,0){.06984262}}
	\put(64.217,49.862){\circle*{1.019}}
	\multiput(78.488,42.472)(.07550553,-.04179771){63}{\line(1,0){.07550553}}
	\multiput(64.133,49.947)(.08625055,-.04181845){65}{\line(1,0){.08625055}}
	\multiput(70.503,46.804)(.08279325,-.04193424){79}{\line(1,0){.08279325}}
\end{picture}
\vspace{-0.5in}
\begin{center}
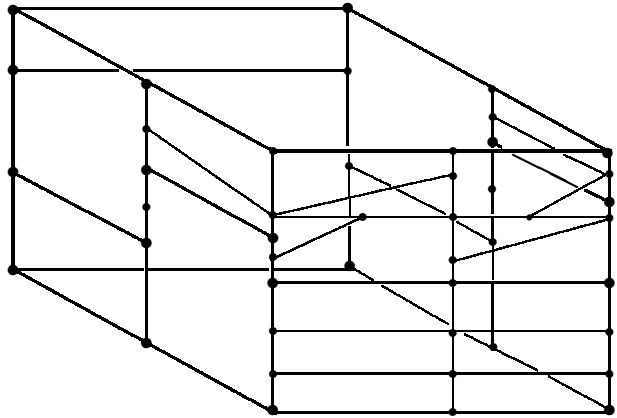
	Figure 19. The bimultiplihedra $JJ_{2,3}$ as a subdivision of $P_{4}$.
\end{center}
\bigskip

\section{Morphisms and the Transfer of $A_\infty$-Structure}

In this section we define the morphisms of $A_{\infty}$-bialgebras. But before
we begin, we mention three settings in which $A_{\infty}$-bialgebras naturally
appear:\vspace{0.1in}

\noindent(1) Let $X$ be a space. The bar construction $\Omega S_{\ast}(X)$
on the simplicial singular chain complex of $X$ is a DG bialgebra with
coassociative coproduct \cite{Baues1}, \cite{CM}, \cite{KS1}, but whether or
not $\Omega^{2}S_{\ast}(X)$ admits a coassociative coproduct is unknown.
However, there is an $A_{\infty}$-coalgebra structure on $\Omega^{2}S_{\ast
}(X),$ which is compatible with the product, and $\Omega^{2}S_{\ast}(X)$ is an
$A_{\infty}$-bialgebra.\vspace{0.1in}

\noindent(2) Let $R$ be a commutative ring with unity, let $H$ be a graded $R$-bialgebra with nontrivial product and
coproduct, and let $\rho:RH\longrightarrow H$ be a (bigraded) free resolution as an $R$-algebra. Since $RH$ cannot be simultaneously free and cofree, it is
difficult to introduce a coassociative coproduct on $RH$ so that $\rho$ is a
bialgebra map. However, there is always an $A_{\infty}$-bialgebra structure on
$RH$ such that $\rho$ is a morphism of $A_{\infty}$-bialgebras.\vspace{0.1in}

\noindent(3) If $A$ is an $A_{\infty}$-bialgebra over a field, and
$g:H(A)\rightarrow A$ is a cycle-selecting homomorphism, there is an
$A_{\infty}$-bialgebra structure on $H(A),$ which is unique up to isomorphism,
and a morphism $G:H(A)\Rightarrow A$ of $A_{\infty}$-bialgebras extending $g$
(see Theorem \ref{AAHopf}).\vspace*{0.1in}

\begin{definition}\label{gmorphism}
	Let $(A,\omega_{A})$ and $(B,\omega_{B})$ be $A_{\infty}$-bialgebras. An
	element
	\[
	G=\{g_{m}^{n}\in Hom^{m+n-2}(A^{\otimes m},B^{\otimes n})\}_{m,n\geq1}\in
	U_{A,B}%
	\]
	is an $A_{\infty}$-\textbf{bialgebra morphism from} $A$ \textbf{to} $B$ if the
	map $\mathfrak{f}_{m}^{n}\mapsto g_{m}^{n}$ extends to a map $r\mathcal{H}%
	_{\infty}\rightarrow U_{A,B}$ of relative prematrads. When $G$ is an $A_{\infty}$-bialgebra morphism from $A$ \textbf{to} $B$	we write $G:A\Rightarrow B$.
	An $A_{\infty}$-bialgebra morphism $\Phi
	=\{\phi_{m}^{n}\}_{m,n\geq1}:A\Rightarrow B$ is an \textbf{isomorphism} if
	$\phi_{1}^{1}:A\rightarrow B{\ }$is an isomorphism of underlying modules.
\end{definition}

If $A$ is a free DGM, $B$ is an $A_{\infty}$-coalgebra, and $g:A\rightarrow B$
is a homology isomorphism (weak equivalence) with a right-homotopy inverse,
the Coalgebra Perturbation Lemma (CPL) transfers the $A_{\infty}$-coalgebra
structure from $B$ to $A$ (see \cite{Hueb-Kade}, \cite{Markl1}). When $B$ is
an $A_{\infty}$-bialgebra, Theorem \ref{transfer} generalizes the CPL in two directions:

\begin{enumerate}
	\item The $A_{\infty}$-bialgebra structure transfers from $B$ to $A.$
	
	\item Neither freeness nor a existence of a right-homotopy inverse is required.
\end{enumerate}

\noindent Note that (2) formulates the transfer of $A_{\infty} $-algebra
structure in maximal generality (see Remark \ref{twoinfty}).

\begin{proposition}
	\label{homology-iso}Let $A$ and $B$ be DGMs. If $g:A\rightarrow B$ is a chain
	map and $u\in Hom\left(  A^{\otimes m},A^{\otimes n}\right)  $, the induced
	map $\tilde{g}:U_{A}\rightarrow U_{A,B}$ defined by $\tilde{g}\left(
	u\right)  =g^{\otimes n}u$ is a cochain map. If in addition, $g$ is a homology
	isomorphism, $\tilde{g}$ is a cohomology isomorphism if either
	
	\begin{enumerate}
		\item[\textit{(i)}] $A$ is free as an $R$-module or
		
		\item[\textit{(ii)}] for each $n\geq1,$ there is a DGM\ $X\left(  n\right)  $
		and a splitting $B^{\otimes n}=A^{\otimes n}\oplus X(n)$ as a chain complex
		such that $H^{\ast}Hom\left(  A^{\otimes k},X\left(  n\right)  \right)  =0$
		for all $k\geq1.$
	\end{enumerate}
\end{proposition}

The proof is left to the reader.

\begin{theorem}
	\label{transfer}Let $\left(  A,d_{A}\right)  $
	be a DGM, let $(B,d_{B},\omega_{B})$ be an $A_{\infty}$-bialgebra, and let
	$g:A\rightarrow B $ be a chain map/homology isomorphism. If $\tilde{g}$ is a
	cohomology isomorphism, then
	
	\begin{enumerate}
		\item[\textit{(i)}] (Existence) $g$ induces an $A_{\infty}$-bialgebra
		structure $\omega_{A}=\{\omega_{A}^{n,m}\}$ on $A$, and extends to a map
		$G=\{g_{m}^{n}\mid g_{1}^{1}=g\}:A\Rightarrow B$\ of $A_{\infty}$-bialgebras.
		
		\item[\textit{(ii)}] (Uniqueness) $\left(  \omega_{A},G\right)  $ is unique up
		to isomorphism, i.e., if $\left(  \omega_{A},G\right)  $ and $\left(
		\bar{\omega}_{A},\bar{G}\right)  $ are induced by chain homotopic maps $g$ and
		$\bar{g}$, there is an isomorphism\ of $A_{\infty}$-bialgebras $\Phi:\left(
		A,\bar{\omega}_{A}\right)  \Rightarrow\left(  A,\omega_{A}\right)  $ and a
		chain homotopy $T:\bar{G}\simeq G\circ\Phi.$
	\end{enumerate}
\end{theorem}

\begin{proof} \ \textbf{(The Transfer Algorithm)}
	We obtain the desired structures by simultaneously constructing a map of
	matrads $\alpha_{A}:C_{\ast}\left(  KK\right)  \rightarrow(U_{A},\nabla)$ and
	a map of relative matrads $\beta:C_{\ast}\left(  JJ\right)  \rightarrow
	(U_{A,B},\nabla)$. Thinking of $JJ_{n,m}$ as a subdivision of the cylinder
	$KK_{n,m}\times I,$ identify the top dimensional cells of $KK_{n,m}$ and
	$JJ_{n,m}$ with $\theta_{m}^{n}$ and $\mathfrak{f}_{m}^{n}$, and the faces
	$KK_{n,m}\times0$ and $KK_{n,m}\times1$ of $JJ_{n,m}$ with $\theta_{m}%
	^{n}\left(  \mathfrak{f} _{1}^{1}\right)  ^{\otimes m}$ and $\left(
	\mathfrak{f}_{1}^{1}\right)  ^{\otimes n}\theta_{m}^{n},$ respectively. By
	hypothesis, there is a map of matrads $\alpha_{B}:C_{\ast}(KK)\rightarrow
	(U_{B},\nabla)$ such that $\alpha_{B}(\theta_{m}^{n})=\omega_{B}^{n,m}.$
	
	To initialize the induction, define $\mathcal{\beta}:C_{\ast}\left(
	JJ_{1,1}\right)  \rightarrow Hom^{0}\left(  A,B\right)  $ by $\beta\left(
	\mathfrak{f}_{1}^{1}\right)  =g_{1}^{1}=g$, and extend $\mathcal{\beta}$ to
	$C_{\ast}\left(  JJ_{1,2}\right)  \rightarrow Hom^{1}\left(  A^{\otimes
		2},B\right)  $ and $C_{\ast}\left(  JJ_{2,1}\right)  \rightarrow
	Hom^{1}\left(  A,B^{\otimes2}\right)  $ in the following way: On the vertices
	$\theta_{2}^{1}\left(  \mathfrak{f}_{1}^{1}\otimes\mathfrak{f}_{1}^{1}\right)
	\in JJ_{1,2}$ and $\theta_{1}^{2}\mathfrak{f}_{1}^{1}\in JJ_{2,1}$, define
	$\beta\left(  \theta_{2}^{1}\left(  \mathfrak{f}_{1}^{1}\otimes\mathfrak{f}
	_{1}^{1}\right)  \right)  =\omega_{B}^{1,2}\left(  g\otimes g\right)  $ and
	$\beta\left(  \theta_{1}^{2}\mathfrak{f}_{1}^{1}\right)  =\omega_{B}^{2,1}g.$
	Since $\omega_{B}^{1,2}\left(  g\otimes g\right)  $ and $\omega_{B}^{2,1}g$
	are $\nabla$-cocycles, and $\tilde{g}_{\ast}$ is an isomorphism, there exist
	cocycles $\omega_{A}^{1,2}$ and $\omega_{A}^{2,1}$ in $U_{A}$ such that
	\[
	\tilde{g}_{\ast}[\omega_{A}^{1,2}]=[\omega_{B}^{1,2}\left(  g\otimes g\right)
	]\text{ \ and \ }\tilde{g}_{\ast}[\omega_{A}^{2,1}]=[\omega_{B}^{2,1}g].
	\]
	Thus $\left[  \omega_{B}^{1,2}\left(  g\otimes g\right)  -g\omega_{A}
	^{1,2}\right]  =\left[  \omega_{B}^{2,1}g-\left(  g\otimes g\right)
	\omega_{A}^{2,1}\right]  =0,$ and there exist cochains $g_{2}^{1}$ and
	$g_{1}^{2}$ in $U_{A,B} $ such that
	\[
	\nabla g_{2}^{1}=\omega_{B}^{1,2}\left(  g\otimes g\right)  -g\omega_{A}
	^{1,2}\text{ \ and \ }\nabla g_{1}^{2}=\omega_{B}^{2,1}g-\left(  g\otimes
	g\right)  \omega_{A}^{2,1}.
	\]
	For $m=1,2$ and $n=3-m,$ define $\alpha_{A}:C_{\ast}\left(  KK_{n,m}\right)
	\rightarrow Hom\left(  A^{\otimes m},A^{\otimes n}\right)  $ by $\alpha
	_{A}(\theta_{m}^{n})=\omega_{A}^{n,m}$, and define $\beta:C_{\ast}\left(
	JJ_{n,m}\right)  \rightarrow Hom\left(  A^{\otimes m},B^{\otimes n}\right)  $
	by
	\[%
	\begin{array}
		[c]{rllll}%
		\beta( \mathfrak{f}_{m}^{n}) & = & g_{m}^{n} \vspace{1mm} &  & \\
		\beta( \mathfrak{f}_{1}^{1}\, \theta_{2}^{1}) & = & g\, \omega_{A}^{1,2} & (
		m=2 ) \vspace{1mm} & \\
		\beta( ( \mathfrak{f}_{1}^{1}\otimes\mathfrak{f}_{1}^{1})\, \theta_{1}^{2}) &
		= & ( g\otimes g )\, \omega_{A}^{2,1} & (m=1). &
	\end{array}
	\
	\]

	Inductively, given $\left(  m,n\right)  ,$ $m+n\geq4,$ assume that for
	$i+j<m+n$ there exists a map of matrads $\alpha_{A}:C_{\ast}\left(
	KK_{j,i}\right)  \rightarrow Hom\left(  A^{\otimes i},A^{\otimes j}\right)  $
	and a map of relative matrads $\beta:C_{\ast}\left(  JJ_{j,i}\right)
	\rightarrow Hom\left(  A^{\otimes i},B^{\otimes j}\right)  $ such that
	$\alpha_{A}(\theta_{i}^{j})=\omega_{A}^{j,i}$ and $\beta(\mathfrak{f}_{i}%
	^{j})=g_{i}^{j}$. Thus we are given chain maps $\alpha_{A}:C_{\ast}\left(
	\partial KK_{n,m}\right)  \rightarrow Hom\left(  A^{\otimes m},A^{\otimes
		n}\right)  $ and $\beta:C_{\ast}\left(  \partial JJ_{n,m}\smallsetminus
	\operatorname*{int}KK_{n,m}\times1\right)  \rightarrow Hom\left(  A^{\otimes
		m},B^{\otimes n}\right)  ;$ we wish to extend $\alpha_{A}$ to the top cell
	$\theta_{m}^{n}$ of $KK_{n,m}$, and $\beta$ to the codimension 1 cell $\left(
	\mathfrak{f}_{1}^{1}\right)  ^{\otimes n}\theta_{m}^{n}$ and the top cell
	$\mathfrak{f}_{m}^{n}$ of $JJ_{n,m}$. Since $\alpha_{A}$ is a map of matrads,
	the components of the cocycle
	\[
	z=\alpha_{A}\left(  C_{\ast}(\partial KK_{n,m})\right)  \in Hom^{m+n-4}\left(
	A^{\otimes m},A^{\otimes n}\right)
	\]
	are expressed in terms of $\omega_{A}^{j,i}$ with $i+j<m+n;$ similarly, since
	$\beta$ is a map of relative matrads, the components of the cochain
	\[
	\varphi=\beta\left(  C_{\ast}(\partial JJ_{n,m}\smallsetminus
	\operatorname*{int}KK_{n,m}\times1)\right)  \in Hom^{m+n-3}\left(  A^{\otimes
		m},B^{\otimes n}\right)
	\]
	are expressed in terms of $\omega_{B},$ $\omega_{A}^{j,i}$ and $g_{i}^{j}$
	with $i+j<m+n.$ Clearly $\tilde{g}\left(  z\right)  =\nabla\varphi;$ and
	$\left[  z\right]  =\left[  0\right]  $ since $\tilde{g}$ is a homology
	isomorphism. Now choose a cochain $b\in Hom^{m+n-3}\left(  A^{\otimes
		m},A^{\otimes n}\right)  $ such that $\nabla b=z.$ Then
	\[
	\nabla\left(  \tilde{g}\left(  b\right)  -\varphi\right)  =\nabla\tilde
	{g}\left(  b\right)  -\tilde{g}\left(  z\right)  =0.
	\]
	Choose a class representative $u\in\tilde{g}_{\ast}^{-1}\left[  \tilde
	{g}\left(  b\right)  -\varphi\right]  ,\ $set $\omega_{A}^{n,m}=b-u,$ and
	define $\alpha_{A}\left(  \theta_{m}^{n}\right)  =\omega_{A}^{n,m}.$ Then
	$\left[  \tilde{g}\left(  \omega_{A}^{n,m}\right)  -\varphi\right]  =\left[
	\tilde{g}\left(  b-u\right)  -\varphi\right]  =\left[  \tilde{g}\left(
	b\right)  -\varphi\right]  -\left[  \tilde{g}\left(  u\right)  \right]
	=\left[  0\right]  .$ Choose a cochain $g_{m}^{n}\in Hom^{m+n-2}$ $\left(
	A^{\otimes m},B^{\otimes n}\right)  $ such that
	\[
	\nabla g_{m}^{n}=g^{\otimes n}\omega_{A}^{n,m}-\varphi,
	\]
	and define $\beta\left(  \mathfrak{f}_{m}^{n}\right)=g_{m}^{n}.$ To extend
	$\beta$ as a map of relative matrads, define\linebreak $\beta(\left(
	\mathfrak{f}_{1}^{1}\right)  ^{\otimes n}\theta_{m}^{n})=g^{\otimes n}%
	\omega_{A}^{n,m}.$ Passing to the limit we obtain the desired maps $\alpha
	_{A}$ and $\beta.$
	
	Furthermore, if chain maps $\bar{\alpha}_{A}$ and $\bar{\beta}$ are defined in
	terms of different choices, beginning with a chain map $\bar{g}$ chain
	homotopic to $g,$ let $\bar{\omega}_{A}=\operatorname{Im}\bar{\alpha}_{A}$ and
	$\bar{G}=\operatorname{Im}\bar{\beta}.$ There is an inductively defined
	isomorphism $\Phi=\sum\phi_{m}^{n}:\left(  A,\bar{\omega}_{A}\right)
	\Rightarrow\left(  A,\omega_{A}\right)  $ with $\phi_{1}^{1}=\mathbf{1},$ and
	a chain homotopy $T:\tilde{G}\simeq G\circ\Phi.$ To initialize the induction,
	set $\phi_{1}^{1}=\mathbf{1}${, and note that}
	\[
	\nabla{g}_{2}^{1}=g{\omega}_{A}^{1,2}-{\omega}_{B}^{1,2}(g\otimes g)\text{ and
	}\nabla\bar{g}_{2}^{1}={\bar{g}\bar{\omega}}_{A}^{1,2}-{\omega}_{B}^{1,2}
	(\bar{g}\otimes\bar{g}).
	\]
	Let $s:\bar{g}\simeq g;$ then $c_{2}^{1}={\omega}_{B}^{1,2}\left(  s\otimes
	g+\bar{g}\otimes s\right)  $ satisfies
	\[
	\nabla c_{2}^{1}={\omega}_{B}^{1,2}(g\otimes g)-{\omega}_{B}^{1,2}(\bar
	{g}\otimes\bar{g}).
	\]
	Hence
	\[
	\nabla(g_{2}^{1}-\bar{g}_{2}^{1}+c_{2}^{1})=g{\omega}_{A}^{1,2}-\bar{g}
	{\bar{\omega}}_{A}^{1,2}%
	\]
	and
	\[
	\bar{g}({\omega}_{A}^{1,2}-{\bar{\omega}}_{A}^{1,2})=\nabla(g_{2}^{1}-\bar
	{g}_{2}^{1}+c_{2}^{1}-s\,\omega_{A}^{1,2}).
	\]
	Consequently, there is $\phi_{2}^{1}:A^{\otimes2}\rightarrow A$ such that
	$\nabla\phi_{2}^{1}={\omega}_{A}^{1,2}-${$\bar{\omega}$}$_{A}^{1,2};$ and, as
	above, $\phi_{2}^{1}$ may be chosen so that $\bar{g}\phi_{2}^{1}-(g_{2}
	^{1}-\bar{g}_{2}^{1}+c_{2}^{1}-s\,\omega_{A}^{1,2})$ is cohomologous to zero.
	Thus there is a component $t_{2}^{1}$ of $T$ such that
	\[
	\nabla(t_{2}^{1})=\bar{g}\phi_{2}^{1}-(g_{2}^{1}-g_{2}^{1}+c_{2}^{1}
	+s\,\omega_{A}^{1,2}).
	\]
	
\end{proof}

Since $\tilde{g}$ is a homology isomorphism
whenever $A$ is free (cf. Proposition \ref{homology-iso}) we have:

\begin{corollary}
	\label{m-bialgebra}Let $\left(  A,d_{A}\right)  $ be a free DGM, let
	$(B,d_{B},\omega_{B})$ be an $A_{\infty}$-bialgebra, and let $g:A\rightarrow B
	$ be a chain map/homology isomorphism. Then
	
	\begin{enumerate}
		\item[\textit{(i)}] (Existence) $g$ induces an $A_{\infty}$-bialgebra
		structure $\omega_{A}$ on $A$, and extends to a map $G:A\Rightarrow B$\ of
		$A_{\infty}$-bialgebras.
		
		\item[\textit{(ii)}] (Uniqueness) $\left(  \omega_{A},G\right)  $ is unique up
		to isomorphism.
	\end{enumerate}
\end{corollary}

Given a chain complex $B$ of (not necessarily free) $R$-modules, there is
always a chain complex of free $R$-modules $\left(  A,d_{A}\right)  $, and a
homology isomorphism $g:A\rightarrow B.$ To see this, let $(RH:\cdots
\rightarrow R_{1}H\rightarrow R_{0}H\overset{\rho}{\rightarrow}H,d)$ be a free
$R$-module resolution of $H=H_{\ast}\left(  B\right)  .$ Since $R_{0}H$ is
projective, there is a cycle-selecting homomorphism $g_{0}^{\prime}%
:R_{0}H\rightarrow Z\left(  B\right)  $ lifting $\rho$ through the projection
$Z\left(  B\right)  \rightarrow H$, and extending to a chain map
$g_{0}:\left(  RH,0\right)  \rightarrow\left(  B,d_{B}\right)  .$ If
$RH:0\rightarrow R_{1}H\rightarrow R_{0}H\rightarrow H$ is a short $R$-module
resolution of $H,$ then $g_{0}$ extends to a homology isomorphism $g:\left(
RH,d+h\right)  \rightarrow\left(  B,d_{B}\right)  $ with $\left(
A,d_{A}\right)  =\left(  RH,d\right)  .$ Otherwise, there is a perturbation
$h$ of $d$ such that $g:\left(  RH,d+h\right)  \rightarrow\left(
B,d_{B}\right)  $ is a homology isomorphism with $\left(  A,d_{A}\right)
=\left(  RH,d+h\right)  $ (see \cite{Berikashvili}, \cite{Saneblidze1}). Thus
an $A_{\infty}$-structure on $B$ always transfers to an $A_{\infty}$-structure
on $\left(  RH,d+h\right)  $ via Corollary \ref{m-bialgebra}, and we obtain
our main result concerning the transfer of $A_{\infty}$-structure to homology:

\begin{theorem}
	\label{AAHopf} \textit{Let }$B$\textit{\ be an }$A_{\infty}$\textit{-bialgebra
		with homology }$H=H_{\ast}\left(  B\right)  $\textit{, let }$\left(
	RH,d\right)  $\textit{\ be a free }$R$\textit{-module resolution of }$H,$
	\textit{and let }$h$\textit{\ be} \textit{a perturbation of }$d$\textit{\ such
		that }$g:\left(  RH,d+h\right)  \rightarrow\left(  B,d_{B}\right)
	$\textit{\ is a homology isomorphism. Then}
	
	\begin{enumerate}
		\item[\textit{(i)}] (Existence) $g$ induces an $A_{\infty}$-bialgebra
		structure $\omega_{RH}$ on $RH$, and extends to a map $G:RH\Rightarrow B$ of
		$A_{\infty}$-bialgebras.
		
		\item[\textit{(ii)}] (Uniqueness) $\left(  \omega_{RH},G\right)  $ is unique
		up to isomorphism.
	\end{enumerate}
\end{theorem}

\begin{remark}\label{Aclass}
	Note that $A_{\infty}$-bialgebra structures induced by the Transfer Algorithm
	are isomorphic for all choices of the map $g:\left(  RH,d+h\right)
	\rightarrow\left(  B,d_{B}\right).$ Thus we obtain an isomorphism class of
	$A_{\infty}$-bialgebra structures on $RH.$
\end{remark}

\begin{remark}
	\label{twoinfty} When $H=H_{\ast}(B)$ is a free module, we recover the
	classical results of Kadeishvili \cite{Kadeishvili1}, Markl \cite{Markl1}, and
	others, which transfer a DG (co)algebra structure to an $A_{\infty}%
	$-(co)algebra structure on homology, by setting $RH=H$. Furthermore, any pair
	of $A_{\infty}$-(co)algebra structures $\{{\omega}_{H}^{n,1}\}_{n\geq1}$ and
	$\{{\omega}_{H}^{1,m}\}_{m\geq1}$ on $H$ induced by the same cycle-selecting
	map $g:H\rightarrow B$ extend to an $A_{\infty}$-bialgebra structure $\left(
	H,{\omega}_{H}^{n,m}\right)  ,$ by the proof of Theorem \ref{AAHopf}. For an
	example of a DGA $B$ whose cohomology $H(B)$ is not free, and whose DGA
	structure transfers to an $A_{\infty}$-algebra structure on $H\left(
	B\right)  $ via Theorem \ref{AAHopf} along a map $g:H(B)\rightarrow B$ having
	no right-homotopy inverse, see \cite{Umble2}.
\end{remark}

\section{Applications and Examples}

The applications and examples in this section apply the Transfer Algorithm
given by the proof of Theorem \ref{transfer}. Three kinds of specialized
$A_{\infty}$-bialgebras $\left(  A,\left\{  \omega_{m}^{n}\right\}  \right)  $
are relevant here:

\begin{enumerate}
	\item $\omega_{m}^{1}=0$ for $m\geq3$ (the $A_{\infty}$-algebra substructure
	is trivial).
	
	\item $\omega_{m}^{n}=0$ for $m,n\geq2$ (all higher order structure is
	concentrated in the $A_{\infty}$-algebra and $A_{\infty}$-coalgebra substructures).
	
	\item Conditions (1) and (2) hold simultaneously.
\end{enumerate}

\noindent Of these, $A_{\infty}$-bialgebras of the first and third kind appear
in the applications.

Structure relations defining $A_{\infty}$-bialgebras of the second and third
kind are expressed in terms of the S-U diagonal $\Delta_{K}$ on associahedra
(see Subsection 2.4) and have especially nice form. Structure relations of the second
kind were derived in \cite{Umble}. Structure relations in an $A_{\infty}%
$-bialgebra $\left(  A,\omega\right)  $ of the third kind with $\omega_{1}%
^{1}=0$, $\mu=\omega_{2}^{1}$ and $\psi^{n}=\omega_{1}^{n}$ are a special case
of those derived in \cite{Umble}, and are given by the formula
\begin{equation}
	\left\{  \psi^{n}\mu=\mu^{\otimes n}\Psi^{n}\right\}  _{n\geq2},
	\label{multformula}%
\end{equation}
where the $n$-ary $A_{\infty}$-coalgebra operation
\[
\Psi^{n}=\left(  \sigma_{n,2}\right)  _{\ast}\iota\left(  \xi\otimes
\xi\right)  \Delta_{K}\left(  e^{n-2}\right)  :A\otimes A\rightarrow\left(
A\otimes A\right)  ^{\otimes n}%
\]
is defined in terms of

\begin{itemize}
	\item a map $\xi:C_{\ast}(K)\rightarrow Hom(A,TA)$ of operads that sends  $e^{n-2}\subseteq K_{n}$ to $\psi^{n},$
	
	\item the canonical isomorphism
	$	\iota:Hom\left(  A,A^{\otimes n}\right)  ^{\otimes2}\rightarrow Hom\left(
	A^{\otimes2},\left(  A^{\otimes n}\right)  ^{\otimes2}\right)  ,$

	\item and the induced isomorphism
	\[	\left(  \sigma_{n,2}\right)  _{\ast}:Hom\left(  A^{\otimes2},\left(
	A^{\otimes n}\right)  ^{\otimes2}\right)  \rightarrow Hom\left(  A^{\otimes
		2},\left(  A^{\otimes2}\right)  ^{\otimes n}\right)  .\]
	
\end{itemize}

Structure relations defining a morphism $G=\left\{  g^{n}\right\}
:(A,\omega_{A})\Rightarrow(B,\omega_{B})$ between $A_{\infty}$-bialgebras of
the third kind are expressed in terms of the S-U diagonal $\Delta_{J}$ on
multiplihedra \cite{SU2} by the formula
\begin{equation}
	\left\{  g^{n}\mu_{A}=\mu_{B}^{\otimes n}\mathbf{g}^{n}\right\}  _{n\geq1},
	\label{fmultformula}%
\end{equation}
where
\[
\mathbf{g}^{n}=(\sigma_{n,2})_{\ast}\iota(\upsilon\otimes\upsilon)\Delta
_{J}(e^{n-1}):A\otimes A\rightarrow\left(  B\otimes B\right)  ^{\otimes n},
\]
and $\upsilon:C_{\ast}(J)\rightarrow Hom(A,TB)$ is a map of relative
prematrads sending the top dimensional cell $e^{n-1}\subseteq J_{n}$ to
$g^{n}$ (the maps $\left\{  \mathbf{g}^{n}\right\}  $ define the tensor
product morphism $G\otimes G:\left(  A\otimes A,\Psi_{A\otimes A}\right)
\Rightarrow\left(  B\otimes B,\Psi_{B\otimes B}\right)  $).

Given a simply connected topological space $X,$ consider the Moore loop space
$\Omega X$ and the simplicial singular cochain complex $S^{\ast}(\Omega X;R)$.
Under the hypotheses of the Transfer Algorithm, the DG bialgebra structure of
$S^{\ast}(\Omega X;R)$ transfers to an $A_{\infty}$-bialgebra structure on
$H^{\ast}(\Omega X;R)$. Our next two theorems apply this principle, and
identify some important $A_{\infty}$-bialgebras of the third kind on loop
space (co)homology.

\begin{theorem}
	\label{rational3}If $X$ is simply connected, $H^{\ast}(\Omega X;\mathbb{Q})$
	admits an induced $A_{\infty}$-bialgebra structure of the third kind.
\end{theorem}

\begin{proof}
	Let $\mathcal{A}_{X}$ be a free DG commutative algebra cochain model for $X$
	over $\mathbb{Q}$ (e.g., Sullivan's minimal or Halperin-Stasheff's filtered
	model); then $H^{\ast}\left(  \mathcal{A}_{X}\right)  \approx H^{\ast
	}(X;\mathbb{Q}).$ The bar construction $\left(  B=B\mathcal{A}_{X}
	,d_{B},\Delta_{B}\right)  $ with shuffle product is a cofree DG commutative
	Hopf algebra cochain model for $\Omega X,$ and $H=H^{\ast}(B,d_{B})$ is a Hopf
	algebra with induced coproduct $\psi^{2}=\omega_{1}^{2}$ and free graded
	commutative product $\mu=\omega_{2}^{1}$ (by a theorem of Hopf). Since $H$ is
	a free commutative algebra, there is a multiplicative cocycle-selecting map
	$g_{1}^{1}:H\rightarrow B.$ Consequently, we may set $\omega_{n}^{1}=0$ for
	all $n\geq3$ and $g_{n}^{1}=0$ for all $n\geq2$, and obtain a trivial
	$A_{\infty}$-algebra structure $\left(  H,\mu\right)  $ induced by $g_{1}
	^{1}.$ There is an induced $A_{\infty}$-coalgebra structure $\left(
	H,\psi^{n}\right)  _{n\geq2}$, and an $A_{\infty}$-coalgebra map $G=\left\{
	g^{n}\mid g^{1}=g_{1}^{1}\right\}  _{n\geq1}:H\Rightarrow B$ constructed as
	follows: For $n\geq2,$ assume $\psi^{n}$ and $g^{n-1}$ have been constructed,
	and apply the Transfer Algorithm to obtain candidates $\omega_{1}^{n+1}$ and
	$g_{1}^{n}.$ Restrict $\omega_{1}^{n+1}$ to generators, and let $\psi^{n+1}$
	be the free extension of $\omega_{1}^{n+1}$ to all of $H $ using Formula
	\ref{multformula}. Similarly, restrict $g_{1}^{n}$ to generators, and let
	$g^{n}$ be the free extension of $g_{1}^{n}$ to all of $H$ using Formula
	\ref{fmultformula}.
	
	To complete the proof, we show that all other $A_{\infty}$-bialgebra
	operations may be trivially chosen. Refer to the Transfer Algorithm, and note
	that the Hopf relation $\psi^{2}\mu=\left(  \mu\otimes\mu\right)  \sigma
	_{2,2}\left(  \psi^{2}\otimes\psi^{2}\right)  $ implies $\beta(\partial
	\mathfrak{f}_{2}^{2})|_{C_{1}(JJ_{2,2}\setminus\operatorname*{int}
		(KK_{2,2}\times1))}=0.$ Thus we may choose $\omega_{2}^{2}=g_{2}^{2}=0\ $so
	that $\beta(\partial\mathfrak{f}_{2}^{2})=\beta(\mathfrak{f}_{2}^{2})=0.$
	Inductively, assume that $\omega_{2}^{n-1}=g_{2}^{n-1}=0$ for $n\geq3.$ Then
	$\beta(\partial\mathfrak{f}_{2}^{n})|_{C_{n-1}(JJ_{n,2}\setminus
		\operatorname*{int}(KK_{n,2}\times1))}=0$, and we may choose $\omega_{2}%
	^{n}=0$ and $g_{2}^{n}=0 $ so that $\beta(\partial\mathfrak{f}_{2}^{n}%
	)=\beta(\mathfrak{f}_{2}^{n})=0.$ Finally, for $m\geq3$ set $\omega_{m}^{n}=0$
	and $g_{m}^{n}=0. $ Then $\left(  H,\mu,\psi^{n}\right)  _{n\geq2}$ as an
	$A_{\infty}$-bialgebra of the third kind with structure relations given by
	Formula \ref{multformula}, and $G$ is a map of $A_{\infty}$-bialgebras
	satisfying Formula \ref{fmultformula}.
\end{proof}

Note that the components of the $A_{\infty}$-bialgebra map $G$ in the proof of
Theorem \ref{rational3} are exactly the components of a map of underlying
$A_{\infty}$-coalgebras given by the Transfer Algorithm.

Let $R$ be a PID and let $X$ be a connected space such that $H_{\ast}(X;R)$ is
torsion free. Then the Bott-Samelson Theorem \cite{Bott} asserts that
$H_{\ast}\left(  \Omega\Sigma X;R\right)  $ is isomorphic as an algebra to the
tensor algebra $T^{a}\tilde{H}_{\ast}(X;R)$ generated by the reduced homology
of $X$, and the adjoint $i:X\rightarrow\Omega\Sigma X$ of the identity
$\mathbf{1}:\Sigma X\rightarrow\Sigma X$ induces the canonical inclusion
$i_{\ast}:\tilde{H}_{\ast}\left(  X;R\right)  \hookrightarrow T^{a}\tilde
{H}_{\ast}(X;R)\approx H_{\ast}\left(  \Omega\Sigma X;R\right)  $. Thus if
$\{\psi^{n}\}_{n\geq2}$ is an $A_{\infty}$-coalgebra structure on ${H}_{\ast
}(X;R),$ the tensor algebra $T^{a}\tilde{H}_{\ast}(X;R)$ admits a canonical
$A_{\infty}$-bialgebra structure of the third kind with respect to the free
extension of each $\psi^{n}$ via Formula \ref{multformula}.

Furthermore, the canonical inclusion $t:X\hookrightarrow\Omega\Sigma X$
induces a DG coalgebra map of simplicial singular chains $t_{\#}:S_{\ast
}(X;R)\rightarrow S_{\ast}(\Omega\Sigma X;R)$, which extends to a homology
isomorphism $t_{\#}:T^{a}\tilde{S}_{\ast}(X;R)\approx S_{\ast}(\Omega\Sigma
X;R)$ of DG Hopf algebras. Thus the induced \emph{Bott-Samelson Isomorphism}
$t_{\ast}:T^{a}\tilde{H}_{\ast}\left(  X;R\right)  \approx H_{\ast}
(\Omega\Sigma X;R)$ is an isomorphism of Hopf algebras (\cite{husemoller},
\cite{KS1}), and $T^{a}\tilde{S}_{\ast}(X;R)$ is a free DG Hopf algebra chain
model for $\Omega\Sigma X.$

\begin{theorem}
	\label{B-S}Let $R$ be a PID, and let $X$ be a connected space such that
	$H_{\ast}(X;R)$ is torsion free.
	
	\begin{enumerate}
		\item[\textit{(i)}] Then $T^{a}\tilde{H}_{\ast}(X;R)$ admits an $A_{\infty}$
		-bialgebra structure of the third kind, which is trivial if and only if the
		$A_{\infty}$-coalgebra structure of $H_{\ast}(X;R)$ is trivial.
		
		\item[\textit{(ii)}] The Bott-Samelson Isomorphism $t_{\ast}:T^{a}\tilde
		{H}_{\ast}(X;R)\approx H_{\ast}(\Omega\Sigma X;R)$ extends to an isomorphism
		of $A_{\infty}$-bialgebras.
	\end{enumerate}
\end{theorem}

\begin{proof}
	Since $H_{\ast}(X;R)$ is free as an $R$-module, we may choose a
	cycle-selecting map $\bar{g}={\bar{g}}_{1}^{1}:H_{\ast}(X,R)\rightarrow
	S_{\ast}(X;R)$ and apply the Transfer Algorithm to obtain an induced
	$A_{\infty}$-coalgebra structure $\bar{\omega}=\{{\bar{\omega}}_{1}
	^{n}\}_{n\geq2}$ on $H_{\ast}(X,R)$ and a corresponding map of $A_{\infty}
	$-coalgebras $\bar{G}=\{{\bar{g}}_{1}^{n}\}_{n\geq1}:H_{\ast}(X;R)\Rightarrow
	S_{\ast}(X;R).$ Let $H=T^{a}\tilde{H}_{\ast}(X;R),$ let $B=T^{a}\tilde
	{S}_{\ast}(X;R),$ and consider the free (multiplicative) extension $g=T\left(
	\bar{g}\right)  :H\rightarrow B.$ As in the proof of Theorem \ref{rational3},
	use formulas \ref{multformula} and \ref{fmultformula} to freely extend
	$\bar{\omega}$ and $\bar{G}$ to families $\omega=\left\{  \omega_{1}
	^{n}\right\}  $ and $G=\{{g}_{1}^{n}\mid{g}_{1}^{1}=g\}_{n\geq1}$ defined on
	$H$, and choose all other $A_{\infty}$-bialgebra operations to be zero. Then
	$\bar{\omega}$ lifts to an $A_{\infty}$-bialgebra structure $\left(
	H,\omega,\mu\right)  $ of the third kind with free product $\mu$, and $\bar
	{G}$ lifts to a map $G:H\Rightarrow B$ of $A_{\infty}$-bialgebras.
	Furthermore, restricting $\omega$ to the multiplicative generators $H_{\ast
	}(X;R)$ recovers the $A_{\infty}$-coalgebra operations on $H_{\ast}(X;R)$.
	Thus $A_{\infty}$ -bialgebra structure of $H$ is trivial if and only if the
	$A_{\infty}$-coalgebra structure of $H_{\ast}(X;R)$ is trivial. Finally, since
	$B$ is a free DG Hopf algebra chain model for $\Omega\Sigma X,$ the
	Bott-Samelson Isomorphism $t_{\ast}$ extends to an isomorphism of $A_{\infty}
	$-bialgebras, and identifies the $A_{\infty}$-bialgebra structure of $H_{\ast
	}(\Omega\Sigma X;R)$ with $\left(  H,\omega,\mu\right)  $.
\end{proof}

It is important to note that prior to this work, all known \emph{rational}
homology invariants of $\Omega\Sigma X$ are trivial for any space $X$.
However, we now have the following:

\begin{corollary}
	\label{invariant}A nontrivial $A_{\infty}$-coalgebra structure on $H_{\ast
	}(X;\mathbb{Q})$ induces a nontrivial $A_{\infty}$-bialgebra structure on
	$H_{\ast}(\Omega\Sigma X;\mathbb{Q}).$ Thus the $A_{\infty}$-bialgebra
	structure of $H_{\ast}(\Omega\Sigma X;\mathbb{Q})$ is a nontrivial rational
	homology invariant.
\end{corollary}

\begin{proof}
	First, $H=H_{\ast}(\Omega\Sigma X;\mathbb{Q})$ admits an induced $A_{\infty}
	$-bialgebra structure of the third kind by Theorem \ref{B-S}, which is trivial
	if and only if the $A_{\infty}$-coalgebra structure of $H_{\ast}
	(X;\mathbb{Q})$ is trivial. Second, the dual version of Theorem
	\ref{rational3} imposes an induced $A_{\infty}$-bialgebra structure on $H$
	whose $A_{\infty}$-coalgebra substructure is trivial, and whose $A_{\infty}
	$-algebra substructure is trivial if and only if the $A_{\infty}$-coalgebra
	structure of $H_{\ast}(X;\mathbb{Q})$ is trivial.
\end{proof}

The two $A_{\infty}$-bialgebras identified in the proof of Corollary
\ref{invariant}--one with trivial $A_{\infty}$-coalgebra substructure and the
other with trivial $A_{\infty}$-algebra substructure--are in fact isomorphic,
and represent the same isomorphism class of $A_{\infty}$-bialgebra structures
on $H_{\ast}(\Omega\Sigma X;\mathbb{Q})$ (cf. Remark \ref{Aclass}). Indeed, choose a pair
of isomorphisms for the two underlying $A_{\infty}$-(co)algebra substructures---the first as an isomorphism of $A_{\infty}$-algebras and the second as an isomorphism of $A_{\infty}$-coalgebras (their
component in bidegree $(1,1)$ is the identity $\mathbf{1}:H\rightarrow H$). Since
$\omega_{i}^{j}=0$ for $i,j\geq2,$ these isomorphisms clearly determine an
isomorphism of $A_{\infty}$-bialgebras.

Our next example exhibits an $A_{\infty}$-bialgebra of the first but not the
second kind. Given a $1\,$-connected DGA $\left(  A,d_{A}\right)  ,$ the bar
construction of $A$, denoted by $BA,$ is the cofree DGC $T^{c}\left(
\downarrow\overline{A}\right)  $ whose differential $d\ $and coproduct
$\Delta$ are defined as follows:\ Let $\left\lfloor x_{1}|\cdots
|x_{n}\right\rfloor $ denote the element $\left.  \downarrow x_{1}\mid
\cdots\mid\downarrow x_{n}\right.  \in BA$, and let $e$ denote the unit
$\left\lfloor \ \right\rfloor .$ Then
\[
d\left\lfloor x_{1}|\cdots|x_{n}\right\rfloor =\sum_{i=1}^{n}\pm\left\lfloor
x_{1}|\cdots|d_{A}x_{i}|\cdots|x_{n}\right\rfloor +\sum_{i=1}^{n-1}
\pm\left\lfloor x_{1}|\cdots|x_{i}x_{i+1}|\cdots|x_{n}\right\rfloor \text{;}%
\]
\[
\Delta\left\lfloor x_{1}|\cdots|x_{n}\right\rfloor =e\otimes\left\lfloor
x_{1}|\cdots|x_{n}\right\rfloor +\left\lfloor x_{1}|\cdots|x_{n}\right\rfloor
\otimes e+\sum_{i=1}^{n-1}\left\lfloor x_{1}|\cdots|x_{i}\right\rfloor
\otimes\left\lfloor x_{i}|\cdots|x_{n}\right\rfloor .
\]
Given an $A_{\infty}$-coalgebra $\left(  C,\Delta^{n}\right)  _{n\geq1},$ the
tilde-cobar construction of $C$, denoted by $\tilde{\Omega}C,$ is the free DGA
$T^{a}\left(  \uparrow\overline{C}\right)  $ with differential $d_{\tilde
	{\Omega}A}$ given by extending $\sum_{i\geq1}\Delta^{n}$ as a derivation. Let
$\left\lceil x_{1}|\cdots|x_{n}\right\rceil $ denote $\left.  \uparrow
x_{1}\mid\cdots\mid\uparrow x_{n}\right.  \in\tilde{\Omega}H.$

\begin{example}
	\label{hga}Consider the DGA $A=\mathbb{Z}_{2}\left[  a,b\right]  /\left(
	a^{4},ab\right)  $ with $\left\vert a\right\vert =3,$ $\left\vert b\right\vert
	\!=\!5$ and trivial differential. Define a homotopy Gerstenhaber algebra
	$($HGA$)$ structure $\{E_{p,q}:A^{\otimes p}\otimes A^{\otimes q}\rightarrow
	A\}_{p,q\geq0;\text{ }p+q>0}$ with $E_{p,q}$ acting trivially except
	$E_{1,0}=E_{0,1}=\mathbf{1}$ and $E_{1,1}(b;b)=a^{3}$ $($cf. \cite{Voronov},
	\cite{KS1}$)$. Form the tensor coalgebra $BA\otimes BA$ with coproduct
	$\psi=\sigma_{2,2}\left(  \Delta\otimes\Delta\right)  ,$ and consider the
	induced map
	\[ \phi=E_{1,0}^{\prime}+E_{0,1}^{\prime}+E_{1,1}^{\prime}:BA\otimes
	BA\rightarrow A
	\]
	of degree $+1,$ which acts trivially except for $E_{1,0}^{\prime}(\lfloor
	x\rfloor\otimes e)=E_{0,1}^{\prime}(e\otimes\lfloor x\rfloor)=x$ for all $x\in
	A,$ and $E_{1,1}^{\prime}(\lfloor b\rfloor\otimes\lfloor b\rfloor)=a^{3}.$
	Since $E_{p,q}$ is an HGA structure, $\phi$ is a twisting cochain, which lifts
	to a chain map of DG coalgebras $\mu:BA\otimes BA\rightarrow BA$ defined by
	\[
	\mu=\sum_{k=0}^{\infty}\downarrow^{\otimes k+1}\phi^{\otimes k+1}\bar{\psi
	}^{\left(  k\right)  },
	\]
	where $\bar{\psi}^{(0)}=\mathbf{1,}$ $\bar{\psi}^{\left(  k\right)  }=\left(
	\bar{\psi}\otimes\mathbf{1}^{\otimes k-1}\right)  \cdots\left(  \bar{\psi
	}\otimes\mathbf{1}\right)  \bar{\psi}$ for $k>0,$ and $\bar{\psi}$ is the
	reduced coproduct on $BA\otimes BA.$ Then, for example, $\mu\left(
	\left\lfloor b\right\rfloor \!\otimes\! \left\lfloor b\right\rfloor \right)
	=\left\lfloor a^{3}\right\rfloor $ and $\mu\left(  \left\lfloor b\right\rfloor
	\! \otimes\! \left\lfloor a|b\right\rfloor \right)  =\left\lfloor a|a^{3}
	\right\rfloor +\left\lfloor b|a|b\right\rfloor .$ It follows that
	$(BA,d,\Delta,\mu)$ is a DG\ Hopf algebra. Let $\mu_{H}$ and $\Delta_{H}$ be
	the product and coproduct on $H=H^{\ast}\left(  BA\right)  $ induced by $\mu$
	and $\Delta;$ then $\left(  H,\Delta_{H},\mu_{H}\right)  $ is a graded
	bialgebra. Let $\alpha={\operatorname*{cls}}\left\lfloor a\right\rfloor $ and
	$z={\operatorname*{cls}}\left\lfloor a|a^{3}\right\rfloor $ in $H,$ and note
	that $\left\lfloor a^{3}\right\rfloor =d\left\lfloor a|a^{2}\right\rfloor .$
	Let $g:H\rightarrow BA$ be a cycle-selecting map such that $g\left(
	\operatorname*{cls}\left\lfloor x_{1}|\cdots|x_{n}\right\rfloor \right)
	=\left\lfloor x_{1}|\cdots|x_{n}\right\rfloor .$ Then
	\[
	\bar{\Delta}_{H}\left(  z\right)  ={\operatorname*{cls}\bar{\Delta}
	}\left\lfloor a|a^{3}\right\rfloor ={\operatorname*{cls}}\left\{  \left\lfloor
	a\right\rfloor \otimes\left\lfloor a^{3}\right\rfloor \right\}  =0
	\]
	so that
	\[
	\left\{  \Delta g+\left(  g\otimes g\right)  \Delta_{H}\right\}  \left(
	z\right)  =\left\lfloor a\right\rfloor \otimes\left\lfloor a^{3}\right\rfloor
	.
	\]
	By the Transfer Algorithm, we may choose a map $g^{2}:H\rightarrow BA\otimes BA
	$ such that $g^{2}\left(  z\right)  =\left\lfloor a\right\rfloor
	\otimes\left\lfloor a|a^{2}\right\rfloor ;$ then
	\[
	\nabla g^{2}\left(  z\right)  =\left\{  \Delta g+\left(  g\otimes g\right)
	\Delta_{H}\right\}  \left(  z\right)  .
	\]
	Furthermore, note that
	\[
	\left\{  \left(  g^{2}\otimes g+g\otimes g^{2}\right)  \Delta_{H}+\left(
	\Delta\otimes\mathbf{1}+\mathbf{1}\otimes\Delta\right)  g^{2}\right\}  \left(
	z\right)  =\left\lfloor a\right\rfloor \otimes\left\lfloor a\right\rfloor
	\otimes\left\lfloor a^{2}\right\rfloor .
	\]
	Since $\left\lfloor a^{2}\right\rfloor =d\left\lfloor a|a\right\rfloor ,$
	there is an $A_{\infty}$-coalgebra operation $\Delta_{H}^{3}:H\rightarrow
	H^{\otimes3}$, and a map $g^{3}:H\rightarrow\left(  BA\right)  ^{\otimes3}$
	satisfying the general relation on $J_{3}$ such that $\Delta_{H}^{3}(z)=0$ and
	$g^{3}\left(  z\right)  =\left\lfloor a\right\rfloor \otimes\left\lfloor
	a\right\rfloor \otimes\left\lfloor a|a\right\rfloor .$ In fact, we may choose
	$\Delta_{H}^{3}$ to be identically zero on $H$ so that
	\[
	\nabla g^{3}=\left(  \Delta\otimes\mathbf{1}+\mathbf{1}\otimes\Delta\right)
	g^{2}+\left(  g^{2}\otimes g+g\otimes g^{2}\right)  \Delta_{H}.
	\]

	Now the potentially non-vanishing terms in the image of $J_{4}$ are
	\[
	\left(  g^{3}\otimes g+g^{2}\otimes g^{2}+g\otimes g^{3}\right)  \Delta
	_{H}+\left(  \Delta\otimes\mathbf{1}\otimes\mathbf{1}+\mathbf{1}\otimes
	\Delta\otimes\mathbf{1}+\mathbf{1}\otimes\mathbf{1}\otimes\Delta\right)
	g^{3},
	\]
	and evaluating at $z$ gives $\left\lfloor a\right\rfloor \otimes\left\lfloor
	a\right\rfloor \otimes\left\lfloor a\right\rfloor \otimes\left\lfloor
	a\right\rfloor .$ Thus there is an $A_{\infty}$-coalgebra operation
	$\Delta_{H}^{4}$ and a map $g^{4}:H\rightarrow\left(  BA\right)  ^{\otimes4}$
	satisfying the general relation on $J_{4}$ such that $\Delta_{H}^{4}\left(
	z\right)  =\alpha\otimes\alpha\otimes\alpha\otimes\alpha$ and $g^{4}\left(
	z\right)  =0.$ Now recall that the induced $A_{\infty}$-coalgebra structure on
	$H\otimes H$ is given by
	\begin{align*}
		\Delta_{H\otimes H}  &  =\sigma_{2,2}\left(  \Delta_{H}\otimes\Delta
		_{H}\right) \\
		\Delta_{H\otimes H}^{4}  &  =\sigma_{4,2}[\Delta_{H}^{4}\otimes(\mathbf{1}
		^{\otimes2}\otimes\Delta_{H})(\mathbf{1}\otimes\Delta_{H})\Delta_{H}
		+(\Delta_{H}\otimes\mathbf{1}^{\otimes2})(\Delta_{H}\otimes\mathbf{1}
		)\Delta_{H}\otimes\Delta_{H}^{4}]\\
		&  \vdots
	\end{align*}
	Let $\beta={\operatorname*{cls}}\left\lfloor b\right\rfloor ,$
	$u={\operatorname*{cls}}\left\lfloor a|b\right\rfloor ,$
	$v={\operatorname*{cls}}\left\lfloor b|a\right\rfloor ,$ and
	$w={\operatorname*{cls}}\left\lfloor b|a|b\right\rfloor $ in $H,$ and consider
	the induced map of tilde-cobar constructions
	\[
	\widetilde{\mu_{H}}=\sum_{n\geq1}\left(  \uparrow\mu_{H}\downarrow\right)
	^{\otimes n}:\tilde{\Omega}(H\otimes H)\rightarrow\tilde{\Omega}H.
	\]
	Then
	\[
	\widetilde{\mu_{H}}\left\lceil \beta\otimes u\right\rceil =\left\lceil \mu
	_{H}\left(  \beta\otimes u\right)  \right\rceil =\left\lceil
	{\operatorname*{cls}}\text{ }\mu\left(  \left\lfloor b\right\rfloor
	\otimes\left\lfloor a|b\right\rfloor \right)  \right\rceil =\left\lceil
	z+w\right\rceil
	\]
	so that
	\[
	d_{\tilde{\Omega}H}\widetilde{\mu_{H}}\left\lceil \beta\otimes u\right\rceil
	=d_{\tilde{\Omega}H}\left\lceil z+w\right\rceil =\left\lceil \alpha
	|\alpha|\alpha|\alpha\right\rceil +\left\lceil \beta|u\right\rceil
	+\left\lceil v|\beta\right\rceil .
	\]
	But on the other hand,
	\begin{align*}
		d_{\tilde{\Omega}(H\otimes H)}\left\lceil \beta\otimes u\right\rceil  &
		=\left\lceil \Delta_{H\otimes H}\left(  \beta\otimes u\right)  \right\rceil \\
		&  =\left\lceil e\otimes u\mid\beta\otimes e\right\rceil +\left\lceil
		\beta\otimes e\mid e\otimes u\right\rceil \\
		&  \hspace*{0.5in}+\left\lceil \beta\otimes\alpha\mid e\otimes\beta
		\right\rceil +\left\lceil e\otimes\alpha\mid\beta\otimes\beta\right\rceil
	\end{align*}
	so that
	\[
	\widetilde{\mu_{H}}d_{\tilde{\Omega}(H\otimes H)}\left\lceil \beta\otimes
	u\right\rceil =\left\lceil \beta|u\right\rceil +\left\lceil v|\beta
	\right\rceil .
	\]
	Although $\widetilde{\mu_{H}}$ fails to be a chain map, the Transfer Algorithm
	implies there is a chain map $\widetilde{\mu_{H}}^{2}:\tilde{\Omega}(H\otimes
	H)\rightarrow\tilde{\Omega}H$ such that $\widetilde{\mu_{H}}^{2}\left\lceil
	e\otimes\alpha\mid\beta\otimes\beta\right\rceil =\left\lceil \alpha
	|\alpha|\alpha|\alpha\right\rceil ,$ which can be realized by defining
	\[
	\widetilde{\mu_{H}}^{2}=\sum_{n\geq1}\left(  \uparrow\mu_{H}\downarrow
	+\uparrow^{\otimes3}\omega_{2}^{3}\downarrow\right)  ^{\otimes n},
	\]
	where $\omega_{2}^{3}\left(  \beta\otimes\beta\right)  =\alpha\otimes
	\alpha\otimes\alpha.$ Indeed, to see that the required equality holds, note
	that $\mu_{H}\left(  \beta\otimes\beta\right)  =0$ since $\left\lfloor
	a^{3}\right\rfloor =d\left\lfloor a|a^{2}\right\rfloor .$ Thus there is a map
	$g_{2}:H\otimes H\rightarrow BA$ such that $g_{2}\left(  \beta\otimes
	\beta\right)  =\left\lfloor a|a^{2}\right\rfloor $, and $\nabla g_{2}=g\mu
	_{H}+\mu\left(  g\otimes g\right)  .$ Furthermore, there is the following
	general relation on $JJ_{2,2}:$
	\begin{align}
		\nabla g_{2}^{2}  &  =\omega_{BA}^{2,2}\left(  g\otimes g\right)  +\left(
		\mu\otimes\mu\right)  \sigma_{2,2}\left(  \Delta g\otimes g^{2}+g^{2}
		\otimes\left(  g\otimes g\right)  \Delta_{H}\right)  +g^{2}\mu_{H}%
		\label{three}\\
		&  +\left(  \mu\left(  g\otimes g\right)  \otimes g_{2}+g_{2}\otimes g\mu
		_{H}\right)  \sigma_{2,2}\left(  \Delta_{H}\otimes\Delta_{H}\right)  +\Delta
		g_{2}+\left(  g\otimes g\right)  \omega_{2}^{2}.\nonumber
	\end{align}
	The first expression on the right-hand side vanishes since $BA$ has trivial
	higher order structure, and the next two expressions vanish since $\mu
	_{H}\left(  \beta\otimes\beta\right)  =0$ and $g^{2}(\beta)=0$ ($\beta$ is
	primitive). However,
	\[
	\left\{  \left(  \mu\left(  g\otimes g\right)  \otimes g_{2}+g_{2}\otimes
	g\mu_{H}\right)  \sigma_{2,2}\left(  \Delta_{H}\otimes\Delta_{H}\right)
	+\Delta g_{2}\right\}  \left(  \beta\otimes\beta\right)\]\[
	=\bar{\Delta}\,g_{2}(\beta\otimes\beta)=\left\lfloor a\right\rfloor
	\otimes\left\lfloor a^{2}\right\rfloor .
	\]
	Since $d\left\lfloor a|a\right\rfloor =\left\lfloor a^{2}\right\rfloor ,$
	there an operation $\omega_{2}^{2}:H^{\otimes2}\rightarrow H^{\otimes2}$, and
	a map $g_{2}^{2}:H^{\otimes2}\rightarrow\left(  BA\right)  ^{\otimes2}$
	satisfying relation (\ref{three}) such that $\omega_{2}^{2}\left(
	\beta\otimes\beta\right)  =0$ and $g_{2}^{2}\left(  \beta\otimes\beta\right)
	=\left\lfloor a\right\rfloor \otimes\left\lfloor a|a\right\rfloor $.
	Similarly, there is an operation $\omega_{2}^{3}:H^{\otimes2}\rightarrow
	H^{\otimes3}$, and a map $g_{2}^{3}:H^{\otimes2}\rightarrow\left(  BA\right)
	^{\otimes3}$ satisfying the general relation on $JJ_{3,2}$ such that
	$\omega_{2}^{3}\left(  \beta\otimes\beta\right)  =\alpha\!\otimes
	\!\alpha\!\otimes\!\alpha$ and $g_{2}^{3}\left(  \beta\otimes\beta\right)
	=0.$ Thus $\left(  H,\mu_{H},\Delta_{H},\,\Delta_{H}^{4},\omega_{2}
	^{3},...\right)  $ is an $A_{\infty}$-bialgebra of the first kind.
\end{example}

One can think of the algebra $A$ in Example \ref{hga} as the singular
$\mathbb{Z}_{2}$-cohomology algebra of a space $X$ with the Steenrod algebra
$\mathcal{A}_{2}$ acting nontrivially via $Sq_{1}b=a^{3}$ (recall that
$Sq_{1}:H^{n}\left(  X;\mathbb{Z}_{2}\right)  \rightarrow H^{2n-1}\left(
X;\mathbb{Z}_{2}\right)  $ is a homomorphism defined by $Sq_{1}\left[
x\right]  =\left[  x\smile_{1}x\right]  $). Recall that a space $X$ is
$\mathbb{Z}_{2}$-\emph{formal} if there exists a DGA $B$ and cohomology
isomorphisms $C^{\ast}(X;\mathbb{Z}_{2})\leftarrow B\rightarrow H^{\ast
}\left(  X;\mathbb{Z}_{2}\right)  .$ Thus, when $X$ is $\mathbb{Z}_{2}$-formal,
$H^{\ast}\left(  BA\right)  \approx H^{\ast}(\Omega X;\mathbb{Z}_{2})$ as
graded coalgebras. Now consider a $\mathbb{Z}_{2}$-formal space $X$ whose
cohomology $H^{\ast}\left(  X;\mathbb{Z}_{2}\right)  $ is generated
multiplicatively by $\left\{  a_{1},\ldots,a_{n+1},b\right\}  ,$ $n\geq2.$
Then Example \ref{hga} suggests the following conditions on $X,$ which if
satisfied, give rise to a nontrivial operation $\omega_{2}^{n}$ with $n\geq2,$
on the loop cohomology $H^{\ast}(\Omega X;\mathbb{Z}_{2})$:

\begin{enumerate}
	\item $a_{1}b=0;$
	
	\item $a_{1}\cdots a_{n+1}=0;$
	
	\item $a_{i_{1}}\cdots a_{i_{k}}\neq0$ whenever $k\leq n$ and $i_{p}\neq
	i_{q}$ for all $p\neq q;$
	
	\item $Sq_{1}(b)=a_{2}\cdots a_{n+1}$.
\end{enumerate}

\noindent To see this, consider the non-zero classes $\alpha_{i}
=\operatorname*{cls}\left\lfloor a_{i}\right\rfloor ,$ $\beta
=\operatorname*{cls}\left\lfloor b\right\rfloor ,$ $u=\operatorname*{cls}
\left\lfloor a_{1}|b\right\rfloor ,$ $w=\operatorname*{cls}\left\lfloor
b|a_{1}|b\right\rfloor ,$ and $z=\operatorname*{cls}\left\lfloor a_{1}
|a_{2}\cdots a_{n+1}\right\rfloor $ in $H=H^{\ast}\left(  BA\right)  .$
Conditions (2) and (3) give rise to an induced $A_{\infty}$-coalgebra
structure $\left\{  \Delta_{H}^{k}:H\rightarrow H^{\otimes k}\right\}  $ such
that $\Delta_{H} ^{k}(z)=0$ for $3\leq k\leq n$, and $\Delta_{H}^{n+1}
(z)=\alpha_{1} \otimes\cdots\otimes\alpha_{n+1}$ with $g^{k}(z)=\left\lfloor
a_{1}\right\rfloor \otimes\cdots\otimes\left\lfloor a_{k-1}\right\rfloor
\otimes\left\lfloor a_{k}|a_{k+1}\cdots a_{n+1}\right\rfloor $ for $2\leq
k\leq n$ and $g^{n+1}(z)=0.$ Next, condition (4) implies $\beta\smile u=w+z,$
and we can define $\omega_{2}^{k}(\beta\otimes\beta)=0$ for $2\leq k<n$ and
$\omega_{2}^{n}(\beta\otimes\beta)=\alpha_{2}\otimes\cdots\otimes\alpha_{n+1}
$ with $g_{2}^{1}(\beta\otimes\beta)=\left\lfloor a_{2}|a_{3}\cdots
a_{n+1}\right\rfloor ,$ $g_{2}^{k}(\beta\otimes\beta)=\left\lfloor
a_{2}\right\rfloor \otimes\cdots\otimes\left\lfloor a_{k}\right\rfloor
\otimes\left\lfloor a_{k+1}\right.  \left.  | a_{k+2}\cdots a_{n+1}
\right\rfloor $ for $2\leq k\leq n-1$ and $g_{2}^{n}(\beta\otimes\beta)=0.$
Indeed, the Transfer Algorithm implies the existence of an $A_{\infty}
$-bialgebra structure in which $\omega_{2}^{n}$ satisfies the required
structure relation on $JJ_{n,2}$. Note that the $\mathbb{Z}_{2} $-formality
assumption is in fact superfluous here, as it is sufficient for $\alpha_{i},$
$\beta,$ and $u$ to be non-zero.

Spaces $X$ with $\mathbb{Z}_{2}$-cohomology satisfying conditions (1)-(4) abound.

\begin{example}
	\label{ex7.5}Given an integer $n\geq2,$ choose positive integers $r_{1}
	,\ldots,r_{n+1}$ and $m\geq2$ such that $r_{2}+\cdots+r_{n+1}=4m-3.$ Consider
	the \textquotedblleft thick bouquet\textquotedblright\ of spheres $S^{r_{1}}\veebar
	\cdots\veebar S^{r_{n+1}},$ i.e., $S^{r_{1}}\times\cdots\times S^{r_{n+1}}$
	with top dimensional cell removed, and generators $\bar{a}_{i}\in H^{r_{i}
	}(S^{r_{i}};\mathbb{Z}_{2}).$ Also consider the suspension of complex
	projective space $\Sigma\mathbb{C}P^{2m-2}$ with generators $\bar{b}\in
	H^{2m-1}(\Sigma\mathbb{C}P^{2m-2};\mathbb{Z}_{2})$ and $Sq_{1}\bar{b}\in
	H^{4m-3}(\Sigma\mathbb{C}P^{2m-2};\mathbb{Z}_{2})$. Let $Y_{n}=S^{r_{1}
	}\veebar\cdots\veebar S^{r_{n+1}}\vee\Sigma\mathbb{C}P^{2m-2},$ and choose a
	map $f:Y_{n}\rightarrow K(\mathbb{Z}_{2},4m-3)$ such that $f^{\ast}
	(\iota_{4m-3})=\bar{a}_{2}\cdots\bar{a}_{n+1}+Sq_{1}\bar{b}. $ Finally,
	consider the pullback $p:X_{n}\rightarrow Y_{n}$ of the following path
	fibration:
	\[%
	\begin{array}
		[c]{ccccc}%
		K\left(  \mathbb{Z}_{2},4m-4\right)  & \longrightarrow & X_{n}\medskip &
		\longrightarrow & \mathcal{L}K\left(  \mathbb{Z}_{2},4m-3\right) \\
		&  & p\downarrow\medskip\text{ } &  & \downarrow\\
		&  & Y_{n} & \overset{f}{\longrightarrow} & K\left(  \mathbb{Z}_{2}
		,4m-3\right) \\
		&  & \bar{a}_{2}\cdots\bar{a}_{n+1}+Sq_{1}\bar{b} & \underset{f^{\ast
		}}{\longleftarrow} & \iota_{4m-3}%
	\end{array}
	\]
	Let $a_{i}=p^{\ast}\left(  \bar{a}_{i}\right)  $ and $b=p^{\ast}\left(
	\bar{b}\right)  ;$ then $a_{1},\ldots,a_{n+1},b$ are multiplicative generators
	of $H^{\ast}(X_{n};\mathbb{Z}_{2})$ satisfying conditions (1) - (4) above. We
	remark that one can also obtain a space $X_{2}^{\prime}$ with a non-trivial
	$\omega_{2}^{2}$ on its loop cohomology by setting $Y_{2}^{\prime}=\left(
	S^{2}\times S^{3}\right)  \vee\Sigma\mathbb{C}P^{2}$ in the construction above
	(see \cite{Umble2} for details).
\end{example}

Finally, we note that the cohomology of Eilenberg-MacLane spaces and Lie
groups fail to satisfy all of (1)-(4), and it would not be surprising to find
that the operations $\omega_{2}^{n}$ vanish in their loop cohomologies for all
$n\geq2.$ In the case of Eilenberg-MacLane spaces, results due to Berciano and
the second author \cite{BU} seems to support this conjecture. Indeed, each tensor factor
$A=E\left(  v,2n+1\right)  \otimes\Gamma(w,2np+2)\subset H_{\ast}\left(
K\left(  \mathbb{Z},n\right)  ;\mathbb{Z}_{p}\right)  ,$ where $n\geq3$ and
$p$ an odd prime, is an $A_{\infty}$-bialgebra of the third kind of the form
$\left(  A,\Delta^{2},\Delta^{p},\mu\right)  $.

HGAs with nontrivial actions of the Steenrod algebra $\mathcal{A}_{2}$ were
first considered by the first author in \cite{Saneblidze3}. In general, the Steenrod $\smile_{1}$-cochain operation together with other higher cochain operations induce a nontrivial HGA structure on $S^{\ast}\left(  X\right)  $, but the failure of the differential to be a $\smile_{1}$-derivation prevents an immediate lifting of this HGA structure to cohomology (for some remarks on the history of lifting of $\smile_{1}$-operation  on the homology level see \cite {KS1} and \cite{Saneblidze3}).

\begin{example}
	\label{hopfinv}Let $g:S^{2n-2}\rightarrow S^{n}$ be a map of spheres, and let
	$Y_{m,n}=S^{m}\times\left(  e^{2n-1}\cup_{g}S^{n}\right)  .$ Let $\ast$ be the
	wedge point of $S^{m}\vee S^{n}\subset Y_{m,n},$ let $f:S^{2m-1}\rightarrow
	S^{m}\times\ast,$ and let $X_{m,n}=e^{2m}\cup_{f}Y_{m,n}.$ Then $X_{m,n}$ is
	$\mathbb{Z}_{2}$-formal for each $m$ and $n$ (by a dimensional argument), and
	we may consider $A=H^{\ast}(X_{m,n};\mathbb{Z}_{2})$ and $H=H^{\ast
	}(BA)\approx H^{\ast}(\Omega X_{m,n};\mathbb{Z}_{2}).$ Below we prove that:
	
	\begin{enumerate}
		\item[\textit{(i)}] The $A_{\infty}$-coalgebra structure of $H$ is nontrivial
		if and only if the Hopf invariant $h(f)=1,$ in which case $m=2,4,8 $.
		
		\item[\textit{(ii)}] If $h\left(  f\right)  =1,$ the $A_{\infty}$-coalgebra
		structure on $H$ extends to a nontrivial $A_{\infty}$-bialgebra structure on
		$H.$ Furthermore, let $a\in A^{n}$ and $c\in A^{2n-1}$ be multiplicative
		generators; then there is a perturbed multiplication $\varphi$ on $H$ if and
		only if $Sq_{1}a=c,$ in which case $n=3,5,9$; otherwise $\varphi$ is induced
		by the shuffle product on $BA$.
	\end{enumerate}
\end{example}

\begin{proof}
	Suppose $h\left(  f\right)  =1.$ Then $A$ is generated multiplicatively by
	$a\in A^{n},$ $b\in A^{m},$ and $c\in A^{2n-1}$ subject to the relations
	$a^{2}=c^{2}=ac=ab^{2}=b^{2}c=b^{3}=0.$\
	Let $\alpha=\operatorname*{cls}\left\lfloor a\right\rfloor ,$ $\beta
	=\operatorname*{cls}\left\lfloor b\right\rfloor ,$ $\gamma=\operatorname*{cls}
	\left\lfloor c\right\rfloor ,$ and $z=\operatorname*{cls}\left\lfloor
	b^{2}\right\rfloor \in H=H^{\ast}\left(  BA\right)  .$ Given $x_{i}
	=\operatorname*{cls}\left\lfloor u_{i}\right\rfloor \in H$ with $u_{i}
	u_{i+1}=0,$ let $x_{1}|\cdots|x_{n}=\operatorname*{cls}\left\lfloor
	u_{1}|\cdots|u_{n}\right\rfloor .$ Note that $x=\alpha|z=z|\alpha$ and
	$y=\gamma|z=z|\gamma.$ Let $\Delta_{H}\ $denote the coproduct in $H$ induced
	by the cofree coproduct $\Delta$ in $BA.$ Then $x$ and $y$ are primitive, and
	$\Delta_{H}\left(  \alpha|z|\alpha\right)  =e\otimes\alpha|z|\alpha
	+x\otimes\alpha+\alpha\otimes x+\alpha|z|\alpha\otimes e.$ Define $g\left(
	x\right)  =\left\lfloor a|b^{2}\right\rfloor $ and $g^{2}\left(  x\right)
	=\left\lfloor a\right\rfloor \otimes\left\lfloor b|b\right\rfloor ;$ define
	$g\left(  \alpha|z|\alpha\right)  =\left\lfloor a|b^{2}|a\right\rfloor $ and
	$g^{2}\left(  \alpha|z|\alpha\right)  =\left\lfloor a\right\rfloor
	\otimes\left(  \left\lfloor a|b|b\right\rfloor +\left\lfloor
	b|a|b\right\rfloor +\left\lfloor b|b|a\right\rfloor \right)  .$ There is an
	induced $A_{\infty}$-coalgebra operation $\Delta_{H}^{3}:H\rightarrow
	H^{\otimes3},$ which vanishes except on elements of the form $\cdots
	|z|\cdots,$ and may be defined on the elements $x,$ $y,$ and $\alpha|z|\alpha$
	by
	$\Delta_{H}^{3}(x)=\alpha\otimes\beta\otimes\beta,\ \Delta_{H}^{3}(y)=\gamma\otimes\beta\otimes\beta,$ and
	$\Delta_{H}^{3}(\alpha|z|\alpha)=\alpha\otimes(\alpha|\beta+\beta
	|\alpha)\otimes\beta+\alpha\otimes\beta\otimes(\alpha|\beta+\beta|\alpha).$\
	Then $\left\{  \Delta_{H},\Delta_{H}^{3}\right\}  $ defines an $A_{\infty}
	$-coalgebra structure on $H$. Furthermore, if $Sq_{1}a=c,$ which can only
	occur when $n=3,5,9,$ the induced HGA structure on $A$ is determined by
	$Sq_{1}$, and induces a perturbation of the shuffle product $\mu:BA\otimes
	BA\rightarrow BA$ with $\mu\left(  \left\lfloor a\right\rfloor \otimes
	\left\lfloor a\right\rfloor \right)  =\left\lfloor c\right\rfloor .$ The
	product $\mu$ lifts to a perturbed product $\varphi$ on $H$ such that
	$\varphi(\alpha\otimes\alpha|z)=\alpha|z|\alpha+\gamma|z,$
	and the $A_{\infty}$-coalgebra structure $\left(  H,\Delta_{H},\Delta_{H}
	^{3}\right)  $ extends to an $A_{\infty}$-bialgebra structure $\left(
	H,\Delta_{H},\Delta_{H}^{3},\varphi\right)  $ as in Example \ref{hga}. On the
	other hand, if $Sq_{1}a=0,$ then is induced by the shuffle product on $BA$ and
	$\varphi(\alpha\otimes\alpha|z)=\alpha|z|\alpha.$ Conversely, if $h(f)=0, $
	then $b^{2}=0$ so that $\Delta_{H}^{k}=0,\ $for all $k\geq3.$
\end{proof}

We conclude with an investigation of the $A_{\infty}$-bialgebra structure on
the double cobar construction. To this end, we first prove a more general
fact, which follows our next definition:

\begin{definition}
	Let $\left(  A,d,\psi,\varphi\right)  $ be a free DG\ bialgebra, i.e., free as
	a DGA. An \textbf{acyclic cover of }$A$ is a collection of acyclic DG
	submodules
	$\mathcal{C}\left(  A\right) :=\{  C^{a}\subseteq A: a\text{ is a
		monomial of }A\}$
	such that $\psi\left(  C^{a}\right)  \subseteq C^{a}\otimes C^{a}$ \textit{and
	}$\varphi\left(  C^{a}\otimes C^{b}\right)  \subseteq C^{ab}.$
\end{definition}

\begin{proposition}
	\label{prop}Let $\left(  A,d,\psi,\varphi\right)  $ be a free DG\ bialgebra
	with acyclic cover $\mathcal{C}\left(  A\right)  .$
	
	\begin{enumerate}
		\item[\textit{(i)}] Then $\varphi\ $and $\psi$ extend to an $A_{\infty}
		$-bialgebra structure of the third kind.
		
		\item[\textit{(ii)}] Let $\left(  A^{\prime},d^{\prime},\psi^{\prime}
		,\varphi^{\prime}\right)  $ be a free DG\ bialgebra with acyclic cover
		$\mathcal{C}\left(  A^{\prime}\right)  ,$ and let $f:A\rightarrow A^{\prime}$
		be a DGA\ map such that $f(C^{a})\subseteq C^{f\left(  a\right)  }$ for all
		$C^{a}\in\mathcal{C}\left(  A\right)  .$ Then $f $ extends to a morphism of
		$A_{\infty}$-bialgebras.
	\end{enumerate}
\end{proposition}

\begin{proof}
	Define an $A_{\infty}$-coalgebra structure as follows: Let $\psi^{2}=\psi;$
	arbitrarily define $\psi^{3}$ on multiplicative generators, and extend
	$\psi^{3}$ to decomposables via
	$\psi^{3}\mu_{A}=\mu_{A}^{\otimes3}\sigma_{3,2}\left(  \psi^{3}\otimes\psi
	^{3}\right)  .\ $
	Inductively, if $\left\{  \psi^{i}\right\}  _{i<n}$ have been constructed,
	arbitrarily define $\psi^{n}$ on multiplicative generators, and use
	(\ref{multformula}) to extend $\psi^{n}$ to decomposables. Since each
	$\psi^{n}$ preserves $\mathcal{C}\left(  A\right)  $ by hypothesis, $\left\{
	\varphi,\psi^{2},\psi^{3},\ldots\right\}  $ is an $A_{\infty}$-bialgebra
	structure of the third kind, as desired. The proof of part (ii) is similar.
\end{proof}

Given a space $X,$ choose a base point $y\in X.$ Let $\operatorname{Sing}
^{2}X$ denote the Eilenberg 2-subcomplex of $\operatorname{Sing}X$, and let
$C_{\ast}(X)=C_{\ast}(\operatorname{Sing}^{2}X)/C_{>0}(\operatorname{Sing}
\,y).$ In \cite{SU2} we constructed an explicit (non-coassociative) coproduct
on the double cobar construction $\Omega^{2}C_{\ast}(X),$ which imposes
a DG bialgebra structure. Let $\mathbf{\Omega}{^{2}}$ denote the functor from the
category of (2-reduced) simplicial sets to the category of permutahedral sets
(\cite{SU2}, \cite{KS2}) such that $\Omega^{2}C_{\ast}(X)=
C_{\ast}^{\diamondsuit}(\mathbf{\Omega}{^{2}}\operatorname{Sing}^{2}X),$ where
$C_{\ast}^{\diamondsuit}\left(  Y\right)  \!=\!C_{\ast}\!\left(
\operatorname{Sing}^{M}Y\right)  \!/\!\left\langle \text{degeneracies}
\right\rangle \!$ and $\operatorname{Sing}^{M}Y$ is the multipermutahedral
singular complex of $Y$ (see Definition 15 in \cite{SU2}; cf. \cite{Baues1},\cite{CM}). Now consider the monoidal permutahedral set ${\mathbf{\Omega}^{2}}\operatorname{Sing}^{2}X,$ and let $V_{\ast} $ be its monoidal
(non-degenerate) generators. For each $a\in V_{n},$ let
$C^{a}=R\left\langle r\text{-faces of}\text{ }a: 0\leq r\leq n\right\rangle.$
Then $\left\{  C^{a}\right\}  $ is an acyclic cover, and by Proposition
\ref{prop} we have

\begin{theorem}
	\label{double}The DG bialgebra structure on the double cobar construction
	$\Omega^{2}C_{\ast}(X)$ extends to an $A_{\infty}$-bialgebra of the third kind.
\end{theorem}

\begin{conjecture}
	Given a $2$-connected space $X,$ the chain complex $C_{\ast}^{\diamondsuit
	}(\Omega^{2}X)$ admits an $A_{\infty}$-bialgebra structure extending the DG
	bialgebra structure constructed in \cite{SU2}. Moreover, there exists a
	morphism
	$G=\left\{  g_{m}^{n}\right\}  :\Omega^{2}C_{\ast}(X)\Rightarrow C_{\ast
	}^{\diamondsuit}(\Omega^{2}X)$
	of $A_{\infty}$-bialgebras such that $g_{1}^{1}$ is a homology isomorphism.\vspace{0.2in}
\end{conjecture}

\noindent\textbf{Acknowledgements.}
The second author wishes to thank the Institute of Mathematics of the Czech Academy of Sciences and the Institute of Mathematics of University of Seville (IMUS) for providing financial support and residential accommodations during various stages of this project.

\end{document}